%% file: main_2.tex
\documentclass[10pt]{article}
\usepackage[margin=1in]{geometry}

\usepackage{amsmath}
\usepackage{amssymb}
\usepackage{amsthm}
\usepackage{mathrsfs}
\usepackage{bm}
\usepackage{enumerate}
\usepackage{xcolor}         
\usepackage[colorlinks,linkcolor=black,citecolor=black,urlcolor=black]{hyperref}
\usepackage{graphicx}
\usepackage[sort, compress]{cite}

\input{preamble}

\begin{document}

\title{Dynamical mean-field analysis of adaptive Langevin diffusions:
Replica-symmetric fixed point and empirical Bayes}

\author{Zhou Fan\thanks{Department of Statistics and Data Science, Yale University}, Justin Ko\thanks{Department of Statistics and Actuarial Science, University of Waterloo}, Bruno Loureiro\thanks{Departement d'Informatique,  École Normale Supérieure, PSL \& CNRS}, Yue M.\ Lu\thanks{Departments of Electrical Engineering and Applied Mathematics, Harvard University}, Yandi Shen\thanks{Department of Statistics and Data Science, Carnegie Mellon University}}

\date{}

\maketitle

\begin{abstract}
In many applications of statistical estimation via sampling,
one may wish to sample from a high-dimensional target distribution
that is adaptively evolving to the samples already seen. We study an example
of such dynamics, given by a Langevin diffusion for posterior sampling in a
Bayesian linear regression model with i.i.d.\ regression design,
whose prior continuously adapts to the
Langevin trajectory via a maximum marginal-likelihood scheme. Results of
dynamical mean-field theory (DMFT) developed in our companion paper establish a
precise high-dimensional asymptotic limit for the joint evolution of the prior
parameter and law of the Langevin sample. In this work, we carry out an analysis
of the equations that describe this DMFT limit, under conditions of
approximate time-translation-invariance which include, in particular, settings
where the posterior law satisfies a log-Sobolev inequality. In such
settings, we show that this adaptive Langevin trajectory converges on a
dimension-independent time horizon to an equilibrium state that is
characterized by a system of scalar fixed-point equations, and the associated
prior parameter converges to a critical point of a replica-symmetric limit for
the model free energy. As a by-product of our analyses, we obtain a new
dynamical proof that this replica-symmetric limit for the free energy is exact, 
in models having a possibly misspecified prior and where a log-Sobolev
inequality holds for the posterior law.
\end{abstract}

\setcounter{tocdepth}{3}
\tableofcontents

\input{introduction}

\input{results}

\input{equilibrium}

\input{fixedprior}

\input{adaptiveprior}

\appendix

\input{globalconditions}

\input{dmftgaussian}

\input{appendix_LSI}

\input{auxiliary}

\bibliographystyle{unsrt}
\bibliography{eb_dmft}
\end{document}

%% file: preamble.tex
\def\R{\mathbb{R}}
\def\Z{\mathbb{Z}}
\def\P{\mathbb{P}}
\def\E{\mathbb{E}}

\def\cB{\mathcal{B}}

\def\N{\mathcal{N}}
\def\cF{\mathcal{F}}
\def\cG{\mathcal{G}}
\def\cM{\mathcal{M}}
\def\cP{\mathcal{P}}

\def\b{\mathbf{b}}
\def\e{\mathbf{e}}

\def\u{\mathbf{u}}
\def\v{\mathbf{v}}

\def\x{\mathbf{x}}
\def\y{\mathbf{y}}
\def\z{\mathbf{z}}

\def\D{\mathbf{D}}

\def\I{\mathbf{I}}

\def\V{\mathbf{V}}

\def\X{\mathbf{X}}

\def\bC{\mathbf{C}}

\def\bR{\mathbf{R}}
\def\sP{\mathsf{P}}
\def\sQ{\mathsf{Q}}

\def\btheta{\boldsymbol{\theta}}

\def\boldeta{\boldsymbol{\eta}}

\def\bvarphi{\boldsymbol{\varphi}}
\def\bLambda{\boldsymbol{\Lambda}}
\def\bSigma{\boldsymbol{\Sigma}}
\def\eps{\varepsilon}
\def\beps{\boldsymbol{\varepsilon}}
\def\d{\mathrm{d}}
\def\cA{\mathcal{A}}
\def\ball{\mathcal{B}}

\def\event{\mathcal{E}}
\def\LSI{\mathrm{LSI}}
\def\L{\mathrm{L}}
\def\1{\mathbf{1}}
\def\tti{\mathrm{tti}}
\def\init{\mathrm{init}}
\def\dist{\operatorname{dist}}

\newcommand{\floor}[1]{\left\lfloor #1 \right\rfloor}

\newcommand{\pnorm}[2]{\lVert #1\rVert_{#2}}

\DeclareMathOperator{\Var}{Var}
\DeclareMathOperator{\Cov}{Cov}
\DeclareMathOperator{\Ent}{Ent}
\DeclareMathOperator{\diag}{diag}
\DeclareMathOperator{\sign}{sign}
\DeclareMathOperator{\Tr}{Tr}
\DeclareMathOperator{\DKL}{D_{\mathrm{KL}}}

\DeclareMathOperator{\mse}{mse}
\DeclareMathOperator{\ymse}{ymse}
\DeclareMathOperator{\MSE}{MSE}

\DeclareMathOperator{\YMSE}{YMSE}
\DeclareMathOperator{\Crit}{Crit}
\DeclareMathOperator{\op}{op}

\def\op{\mathrm{op}}

\def\FI{\mathrm{FI}}
\def\grad{\mathrm{grad}}

\newtheorem{theorem}{Theorem}[section]
\newtheorem{lemma}[theorem]{Lemma}
\newtheorem{proposition}[theorem]{Proposition}
\newtheorem{corollary}[theorem]{Corollary}
\theoremstyle{definition}
\newtheorem{definition}[theorem]{Definition}
\newtheorem{assumption}[theorem]{Assumption}
\newtheorem{remark}[theorem]{Remark}
\newtheorem{example}[theorem]{Example}

\ExplSyntaxOn
\keys_define:nn { miguel/label }
{
	label   .tl_set:N = \l_miguel_label_tl,
	unknown .code:n   = \clist_put_right:Nx \l_miguel_label_clist
	{ \l_keys_key_tl = \exp_not:n { #1 } }
}
\clist_new:N \l_miguel_label_clist
\box_new:N \l_miguel_label_box
\box_new:N \l_miguel_label_image_box
\NewDocumentCommand{\xincludegraphics}{O{}m}
{
	\tl_clear:N \l_miguel_label_tl
	\clist_clear:N \l_miguel_label_clist
	\keys_set:nn { miguel/label } { #1 }
	\tl_if_empty:NTF \l_miguel_label_tl
	{
		\miguel_includegraphics:Vn \l_miguel_label_clist { #2 }
	}
	{
		\hbox_set:Nn \l_miguel_label_image_box
		{
			\miguel_includegraphics:Vn \l_miguel_label_clist { #2 }
		}
		\hbox_set:Nn \l_miguel_label_box
		{
			\skip_horizontal:n { -15pt }
			\fcolorbox{white}{white}{\footnotesize \tl_use:N \l_miguel_label_tl}
		}
		\leavevmode
		\box_use:N \l_miguel_label_image_box
		\skip_horizontal:n { -\box_wd:N \l_miguel_label_image_box }
		\hbox_overlap_right:n
		{
			\box_move_up:nn
			{
				\box_ht:N \l_miguel_label_image_box - 
				\box_ht:N \l_miguel_label_box - -3pt
			}
			{ \box_use_drop:N \l_miguel_label_box }
		}
		\skip_horizontal:n { \box_wd:N \l_miguel_label_image_box }
	}
}
\cs_new_protected:Nn \miguel_includegraphics:nn
{
	\includegraphics[#1]{#2}
}
\cs_generate_variant:Nn \miguel_includegraphics:nn { V }
\ExplSyntaxOff

%% file: introduction.tex
\section{Introduction}

Parameter estimation via Monte Carlo sampling is a common paradigm in statistical learning, arising for example in stochastic implementations of Expectation-Maximization estimation in latent variable models
\cite{wei1990monte,levine2001implementations}, and contrastive-divergence \cite{hinton2002training} and diffusion-based
learning \cite{sohl2015deep,song2019generative,ho2020denoising,song2020score}
of generative models for data. In these applications, one wishes to learn a
parameter using Monte Carlo
samples from an associated distribution on a high-dimensional space. Monte Carlo
methods whose target distribution continuously adapts to the learned
parameter are natural for such tasks, and we refer to
\cite{kuntz2023particle,akyildiz2023interacting,fan2023gradient,sharrock2024tuning,marion2024implicit}
for several recent proposals of this form.

The goal of our current work is to study the learning dynamics in a particular
(classical) instance of this paradigm, namely the estimation of the
distribution of regression coefficients in a high-dimensional regression model
\cite{mcculloch1994maximum,mcculloch1997maximum}.
We will focus on the linear model
\[\y=\X\btheta^*+\beps\]
with a latent and high-dimensional coefficient vector
$\btheta^* \in \R^d$, whose coordinates have an unknown ``prior'' distribution
$g_*$. Estimation of this prior distribution is a classical example
of empirical Bayes inference \cite{Robbins56,efron2012large}, and arises ubiquitously in genetic
association analyses where $g_*$ represents the distribution of
genetic effect sizes in linear mixed models for complex traits
\cite{meuwissen2001prediction,yang2011gcta,loh2015efficient,moser2015simultaneous,lloyd2019improved,ge2019polygenic,spence2022flexible,morgante2023flexible}. Two
recent works \cite{mukherjee2023mean,fan2023gradient} have established the statistical consistency of
nonparametric maximum marginal-likelihood estimators of $g_*$ 
in settings of high-dimensional
regression designs $\X \in \R^{n \times d}$, as $n,d \to \infty$.
However, direct computation of this maximum marginal-likelihood estimate is
intractable for general regression designs, motivating approaches based on
approximate posterior inference schemes.

We will investigate in this work a parametric analogue of a learning procedure proposed in \cite{fan2023gradient}, modeling the prior distribution via a parametric model $g(\,\cdot\,,\alpha)$
and applying an adaptive diffusion
to estimate the parameter
$\alpha \in \R^K$. This procedure will take the form of a Langevin diffusion 
\begin{equation}\label{eq:langevin_intro}
\d \btheta^t=\nabla_{\btheta} \log \sP_{g(\cdot,\widehat
\alpha^t)}(\btheta^t \mid \X,\y)\d t+\sqrt{2}\,\d \b^t
\end{equation}
for sampling from the posterior distribution
$\sP_{g(\cdot,\widehat \alpha^t)}(\btheta \mid \X,\y)$ of the regression
coefficients, under a prior law $g(\,\cdot\,,\widehat \alpha^t)$ whose
parameter evolves according to a coupled continuous-time dynamics
\begin{equation}\label{eq:alpha_intro}
\d\widehat \alpha^t=\cG\Big(\widehat\alpha^t,\frac{1}{d}\sum_{j=1}^d \delta_{\theta_j^t}\Big)\d
t.
\end{equation}
Here $\cG(\cdot)$ is a map that implements gradient-based maximum marginal-likelihood learning of $\alpha$ via the empirical distribution of coordinates of $\btheta^t$, and we defer a discussion of this motivation to Section~\ref{sec:results}.
The procedure may be understood as an approximation to an idealized dynamics
\begin{equation}\label{eq:alpha_intro_ideal}
\d\alpha^t=\cG(\alpha^t,\sP(\theta^t))\d t
\end{equation}
where $\sP(\theta^t)$ denotes the average law of the coordinates
$\theta_1^t,\ldots,\theta_d^t$. For these idealized dynamics, the
analyses of \cite{fan2023gradient} may be adapted to show that the
prior parameter $\alpha^t$
converges to a fixed point of the marginal log-likelihood,
under certain conditions for the noise and regression design.
Related results have also been shown recently
in more general latent variable models
in \cite{de2021efficient,kuntz2023particle,akyildiz2023interacting,marion2024implicit}, which, in addition, provide convergence guarantees for particle
approximations of the McKean-Vlasov type
$\d\alpha^t=\cG(\alpha^t,\frac{1}{dM}\sum_{j=1}^d \sum_{m=1}^M
\delta_{\theta_j^{m,t}})\d t$, having $M$ parallel sampling chains
$\{\btheta^{1,t}\}_{t \geq 0},\ldots,\{\btheta^{M,t}\}_{t \geq 0}$
for the latent variable $\btheta\in \R^d$,
in the limit $M \to \infty$.

The aforementioned results are not fully satisfactory in our context of a
high-dimensional regression model, and leave open the following two interesting
questions about the original dynamics
(\ref{eq:langevin_intro}--\ref{eq:alpha_intro}):
\begin{enumerate}
\item \emph{Single chain propagation-of-chaos.} In the limit of increasing
dimensions $d \to \infty$, are the idealized dynamics
(\ref{eq:alpha_intro_ideal}) well-approximated by (\ref{eq:alpha_intro}) using
just a single Langevin chain $\{\btheta^t\}_{t \geq 0}$ in $\R^d$?
\item \emph{Characterization of fixed points.} Can the fixed points
$\widehat\alpha$ of (\ref{eq:alpha_intro}) be explicitly characterized?
Does (\ref{eq:alpha_intro}) exhibit dimension-free
convergence to these fixed points, and in what settings is the fixed point
representing the maximum marginal-likelihood estimator of $\alpha$ unique?
\end{enumerate}
The purpose of our work is to provide answers to these questions in the context 
of an i.i.d.\ regression design.
Question 1 is addressed in our companion paper \cite{paper1}, which build upon
the recent results of \cite{celentano2021high,gerbelot2024rigorous} to formalize a dynamical mean-field theory
(DMFT) approximation of (\ref{eq:alpha_intro}) by (\ref{eq:alpha_intro_ideal})
over dimension-independent time horizons $t \in [0,T]$, for a general class of
such adaptive Langevin dynamics procedures. Our current paper addresses Question
2 by carrying out an analysis of the resulting DMFT system, under an assumption
of a uniform log-Sobolev inequality for the posterior law.

\subsection{Summary of results}

Our main results provide an analysis of the DMFT equations that approximate
the empirical Bayes Langevin dynamics
(\ref{eq:langevin_intro}--\ref{eq:alpha_intro}) in the high-dimensional limit
as $n,d \to \infty$ proportionally.
En route to this analysis, we obtain also new results for the DMFT
approximation of the standard non-adaptive Langevin diffusion
(\ref{eq:langevin_intro}) with a fixed prior
$g(\cdot) \equiv g(\,\cdot\,,\alpha)$. We summarize these results as follows:

\begin{enumerate}
\item In the setting of a non-adaptive Langevin diffusion,
we formalize a condition of \emph{approximate time-translation-invariance} (TTI) for
the DMFT system. We perform an analysis of the
dynamical fixed-point equations for the DMFT correlation and response functions
under this condition, and show that they
recover the static fixed-point equations for the free energy and
posterior mean-squared-error predicted by a replica-symmetric ansatz \cite{guo2005randomly,kabashima2008inference}.
\item We show that a log-Sobolev inequality (LSI) for the posterior law provides
a sufficient condition to guarantee the above approximate-TTI property
for the DMFT system, and we discuss several settings of log-concavity,
high noise, or large sample size where such an LSI holds.
As a consequence, we obtain a new dynamical proof of the
validity of the replica-symmetric predictions for the free energy and MSE in the Bayesian linear model with a possibly
misspecified prior law, under such an LSI condition.
\item When the LSI holds
uniformly over the posterior laws corresponding to the deterministic DMFT
trajectory of $\{\alpha^t\}_{t \geq 0}$, we show that the empirical Bayes
estimate $\widehat\alpha^t$ converges on a dimension-free time horizon to a
critical point $\alpha^\infty$ of the replica-symmetric limit for the free
energy. This is explicitly characterized by a system of scalar fixed-point
equations, and we discuss examples of models where this critical point
may or may not be unique.
\end{enumerate}

We present and discuss these results and examples
in further detail in Section \ref{sec:results}.

\subsection{Further related literature}
Approximating the dynamical behavior of many degrees-of-freedom by an effective single-particle problem interacting self-consistently with its environment is an old idea in the statistical physics literature. Relevant to our work is the development of this idea in the context of disordered systems, and in particular the study of high-dimensional Langevin dynamics of soft-spin variants of the
Sherrington-Kirkpatrick model
\cite{sompolinsky1981dynamic,sompolinsky1982relaxational} and the spherical
p-spin model \cite{kirkpatrick1987dynamics, crisanti1993spherical, cugliandolo1993analytical, cugliandolo1994out}.
Mathematical proofs of these approximations were first shown for such
models in the works of
\cite{arous1995large,arous1997symmetric,guionnet1997averaged} using large deviations techniques,
and more recently in generalized linear models close to our setting
by \cite{celentano2021high,gerbelot2024rigorous}
using different methods around Approximate Message Passing algorithms and iterative
Gaussian conditioning. In recent years, DMFT analyses have been applied to
study Langevin dynamics and gradient-based optimization in
many statistical models and applications, including
Gaussian mixture classification \cite{mignacco2020dynamical}, 
matrix and tensor PCA \cite{mannelli2019passed,sarao2020marvels,liang2023high},
phase retrieval and generalized linear models
\cite{mignacco2021stochasticity,han2024gradient}, and learning in perceptron
and neural network models \cite{agoritsas2018out,bordelon2022self,sclocchi2022high, kamali2023stochastic, kamali2023dynamical, dandi2024benefits,bordelon2024infinite,bordelon2024dynamical,montanari2025dynamical}. These
analyses have uncovered surprising phenomena about the efficacy of
gradient-based methods and relationships to landscape complexity for
high-dimensional non-convex problems \cite{mannelli2019passed}.

Understanding the long-time behavior of DMFT systems, in particular in
low-temperature regimes characterized by
aging or metastability, has been a primary goal in
both the physics and mathematics literature since the
original inception of these methods (see \cite{altieri2020dynamical, Crisanti2023} and references within for a review). Mathematically rigorous analyses of
long-time dynamics have been obtained previously for spherical 2-spin models in
\cite{arous2001aging} and related statistical models in
\cite{bodin2021rank,liang2023high} by leveraging the rotational invariance of
these models. However, such
analyses of DMFT are (to our knowledge) quite rare in more general settings.
Our work takes a step towards filling this gap, by providing a rigorous
analysis of the DMFT approximation to Langevin dynamics in a more general model
without a rotationally invariant prior, in settings where approximate-TTI holds.

As a by-product of our analyses, we obtain a new proof of a replica formula
\cite{tanaka2002statistical} for the free energy and posterior MSE in the
Bayesian linear model. This proof is different
from several existing proofs of this result
\cite{montanari2006analysis,reeves2019replica,barbier2020mutual,barbier2019adaptive,barbier2019optimal}
and from the Gaussian interpolation methods of Guerra-Talagrand
\cite{guerra2003broken,talagrand2006parisi}, and is based instead on deducing a
static fixed-point equation from the dynamical fixed-point equations of DMFT. Our current result
is specific to a high-temperature regime where a LSI holds for the posterior
law, but it applies to models where the prior law is misspecified
\cite{guo2005randomly,kabashima2008inference,takahashi2022macroscopic}. In this
misspecified context, the closest mathematical result of which we are aware is
\cite{barbier2021performance} which proved the replica-symmetric predictions
in a setting where the posterior is log-concave. A complete large
deviations analysis of the free energy in a related rank-one matrix
estimation model with misspecified prior and noise was carried out in
\cite{guionnet2025estimating}, showing that in
general the asymptotic free energy is characterized by a Parisi-type
variational problem whose solution may not be replica-symmetric. Our results imply for the linear model that this solution must be replica-symmetric under our assumed condition of a LSI for the posterior law.

In the context of adaptive empirical Bayes Langevin dynamics, our results
complement the previous analyses of \cite{fan2023gradient} for more general
regression designs, and of \cite{kuntz2023particle,marion2024implicit} in 
general latent variable models. We deduce a dimension-free convergence rate,
in contrast to the results of \cite{fan2023gradient} that established
convergence (for a nonparametric variant of this algorithm) 
on a time horizon growing linearly with $n,d$, and without employing a
time-dependent and decaying learning rate as in \cite{marion2024implicit}.
Under the additional mean-field structure of our current model, we are
able to establish convergence of a single-chain implementation of the
empirical Bayes Langevin dynamics using (\ref{eq:alpha_intro}), rather than
for an idealized dynamics as studied in
\cite{fan2023gradient,marion2024implicit} or for an implementation using $M$
parallel chains as studied in \cite{kuntz2023particle}. We are also able to
give an explicit characterization and analysis of the fixed points to which
the dynamics of $\{\widehat\alpha^t\}_{t \geq 0}$ may converge.

\subsubsection*{Acknowledgments}

This research was initiated during the ``Huddle on Learning and Inference from
Structured Data'' at ICTP Trieste in 2023. We'd like to thank the huddle
organiers Jean Barbier, Manuel S\'aenz, Subhabrata Sen, and Pragya Sur for
their hospitality and many helpful discussions. We are very grateful to
Andrea Montanari who suggested to us the approach to prove
Theorem \ref{thm:dmft_equilibrium}. We'd like to thank also Pierfrancesco Urbani, Francesca Mignacco and Emanuele Troiani for useful discussions. Z. Fan was supported in part by NSF DMS--2142476 and a Sloan Research Fellowship. J. Ko was supported in part by the Natural Sciences and Engineering Research Council of Canada (NSERC), the Canada Research Chairs programme, and the Ontario Research Fund [RGPIN-2020-04597, DGECR-2020-00199]. B. Loureiro was supported by the French government, managed by the National Research Agency (ANR), under the France 2030 program with the reference ``ANR-23-IACL-0008'' and the Choose France - CNRS AI Rising Talents program. Y. M. Lu was supported in part by the Harvard FAS Dean's Competitive Fund for Promising Scholarship and by a Harvard College Professorship.

\subsubsection*{Notational conventions}

In the context of the posterior law $\sP_g(\btheta \mid \X,\y)$ for a given
prior $g(\cdot)$, we will write 
\[\langle f(\btheta) \rangle=\E[f(\btheta) \mid \X,\y]\]
for the posterior expectation conditioning on the ``quenched'' variables
$\X,\y$. In the context of Langevin dynamics, we will write
similarly
\[\langle f(\btheta^t) \rangle=\E[f(\btheta^t) \mid \X,\y]\]
also for an expectation conditioning on $\X,\y$.
In some arguments it is convenient to consider the expectation also conditioned
on the initial condition $\btheta^0$, and we will denote this by
\[\langle f(\btheta^t) \rangle_\x=\E[f(\btheta^t) \mid \X,\y,\,\btheta^0=\x].\]
We reserve $\E$ and $\P$ for the full expectation and probability also over
$\X,\y,\btheta^*,\beps$.

Constants $C,C',c,c'>0$ throughout are independent of the dimensions $n,d$.
For any random variable $\xi$ in a complete and separable normed vector space $(\cM,\|\cdot\|)$,
we will use $\sP(\xi)$ to denote its law. $\cP_2(\cM)$ is the space of
probability distributions $\sP$ on $(\cM,\|\cdot\|)$ such that
$\E_{\xi \sim \sP} \|\xi\|^2<\infty$, and
$W_1(\cdot)$ and $W_2(\cdot)$ denote
the Wasserstein-1 and Wasserstein-2 metrics on $\cP_2(\cM)$.

For $f:\R^d \to \R$, $\nabla f \in \R^d$ is its gradient, $\nabla^2 f \in \R^{d \times d}$ its Hessian, and $\nabla^3 f \in \R^{d \times d \times d}$ the symmetric tensor of its $3^\text{rd}$-order partial derivatives. For $f:\R \times \R^K \to \R$,
$\partial_\theta f(\theta,\alpha)$ and $\nabla_\alpha f(\theta,\alpha)$ denote its partial derivatives 
with respect to $\theta \in \R$ and $\alpha \in \R^K$. $\|\cdot\|_2$ is the Euclidean norm for vectors and vectorized Euclidean norm
for matrices and tensors. $\Tr M$ and $\|M\|_\op$ are the matrix
trace and Euclidean operator norm.
$C([0,T],\R^d)$ is the space of continuous functions
$f:[0,T] \to \R^d$ equipped with the norm of uniform convergence
$\|f\|_\infty=\sup_{t \in [0,T]} \|f(t)\|_2$. 
$C^k(\R^d,\R^m)$ is the space of functions $f:\R^d \to \R^m$
that are $k$-times continuously-differentiable. For two probability densities $p,q$ on $\R^d$,
$\DKL(p\|q)=\int q(\log q-\log p)$ is the Kullback-Leibler divergence. For a
scalar random variable $X$, $\Var X=\E X^2-(\E X)^2$ and
$\Ent X=\E X \log X-\E X \log \E X$. For a random vector $X \in \R^k$, $\Cov X=\E XX^\top-(\E X)(\E X)^\top \in \R^{k \times k}$.

%% file: results.tex
\section{Model and main results}\label{sec:results}

\subsection{Bayesian linear model and adaptive Langevin dynamics}

We study a linear model
\begin{equation}\label{eq:linearmodel}
\y=\X\btheta^* + \beps \in \R^n
\end{equation}
with random effects $\btheta^* \in \R^d$.
Modeling $\theta_1^*,\ldots,\theta_d^* \overset{iid}{\sim}
g$ for a prior density $g(\cdot)$ on the real line and modeling
$\beps \sim \N(0,\sigma^2\I)$ as Gaussian noise,
Bayesian inference for $\btheta^*$ is based upon the posterior density
\begin{equation}\label{eq:posterior}
\sP_g(\btheta \mid \X,\y)
=\frac{1}{\sP_g(\y \mid \X)}
\frac{1}{(2\pi\sigma^2)^{n/2}}
\exp\bigg({-}\frac{1}{2\sigma^2}\|\y-\X\btheta\|_2^2\bigg)
\prod_{j=1}^d g(\theta_j).
\end{equation}
Here $\sP_g(\y \mid \X)$ is the marginal likelihood of $\y$ (i.e.\ model
evidence or partition function), given by
\begin{equation}\label{eq:marginallikelihood}
\sP_g(\y \mid \X)=\int \frac{1}{(2\pi\sigma^2)^{n/2}}
\exp\bigg({-}\frac{1}{2\sigma^2}\|\y-\X\btheta\|_2^2\bigg)
\prod_{j=1}^d g(\theta_j)\d \theta_j.
\end{equation}

We will study (overdamped) Langevin dynamics for sampling from the posterior
density (\ref{eq:posterior}) in two settings, the first in which the
prior law $g(\cdot)$ is fixed but may be misspecified, and the second in which
this prior law may adapt to the Langevin
trajectory to implement empirical Bayes learning from the observed
data $(\X,\y)$. In the former setting, we consider the Langevin dynamics
\begin{equation}\label{eq:langevinfixedprior}
\d\btheta^t=\nabla_{\btheta}\bigg({-}\frac{1}{2\sigma^2}\|\y-\X\btheta^t\|_2^2
+\sum_{j=1}^d \log g(\theta_j^t)\bigg)\d t+\sqrt{2}\,\d\b^t
\end{equation}
where $\{\b^t\}_{t \geq 0}$ is a standard Brownian motion on $\R^d$. In the
latter setting, we will model the prior via a parametric model
\begin{equation}\label{eq:gmodel}
\big\{g(\;\cdot\;,\alpha):\alpha \in \R^K\big\}
\end{equation}
and consider the empirical Bayes Langevin dynamics
\begin{align}
\d\btheta^t&=\nabla_{\btheta}\bigg({-}\frac{1}{2\sigma^2}\|\y-\X\btheta^t\|_2^2
+\sum_{j=1}^d \log g(\theta_j^t,\widehat\alpha^t)\bigg)\d t+\sqrt{2}\,\d\b^t
\label{eq:langevin_sde}\\
\d\widehat\alpha^t&=\nabla_{\alpha} \bigg(\frac{1}{d}\sum_{j=1}^d \log
g(\theta_j^t,\widehat \alpha^t)- R(\widehat \alpha^t)\bigg)
\d t.\label{eq:gflow}
\end{align}
The equation (\ref{eq:gflow}) describes a continuous-time evolution of the
prior parameter $\alpha \in \R^K$ that is coupled to the Langevin diffusion
(\ref{eq:langevin_sde}) of the posterior sample,
and $R:\R^K \to \R$ is a possible smooth regularizer.
(In this work, we will be interested mostly in the behavior of
these dynamics when $R(\alpha) \equiv 0$, and we introduce $R(\alpha)$ 
for theoretical purposes to confine the dynamics of $\widehat\alpha^t$ in
certain examples.)

To motivate the dynamics (\ref{eq:langevin_sde}--\ref{eq:gflow}) as a procedure
that implements maximum marginal-likelihood learning of $\alpha \in \R^K$, we
may consider the free energy (i.e.\ negative marginal log-likelihood)
\begin{equation}\label{eq:hatF}
\widehat F(\alpha)={-}\frac{1}{d}\log \sP_{g(\cdot,\alpha)}(\y \mid \X)
\end{equation}
as a function of the prior parameter $\alpha \in \R^K$. By the Gibbs
variational principle (c.f.\ \cite[Proposition 4.7]{polyanskiy2025information}),
\begin{equation}\label{eq:gibbsvariational}
\widehat F(\alpha)=\inf_{q \in \cP_*(\R^d)} V(q,\alpha)
\end{equation}
where $\cP_*(\R^d)$ is the space of all probability densities on $\R^d$, and
\begin{equation}\label{eq:gibbsenergy}
V(q,\alpha)=\frac{1}{d}\int \Big(\frac{1}{2\sigma^2}\|\y-\X\btheta\|^2
-\sum_{j=1}^d \log g(\theta_j,\alpha)+\log q(\btheta)\Big)q(\btheta)\d\btheta
+\frac{n}{2d}\log 2\pi\sigma^2
\end{equation}
is the Gibbs free energy corresponding to the prior
$g(\cdot)=g(\,\cdot\,,\alpha)$. We propose to implement
maximum-likelihood learning of
$\alpha \in \R^K$ by minimizing the regularized Gibbs free energy
$V(q,\alpha)+R(\alpha)$ jointly over $(q,\alpha)$, via a gradient flow in the
Wasserstein-2 geometry for $q \in \cP_*(\R^d)$ and the standard Euclidean
geometry for $\alpha \in \R^K$. The resulting gradient flow equations
take the form
\begin{align}
\frac{\d}{\d t}q_t&={-}d \cdot \grad_q^{W_2} V(q_t,\alpha^t)
:=\nabla_{\btheta} \cdot \bigg[q_t(\btheta)
\nabla_{\btheta}\bigg(\frac{1}{2\sigma^2}\|\y-\X\btheta\|_2^2
-\sum_{j=1}^d \log g(\theta_j,\alpha^t)\bigg)\bigg]+\Tr \nabla_{\btheta}^2
q_t(\btheta),\label{eq:ideal_qflow}\\
\frac{\d}{\d t}\alpha^t&=
{-}\nabla_\alpha[V(q_t,\alpha^t)+R(\alpha^t)]
=\nabla_\alpha\bigg(\int \frac{1}{d}\sum_{j=1}^d \log
g(\theta_j,\alpha^t)\,q_t(\btheta)\d\btheta-R(\alpha^t)\bigg).\label{eq:ideal_alphaflow}
\end{align}
In (\ref{eq:ideal_qflow}), we identify $\grad_q^{W_2} V(q,\alpha)$ with
the Fokker-Planck equation for the density evolution of $\btheta^t$
under the Langevin diffusion (\ref{eq:langevinfixedprior}) with prior law
$g(\cdot)=g(\,\cdot\,,\alpha)$, via its variational interpretation put forth
in \cite{jordan1998variational}. Then (\ref{eq:langevin_sde}--\ref{eq:gflow})
may be understood as a particle implementation of
(\ref{eq:ideal_qflow}--\ref{eq:ideal_alphaflow}) that
uses a single Langevin trajectory $\{\btheta^t\}_{t \geq 0}$ to simulate
the dynamics of $q_t$ in (\ref{eq:ideal_qflow}), and that uses the
empirical distribution
$\frac{1}{d}\sum_{j=1}^d \delta_{\theta_j^t}$ to approximate
the expectation over $\btheta \sim q_t$ in the dynamics of $\alpha^t$ in
(\ref{eq:ideal_alphaflow}). An algorithm similar to
(\ref{eq:langevin_sde}--\ref{eq:gflow}) was introduced in
\cite{fan2023gradient}, with some additional reparametrization ideas to allow
for nonparametric modeling of the prior $g(\cdot)$. Here,
to simplify technical considerations, we restrict our study to parametric
prior models of the form (\ref{eq:gmodel}).

\subsection{DMFT equations}

The empirical Bayes Langevin diffusion (\ref{eq:langevin_sde}--\ref{eq:gflow})
is an example of a general class of adaptive Langevin diffusions that we
study in our companion work \cite{paper1}. In particular, the gradient
equation (\ref{eq:gflow}) for $\widehat\alpha^t$ is a function of the
empirical distribution of coordinates $\btheta^t$,
\[\d\widehat\alpha^t=\cG\Big(\widehat\alpha^t,\frac{1}{d}\sum_{j=1}^d
\delta_{\theta_j^t}\Big)\d t\]
where $\cG:\R^K \times \cP_2(\R) \to \R^K$ is the gradient map
\[\cG(\alpha,\sP)=\E_{\theta \sim \sP}[\nabla_\alpha \log g(\theta,\alpha)]
-\nabla R(\alpha).\]
Our analyses will rely on a system of dynamical mean-field theory (DMFT)
equations, formalized in \cite{paper1} and building upon the results of
\cite{celentano2021high,gerbelot2024rigorous}, that describes a
deterministic evolution of a prior parameter $\alpha^t \in \R^K$ and a
univariate law $\sP(\theta^t) \in \cP_2(\R)$ that approximate
$(\widehat\alpha^t,\frac{1}{d}\sum_{j=1}^d
\delta_{\theta^t_j})$ in the large system limit $n,d \to \infty$.
This approximation will hold under the following assumptions, which
we assume throughout this work.

\begin{assumption}[Linear model and initial conditions]\label{assump:model}
\phantom{}
\begin{enumerate}[(a)]
\item (Asymptotic scaling)
$\lim_{n,d \to \infty} \frac{n}{d}=\delta \in (0,\infty)$.
\item (Random design)
$\X=(x_{ij}) \in \R^{n \times d}$ has independent entries satisfying $\E[x_{ij}]=0$,
$\E[x_{ij}^2]=\frac{1}{d}$, and $\|\sqrt{d}x_{ij}\|_{\psi_2} \leq C$ for a constant $C>0$, where $\|\cdot\|_{\psi_2}$ is the sub-Gaussian norm.
\item (Bayesian linear model) $\btheta^*,\beps$ are independent of each other and of $\X$, and $\y=\X\btheta^*+\beps$.
The entries of $\btheta^*,\beps$ are distributed as
\begin{equation}\label{eq:thetaepsdistribution}
\theta_1^*,\ldots,\theta_d^* \overset{iid}{\sim} g_*,
\qquad \eps_1,\ldots,\eps_n \overset{iid}{\sim} \N(0,\sigma^2)
\end{equation}
for some $\sigma^2>0$ and probability density $g_*$ 
(both fixed and independent of $n,d$), where $g_*$ satisfies the log-Sobolev inequality
\begin{equation}\label{eq:LSIprior}
\Ent_{\theta^* \sim g_*}[f(\theta^*)^2] \leq C_\LSI\,\,
\E_{\theta^* \sim g_*} [(f'(\theta^*))^2]
\text{ for all } f \in C^1(\R).
\end{equation}
\item (Initial conditions) The initialization $\btheta^0$ is independent of $\X,\btheta^*,\beps$, and
\begin{equation}\label{eq:theta0distribution}
\theta^0_1,\ldots,\theta^0_d \overset{iid}{\sim} g_0
\end{equation}
for some probability density $g_0$ (fixed and independent of $n,d$) with finite entropy $\int g_0\log g_0$ and finite moment-generating-function in a neighborhood of 0. The initialization
$\widehat \alpha^0$ satisfies
$\lim_{n,d \to \infty} \widehat \alpha^0=\alpha^0$ a.s.\ for a
deterministic parameter $\alpha^0 \in \R^K$.
\end{enumerate}
\end{assumption}

\begin{assumption}[Prior model and regularizer]\label{assump:prior}
\phantom{}
\begin{enumerate}[(a)]
\item In the context of a fixed prior, $g(\theta)$ is strictly positive and
thrice continuously-differentiable, and
$(\log g)'''(\theta)$ is uniformly H\"older continuous over $\theta \in \R$.
For a constant $C>0$ and all $\theta \in \R$,
\[|(\log g)'(\theta)| \leq C(1+|\theta|),
\qquad |(\log g)''(\theta)|\leq C,
\qquad |(\log g)'''(\theta)|\leq C,\]
and for some constants $r_0,c_0>0$,
\[{-}(\log g)''(\theta) \geq c_0 \text{ for all } |\theta| \geq
r_0.\]
\item In the context of an adaptive prior, $g(\theta,\alpha)$ is strictly positive, $R(\alpha)$ is nonnegative, and both are
thrice continuously-differentiable.
For a constant $C>0$ and all $(\theta,\alpha) \in \R \times \R^K$,
\begin{equation}\label{eq:logggradientbound}
\begin{aligned}
\|\nabla_{(\theta,\alpha)} \log g(\theta,\alpha)\|_2 &\leq
C(1+|\theta|+\|\alpha\|_2),\\
\|\nabla R(\alpha)\|_2 &\leq C(1+\|\alpha\|_2).
\end{aligned}
\end{equation}
Furthermore, for each compact subset $S \subset \R^K$,
$\theta \mapsto \nabla_{(\theta,\alpha)}^3 \log g(\theta,\alpha)$ is H\"older-continuous uniformly over $(\theta,\alpha) \in \R \times S$, and for some constants
$C(S),r_0(S),c_0(S)>0$,
\begin{equation}\label{eq:localconditions}
\begin{gathered}
\|\nabla_{(\theta,\alpha)}^2 \log g(\theta,\alpha)\|_2 \leq C(S),
\;
\|\nabla_{(\theta,\alpha)}^3 \log g(\theta,\alpha)\|_2 \leq C(S)
\text{ for all } (\theta,\alpha) \in \R \times S,\\
{-}\partial_\theta^2 \log g(\theta,\alpha) \geq c_0(S)
\text{ for all } |\theta| \geq r_0(S) \text{ and } \alpha \in S,\\
\|\nabla^2 R(\alpha)\|_2 \leq C(S),\;
\|\nabla^3 R(\alpha)\|_2 \leq C(S) \text{ for all } \alpha \in S.
\end{gathered}
\end{equation}
\end{enumerate}
\end{assumption}

In particular, Assumption \ref{assump:prior}(b) requires that
$g(\cdot)\equiv g(\,\cdot\,,\alpha)$ satisfies Assumption \ref{assump:prior}(a) for each fixed
prior parameter $\alpha \in \R^K$.
We assume the LSI condition
(\ref{eq:LSIprior}) for the true prior $g_*$ to ensure concentration of the free
energy (c.f.\ Proposition \ref{prop:mseconcentration}), and we clarify that the
conditions of Assumption \ref{assump:prior}(a) imply that the modeled prior
$g(\cdot)$ must also satisfy an LSI of the form (\ref{eq:LSIprior}), as 
reviewed in Lemma \ref{lemma:univariateLSI}.

Under the above Assumptions \ref{assump:model} and \ref{assump:prior}(b), the
DMFT limit for (\ref{eq:langevin_sde}--\ref{eq:gflow})
is described by the following construction: Let
\begin{equation}\label{eq:DMFTscalars}
\theta^* \sim g_*, \qquad \theta^0 \sim g_0,
\qquad \eps \sim \N(0,\sigma^2)
\end{equation}
denote independent scalar variables with the distributions
(\ref{eq:thetaepsdistribution}) and (\ref{eq:theta0distribution}), and let
$\delta=\lim \frac{n}{d}$
be as in Assumption \ref{assump:model}. Let
$\{b^t\}_{t \geq 0}$, $\{u^t\}_{t \geq 0}$, and $(w^*,\{w^t\}_{t \geq 0})$
be centered univariate Gaussian processes independent of each other and of
$(\theta^*,\theta^0,\eps)$, where $\{b^t\}_{t \geq 0}$ is a standard Brownian
motion on $\R$, and $\{u^t\}_{t \geq 0}$ and $(w^*,\{w^t\}_{t \geq 0})$
have covariance kernels
\begin{equation}\label{def:dmft_covuw}
\E[u^t u^s]=C_\eta(t,s),
\qquad \E[w^t w^s]=C_\theta(t,s),
\qquad \E[w^t w^*]=C_\theta(t,*),
\qquad \E[(w^*)^2]=C_\theta(*,*)
\end{equation}
defined self-consistently in (\ref{def:covarianceresponse}) below.
We consider a system of stochastic differential equations
\begin{align}
\d\theta^t&=\Big[{-}\frac{\delta}{\sigma^2}(\theta^t-\theta^*) +
\partial_\theta \log g(\theta^t,\alpha^t)
+\int_0^t R_\eta(t, s)(\theta^{s} - \theta^\ast)\d s + u^t\Big]\d t + \sqrt{2}\,\d b^t\label{def:dmft_langevin_cont_theta}\\
\d\Big(\frac{\partial\theta^t}{\partial u^s}\Big) &=
\bigg[{-}\bigg(\frac{\delta}{\sigma^2} - \partial_\theta^2 \log
g(\theta^t,\alpha^t)\bigg)\frac{\partial \theta^{t}}{\partial u^s} + \int_s^{t}
R_\eta(t,s')\frac{\partial \theta^{s'}}{\partial u^s}\d s'\bigg]\d t
\label{def:response_theta}
\end{align}
for univariate processes $\{\theta^t\}_{t \geq 0}$ and
$\{\frac{\partial \theta^t}{\partial u^s}\}_{t \geq s \geq 0}$
adapted to the filtration
$\cF_t^\theta:=\cF(\{b^s\}_{s \leq t},\{u^s\}_{s \leq t},\theta^*,\theta^0)$,
with the initial conditions
\[\theta^t|_{t=0}=\theta^0,
\qquad \frac{\partial\theta^t}{\partial u^s}\bigg|_{t=s}=1.\]
These are driven by a deterministic and continuous
$\R^K$-valued process $\{\alpha^t\}_{t \geq 0}$ representing the asymptotic limit of $\{\widehat\alpha^t\}_{t \geq 0}$.
We consider likewise univariate processes $\{\eta^t\}_{t \geq 0}$
and $\{\frac{\partial \eta^t}{\partial w^s}\}_{t \geq s \geq 0}$ defined by
\begin{align}
\eta^t &= -\frac{1}{\sigma^2}\int_0^t R_\theta(t,s)\big(\eta^s + w^* -
\eps\big)\d s - w^t\label{def:dmft_langevin_cont_eta}\\
\frac{\partial \eta^t}{\partial w^s} &= {-}\frac{1}{\sigma^2}\Big[\int_{s}^{t}
R_\theta(t,s')\frac{\partial \eta^{s'}}{\partial w^s}\d s'-R_\theta(t,s)\Big]\label{def:response_g}
\end{align}
adapted to the filtration $\cF_t^\eta:=\cF(\{w^s\}_{s \leq t},w^*,\eps)$.
The deterministic process $\{\alpha^t\}_{t \geq 0}$ above is defined self-consistently by
\begin{equation}
\frac{\d}{\d t}\alpha^t = \cG(\alpha^t,\sP(\theta^t)),
\qquad \cG(\alpha,\sP)=\E_{\theta \sim \sP}[\nabla_\alpha \log g(\theta,\alpha)]-\nabla R(\alpha)
\label{def:dmft_langevin_alpha}
\end{equation}
with initial condition
$\alpha^t|_{t=0}=\alpha^0$
given in Assumption \ref{assump:prior},
where $\sP(\theta^t)$ is the law of $\theta^t$ in (\ref{def:dmft_langevin_cont_theta}).
The covariance and response functions $C_\theta,C_\eta,R_\theta,R_\eta$
are also defined for all $t \geq s \geq 0$ self-consistently via the above processes by
\begin{equation}\label{def:covarianceresponse}
\begin{gathered}
C_\theta(t,s)=\E[\theta^t \theta^s],
\quad C_\theta(t,*)=\E[\theta^t \theta^*],
\quad C_\theta(*,*)=\E[(\theta^*)^2],\\
C_\eta(t,s)=\frac{\delta}{\sigma^4}\E[(\eta^t + w^*-\eps)(\eta^s + w^*-\eps)]\\
R_\theta(t,s)=\E\Big[\frac{\partial \theta^t}{\partial u^s}\Big],
\quad R_\eta(t,s)=\frac{\delta}{\sigma^2}\E\Big[
\frac{\partial \eta^t}{\partial w^s}\Big].
\end{gathered}
\end{equation}

This DMFT system (\ref{def:dmft_covuw}--\ref{def:covarianceresponse})
describes the $n,d \to \infty$ limit of the empirical Bayes Langevin dynamics 
(\ref{eq:langevin_sde}--\ref{eq:gflow}).
In the setting of a fixed prior $g(\cdot) \equiv g(\,\cdot\,,\alpha^0)$,
the DMFT limit for the standard Langevin diffusion (\ref{eq:langevinfixedprior})
is the same, upon replacing $\cG(\alpha,\sP)$ in
(\ref{def:dmft_langevin_alpha}) by $\cG(\alpha,\sP)=0$ so that
$\alpha^t=\alpha^0$ for all $t \geq 0$.

Fixing any time horizon $T>0$, let us set
\[\eta^*={-}w^*\]
and denote by
\[\sP(\theta^*,\{\theta^t\}_{t \in [0,T]}) \in \cP_2(\R \times C([0,T],\R)),
\quad \sP(\eta^*,\eps,\{\eta^t\}_{t \in [0,T]}) \in \cP_2(\R \times \R \times
C([0,T],\R))\]
the joint laws of sample paths $(\theta^*,\{\theta^t\}_{t \in [0,T]})$
and $(\eta^*,\eps,\{\eta^t\}_{t \in [0,T]})$ in this DMFT system.
We write $\theta_j^*,\theta_j^t,\eps_i,\eta_i^*,\eta_i^t$ for the coordinates of
$\btheta^*,\btheta^t,\beps,\boldeta^*=\X\btheta^*,\boldeta^t=\X\btheta^t$, and
$\overset{W_2}{\to}$ for Wasserstein-2 convergence in the spaces
$\cP_2(\R \times C([0,T],\R))$ and $\cP_2(\R \times \R \times C([0,T],\R))$
as $n,d \to \infty$. The main result we will use from our companion work
\cite{paper1} is summarized in the following theorem.

\begin{theorem}\label{thm:dmft_approx}
\begin{enumerate}[(a)]
\item Suppose Assumptions \ref{assump:model} and \ref{assump:prior}(a)
hold, and identify $g(\cdot) \equiv g(\,\cdot\,,\alpha^0)$.
Let $\{\btheta^t\}_{t \geq 0}$ be the solution to the dynamics
(\ref{eq:langevinfixedprior}) with fixed prior $g(\cdot)$,
and denote $\boldeta^*=\X\btheta^*$ and $\boldeta^t=\X\btheta^t$.
Then for each fixed $T \geq 0$, there exists a solution up to time $T$ of the DMFT system
(\ref{def:dmft_covuw}--\ref{def:covarianceresponse}) with 
(\ref{def:dmft_langevin_alpha}) replaced by $\cG(\alpha,\sP)=0$,
such that almost surely as $n,d \to \infty$,
\begin{equation}\label{eq:dmftW2convergence}
\frac{1}{d}\sum_{j=1}^d \delta_{\theta_j^*,\{\theta_j^t\}_{t \in [0,T]}}
\overset{W_2}{\to} \sP(\theta^*,\{\theta^t\}_{t \in [0,T]}),
\quad \frac{1}{n}\sum_{i=1}^n \delta_{\eta_i^*,\eps_i,\{\eta_i^t\}_{t \in [0,T]}}
\overset{W_2}{\to} \sP(\eta^*,\eps,\{\eta^t\}_{t \in [0,T]}).
\end{equation}

\item Suppose Assumptions \ref{assump:model} and \ref{assump:prior}(b) hold, 
and let $\{\btheta^t,\widehat\alpha^t\}_{t \geq 0}$ be the solution to the empirical
Bayes Langevin dynamics
(\ref{eq:langevin_sde}--\ref{eq:gflow}).
Then for each fixed $T>0$, there exists a solution up to time $T$ of the DMFT system
(\ref{def:dmft_covuw}--\ref{def:covarianceresponse}) such that
almost surely as $n,d \to \infty$,
(\ref{eq:dmftW2convergence}) holds and also
\[\{\widehat \alpha^t\}_{t \in [0,T]} \to \{\alpha^t\}_{t \in [0,T]}
\text{ in } C([0,T],\R^K).\]
\end{enumerate}
\end{theorem}
Both parts of this theorem follow from \cite[Theorem 2.5]{paper1}, and we explain the details of the reduction to \cite[Theorem 2.5]{paper1} in Appendix \ref{appendix:globalconditions}.
The above solutions to the DMFT systems are unique in certain domains of
exponential growth for $\{\alpha^t\}$ and for the correlation and response functions, and we refer
readers to \cite[Theorem 2.4]{paper1} for details of this uniqueness claim.

For our analyses of dynamics with fixed prior $g(\cdot)$
in the setting of Theorem \ref{thm:dmft_approx}(a), we will require a second
result from \cite{paper1} that gives an interpretation for the DMFT response
functions $R_\theta(t,s)$ and $R_\eta(t,s)$ as coordinate averages of
single-particle responses in the Langevin diffusion
(\ref{eq:langevinfixedprior}). We defer a statement of this result to
Section \ref{sec:langevinproperties}.

\subsection{Replica-symmetric characterization of equilibrium for a fixed prior}

This and the next section describe the main results of our current paper.
We discuss results pertaining to the dynamics
(\ref{eq:langevinfixedprior}) with a fixed prior $g(\cdot)$ in this section,
and results pertaining to the empirical Bayes dynamics
(\ref{eq:langevin_sde}--\ref{eq:gflow}) in Section \ref{subsec:adaptive} to
follow.

\subsubsection{Approximately-TTI DMFT systems}

We first introduce a set of conditions for the correlation and
response functions of the DMFT system that characterize an approximate
time-translation-invariance (TTI) property. Under these conditions, we
establish convergence of the joint
law of $(\theta^*,\theta^t)$ in the DMFT equations to a replica-symmetric
fixed point as $t \to \infty$.

\begin{definition}\label{def:regular}
In the setting of a fixed prior $g(\cdot)$ [i.e.\ with
$\cG(\alpha,\sP)=0$ in (\ref{def:dmft_langevin_alpha})],
the solution of the DMFT system
(\ref{def:dmft_covuw}--\ref{def:covarianceresponse}) is
\emph{approximately-TTI} if it
satisfies the following conditions:
\begin{enumerate}
\item There exists a scalar value $c_\theta(*) \in \R$ and
functions $c_\theta^\tti,c_\eta^\tti:[0,\infty)
\to \R$ such that, for some $\eps:[0,\infty) \to [0,\infty)$
satisfying $\lim_{s \to \infty} \eps(s)=0$ and for all $t \geq s \geq 0$,
\begin{align}
|C_\theta(t,s)-c_\theta^\tti(t-s)|
&\leq \eps(s),\label{eq:Cthetaapproxinvariant}\\
|C_\eta(t,s)-c_\eta^\tti(t-s)| &\leq
\eps(s),\label{eq:Cetaapproxinvariant}\\
|C_\theta(s,*)-c_\theta(*)| &\leq \eps(s).\label{eq:Cthetastarlim}
\end{align}
Furthermore, there exist values $c_\theta^\tti(\infty),c_\eta^\tti(\infty) \geq 0$ and
finite positive measures $\mu_\theta,\mu_\eta$
supported on $[\iota,\infty)$ for some $\iota>0$ (strictly) such that
\begin{equation}\label{eq:cttiforms}
c_\theta^\tti(\tau)=c_\theta^\tti(\infty)+\int_\iota^\infty e^{-a\tau}\d\mu_\theta(a),
\qquad c_\eta^\tti(\tau)=c_\eta^\tti(\infty)+\int_\iota^\infty
e^{-a\tau}\d\mu_\eta(a).
\end{equation}
\item There exist functions
$r_\theta^\tti,r_\eta^\tti:[0,\infty) \to \R$ such that, for some $\eps:[0,\infty) \to [0,\infty)$
satisfying $\lim_{t \to \infty} \eps(t)=0$,
\begin{align}
\int_0^t |R_\theta(t,s)-r_\theta^\tti(t-s)|\d s &\leq
\eps(t),\label{eq:Rthetaapproxinvariant}\\
\int_0^t |R_\eta(t,s)-r_\eta^\tti(t-s)|\d s &\leq \eps(t).\label{eq:Retaapproxinvariant}
\end{align}
Furthermore, $r_\theta^\tti,r_\eta^\tti,c_\theta^\tti,c_\eta^\tti$ satisfy
the fluctuation-dissipation relations
\begin{equation}\label{eq:crttiFDT}
r_\theta^\tti(\tau) = {-}{c_\theta^\tti}'(\tau),
\qquad r_\eta^\tti(\tau) = {-}{c_\eta^\tti}'(\tau).
\end{equation}
\end{enumerate}
\end{definition}

We show that if the DMFT system is approximately-TTI
in the above sense, then its $t \to \infty$ limit is
characterized by a system of ``static'' scalar fixed-point equations. To
describe this characterization, consider a scalar Gaussian convolution model
\begin{equation}\label{eq:scalarchannel}
y=\theta^*+z \in \R.
\end{equation}
Let
\begin{equation}\label{eq:scalarchannel_gibbsavg}
\sP_{g,\omega}(\theta \mid y)=\frac{1}{\sP_{g,\omega}(y)}
\sqrt{\frac{\omega}{2\pi}}
\exp\Big({-}\frac{\omega}{2}(y-\theta)^2\Big)g(\theta)
\end{equation}
be the posterior distribution of $\theta$ in this model, assuming a prior law
$\theta \sim g(\cdot)$ and independent Gaussian noise $z \sim
\N(0,\omega^{-1})$, where
\begin{equation}\label{eq:scalarchannelmarginal}
\sP_{g,\omega}(y)=\int \sqrt{\frac{\omega}{2\pi}}
\exp\Big({-}\frac{\omega}{2}(y-\theta)^2\Big)g(\theta)\d\theta
\end{equation}
denotes the marginal density of $y$ under these assumptions. Let the
true model be $y=\theta^*+z$ with $\theta^* \sim g_*$
and independent noise $z \sim \N(0,\omega_*^{-1})$, and denote by
\begin{equation}\label{eq:thetalimitlaw}
\sP_{g_*,\omega_*;g,\omega}(\theta^*,\theta)
\end{equation}
the joint law of the true parameter $\theta^*$ and a posterior sample
$\theta$ under the generating process
\begin{equation}\label{eq:scalarchannelprocess}
\theta^* \sim g_*,\,z \sim \N(0,\omega_*^{-1}) \text{ (independent)}
\quad \Rightarrow \quad
y=\theta^*+z
\quad \Rightarrow \quad \theta \mid y \sim \sP_{g,\omega}(\,\cdot\,\mid y)
\end{equation}
(where $\theta \mid y$ is defined with misspecified prior law $g(\cdot)$ and
misspecified noise variance $\omega^{-1}$).
We write $\langle f(\theta) \rangle_{g,\omega}$ for the posterior average with
respect to $\sP_{g,\omega}(\,\cdot\, \mid y)$ depending implicitly on $y$,
and $\E_{g_*,\omega_*} f(y)$ for the expectation under the true model
$y=\theta^*+z$. Thus, an expectation over the joint law 
$\sP_{g_*,\omega_*;g,\omega}$ in (\ref{eq:thetalimitlaw}) takes the form
\[\E_{(\theta^*,\theta) \sim \sP_{g_*,\omega_*;g,\omega}} f(\theta^*,\theta)
=\E_{g_*,\omega_*} \langle f(\theta^*,\theta) \rangle_{g,\omega}.\]

\begin{theorem}\label{thm:dmft_equilibrium}
Suppose Assumptions \ref{assump:model} and \ref{assump:prior}(a) hold.
Consider the Langevin diffusion (\ref{eq:langevinfixedprior}) with a
fixed prior $g(\cdot)$, and suppose
that the corresponding solution of the DMFT system in
Theorem \ref{thm:dmft_approx}(a) is approximately-TTI.
Define, from the quantities of Definition \ref{def:regular},
\begin{equation}\label{eq:mmse}
\begin{gathered}
\mse=c_\theta^\tti(0)-c_\theta^\tti(\infty),
\qquad \mse_*=\E[{\theta^*}^2]-2c_\theta(*)+c_\theta^\tti(\infty),\\
\ymse=\frac{\sigma^4}{\delta}\big(c_\eta^{\tti}(0)-c_\eta^{\tti}(\infty)\big),
\qquad
\ymse_*=\frac{\sigma^4}{\delta}\big(2c_\eta^\tti(0)-c_\eta^{\tti}(\infty)\big)-\sigma^2.
\end{gathered}
\end{equation}
Then there are unique values $\omega,\omega_*>0$ (given $\mse,\mse_*$) for which
$\mse,\mse_*,\omega,\omega_*$ satisfy the fixed-point equations
\begin{equation}\label{eq:static_fixedpoint}
\begin{gathered}
\omega=\delta(\sigma^2+\mse)^{-1}, \qquad
\omega_*=\delta(\sigma^2+\mse_*)^{-1},\\
\mse=\E_{g_*,\omega_*}[\langle (\theta-\langle \theta \rangle_{g,\omega})^2
\rangle_{g,\omega}],
\qquad \mse_*=\E_{g_*,\omega_*}[(\theta^*-\langle \theta
\rangle_{g,\omega})^2].
\end{gathered}
\end{equation}
The quantities $\ymse,\ymse_*$ are related to these fixed points by
\begin{equation}\label{eq:ymmseomega}
\ymse=\sigma^2\Big(1-\frac{\omega\sigma^2}{\delta}\Big), \qquad
\ymse_*=\sigma^2+\frac{\omega\sigma^4}{\delta}\Big(\frac{\omega}{\omega_*}-2\Big).
\end{equation}
Furthermore, letting $\sP(\theta^*,\theta^t)$ be the joint law of
$(\theta^*,\theta^t)$ in the DMFT system, as $t \to \infty$,
\begin{equation}\label{eq:dmftthetaconvergence}
\sP(\theta^*,\theta^t) \overset{W_2}{\to}
\sP_{g_*,\omega_*;g,\omega}.
\end{equation}
\end{theorem}

\begin{remark}\label{remark:replica}
Let $\langle f(\btheta) \rangle$
and $\langle f(\btheta,\btheta') \rangle$ 
denote the expectation over independent samples
$\btheta,\btheta' \sim \sP_g(\cdot \mid \X,\y)$ from the posterior law
(\ref{eq:posterior}) with a
fixed prior $g(\cdot)$. Then the asymptotic overlaps
\[\lim_{n,d \to \infty} d^{-1}\langle \btheta^\top \btheta \rangle,
\qquad \lim_{n,d \to \infty} d^{-1}\langle \btheta^\top \btheta' \rangle,
\qquad \lim_{n,d \to \infty}
d^{-1}\langle \btheta^\top \btheta^* \rangle\]
are predicted in the DMFT system, respectively, by
\[c_\theta^\tti(0)=\lim_{t \to \infty} C_\theta(t,t), \qquad
c_\theta^\tti(\infty)=\lim_{t,\tau \to \infty} C_\theta(t,t+\tau),
\qquad c_\theta(*)=\lim_{t \to \infty} C_\theta(t,*).\]
Thus $\mse$ and $\mse_*$ as defined in (\ref{eq:mmse}) represent the
DMFT predictions for
\[\lim_{n,d \to \infty}
d^{-1}\langle \|\btheta-\langle \btheta \rangle\|_2^2 \rangle, \qquad
\lim_{n,d \to \infty} d^{-1}\|\btheta^*-\langle \btheta \rangle\|_2^2.\]
Similarly, one may check that $\ymse$ and $\ymse_*$ as defined in
(\ref{eq:mmse}) represent the DMFT predictions for
\[\lim_{n,d \to \infty}
n^{-1}\langle \|\X\btheta-\X\langle \btheta \rangle\|_2^2 \rangle, \qquad
\lim_{n,d \to \infty} n^{-1}\|\X\btheta^*-\X\langle \btheta \rangle\|_2^2.\]
These fixed-point equations (\ref{eq:static_fixedpoint}) that characterize
$\mse$ and $\mse_*$ coincide with those derived via the replica method (with
misspecified prior) under a replica-symmetric ansatz, c.f.\
\cite{guo2005randomly,kabashima2008inference,takahashi2022macroscopic}.
\end{remark}

We clarify that Theorem \ref{thm:dmft_equilibrium} does not claim that the
joint solution $(\mse,\mse_*,\omega,\omega_*)$ of the fixed-point equations
(\ref{eq:static_fixedpoint}) is unique. In settings with multiple
such fixed points, the theorem pertains to the specific choice of this fixed
point that arises from the $t \to \infty$ limit of the DMFT dynamics.

\subsubsection{Asymptotic MSE and free energy under a posterior LSI}

To motivate Definition \ref{def:regular}, it is illustrative to consider the
example of a fixed Gaussian prior $g(\cdot)$, where the Langevin
diffusion for $\btheta^t$ is a linear Ornstein-Uhlenbeck process.
Then $C_\theta,C_\eta,R_\theta,R_\eta$ of the DMFT system may be computed
explicitly, as we show in Appendix \ref{appendix:gaussian}, and it is
directly checked from their explicit forms that the DMFT system is indeed
approximately-TTI.

Generalizing this Gaussian prior example, we consider a setting
where the posterior distribution (\ref{eq:posterior})
satisfies a log-Sobolev inequality.

\begin{assumption}\label{assump:LSI}
There exists a constant $C_\LSI>0$ and a $\X$-dependent event
$\event(\X)$ holding almost surely for all large $n,d$, for which
\begin{enumerate}[(a)]
\item (LSI for posterior) On $\event(\X)$, for all $\y \in \R^n$,
the posterior distribution $\sP_g(\btheta \mid \X,\y)$ satisfies
\begin{equation}\label{eq:LSI}
\Ent[f(\btheta)^2 \mid \X,\y] \leq C_\LSI\,
\,\E[\|\nabla f(\btheta)\|_2^2 \mid \X,\y] \text{ for all } f \in C^1(\R^d).
\end{equation}
\item (LSI for larger noise) On $\event(\X)$, for every noise variance
$\tilde \sigma^2 \in [\sigma^2,\infty)$, (\ref{eq:LSI}) holds 
also for the posterior law $\sP_g(\btheta \mid \X,\y)$ defined
with $\tilde \sigma^2$ in place of $\sigma^2$
(with a uniform constant $C_\text{LSI}>0$ for all $\tilde\sigma^2 \geq
\sigma^2$).
\end{enumerate}
\end{assumption}

For clarity of interpretation,
we list in the following proposition three
concrete settings in which these LSI conditions hold by currently known techniques.
A proof of Proposition \ref{prop:LSI} is given in Appendix \ref{appendix:LSI}.

\begin{proposition}\label{prop:LSI}
Suppose $\X$ satisfies Assumption \ref{assump:model}(a--b), and
$g(\cdot)$ satisfies Assumption \ref{assump:prior}(a).
Let $C,r_0,c_0>0$ be the constants of Assumption \ref{assump:prior}(a), and
define
\[C_0=\frac{2.01}{c_0}\exp\Big(\frac{8r_0^2(c_0+C)^2}{\pi c_0}\Big).\]
Suppose, in addition, that at least one of the following conditions hold:
\begin{enumerate}[(a)]
\item (global log-concavity)
${-}(\log g)''(\theta) \geq c_0$ for all $\theta \in \R$, or
\item (high noise)
$\sigma^2>C_0(4\sqrt{\delta}\,\1\{\delta>1\}+(\sqrt{\delta}+1)^2\1\{\delta \leq
1\})$, or
\item (large sample size) $\delta>1$ and
$(\sqrt{\delta}-1)^2>4C_0C\sqrt{\delta}$.
\end{enumerate}
Then Assumption \ref{assump:LSI} holds for a constant
$C_\LSI>0$ depending only on $\delta,C,r_0,c_0$.
\end{proposition}

Under the posterior LSI condition of Assumption \ref{assump:LSI}(a),
we verify that the solution of the DMFT system must be approximately-TTI
in the sense of Definition \ref{def:regular}. 

\begin{theorem}\label{thm:fixedalpha_dynamics}
Consider the dynamics (\ref{eq:langevinfixedprior}) with a fixed prior
$g(\cdot)$, and suppose Assumptions \ref{assump:model}, \ref{assump:prior}(a),
and \ref{assump:LSI}(a) hold. Then the DMFT system
given by Theorem \ref{thm:dmft_approx}(a) is approximately-TTI, where
the statements of Definition \ref{def:regular} hold with $\eps(t)=Ce^{-ct}$ and
some constants $C,c>0$.
\end{theorem}

As a consequence, we obtain the following corollary showing that the
asymptotic free energy and mean-squared-errors associated to the posterior
distribution $\sP_g(\btheta \mid \X,\y)$ in the linear model
(with a possibly misspecified prior) are given by their replica-symmetric
predictions, and furthermore the joint empirical distribution of coordinates of
$\btheta^*$ and a posterior sample $\btheta \sim \sP_g(\cdot \mid \X,\y)$
converges to the preceding law $\sP_{g_*,\omega_*;g,\omega}$ in the scalar
Gaussian convolution model.
(Our analysis for the free energy uses an I-MMSE relation, for
which we require the posterior LSI condition of
Assumption \ref{assump:LSI}(b) for an extended range of noise variances.)

\begin{corollary}\label{cor:fixedalpha_dynamics}
Suppose Assumptions \ref{assump:model}, \ref{assump:prior}(a), and
\ref{assump:LSI}(a) hold
for dynamics (\ref{eq:langevinfixedprior}) with a fixed prior $g(\cdot)$.
Let $\sP_g(\y \mid \X)$ be the marginal likelihood of $\y$ in
(\ref{eq:marginallikelihood}), let
$\langle f(\btheta) \rangle$ denote the posterior expectation
under $\sP_g(\btheta \mid \X,\y)$, and define
\begin{equation}\label{eq:MMSE}
\begin{gathered}
\MSE=d^{-1}\langle \|\btheta-\langle \btheta \rangle\|_2^2
\rangle,
\qquad \MSE_*=d^{-1}\|\btheta^*-\langle \btheta \rangle\|_2^2\\
\YMSE=n^{-1}\langle \|\X\btheta-\langle \X\btheta \rangle\|_2^2 \rangle,
\qquad \YMSE_*=n^{-1}\|\X\btheta^*-\langle \X\btheta \rangle\|_2^2.
\end{gathered}
\end{equation}
Let $\mse,\mse_*,\omega,\omega_*,\ymse,\ymse_*$ be as defined by
(\ref{eq:mmse}--\ref{eq:static_fixedpoint}) for the corresponding
(approximately-TTI) DMFT system, let
$\sP_{g,\omega}(y)$ be the marginal density of $y$ in
(\ref{eq:scalarchannelmarginal}), and let $\E_{g_*,\omega_*}$ denote the
expectation over $y=\theta^*+z$ in (\ref{eq:scalarchannel}) with
$\theta^* \sim g_*$ and $z \sim \N(0,\omega_*^{-1})$ .

\begin{enumerate}[(a)]
\item Almost surely,
\[\lim_{n,d \to \infty} \MSE=\mse, \qquad \lim_{n,d \to \infty}
\MSE_*=\mse_*,\]
\[\lim_{n,d \to \infty} \YMSE=\ymse, \qquad \lim_{n,d \to \infty}
\YMSE_*=\ymse_*,\]
\begin{equation}\label{eq:fixedalpha_equilibriumlaw}
\lim_{n,d \to \infty} \bigg\langle
W_2\bigg(\frac{1}{d}\sum_{j=1}^d \delta_{(\theta_j^*,\theta_j)}
,\sP_{g_*,\omega_*;g,\omega}\bigg)^2\bigg\rangle=0.
\end{equation}
\item If furthermore Assumption \ref{assump:LSI}(b) holds, then almost surely,
\[\lim_{n,d \to \infty}
\frac{1}{d}\log \sP_g(\y \mid \X)=\E_{g_*,\omega_*} \log \sP_{g,\omega}(y)
+\frac{1}{2}\left(\delta+\log \frac{2\pi}{\omega}
-\delta \log \frac{2\pi \delta}{\omega}
+(1-\delta)\frac{\omega}{\omega_*}
+\omega\sigma^2\Big(\frac{\omega}{\omega_*}-2\Big)\right).\]
\end{enumerate}
\end{corollary}

As discussed in Remark \ref{remark:replica}, the fixed-point equations
characterizing these limits of the mean-squared-error quantities
$\MSE,\MSE_*,\YMSE,\YMSE_*$ are those derived via
the replica method under an assumption of replica symmetry.
One may verify that the limit of the free energy in part (b)
agrees also with the replica prediction that was computed in
\cite[Eq.\ (20)]{kabashima2008inference}.

The proof of Theorem \ref{thm:dmft_equilibrium} is given in Section
\ref{sec:equilibrium}, and the proofs of Theorem \ref{thm:fixedalpha_dynamics}
and Corollary \ref{cor:fixedalpha_dynamics} are given in Section
\ref{sec:fixedprior}.

\subsection{Convergence of empirical Bayes Langevin dynamics}\label{subsec:adaptive}

We now discuss results pertaining to the empirical Bayes Langevin
dynamics (\ref{eq:langevin_sde}--\ref{eq:gflow}) with a data-adaptive evolution
of the prior law.

\subsubsection{A general condition for dimension-free convergence}

We impose the following strengthening of Assumption
\ref{assump:LSI}, ensuring that $\{\alpha^t\}_{t \geq 0}$ of the DMFT solution
remains confined to a bounded domain where the posterior
log-Sobolev conditions of Assumption \ref{assump:LSI} hold uniformly. 

\begin{assumption}\label{assump:compactalpha}
Let $\{\alpha^t\}_{t \geq 0}$ be the $\alpha$-component of the DMFT system.
There exists a compact subset $S \subset \R^K$ such that
\[\alpha^t \in S \text{ for all } t \geq 0.\]
Furthermore, there exists
a (bounded) open neighborhood $O \supset S$ and an $\X$-dependent event
$\event(\X)$ on which the statements of Assumption \ref{assump:LSI}(a--b)
hold with a uniform constant $C_\text{LSI}>0$ for every prior
$g \in \{g(\,\cdot\,,\alpha):\alpha \in O\}$.
\end{assumption}

Under this condition, we will show dimension-free convergence of
the prior parameter $\{\alpha^t\}_{t \geq 0}$ to a fixed point of the
replica-symmetric free energy. To state this result,
let us recall the free energy
\[\widehat F(\alpha)={-}\frac{1}{d}\log \sP_{g(\cdot,\alpha)}(\y \mid \X)\]
of the linear model from (\ref{eq:hatF}), and denote by
\begin{equation}\label{eq:limitF}
F(\alpha)={-}\E_{g_*,\omega_*} \log \sP_{g(\cdot,\alpha),\omega}(Y)
-\frac{1}{2}\left(\delta+\log \frac{2\pi}{\omega}
-\delta \log \frac{2\pi\delta}{\omega}
+(1-\delta)\frac{\omega}{\omega_*}
+\omega\sigma^2\Big(\frac{\omega}{\omega_*}-2\Big)\right)
\end{equation}
its asymptotic limit prescribed by
Corollary \ref{cor:fixedalpha_dynamics}, both viewed as a function of $\alpha
\in O \subset \R^K$. Here, the fixed points
$(\omega,\omega_*) \equiv (\omega(\alpha),\omega_*(\alpha))$ 
implicitly depend on $\alpha$ and are well-defined
by Theorem \ref{thm:fixedalpha_dynamics} for all $\alpha \in O$.
Recalling the law $\sP_{g_*,\omega_*;g,\omega}(\theta^*,\theta)$
from (\ref{eq:thetalimitlaw}), let us abbreviate this law
with $g \equiv g(\cdot,\alpha)$ and fixed points
$(\omega(\alpha),\omega_*(\alpha))$ as
\begin{equation}\label{eq:limitlawalpha}
\sP_\alpha \equiv \sP_{g_*,\omega_*(\alpha);g(\cdot,\alpha),\omega(\alpha)}.
\end{equation}
We write $\theta \sim \sP_\alpha$ as shorthand for the $\theta$-marginal of
$(\theta^*,\theta) \sim \sP_\alpha$. We write also
$\langle \cdot \rangle_\alpha$ for the expectation under the posterior law
$\sP_{g(\cdot,\alpha)}(\btheta \mid \X,\y)$ in the linear model.
The following lemma strengthens Corollary \ref{cor:fixedalpha_dynamics}(b)
to convergence of $F(\alpha)$ and its gradient, uniformly over the compact
subset $S \subset O$ containing $\{\alpha^t\}_{t \geq 0}$, and shows also that
a true prior parameter $\alpha^* \in O$ must be a global minimizer of
$F(\alpha)$.

\begin{lemma}\label{lemma:gradalphaF}
Suppose Assumptions \ref{assump:model}, \ref{assump:prior}(b),
and \ref{assump:compactalpha} hold, and let $S \subset O \subset \R^K$
be the domains of Assumption \ref{assump:compactalpha}. Then
\begin{enumerate}[(a)]
\item $\widehat F(\alpha)$ and $F(\alpha)$ are continuously differentiable
on $O$ with gradients
\begin{equation}\label{eq:gradF}
\nabla \widehat F(\alpha)={-}\bigg\langle
\frac{1}{d}\sum_{j=1}^d \nabla_\alpha \log g(\theta_j,\alpha)
\bigg\rangle_\alpha,
\qquad \nabla F(\alpha)={-}\E_{\theta \sim \sP_\alpha}
[\nabla_\alpha \log g(\theta,\alpha)].
\end{equation}
\item Almost surely
\[\lim_{n,d \to \infty} \sup_{\alpha \in S}
|\widehat F(\alpha)-F(\alpha)|=0,
\qquad \lim_{n,d \to \infty} \sup_{\alpha \in S}
\|\nabla \widehat F(\alpha)-\nabla F(\alpha)\|_2=0.\]
\item If $g_*(\cdot)=g(\,\cdot\,,\alpha^*)$ for some $\alpha^* \in O$,
then $F(\alpha^*)=\inf_{\alpha \in O} F(\alpha)$. 
\end{enumerate}
\end{lemma}

We now show that under the uniform LSI condition of Assumption
\ref{assump:compactalpha},
the DMFT solution $\{\alpha^t\}_{t \geq 0}$ converges as $t \to \infty$ to
a critical point $\alpha^\infty$ of the asymptotic free energy
$F(\alpha)$ (with possible additional regularization by $R(\alpha)$).
Consequently, for a dimension-independent
time horizon $T>0$ and large system sizes $n,d$, the learned prior parameter
$\widehat \alpha^T$ will be close to $\alpha^\infty$, and the Langevin sample
$\widehat \btheta^T$ will have entrywise statistics close to those in the
scalar Gaussian convolution model described by
Theorem \ref{thm:dmft_equilibrium} for the limiting prior
$g(\cdot)=g(\,\cdot\,,\alpha^\infty)$.

\begin{theorem}\label{thm:adaptivealpha_dynamics}
Suppose Assumptions \ref{assump:model}, \ref{assump:prior}(b),
and \ref{assump:compactalpha} hold.
Let $O \subset \R^K$ be as in Assumption \ref{assump:compactalpha},
define $F(\alpha)$ for $\alpha \in O$ by (\ref{eq:limitF}), and denote 
\[\Crit=\{\alpha \in S:\nabla F(\alpha)+\nabla R(\alpha)=0\}.\]
Consider the empirical Bayes Langevin dynamics
(\ref{eq:langevin_sde}--\ref{eq:gflow}), and let $\{\alpha^t\}_{t \geq 0}$ be
the deterministic approximation of $\{\widehat\alpha^t\}_{t \geq 0}$ in the
solution of the DMFT system in Theorem \ref{thm:dmft_approx}(b). Then
$\{\alpha^t\}_{t \geq 0}$ satisfies
\[\lim_{t \to \infty} \dist(\alpha^t,\Crit)=0.\]
In particular, if all points of $\Crit$ are isolated, then
there exists a limit
\begin{equation}\label{eq:alphatconverges}
\alpha^\infty=\lim_{t \to \infty} \alpha^t \in \Crit.
\end{equation}

Consequently, for any $\eps>0$, there exists a time horizon
$T:=T(\eps)>0$ independent of $n,d$ such that for any fixed $t>T(\eps)$,
the solution $\{(\btheta^t,\widehat\alpha^t)\}_{t \geq 0}$ of
(\ref{eq:langevin_sde}--\ref{eq:gflow}) satisfies almost surely
\begin{equation}\label{eq:posteriorconverges}
\limsup_{n,d \to \infty} \|\widehat \alpha^t-\alpha^\infty\|_2<\eps,
\qquad \limsup_{n,d \to \infty}
W_2\bigg(\frac{1}{d}\sum_{j=1}^d \delta_{(\theta_j^*,\theta_j^t)}
,\sP_{\alpha^\infty}\bigg)<\eps.
\end{equation}
\end{theorem}

The proof of Theorem \ref{thm:adaptivealpha_dynamics} is given in Section
\ref{sec:adaptiveprior}.

Supposing that $g_*(\cdot)=g(\,\cdot\,,\alpha^*)$ for a true prior
parameter $\alpha^* \in O$, in settings where $R(\alpha)=0$ and the critical
point $\alpha^\infty \in \Crit$ of $F(\alpha)$
is unique, Lemma \ref{lemma:gradalphaF}(c)
ensures that $\alpha^\infty=\alpha^*$, and Theorem
\ref{thm:adaptivealpha_dynamics} then provides a guarantee for estimation of
this true prior parameter as $n,d \to \infty$.
In general, $F(\alpha)$ may have multiple critical points.
Theorem \ref{thm:adaptivealpha_dynamics} ensures convergence to a
point $\alpha^\infty \in \Crit$ that is specified deterministically by
the initial conditions of Assumption \ref{assump:model}(d), and successful
learning of $\alpha^*$ may require multiple initializations from different
starting values of $\widehat\alpha^0$. We discuss both
types of settings in the following examples.

\subsubsection{Examples}\label{sec:examples}

We develop some further implications of Theorem \ref{thm:adaptivealpha_dynamics}
in a few specific examples of parametric models for
$g(\,\cdot\,,\alpha)$. We explore also via numerical simulation
the convergence of $(\btheta^t,\widehat\alpha^t)$, the landscape of the
replica-symmetric free energy $F(\alpha)$, and the nature of its critical point
set $\Crit$ in a few settings where a posterior log-Sobolev inequality may not hold.

\begin{example}\label{ex:gaussianmean}
Consider the Gaussian prior
\[g(\theta,\alpha)=\sqrt{\frac{\omega_0}{2\pi}}\exp\Big({-}\frac{\omega_0}{2}(\theta-\alpha)^2\Big)\]
with varying mean $\alpha \in \R$ and a fixed and known
prior variance $\omega_0^{-1}$, and suppose $g_*(\theta)=g(\theta,\alpha^*)$.
Consider the empirical Bayes dynamics driven by
\[\cG(\alpha,\sP)=\E_{\theta \sim \sP}[\partial_\alpha \log g(\theta,\alpha)]\]
in (\ref{def:dmft_langevin_alpha}), with no regularizer (i.e.\ $R(\alpha) \equiv
0$).

We verify in Section \ref{sec:exampledetails} that
Assumptions \ref{assump:prior}(b) and \ref{assump:compactalpha} hold for this
example, for a subset $O \subset \R$ containing $\alpha^*$.
The posterior mean in the Gaussian convolution model (\ref{eq:scalarchannel})
is given explicitly by
\[\langle \theta \rangle_{g(\cdot,\alpha),\omega}
=\frac{\omega_0}{\omega_0+\omega}\alpha
+\frac{\omega}{\omega_0+\omega}y.\]
Then the condition $\alpha \in \Crit$ is
$0=\nabla F(\alpha)=\E_{\theta \sim \sP_\alpha}[\omega_0(\alpha-\theta)]$, i.e.\
\[\alpha=\E_{\theta \sim \sP_\alpha}[\theta]
=\E_{g_*,\omega_*}[\langle \theta \rangle_{g(\cdot,\alpha),\omega}]
=\E_{g_*,\omega}\Big[\frac{\omega_0}{\omega_0+\omega}\alpha
+\frac{\omega}{\omega_0+\omega}y\Big]
=\frac{\omega_0}{\omega_0+\omega}\alpha
+\frac{\omega}{\omega_0+\omega}\alpha^*,\]
so $\Crit$ consists of the unique critical point $\alpha^*$.
Theorem \ref{thm:adaptivealpha_dynamics} then holds with
$\alpha^\infty=\alpha^*$, i.e.\ over a dimension-independent time horizon,
$\widehat\alpha^t$ converges to $\alpha^*$ (in the limit $n,d \to \infty$
followed by $t \to \infty$ as described in
Theorem \ref{thm:adaptivealpha_dynamics}), and the empirical distribution
of coordinates of the
Langevin sample $\btheta^t$ converges to that of the posterior distribution
for the true prior $\N(\alpha^*,\omega_0^{-1})$. \qed
\end{example}

\begin{example}\label{ex:logconcave}
Consider more generally a log-concave location prior
\[g(\theta,\alpha)=\exp\big({-}f(\theta-\alpha)\big)\]
where $\alpha \in \R$ and $f:\R \to \R$ is a fixed strongly convex function,
such that $f$ is thrice continuously-differentiable with H\"older-continuous
third derivative, and
\[f'(0)=0, \qquad C \geq f''(x) \geq c_0, \qquad |f'''(x)| \leq C\]
for some constants $C,c_0>0$ and all $x \in \R$.
Suppose again $g_*(\theta)=g(\theta,\alpha^*)$, and consider the empirical
Bayes dynamics driven by
\[\cG(\alpha,\sP)=\E_{\theta \sim \sP}[\partial_\alpha \log g(\theta,\alpha)]\]
with no regularizer.

We verify in Section \ref{sec:exampledetails} that
Assumptions \ref{assump:prior}(b) and
\ref{assump:compactalpha} hold for this example, for a subset
$O \subset \R$ containing $\alpha^*$. Furthermore, we show in
Section \ref{sec:exampledetails} via an adaptation of the Brascamp-Lieb
argument of \cite[Theorem 3]{kuntz2023particle} that $F(\alpha)$ must be
strongly convex on $O$.
Hence, $\Crit$ consists again of the unique critical point
$\alpha=\alpha^*$,
and Theorem \ref{thm:adaptivealpha_dynamics} holds for $\alpha^\infty=\alpha^*$.
\qed
\end{example}

We next consider two canonical examples where the prior $g(\theta,\alpha)$ is a
Gaussian mixture model that is not log-concave in $\theta$, and where the
landscape of $F(\alpha)$ is also not necessarily convex in $\alpha$.
We will check the uniform
log-Sobolev condition of Assumption \ref{assump:compactalpha} and also
characterize analytically the landscape of the free energy $F(\alpha)$
for sufficiently large $\delta=\lim \frac{n}{d}$,
and explore by simulation the learning dynamics and free energy landscape
for some smaller values of $\delta$.

The sub-level sets of $F(\alpha)$ may not be bounded in these examples.
To confine $\{\alpha^t\}_{t \geq 0}$ to a bounded subset of $\R^K$,
we introduce an additional regularizer:
Fix a radius $D>0$, and let $\ball(D)=\{\alpha \in \R^K:\|\alpha\|_2<D\}$ be
the open ball of radius $D$. For a smooth function $r:[0,\infty) \to
[0,\infty)$ having bounded derivatives of all orders and satisfying
\begin{equation}\label{eq:ralpha}
r(x)=0 \text{ for all } x \in [0,D], 
\quad r(x) \geq x-D \text{ for all } x \geq D+1,
\quad r'(x)>0 \text{ for all } x>D,
\end{equation}
we fix the regularizer $R:\R^K \to \R$ as
\begin{equation}\label{eq:Ralpha}
R(\alpha)=r(\|\alpha\|_2).
\end{equation}
Note that $R(\alpha)=0$ for all $\alpha \in \ball(D)$, so
adding such a regularizer does not change the critical points
$\alpha \in \Crit \cap\,\ball(D)$. We show in Proposition
\ref{prop:alphaconfined} of Section \ref{sec:exampledetails} that
adding such a regularizer indeed confines the dynamics of
$\{\alpha^t\}_{t \geq 0}$ to a bounded domain.

We will study analytically a large-$\delta$ limit under a reparametrization of
the noise variance $\sigma^2$ by $s^2=\sigma^2/\delta$, corresponding to a
rescaling of the regression design $\X$ to have entries of variance
$1/n$ and a rescaling of the noise $\beps$ to have entries $\N(0,s^2)$.
The setting $\delta \to \infty$ with fixed $s^2>0$ is a limiting regime in which
each coordinate of the posterior distribution of
$\btheta$ does not contract around its mode,
the Bayes-optimal mean-squared-error for estimating
$\btheta$ remains bounded away from 0, and
the landscape of $F(\alpha)$ approaches (up to an additive
constant) the log-likelihood landscape in the scalar Gaussian
convolution model $y=\theta+z$ where $\theta \sim g(\cdot,\alpha)$ and
$z \sim \N(0,s^2)$. We denote by
\begin{equation}\label{eq:Galpha}
G_{s^2}(\alpha)={-}\E_{g_*,s^{-2}}[\log \sP_{g(\cdot,\alpha),s^{-2}}(y)]
\end{equation}
the negative population log-likelihood in this model as a function of the prior
parameter $\alpha$, when the true distribution of $y$ is given by
$y=\theta^*+z$ with $\theta^* \sim g_*$.

\begin{proposition}\label{prop:largedelta}
Suppose Assumptions \ref{assump:model} and \ref{assump:prior}(b) hold,
and the regularizer $R(\alpha)$ is given by
(\ref{eq:ralpha}--\ref{eq:Ralpha})
with $\alpha^0 \in \ball(D)$. Fix $s^2=\sigma^2/\delta$, and define
\[\Crit_G=\{\alpha \in \ball(D):\nabla G_{s^2}(\alpha)=0\}.\]
Then, for any $s^2>0$, there exists a constant $\delta_0:=\delta_0(s^2)>0$
and a function $\iota:[\delta_0,\infty) \to (0,\infty)$ with $\lim_{\delta \to
\infty} \iota(\delta)=0$ such that if $\delta>\delta_0$, then
Assumption \ref{assump:compactalpha} holds. Furthermore,
\begin{enumerate}
\item Each point of $\Crit \cap \,\ball(D)$ belongs to a ball of radius
$\iota(\delta)$ around some point of $\Crit_G$.
\item For each point $\alpha \in \Crit_G$ where $\nabla^2 G_{s^2}(\alpha)$ is
non-singular, there is exactly one point of $\Crit$ in the ball of radius
$\iota(\delta)$ around $\alpha$.
\end{enumerate}
In particular, if $g_*=g(\,\cdot\,,\alpha^*)$ for some $\alpha^* \in
\ball(D)$, and if $\alpha^*$ is the unique point of $\Crit_G$ and
$\nabla^2 G_{s^2}(\alpha^*)$ is non-singular, then $\alpha^*$ is also the
unique point of $\Crit \cap \,\ball(D)$.
\end{proposition}

\begin{figure}[t]
\minipage{0.33\columnwidth}
\xincludegraphics[width=0.85\textwidth,label=(a)]{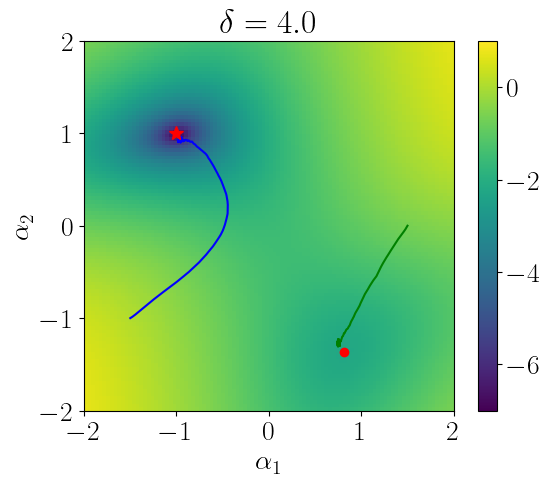}\\
\xincludegraphics[width=0.85\textwidth,label=(d)]{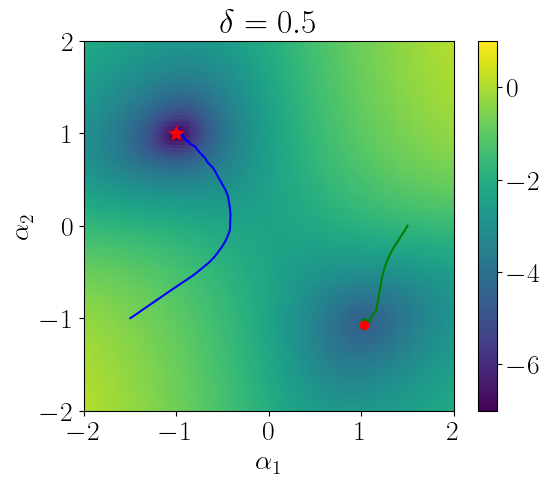}
\endminipage
\minipage{0.33\columnwidth}
\xincludegraphics[width=0.85\textwidth,label=(b)]{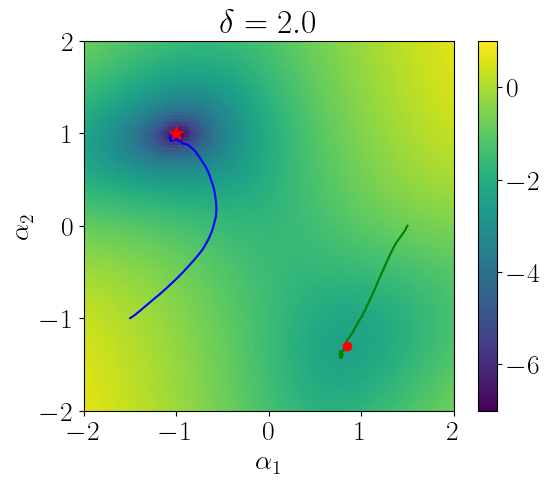}\\
\xincludegraphics[width=0.85\textwidth,label=(e)]{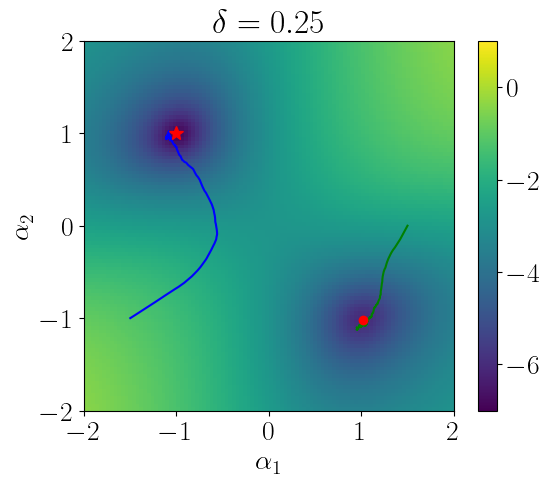}
\endminipage
\minipage{0.33\columnwidth}
\xincludegraphics[width=0.85\textwidth,label=(c)]{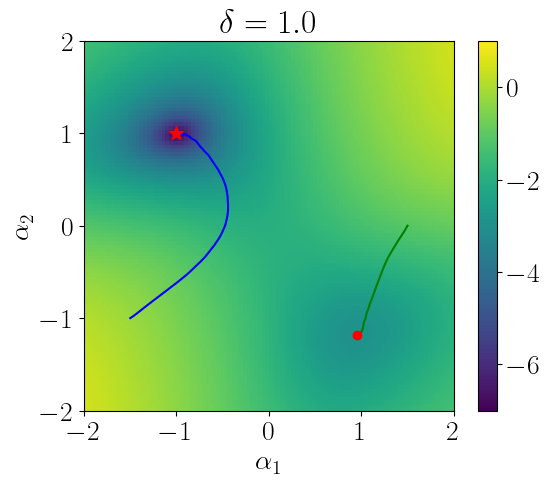}\\
\xincludegraphics[width=0.9\textwidth,label=(f)]{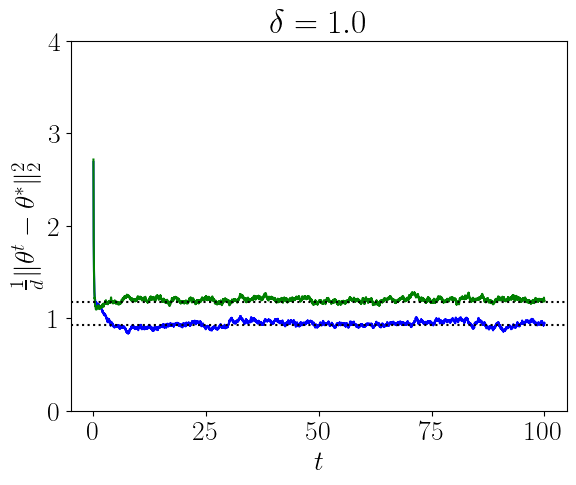}
\endminipage
\caption{Simulations for the Gaussian mixture prior model
$\frac{1}{2}\N(\alpha_1,1)+\frac{1}{2}\N(\alpha_2,0.25)$ of Example
\ref{ex:gaussianmeanmixture}, with true mixture means $\alpha^*=(1,-1)$ and linear model
noise variance $\sigma^2=\delta s^2$ for $s=0.5$.
Empirical Bayes Langevin dynamics is run for a single 
instance $(\X,\y)$ with 
$\max(n,d)=5000$, initialization
$\theta_j^0 \overset{iid}{\sim} \N(0,1)$, and an Euler-Maruyama
discretization of the dynamics.
(a--e) Landscape of the replica-symmetric
free energy $F(\alpha)$ plotted (for visual clarity) as
$\log(F(\alpha)-F(\alpha^*)+10^{-3})$, for $\delta \in
\{4,2,1,0.5,0.25\}$. Two stable fixed points of
$0=\nabla F(\alpha)$ are depicted in red, with star indicating the true parameter $\alpha^*=(-1,1)$ and circle indicating a second fixed point
$\alpha^\dagger$ near $(1,-1)$. Sample paths $\{\widehat \alpha^t\}_{t \geq 0}$
from two different initial states $\widehat\alpha^0$ are shown in blue and green.
(f) Mean-squared-error
$\frac{1}{d}\|\btheta^t-\btheta^*\|_2^2$ across iterations for these same
two initial states, at $\delta=1$. The
predicted value for a posterior
sample $\btheta \sim \sP_{g(\cdot,\alpha)}( \cdot \mid \X,\y)$ is
$\frac{1}{d}\|\btheta-\btheta^*\|_2^2 \approx \mse(\alpha)+\mse_*(\alpha)$, depicted by dashed lines for
$\alpha \in \{\alpha^\dagger,\alpha^*\}$.}
\label{fig:meanmixture}
\end{figure}

\begin{example}\label{ex:gaussianmeanmixture}
Consider a $K$-component Gaussian mixture prior
\[g(\theta,\alpha)=\sum_{k=1}^K p_k
\sqrt{\frac{\omega_k}{2\pi}}\exp\Big({-}\frac{\omega_k}{2}(\theta-\alpha_k)^2
\Big)\]
with fixed mixture weights $p_1,\ldots,p_K$ and variances
$\omega_1^{-1},\ldots,\omega_K^{-1}$, parametrized by the mixture means
$\alpha \in \R^K$. Let us suppose that
$g_*(\theta)=g(\theta,\alpha^*)$ for some $\alpha^* \in \R^K$, and the variances
$\omega_1^{-1},\ldots,\omega_K^{-1}$ are distinct.
We consider the empirical Bayes dynamics driven by
\[\cG(\alpha,\sP)=\E_{\theta \sim \sP}[\nabla_\alpha \log g(\theta,\alpha)]
-\nabla R(\alpha),\]
where $R(\alpha)$ is a regularizer of the form
(\ref{eq:ralpha}--\ref{eq:Ralpha}) for which $\alpha^0,\alpha^* \in \ball(D)$.

We verify in Section \ref{sec:exampledetails} that
Assumption \ref{assump:prior}(b) holds.
Then, for fixed $s^2>0$ and all sufficiently large $\delta$, Proposition \ref{prop:largedelta}
ensures that the confinement and log-Sobolev conditions of
Assumption \ref{assump:compactalpha} also hold, and the
proposition further
establishes a 1-to-1 correspondence between the critical points of $F$ and the (non-singular) critical points of $G_{s^2}(\alpha)$ in $\ball(D)$.
We note that, here, $G_{s^2}(\alpha)$ is the negative population
log-likelihood in the Gaussian mixture model
\begin{equation}\label{eq:marginalmeanmixture}
\sP_{g(\cdot,\alpha),s^{-2}}(y)
=\sum_{k=1}^K p_k \cdot \frac{1}{\sqrt{2\pi(\omega_k^{-1}+s^2)}}
\exp\Big({-}\frac{1}{2(\omega_k^{-1}+s^2)}(y-\alpha_k)^2 \Big)
\end{equation}
having the same mixture means $\alpha \in \R^K$ as the prior, and elevated
mixture variances $\omega_k^{-1}+s^2$. The optimization landscape of
$G_{s^2}(\alpha)$ is well-studied in the literature, see e.g.\
\cite{xu2016global,jin2016local,katsevich2023likelihood,chen2024local}, and in general
$G_{s^2}(\alpha)$ may have local minimizers in $\ball(D)$ that are
different from $\alpha^*$. In such settings, Proposition \ref{prop:largedelta}
implies that $\Crit$ must also have critical points different from $\alpha^*$
for large $\delta$.

We depict in Figure \ref{fig:meanmixture} a simulation of the landscape of
$F(\alpha)$ and the dynamics (\ref{eq:langevin_sde}--\ref{eq:gflow})
across a range of values $\delta \in [0.25,4]$, in a simple setting of
$\frac{1}{2}\N(\alpha_1,1)+\frac{1}{2}\N(\alpha_2,0.25)$
with $K=2$ mixture components and true mixture means $\alpha^*=(-1,1)$.
The Almeida-Thouless condition for stability of
the replica-symmetric phase was computed in \cite[Eq.\
(25)]{kabashima2008inference} to be (in our notation)
\begin{equation}\label{eq:AT}
1-\frac{\omega^2}{\delta}\E_{g_*,\omega_*}[\Var_{g,\omega}[\theta]^2]
\geq 0
\end{equation}
where $g(\cdot)=g(\cdot,\alpha)$ and
$(\omega,\omega_*)=(\omega(\alpha),\omega_*(\alpha))$.
We have verified that this condition holds at each tested $\delta>0$ and parameter $\alpha \in \R^K$
depicted in Figure \ref{fig:meanmixture}, and thus we conjecture that the
depicted replica-symmetric free energy function $F(\alpha)$
is indeed the correct asymptotic limit of
$-\frac{1}{d}\log \sP_{g(\cdot,\alpha)}(\y \mid \X)$ as $n,d \to \infty$ (even in settings where our assumption of a log-Sobolev inequality for the posterior law may not hold).
We observe, not only for large $\delta$ but across a range of
values $\delta \in [0.25,4]$, that the landscape $F(\alpha)$ has two local
minimizers, one fixed at the true parameter $\alpha^*=(-1,1)$ and a
second minimizer $\alpha^\dagger$ whose location depends on $\delta$. As
$\delta$ decreases, this second minimizer approaches $(1,-1)$ --- characterizing
a prior law with mixture means matching those of $g_*=g(\,\cdot\,,\alpha^*)$
but with the mixture variances reversed --- and the free energy difference
$F(\alpha^\dagger)-F(\alpha^*)$ approaches 0, indicating that it becomes
increasingly difficult to distinguish $\alpha^\dagger$ from the true parameter
$\alpha^*$. The dynamics $\{\widehat \alpha^t\}_{t \geq 0}$ follow a smooth
trajectory to one of $\alpha^\dagger$ or $\alpha^*$, depending on the initial
state $\widehat\alpha^0$.
\end{example}

\begin{figure}[t]
\minipage{0.33\columnwidth}
\xincludegraphics[width=0.85\textwidth,label=(a)]{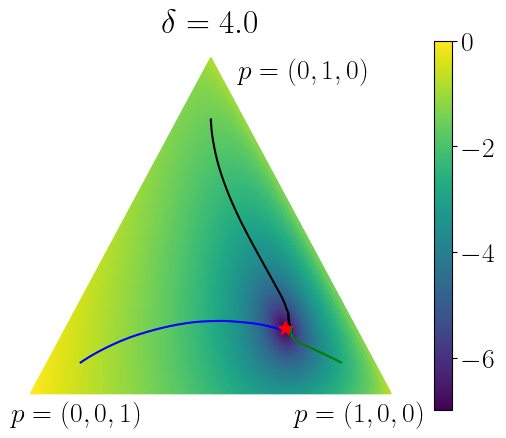}\\
\xincludegraphics[width=0.85\textwidth,label=(d)]{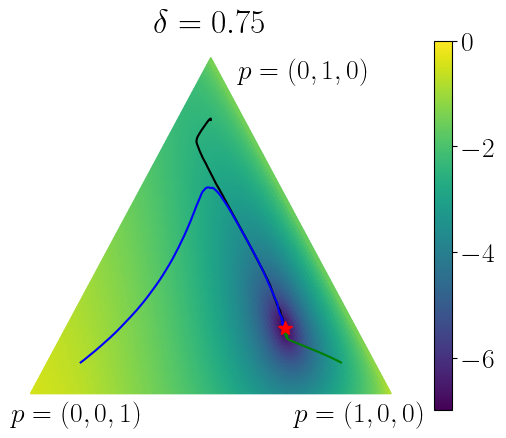}
\endminipage
\minipage{0.33\columnwidth}
\xincludegraphics[width=0.85\textwidth,label=(b)]{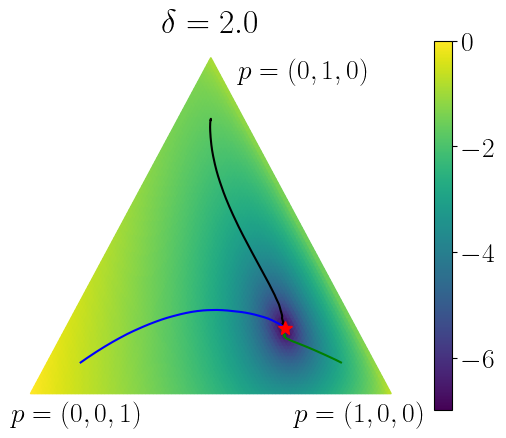}\\
\xincludegraphics[width=0.85\textwidth,label=(e)]{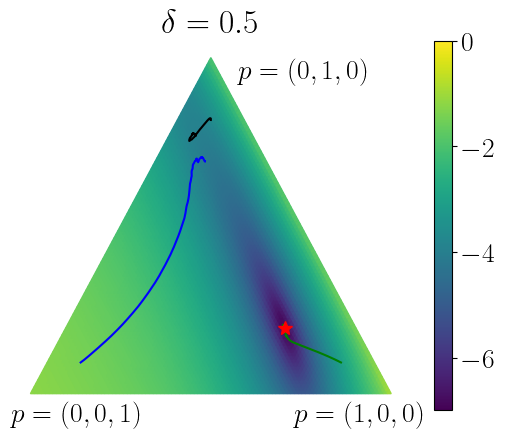}
\endminipage
\minipage{0.33\columnwidth}
\xincludegraphics[width=0.85\textwidth,label=(c)]{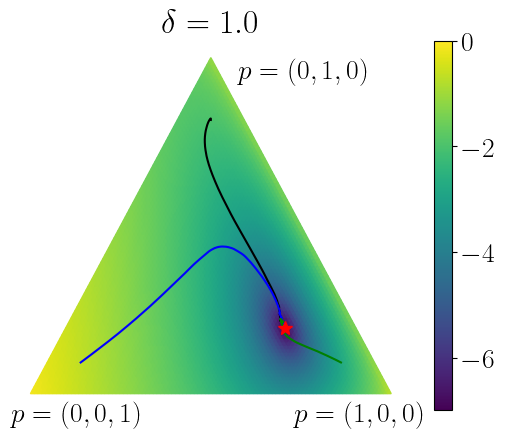}\\
\xincludegraphics[width=0.9\textwidth,label=(f)]{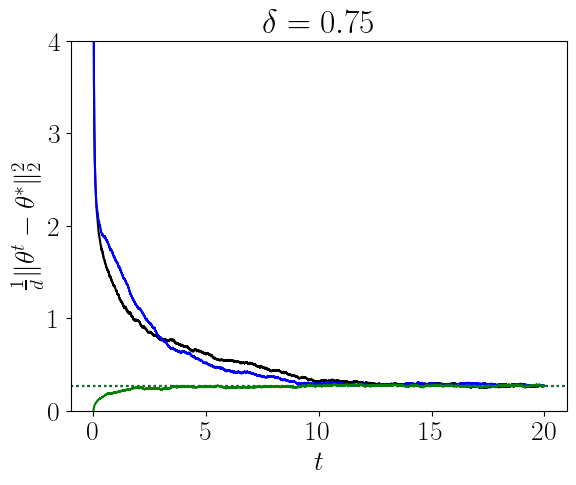}
\endminipage
\caption{Simulations for the Gaussian mixture prior model
$p_1(\alpha)\N(0,0.04)+p_2(\alpha)\N(0,1)+p_3(\alpha)\N(0,25)$ of Example
\ref{ex:gaussianweightmixture}, with true weights
$p(\alpha^*)=(0.6,0.2,0.2)$ and
linear model noise variance $\sigma^2=\delta s^2$ for $s=0.2$.
Empirical Bayes Langevin dynamics are run for two initializations
$\widehat\alpha^0$
with random $\theta_j^0 \overset{iid}{\sim} \N(0,1)$ (black and blue),
and an initialization $\widehat\alpha^0$ near $\alpha^*$
with ground truth $\theta_j^0=\theta_j^*$ (green). The remaining setup is
the same as in Figure \ref{fig:meanmixture}.
(a--e) Landscape of the replica-symmetric
free energy $F(\alpha)$ for $\delta \in \{4,2,1,0.75,0.5\}$,
plotted as
$\log(F(\alpha)-F(\alpha^*)+10^{-3})$ in the coordinates $p(\alpha)$ on the simplex. The unique critical point $p(\alpha^*)$ is
depicted as the red star. Sample paths of $\{p(\widehat \alpha^t)\}_{t \geq 0}$
are shown in green, black, and
blue. (f) Mean-squared-error
$\frac{1}{d}\|\btheta^t-\btheta^*\|_2^2$ across iterations for these same
three initial states, at $\delta=0.75$. The
predicted value of $\mse(\alpha^*)+\mse_*(\alpha^*)$ for a posterior sample is depicted by the
dashed line.}
\label{fig:weightmixture}
\end{figure}

\begin{example}\label{ex:gaussianweightmixture}
Consider a $K+1$-component Gaussian mixture prior
\[g(\theta,\alpha)=\sum_{k=0}^K p_k(\alpha)
\sqrt{\frac{\omega_k}{2\pi}}\exp\Big({-}\frac{\omega_k}{2}(\theta-\mu_k)^2
\Big), \qquad
p_k(\alpha)=\frac{e^{\alpha_k}}{e^{\alpha_0}+\ldots
+e^{\alpha_K}}\]
with fixed means $\mu_0,\ldots,\mu_K$ and variances
$\omega_0^{-1},\ldots,\omega_K^{-1}$, parametrized instead by the mixture
weights $p_k(\alpha)=e^{\alpha_k}/(e^{\alpha_0}+\ldots+e^{\alpha_K})$.
Let us suppose that $g_*(\theta)=g(\theta,\alpha^*)$ for some $\alpha^* \in
\R^{K+1}$, and the parameter pairs $(\mu_0,\omega_0),\ldots,(\mu_K,\omega_K)$
are distinct. We again consider the dynamics driven by
\[\cG(\alpha,\sP)=\E_{\theta \sim \sP}[\nabla_\alpha \log g(\theta,\alpha)]
-\nabla R(\alpha),\]
where $R(\alpha)$ is a regularizer of the form
(\ref{eq:ralpha}--\ref{eq:Ralpha}) such that $\alpha^0,\alpha^* \in \ball(D)$.
This parametrization is overparametrized by a single
parameter --- however, defining the $K$-dimensional linear
subspace $E=\{\alpha \in \R^{K+1}:\alpha_0+\ldots+\alpha_K=0\}$,
a direct calculation (c.f.\ Section \ref{sec:exampledetails}) verifies that
$\nabla_\alpha \log g(\theta,\alpha) \in E$ and $\nabla R(\alpha) \in E$ if $\alpha \in E$. Thus, initializing $\widehat \alpha^0 \in E$ ensures
$\widehat \alpha^t \in E$ for all $t \geq 0$, and we may apply our preceding results upon
identifying $E$ isometrically with $\R^K$.

We verify in Section \ref{sec:exampledetails} that
Assumption \ref{assump:prior}(b) holds.
Then again for fixed $s^2>0$ and all large $\delta$, Proposition \ref{prop:largedelta}
ensures that Assumption \ref{assump:compactalpha} also holds, and there is a 1-to-1 correspondence between the critical points of $F$ and $G_{s^2}(\alpha)$ on $\ball(D)$. Here, $G_{s^2}(\alpha)$
is the negative population log-likelihood of the Gaussian mixture model
\begin{equation}\label{eq:marginalweightmixture}
\sP_{g(\cdot,\alpha),s^{-2}}(y)
=\sum_{k=0}^K p_k(\alpha) \sqrt{\frac{1}{2\pi(\omega_k^{-1}+s^2)}}
\exp\Big({-}\frac{1}{2(\omega_k^{-1}+s^2)}(y-\mu_k)^2\Big).
\end{equation}
Letting
$S=\{(p_0,\ldots,p_K):p_0+\ldots+p_K=1,\,p_0,\ldots,p_K>0\}$ be the open
probability simplex, the mapping $\alpha \in E \mapsto
p(\alpha) \in S$ is a 1-to-1 smooth parametrization with smooth inverse,
and the function $G_{s^2}$ is strictly convex
in the parametrization by $(p_0,\ldots,p_K) \in S$. Thus $p^*=p(\alpha^*) \in S$
is the unique critical point where $\nabla_p G_{s^2}=0$, and the Hessian
$\nabla_p^2 G_{s^2}$ is nonsingular at $p^*$. This implies that
$\alpha^* \in E$ is also the unique critical point where
$\nabla_\alpha G_{s^2}=0$, and $\nabla_\alpha^2 G_{s^2}$ is also non-singular at
$\alpha^*$. So for large $\delta$, Proposition \ref{prop:largedelta} ensures
that $\alpha^*$ must be the unique point of $\Crit \cap\,\ball(D)$.

Figure \ref{fig:weightmixture} depicts the simulated landscape of $F(\alpha)$
and dynamics (\ref{eq:langevin_sde}--\ref{eq:gflow}) in a
scaled-mixture-of-normals model
$p_1(\alpha)\N(0,0.04)+p_2(\alpha)\N(0,1)+p_3(\alpha)\N(0,25)$
with all components having mean 0, across a range of values $\delta \in
[0.5,4]$, and with true weights
$p(\alpha^*)=(0.6,0.2,0.2)$. (We have again verified that the
Almeida-Thouless stability condition (\ref{eq:AT}) holds at each depicted $\delta>0$ and
parameter value $\alpha \in \R^K$ in these figures.)
We observe for all tested values $\delta \in [0.5,4]$ that
$\alpha^*$ is the unique local minimizer and critical point of $F(\alpha)$.
However, as $\delta$ decreases, the landscape of $F(\alpha)$
flattens around $\alpha^*$ along a direction representing a family of priors
$g(\,\cdot\,,\alpha)$ having the same first two moments as $g(\,\cdot\,,\alpha^*)$, reflecting that the
problem of learning $g(\,\cdot\,,\alpha^*)$ beyond its second
moment becomes increasingly ill-conditioned. The
learned parameter $\{\widehat\alpha^t\}_{t \geq 0}$ successfully converges to $\alpha^*$ from several different initial states $\widehat\alpha^0$
when $\delta \geq 0.75$, with mixing of Langevin dynamics becoming increasingly slower as $\delta$ decreases. For $\delta=0.5$, the learned parameter $\{\widehat\alpha^t\}_{t \geq 0}$ fails to converge to $\alpha^*$ under the tested time horizon
from random initializations of
$\btheta^0$, but does converge to
$\alpha^*$ under a ground-truth initialization
$\btheta^0=\btheta^*$ and $\widehat\alpha^0$ close to $\alpha^*$.
\end{example}

%% file: equilibrium.tex
\section{Analysis of approximately-TTI DMFT systems}\label{sec:equilibrium}

In this section, we prove Theorem \ref{thm:dmft_equilibrium} on the equilibrium 
properties of the solution to the DMFT equations under an assumption of
approximate time-translation-invariance (from an out-of-equilibrium
initialization). We assume throughout this section that
Assumptions \ref{assump:model} and \ref{assump:prior}(a) hold,
and that the solution to the DMFT system in Theorem \ref{thm:dmft_approx}(a)
approximating the dynamics (\ref{eq:langevinfixedprior}) with the fixed prior
$g(\cdot)$ is approximately-TTI in the sense of Definition
\ref{def:regular}. We denote by $\{\theta^t\}_{t \geq 0}$,
$\{\eta^t\}_{t \geq 0}$, and $C_\theta,C_\eta,R_\theta,R_\eta$ the components
of this DMFT solution.

\subsection{Analysis of $\theta$-equation}

We first derive, from analysis of the evolution
(\ref{def:dmft_langevin_cont_theta}) for $\{\theta^t\}_{t \geq 0}$,
a representation of
$c_\theta^\tti(0),c_\theta^\tti(\infty),c_\theta(*)$ in terms of
$c_\eta^\tti(0),c_\eta^\tti(\infty)$, assuming a condition
$c_\eta^\tti(0)-c_\eta^\tti(\infty)<\delta/\sigma^2$ which ensures long-time
stability of $\{\theta^t\}_{t \geq 0}$ under
(\ref{def:dmft_langevin_cont_theta}). This condition
will be checked in our subsequent analysis of the evolution of $\{\eta^t\}_{t
\geq 0}$.

\begin{lemma}\label{lem:theta_replica_eq}
Suppose $c_\eta^\tti(0)-c_\eta^\tti(\infty)<\delta/\sigma^2$.
Set $\omega=\delta/\sigma^2-(c_\eta^{\tti}(0)-c_\eta^\tti(\infty))$
and $\omega_\ast=\omega^2/c_\eta^\tti(\infty)$. Then
\begin{align}\label{eq:theta_replica}
c_\theta^{\tti}(0) = \E_{g_\ast,\omega_*} \langle
\theta^2\rangle_{g,\omega},\quad c_\theta^\tti(\infty) =  \E_{g_\ast, \omega_*}
\langle \theta\rangle_{g,\omega}^2, \quad c_\theta(\ast) = \E_{g_\ast,\omega_*}[\langle \theta\rangle_{g,\omega}\theta^\ast].
\end{align}
\end{lemma}

The main idea of the proof is to apply the explicit form of $c_\theta^\tti$
in (\ref{eq:cttiforms}) together with its fluctuation dissipation relation
with $r_\theta^\tti$ in (\ref{eq:crttiFDT}) to
approximate $C_\theta,R_\theta$ at large times by correlation and response
functions $C_\theta^{(M)},R_\theta^{(M)}$ that admit
an interpretation as the effect of marginalization over auxiliary variables
$(x_1^t,\ldots,x_M^t)$ in a \emph{Markovian} joint evolution of
$(\theta^t,x_1^t,\ldots,x_M^t)$ conditional on $\theta^*$.
In contrast to the original high-dimensional
dynamics of $\{\btheta^t\}_{t \geq 0}$ in $\R^d$, here $M$ does not depend on
$(n,d)$, and the dynamics of $\{x_1^t,\ldots,x_M^t\}$ will be decoupled
given $\{\theta^t\}_{t \geq 0}$. This decoupling allows us to provide a
simple explicit form for the $\theta$-marginal of the
stationary distribution of $(\theta,x_1,\ldots,x_M)$ conditional on $\theta^*$,
which in the limit $M \to \infty$ will match the conditional distribution
$\theta \mid \theta^*$ under the limit law
$\sP_{g_*,\omega_*;g,\omega}(\theta,\theta^*)$.

To implement this idea, we will exhibit a coupling of the processes
$\{\theta^t\}_{t \geq 0}$ driven by $C_\theta,R_\theta$ and
$\{\theta_{M,T_0}^t\}_{t \geq 0}$ driven by $C_\theta^{(M)},R_\theta^{(M)}$ from
time $T_0$ onwards, and then analyze the
convergence of $\{\theta_{M,T_0}^t\}_{t \geq 0}$ under the equivalent Markovian
representation of its dynamics. The main technical
challenge is to ensure either that the discretization error $\eps(M)$
obtained by approximating $C_\theta,R_\theta$ by $C_\theta^{(M)},R_\theta^{(M)}$
does not compound exponentially over time, or that the convergence time of
$\{\theta^t_{M,T_0}\}_{t \geq 0}$ in the equivalent Markovian dynamics is independent
of the approximation dimension $M$. We will take the first approach here, by
adapting ideas around sticky and reflection couplings developed in
\cite{eberle2016reflection,eberle2019sticky} to a setting of non-Markovian DMFT 
dynamics for $\{\theta^t\}_{t \geq 0}$ and $\{\theta_{M,T_0}^t\}_{t \geq 0}$.

\subsubsection{Comparison with an auxiliary process}

Let us fix a positive integer $M$ and define two sequences $\{a_m\}_{m=0}^M$
and $\{c_m\}_{m=1}^M$ by
\begin{align}\label{eq:a_grid}
a_m = \iota + \frac{m}{\sqrt{M}} \text{ for } m=0,\ldots,M,
\qquad
\frac{c_m^2}{a_m}=\mu_\eta([a_{m-1}, a_m)) \text{ for } m=1,\ldots,M
\end{align}
where $\mu_\eta$ is given in Definition \ref{def:regular}. We set
\[R_\eta^{(M)}(\tau)=\sum_{m=1}^M c_m^2 e^{-a_m \tau},
\qquad C_\eta^{(M)}(t,s)=\sum_{m=1}^M
\frac{c_m^2}{a_m}(e^{-a_m|t-s|}-e^{-a_m(t+s)})+c_\eta^\tti(\infty).\]
A direct calculation of the covariance shows that
\begin{equation}\label{eq:CetaMinterp}
C_\eta^{(M)}(t,s)=\E[u_M^t u_M^s] \quad \text{ for } \quad
u_M^t=z+\sum_{m=1}^M c_m\int_0^t e^{-a_m(t-s)}\sqrt{2}\,\d b_m^s,
\end{equation}
where $z \sim \N(0,c_\eta^\tti(\infty))$ and
$\{b_1^t\}_{t \geq 0},\ldots,\{b_M^t\}_{t \geq 0}$ are 
standard Brownian motions independent of each other and of $z$. In particular,
$C_\eta^{(M)}(t,s)$ is a positive-semidefinite covariance kernel on
$[0,\infty)$.

For convenience, let us set
\[U(\theta,\theta^\ast)=-\frac{\delta}{\sigma^2}(\theta-\theta^\ast) +
(\log g)'(\theta),\]
so the DMFT equation (\ref{def:dmft_langevin_cont_theta}) reads
 \begin{align}\label{eq:dmft_theta}
 \d\theta^t = \Big[U(\theta^t,\theta^\ast) + \int_0^t R_\eta(t,s)(\theta^s -
\theta^\ast) \d s + u^t\Big]\d t + \sqrt{2}\,\d b^t.
 \end{align}
Let $\{u_M^t\}_{t \geq 0}$ be a centered Gaussian process 
with covariance kernel $C_\eta^{(M)}$, defined in the probability space of
$\{u^t,\theta^t\}_{t \geq 0}$ and independent of $\theta^*$.
Fixing a time $T_0>0$, let $\{\tilde b^t\}_{t \geq T_0}$ be a standard
Brownian motion initialized at $\tilde b^{T_0}=0$,
independent of $\{u^t\}_{t \geq 0}$, $\theta^*$, and $\{\theta^t\}_{t \in
[0,T]}$. We consider
an auxiliary process $\{\theta_{M,T_0}^t\}_{t \geq 0}$ defined by
\begin{equation}\label{eq:dmft_aux_theta}
\begin{aligned}
\theta_{M,T_0}^t&=\theta^t \text{ for } t \in [0,T_0],\\
\d\theta_{M,T_0}^t&=\Big[U(\theta_{M,T_0}^t,\theta^\ast)
+\int_0^t R_\eta^{(M)}(t-s)(\theta_{M,T_0}^s-\theta^\ast)
\d s + u_M^t\Big]\d t + \sqrt{2}\,\d \tilde b^t \text{ for } t \geq T_0.
\end{aligned}
\end{equation}
We proceed to construct a coupling of
$\{u_M^t\}_{t \geq 0}$ with $\{u^t\}_{t \geq 0}$ and 
of $\{\tilde b^t\}_{t \geq T_0}$ with $\{b^t-b^{T_0}\}_{t \geq T_0}$
defining the DMFT solution $\{\theta^t\}_{t \geq 0}$, to yield a coupling of
$\{\theta_{M,T_0}^t\}_{t \geq T_0}$ with $\{\theta^t\}_{t \geq T_0}$.

\begin{lemma}\label{lem:GP_couple}
For any $M,T_0,T>0$, there exists a coupling of
$\{u_M^t\}_{t \geq 0}$ and $\{u^t\}_{t \geq 0}$ such that
\begin{align*}
\sup_{t\in[T_0,T_0+T]} \E(u_M^t - u^t)^2 \leq \eps(M)+\sqrt{T}\,\eps(T_0),
\end{align*}
where $\eps(M)$ does not depend on $T_0,T$ and
$\eps(T_0)$ does not depend on $M,T$, and
$\lim_{M \to \infty} \eps(M)=0$ and $\lim_{T_0 \to \infty} \eps(T_0)=0$.
\end{lemma}
\begin{proof}
Define the covariance kernel $C_\eta^{(\infty)}(t,s)=\int_{\iota}^\infty
\big(e^{-a|t-s|} - e^{-a(t+s)}\big) \mu_\eta(\d a) + c_\eta^\tti(\infty)$
representing the $M \to \infty$ limit of (\ref{eq:CetaMinterp}). We
will couple Gaussian processes with covariance kernels $(C_\eta^{(M)},
C_\eta^{(\infty)})$ and with $(C_\eta^{(\infty)},C_\eta)$ respectively.

\textit{Coupling of $(C_\eta^{(M)}, C_\eta^{(\infty)})$.} Let $M'>M$ be
any positive integer for which $\sqrt{M'}$ is an integer multiple of
$\sqrt{M}$, and let $\{\tilde{a}_m\}_{m=0}^{M'}$ and
$\{\tilde{c}_m\}_{m=1}^{M'}$ be the sequences as defined above with $M'$ in
place of $M$. Note then that the grid points $\{a_j\}_{j=0}^M$
are a subset of the grid points $\{\tilde a_i\}_{i=0}^{M'}$. Let
\begin{align}\label{eq:GP_M_prime}
u_{M'}^t=z+\sum_{i=1}^{M'} \tilde{c}_i\int_0^t e^{-\tilde{a}_i(t-s)}\sqrt{2}\,\d
\tilde b^s_i
\end{align}
where $z \sim \N(0,c_\eta^\tti(\infty))$ and $\{\tilde{b}^t_1\}_{t\geq 0},
\ldots\{\tilde b^t_{M'}\}_{t \geq 0}$ are standard Brownian motions independent
of each other and of $z$. Then (\ref{eq:CetaMinterp}) shows that
$\{u_{M'}^t\}_{t \geq 0}$ has covariance $C_\eta^{(M')}$.
Now, for each $j=1,\ldots,M$, let
\[I_j=\{i:a_{j-1}<\tilde a_i \leq a_j\},
\qquad b^t_j = \sum_{i \in I_j}
\tilde{c}_i \tilde{b}^t_i\bigg/\sqrt{\sum_{i \in I_j} \tilde{c}_i^2}\]
and set
\begin{align*}
u_M^t=z+\sum_{j=1}^M c_j \int_0^t e^{-a_j(t-s)} \sqrt{2}\,\d b^s_j.
\end{align*}
Then $\{b^t_1\}_{t \geq 0},\ldots,\{b^t_M\}_{t \geq 0}$ are standard
Brownian motions independent of each other and of $z$, so
(\ref{eq:CetaMinterp}) shows also that
$\{u^t\}_{t \geq 0}$ is a Gaussian process with covariance $C_\eta^{(M)}$.

We may now bound
\begin{align*}
\E[(u_M^t - u_{M'}^t)^2] &\leq 4\,\E\Big[\Big(\sum_{i:\tilde{a}_i>a_M}
\tilde{c}_i\int_0^t e^{-\tilde{a}_i(t-s)}\d \tilde{b}^s_i\Big)^2\Big]\\
&\hspace{1in} + 4\,
\E\Big[\Big(\sum_{j=1}^M \sum_{i\in I_j} \tilde{c}_i \int_0^t
e^{-\tilde{a}_i(t-s)}\d \tilde{b}^i_s - \sum_{j=1}^M c_j \int_0^t e^{-a_j(t-s)}
\d b^s_j\Big)^2\Big].
\end{align*}
Since $a_M=\iota+\sqrt{M}$, the first term equals $\sum_{i: \tilde{a}_i >\iota+ \sqrt{M}}
\tilde{c}_i^2/\tilde{a}_i=\sum_{i:\tilde a_i>\iota+\sqrt{M}}
\mu_\eta([\tilde a_{i-1},\tilde a_i))$,
which is at most some $\eps_1(M)$ satisfying $\lim_{M \to \infty}
\eps_1(M)=0$, by finiteness of the measure $\mu_\eta$. The second term is
bounded as
\begin{align*}
&\E\Big[\Big(\sum_{j=1}^M \sum_{i\in I_j} \tilde{c}_i \int_0^t
e^{-\tilde{a}_i(t-s)}\d \tilde{b}^s_i - \sum_{j=1}^M c_j \int_0^t e^{-a_j(t-s)}
\d b^s_j\Big)^2\Big]\\
&= \E\Big[\Big(\sum_{j=1}^M \sum_{i\in I_j} \int_0^t \Big(\tilde{c}_i
e^{-\tilde{a}_i(t-s)} - \frac{c_j\tilde{c}_i}{\sqrt{\sum_{\ell\in I_j}
\tilde{c}_\ell^2}}e^{-a_j(t-s)}\Big) \d \tilde{b}^s_i\Big)^2\Big]\\
&= \sum_{j=1}^M\sum_{i\in I_j} \int_0^t \Big(\tilde{c}_i e^{-\tilde{a}_i s} -
\frac{c_j\tilde{c}_i}{\sqrt{\sum_{\ell\in I_j}
\tilde{c}_\ell^2}}e^{-a_j s}\Big)^2 \d s \leq 2(\mathrm{I}+\mathrm{II}),
\end{align*} 
where
\begin{align*}
\mathrm{I} =  \sum_{j=1}^M\sum_{i\in I_j} \int_0^t \tilde{c}_i^2
\big(e^{-\tilde{a}_is} - e^{-a_js}\big)^2 \d s, \qquad \mathrm{II} = \sum_{j=1}^M\sum_{i\in I_j} \int_0^t \tilde{c}_i^2\Big(1-\frac{c_j}{\sqrt{\sum_{\ell\in I_j} \tilde{c}_\ell^2}}\Big)^2e^{-2a_j s}\d s. 
\end{align*}
Let $\Delta=1/\sqrt{M}$ be the spacing of $\{a_j\}_{j=0}^M$. Then,
since $|\tilde a_i-a_j| \leq \Delta$ and $\tilde a_i \leq a_j$
for all $i \in I_j$,
\begin{align*}
\mathrm{I} &\leq \sum_{j=1}^M \sum_{i\in I_j} \int_0^t \tilde{c}_i^2
e^{-2a_{j-1}s}s^2 \Delta^2 \d s \leq \sum_{j=1}^M  \int_0^t \Big(\sum_{i\in I_j}
\frac{\tilde{c}_i^2}{\tilde{a}_i}\Big) a_j e^{-2a_{j-1}s}s^2\Delta^2 \d s
\stackrel{(*)}{=} \sum_{j=1}^M c_j^2\int_0^t e^{-2a_{j-1}s}s^2\Delta^2\d s
\end{align*}
where we use $\sum_{i\in I_j} \tilde{c}_i^2/\tilde{a}_i=
\sum_{i \in I_j} \mu_\eta([\tilde a_{i-1},\tilde a_i))
=\mu_\eta([a_{j-1},a_j))=c_j^2/a_j$ in $(\ast)$. Evaluating this integral, for
an absolute constant $C>0$,
\begin{align*}
\mathrm{I} &\leq C\Delta^2 \sum_{j=1}^M \frac{c_j^2}{a_{j-1}^3} \leq \frac{C\Delta^2}{\iota^2} \sum_{j=1}^M \frac{c_j^2}{a_j}
\leq \frac{C\Delta^2}{\iota^2}\mu_\eta([\iota,\infty)) \leq
\eps_2(M)
\end{align*}
where $\lim_{M \to \infty} \eps_2(M)=0$. For $\mathrm{II}$,
since $c_j^2 = \sum_{\ell\in I_j} \frac{a_j}{\tilde{a}_\ell} \tilde{c}_\ell^2$, we have $|c_j^2 - \sum_{\ell\in I_j} \tilde{c}_\ell^2| \leq \Delta \sum_{\ell\in I_j} \tilde{c}_\ell^2/\tilde{a}_\ell = \Delta c_j^2/a_j$, and hence
\begin{align*}
\mathrm{II} \leq \sum_{j=1}^M \sum_{i\in I_j} \frac{\tilde{c}_i^2}{2a_j}
\frac{(c_j-\sqrt{\sum_{\ell\in I_j} \tilde{c}_\ell^2})^2}{\sum_{\ell\in I_j}
\tilde{c}_\ell^2} \leq \sum_{j=1}^M \frac{(c_j^2 - \sum_{\ell\in I_j}
\tilde{c}_\ell^2)^2}{2a_j c_j^2} \leq \frac{\Delta^2}{2} \sum_{j=1}^M
\frac{c_j^2}{a_j^3} \leq \frac{\Delta^2}{2\iota^2} \sum_{j=1}^M
\frac{c_j^2}{a_j} \leq \eps_3(M)
\end{align*}
where $\lim_{M \to \infty} \eps_3(M)=0$. 
In summary, we have shown that $\sup_{t\geq 0}\E[(u_M^t - u_{M'}^t)^2] \leq
\eps(M)$ for some $\eps(M)\rightarrow 0$ as $M\rightarrow\infty$.

Now note that for any fixed $T_0$ and $T$,
$\{u_{M'}^t\}_{t\in [T_0, T_0+T]}$ has covariance kernel
$C_\eta^{(M')}$ that converges uniformly to $C_\eta^{(\infty)}$ over
$[T_0,T_0+T]$ as $M'\rightarrow\infty$. It is direct to check from its
definitions that $C_\eta^{(\infty)}$ satisfies the condition
(\ref{eq:GP_regularity}) of Lemma \ref{lem:couple_general_cov}.
So by Lemma \ref{lem:couple_general_cov}, there exists a
coupling of $\{u_{M'}^t\}_{t \in [T_0,T_0+T]}$ and a Gaussian process
$\{u_\infty^t\}_{t \in
[T_0,T_0+T]}$ with covariance $\{C_\eta^{(\infty)}(t,s)\}_{s,t \in [T_0,T_0+T]}$
such that $\sup_{t\in [T_0,T_0 + T]} \E (u_{M'}^t - u_\infty^t)^2 \rightarrow 0$
as $M' \rightarrow\infty$. Combining this with the above
bound $\sup_{t\geq 0}\E[(u_M^t - u_{M'}^t)^2] \leq \eps(M)$ and taking $M' \to
\infty$ shows
$\sup_{t\in [T_0,T_0+T]} \E(u_M^t - u_\infty^t)^2 \leq \eps(M)$.

\textit{Coupling of $(C_\eta^{(\infty)},C_\eta)$.} By the approximation
(\ref{eq:Cetaapproxinvariant}) for $C_\eta$ in Definition \ref{def:regular},
we have for any $t \geq s \geq 0$ that
\[|C_\eta(t,s)-C_\eta^{(\infty)}(t,s)|
\leq \eps(s)+\int_\iota^\infty e^{-a(t+s)}\d\mu_\eta(a),\]
so there exists a (different) function $\eps(T_0)$ with $\lim_{T_0 \to \infty}
\eps(T_0)=0$ such that
\[\sup_{s,t \in [T_0,T_0+T]} |C_\eta(t,s)-C_\eta^{(\infty)}(t,s)|
\leq \eps(T_0).\]
Here $C_\eta^{(\infty)}$
satisfies (\ref{eq:GP_regularity}) for a constant $C_0>0$
depending only on $\mu_\eta$,
so by Lemma \ref{lem:couple_general_cov}, there exists a coupling 
of $\{u_\infty^t\}_{t \in [T_0,T_0+T]}$ with $\{u^t\}_{t \in [T_0,T_0+T]}$,
the latter having covariance $\{C_\eta(t,s)\}_{s,t \in [T_0,T_0+T]}$, for which
$\sup_{t\in [T_0,T_0+T]} \E(u_\infty^t-u^t)^2 \leq
C(\sqrt{T\,\eps(T_0)}+\eps(T_0))$
for a constant $C>0$.

Combining these two couplings yields a coupling
of $\{u_M^t\}_{t \in [T_0,T_0+T]}$ with $\{u^t\}_{t \in [T_0,T_0+T]}$ such that
$\sup_{t\in [T_0,T_0+T]} \E(u_M^t-u^t)^2 \leq
\eps(M)+C(\sqrt{T\,\eps(T_0)}+\eps(T_0))$,
and extending this arbitrarily to a full coupling of $\{u_M^t\}_{t \geq 0}$
and $\{u^t\}_{t \geq 0}$ and adjusting the value of $\eps(T_0)$ shows the lemma.
\end{proof}

\begin{lemma}\label{lem:auxiliary_V}
Suppose $c_\eta^\tti(0)-c_\eta^\tti(\infty)<\delta/\sigma^2$. Then
for any $M,T_0,T>0$, there exists a coupling of 
the processes $\{\theta^t\}_{t \geq 0}$ and $\{\theta_{M,T_0}^t\}_{t \geq 0}$
defined by (\ref{eq:dmft_theta}) and (\ref{eq:dmft_aux_theta}) such that
\begin{align*}
 \sup_{t\in[0,T_0+T]}\E|\theta^t-\theta_{M,T_0}^t| \leq
\eps(M)+\sqrt{T}\,\eps(T_0),
 \end{align*}
 where $\eps(M)$ does not depend on $T_0,T$ and
$\eps(T_0)$ does not depend on $M,T$, and
$\lim_{M \to \infty} \eps(M)=0$ and $\lim_{T_0 \to \infty} \eps(T_0)=0$.
\end{lemma}
\begin{proof}
To ease notation, let us write $\tilde \theta^t=\theta_{M,T_0}^t$ and $\tilde
u^t=u_M^t$. We couple $\{u^t\}_{t \geq 0}$ and $\{\tilde u^t\}_{t \geq 0}$ 
according to Lemma \ref{lem:GP_couple}.
By definition, $\{\theta^t\}_{t \in [0,T_0]}$ and
$\{\tilde \theta^t\}_{t \in [0,T_0]}$ coincide up to time $T_0$.

To construct the coupling of $\theta^t$ and $\tilde \theta^t$ for times
$t \in [T_0,T_0+T]$, we adapt the ideas of
\cite{eberle2016reflection,eberle2019sticky}:
Fix some $\eps>0$, and let $h:[0,\infty) \to [0,1]$ be a function
such that $h(0)=0$, $h(x)>0$ for $x>0$, $h(x)=1$ for $x \geq \eps$,
and both $x \mapsto h(x)$ and $x \mapsto \sqrt{1-h(x)^2}$ are Lipschitz.  Let
$\{b^t\}_{t \geq T_0}$ and
$\{\tilde b^t\}_{t \geq T_0}$ be two standard Brownian motions initialized at
$b^{T_0}=\tilde b^{T_0}=0$, independent of each other and of
$\{u^t\}_{t \geq 0}$, $\{\tilde u^t\}_{t \geq 0}$, $\theta^*$, and
$\{\theta^t\}_{t \in [0,T_0]}$. We define a coupling of
$\{\theta^t\}_{t \geq T_0}$ and $\{\tilde \theta^t\}_{t \geq T_0}$ by 
the joint evolutions, for $t \geq T_0$,
\begin{align*}
\d\theta^t&=\Big[U(\theta^t,\theta^*)+\int_0^t R_\eta(t,s)(\theta^s-\theta^*)\d
s+u^t\Big]\d t+h(|\theta^t-\tilde \theta^t|)\sqrt{2}\,\d b^t
+\sqrt{2(1-h(|\theta^t-\tilde \theta^t|)^2)}\,\d \tilde b^t,\\
\d\tilde \theta^t&=\Big[U(\tilde \theta^t,\theta^*)+\int_0^t
R_\eta^{(M)}(t-s)(\tilde \theta^s-\theta^*)\d s+\tilde u^t\Big]\d
t-h(|\theta^t-\tilde \theta^t|)\sqrt{2}\,\d b^t
+\sqrt{2(1-h(|\theta^t-\tilde \theta^t|)^2)}\,\d \tilde b^t.
\end{align*}
Thus the coupling of the Brownian motions defining these processes
is by reflection at times $t \geq T_0$ where $|\theta^t-\tilde \theta^t|
\geq \eps$, and it transitions to a synchronous coupling as $|\theta^t-\tilde
\theta^t| \to 0$. 
L\'evy's characterization of Brownian motion shows that the resulting
marginal laws of $\{\theta^t\}_{t \geq T_0}$ and $\{\tilde \theta^t\}_{t \geq
T_0}$ indeed coincide with those of (\ref{eq:dmft_theta}) and
(\ref{eq:dmft_aux_theta}).

Let us write as shorthand
 \begin{align*}
\xi^t&=\theta^t-\tilde\theta^t\\
v^t&=U(\theta^t,\theta^\ast) + \int_0^t R_\eta(t,s)(\theta^s - \theta^\ast) \d s + u^t\\
\tilde v^t&=U(\tilde \theta^t,\theta^\ast) + \int_0^t
R_\eta^{(M)}(t-s)(\tilde \theta^s-\theta^\ast) \d s + \tilde u^t.
 \end{align*}
We derive a SDE for $|\xi^t|$ that is analogous to \cite[Eq.\
(66)]{eberle2019sticky}:
For any $t \geq T_0$, since $\d\xi^t=(v^t-\tilde v^t)\d t+2\sqrt{2}h(|\xi^t|)\d b^t$,
It\^o's formula yields
\[\d(\xi^t)^2=2\xi^t[(v^t-\tilde v^t)\d t
+2\sqrt{2}h(|\xi^t|)\d b^t]+8h(|\xi^t|)^2\d t.\]
For a small constant $\beta>0$,
let $S_\beta:[0,\infty) \to [0,\infty)$ be a twice continuously-differentiable
approximation to the square root,
satisfying $S_\beta(x)=\sqrt{x}$ for $x \geq \beta$,
$\sup_{0 \leq x \leq \beta} |S_\beta(x)| \leq C$,
$\sup_{0 \leq x \leq \beta} |S_\beta'(x)| \leq C\beta^{-1/2}$,
and $\sup_{0 \leq x \leq \beta} |S_\beta''(x)| \leq C\beta^{-3/2}$
for a universal constant $C>0$. (A specific construction is given
in \cite[Eq.\ (68)]{eberle2019sticky}.) Then again by It\^o's formula,
for any $t \geq T_0$,
\begin{equation}\label{eq:approximatesqrt}
\d S_\beta((\xi^t)^2)
=S_\beta'((\xi^t)^2)\Big[2\xi^t(v^t-\tilde v^t)\d t
+4\sqrt{2}\,\xi^t h(|\xi^t|)\d b^t+8h(|\xi^t|)^2\d t\Big]
+16S_\beta''((\xi^t)^2)(\xi^t)^2 h(|\xi^t|)^2\,\d t.
\end{equation}
We may take the limit $\beta \to 0$ via a dominated convergence argument:
Applying $S_\beta'(x)=x^{-1/2}/2$ for $x \geq \beta$ and
the bound $|S_\beta'(x)| \leq C\beta^{-1/2}$ for $x<\beta$,
we have $|S_\beta'((\xi^t)^2)\xi^t(v^t-\tilde v^t)|
\leq \max(C,1/2)|v^t-\tilde v^t|$. Since $v^t-\tilde v^t$ is continuous and
hence integrable over $[T_0,t]$, by dominated convergence
\[\lim_{\beta \to 0}
\int_{T_0}^t S_\beta'((\xi^t)^2)\xi^t(v^t-\tilde v^t)\d t
=\int_{T_0}^t \lim_{\beta \to 0}
S_\beta'((\xi^t)^2)\xi^t(v^t-\tilde v^t)\d t
=\int_{T_0}^t \frac{\sign(\xi^t)}{2}(v^t-\tilde v^t)\d t.\]
Applying the Lipschitz bound $h(|\xi^t|) \leq h(0)+C|\xi^t|=C|\xi^t|$ and
a similar dominated convergence argument for the other terms of
(\ref{eq:approximatesqrt}), we obtain in the limit $\beta \to 0$
that for $t \geq T_0$,
\[\d|\xi^t|=\sign(\xi^t)(v^t-\tilde v^t)\d t
+2\sqrt{2}\sign(\xi^t)h(|\xi^t|)\d b^t\]
which is the analogue of \cite[Eq.\ (66)]{eberle2019sticky}. (There is no
term corresponding to a local time of $\xi^t$ at 0
that would instead arise under a pure reflection coupling.)

Now let $A:[0,\infty) \rightarrow [0,\infty)$ be any continuously-differentiable
function, and let $f:[0,\infty) \to [0,\infty)$ be any
continuously-differentiable function with absolutely
continuous first derivative (for which It\^o's formula applies,
c.f.\ \cite[Theorem 71]{protter2005stochastic}),
and satisfying $f'(r) \in [0,1]$ and $f''(r) \leq 0$ for all $r \geq 0$. Set
\[r^t=|\xi^t|+\int_0^t A(t-s) |\xi^s|\d s.\]
Then $\d r^t=\d |\xi^t|+[A(0)|\xi^t|+\int_0^t A'(t-s)|\xi^s|\d s]\,\d t$.
Applying It\^o's formula and taking expectations gives, for $t \geq T_0$,
\begin{equation}\label{eq:dfrt}
\frac{\d}{\d t} \E f(r^t)
=\E\Big[f'(r^t)\Big(\sign(\xi^t)(v^t-\tilde v^t)+A(0)|\xi^t|+\int_0^t
A'(t-s)|\xi^s|\d s\Big) + 4f''(r^t)h(|\xi^t|)^2\Big].
\end{equation}
Let us define $\kappa:[0,\infty) \to \R$ by
\begin{equation}\label{eq:kappadef}
\kappa(r)=\inf\Big\{\frac{{-}(\log g)'(x)+(\log
g)'(y)}{x-y}:|x-y|=r\Big\}
\end{equation}
so that $[{-}(\log g)'(\theta^t)+(\log g')(\tilde \theta^t)]/\xi^t \geq
\kappa(|\xi^t|)$. Let us set also
 \begin{align}\label{eq:Delta_t}
 \Delta_t =  \int_0^t \Big(\big|R_\eta(t,s) -
R_\eta^{(M)}(t-s)\big| \cdot \E|\theta^s-\theta^*|\Big)\d s +\E|u^t - \tilde u^t|.
 \end{align}
Then, under our assumption $f'(r) \in [0,1]$, we have the bound
 \begin{align*}
&\E\Big[f'(r^t)\sign(\xi^t)(v^t-\tilde v^t)\Big]\\
&=\E\Big[f'(r^t)\sign(\xi^t)\Big({-}\frac{\delta}{\sigma^2}\xi^t-
\big({-}(\log g)'(\theta^t)+(\log g)'(\tilde \theta^t)\big) \\
&\hspace{0.5in} + \int_0^t
R_\eta^{(M)}(t-s)(\theta^s-\tilde \theta^s)\d s
+\int_0^t \big(R_\eta(t,s) -
R_\eta^{(M)}(t-s)\big)(\theta^s-\theta^*)\d s + (u^t - \tilde u^t)\Big)\Big]\\
 &\leq -\frac{\delta}{\sigma^2}\E[f'(r^t)|\xi^t|] -
\E[f'(r^t)\kappa(|\xi^t|)|\xi^t|
] + \E\Big[f'(r^t)\int_0^t R_\eta^{(M)}(t-s)|\xi^s|\d s\Big] + \Delta_t.
 \end{align*}
Applying this to (\ref{eq:dfrt}), for all $t \geq T_0$,
\begin{align}
\frac{\d}{\d t} \E f(r^t)
&\leq \E\Big[{-}\Big(\frac{\delta}{\sigma^2}-A(0)\Big)f'(r^t)|\xi^t|
-f'(r^t)\kappa(|\xi^t|)|\xi^t|+4f''(r^t)h(|\xi^t|)^2\Big] \notag\\
&\hspace{1in}
+\E\Big[f'(r^t)\int_0^t \big(A'(t-s) + R_\eta^{(M)}(t-s)\big)|\xi^s|\d s\Big]
+ \Delta_t.\label{eq:dEfrt}
\end{align}

Let us now choose the functions $A(\cdot)$ and $f(\cdot)$.
For some small enough $c_0 \in (0,\iota)$, let
\[A(0)=\frac{\delta}{\sigma^2}-c_0, \qquad A(\tau)=A(0)e^{-c_0\tau}
-\int_0^\tau e^{-c_0(\tau-s)}\Big(\sum_{m=1}^M c_m^2 e^{-a_m s}\Big)\d s.\]
This choice of $A(\tau)$ satisfies
$A'(\tau)=-c_0A(\tau)-\sum_{m=1}^M c_m^2 e^{-a_m \tau}$, i.e.\
\begin{equation}\label{eq:Arelation}
A'(\tau)+R_\eta^{(M)}(\tau)=-c_0A(\tau).
\end{equation}
We will require that $A(\tau) \geq 0$ for all $\tau \geq 0$.
To check this condition, observe that explicitly evaluating the integral
defining $A(\tau)$ yields
\[e^{c_0\tau}A(\tau)=A(0)-\sum_{m=1}^M
\frac{c_m^2}{a_m-c_0}\Big(1-e^{-(a_m-c_0)\tau}\Big)
\geq \frac{\delta}{\sigma^2}-c_0-\sum_{m=1}^M \frac{c_m^2}{a_m-c_0}
\geq \frac{\delta}{\sigma^2}-c_0-\sum_{m=1}^M \frac{c_m^2}{a_m} \cdot
\frac{\iota}{\iota-c_0},\]
the last inequality using $a_m \geq \iota \geq c_0$. Further bounding
$\sum_{m=1}^M c_m^2/a_m=\sum_{m=1}^M \mu_\eta([a_{m-1},a_m))
\leq \mu_\eta([\iota,\infty))=c_\eta^\tti(0)-c_\eta^\tti(\infty)$, this shows
$e^{c_0\tau}A(\tau) \geq \frac{\delta}{\sigma^2}-c_0-\frac{\iota}{\iota-c_0}
(c_\eta^\tti(0)-c_\eta^\tti(\infty))$.
Then by the given assumption that
$c_\eta^\tti(0)-c_\eta^\tti(\infty)<\delta/\sigma^2$, we obtain $A(\tau) \geq 0$
for a sufficiently small choice of $c_0 \in (0,\iota)$ and all $\tau \geq 0$,
as desired. Applying (\ref{eq:Arelation}) and $A(0)=\delta/\sigma^2-c_0$ into
(\ref{eq:dEfrt}), and recalling the definition
$r^t=|\xi^t|+\int_0^t A(t-s)|\xi^s|\d s$, we get for all $t \geq T_0$ that
\begin{equation}\label{eq:dEfrtbound}
\frac{\d}{\d t} \E f(r^t)\leq \E\Big[\underbrace{{-}c_0f'(r^t)r^t-f'(r^t)
\kappa(|\xi^t|)|\xi^t| + 4f''(r^t)h(|\xi^t|)^2}_{:=F(r^t,\xi^t)}\Big] + \Delta_t.
\end{equation}

We next proceed to bound the above quantity $\E[F(r^t,\xi^t)]$.
Observe that under the convexity-at-infinity condition for ${-}\log g(\theta)$
in Assumption \ref{assump:prior}(a), there must exist
constants $R_0,\kappa_0>0$ for which
\begin{equation}\label{eq:kappaconditions}
\kappa(r) \geq {-}\kappa_0 \text{ for all } r \geq 0,
\qquad \kappa(r) \geq 0 \text{ for all } r \geq R_0.
\end{equation}
Let us denote $\kappa_-(r)=\max(-\kappa(r),0)$.
Then ${-}\kappa(r)r \leq \kappa_-(r)r$ and
$\kappa_-(r)r \in [0,\kappa_0R_0]$ for all $r \geq 0$.
Recall the constant $\eps>0$ for which $h(x)=1$ when $x \geq \eps$,
and define $K:(\eps,\infty) \to [0,\kappa_0R_0]$ by
\[K(r)=\sup_{t \geq T_0} \E\Big[\kappa_-(|\xi^t|)|\xi^t| \;\Big|\;
r^t=r,\,|\xi^t|>\eps\Big].\]
Then define $f:\R \to \R$ by
\[f(0)=0, \qquad f'(r)=\exp\Big({-}\frac{1}{4}\int_0^{\max(r,2\kappa_0R_0/c_0)}
K(s)ds\Big) \text{ for } r \geq 0.\]
Note that $f'(r)$ is absolutely continuous as required, with
$f'(r) \in [c_1,1]$ for all $r \geq 0$ where
$c_1=\exp({-}\frac{\kappa_0^2R_0^2}{2c_0})$, and
$f''(r)={-}\frac{1}{4}K(r)f'(r)\1\{r<2\kappa_0R_0/c_0\} \leq 0$.
By these definitions, for any $r>\eps$ and $t \geq T_0$, we have
\begin{align*}
\E[F(r^t,\xi^t) \mid r^t=r,|\xi^t|>\eps]
&\leq \E\Big[{-}c_0f'(r^t)r^t+f'(r^t)\kappa_-(|\xi^t|)|\xi^t|
+4f''(r^t)h(|\xi^t|)^2 \;\Big|\; r^t=r,|\xi^t|>\eps\Big]\\
&\leq \E\Big[{-}c_0f'(r)r+f'(r)K(r)+4f''(r)
\;\Big|\; r^t=r,|\xi^t|>\eps\Big].
\end{align*}
When $r \geq 2\kappa_0R_0/c_0$ we may apply $f''(r)=0$
and $K(r) \leq \kappa_0R_0 \leq c_0r/2$, to bound this above by $-(c_0/2)f'(r)r$.
When $r \in (\eps,2\kappa_0R_0/c_0)$ we may instead apply
$f''(r)={-}\frac{1}{4}K(r)f'(r)$ to see that this equals $-c_0f'(r)r$.
Thus for all $r>\eps$ and $t \geq T_0$,
\[\E[F(r^t,\xi^t) \mid r^t=r,|\xi^t|>\eps]
\leq {-}(c_0/2)f'(r)r.\]
For any $r \geq 0$, on the event $|\xi^t| \leq \eps$ (which occurs with
probability 1 when $r^t \leq \eps$ since $A(t) \geq 0$),
let us use $f''(r^t)h(|\xi^t|)^2 \leq 0$ and
${-}f'(r^t)\kappa(|\xi^t|)|\xi^t| \leq \eps\kappa_0$ to bound
\[\E[F(r^t,\xi^t) \mid r^t=r,|\xi^t| \leq \eps]
\leq {-}c_0f'(r)r+\eps \kappa_0.\]
Combining these cases and taking the full expectation over
$\1\{|\xi^t|>\eps\}$ and over $r^t$, we get for all $t \geq T_0$ that
\[\E[F(r^t,\xi^t)] \leq {-}(c_0/2)\E[f'(r^t)r^t]+\eps \kappa_0.\]
Applying $f'(r^t) \geq c_1$ and $r^t \geq f(r^t)$ and putting this bound into
(\ref{eq:dEfrtbound}), for all $t \in [T_0,T_0+T]$,
\[\frac{\d}{\d t} \E f(r^t)
\leq {-}(c_0c_1/2)\E f(r^t)+\eps\kappa_0+\max_{t \in [T_0,T_0+T]} \Delta_t.\]
Since $f(r^{T_0})=f(0)=0$, this differential inequality yields for
all $t \in [T_0,T_0+T]$,
\[\E f(r^t) \leq \Big(\eps\kappa_0+\max_{t \in [T_0,T_0+T]} \Delta_t\Big)
\frac{1-e^{-(c_0c_1/2)(t-T_0)}}{c_0c_1/2}
\leq \frac{2}{c_0c_1}\Big(\eps\kappa_0+\max_{t \in [T_0,T_0+T]} \Delta_t\Big).\]
Since also $r^t \leq c_1f(r^t)$ from the lower bound $f'(r) \geq c_1$, this
gives $\E r^t \leq (2/c_0)(\eps \kappa_0+\max_{t \in [T_0,T_0+T]} \Delta_t)$.
Applying that $A(t) \geq 0$ for all $t \geq 0$, we have
$|\xi^t| \leq r^t$, so this gives finally
\[\max_{t \in [T_0,T_0+T]}
\E|\theta^t-\tilde \theta^t|=\max_{t \in [T_0,T_0+T]}
\E|\xi^t| \leq (2/c_0)\Big(\eps\kappa_0+\max_{t \in [T_0,T_0+T]} \Delta_t\Big).\]

We may choose $\eps$ such that
$\eps\kappa_0 \leq \max_{t \in [T_0,T_0+T]} \Delta_t$. Thus,
to conclude the proof, it remains to show 
\begin{align}\label{eq:Delta_bound}
\sup_{t \in [T_0,T_0+T]} \Delta_t \leq \eps(M)+\sqrt{T}\,\eps(T_0).
\end{align}
We note that under Definition \ref{def:regular},
$\E|\theta^t-\theta^\ast| \leq [\E(\theta^t-\theta^\ast)^2]^{1/2}
=(C_\theta(t,t)-2C_\theta(t,*)+\E{\theta^*}^2)^{1/2} \leq C$
for a constant $C>0$ and all $t \geq 0$. Then for the first term of $\Delta_t$,
by property (\ref{eq:Retaapproxinvariant}) of Definition \ref{def:regular},
\begin{align*}
&\int_0^t |R_\eta(t,s)-R_\eta^{(M)}(t-s)|
\cdot \E|\theta^s-\theta^\ast|\,\d s\\
 &\leq C\int_0^t |R_\eta(t,s) -
r^{\tti}_\eta(t-s)|\d s + C\int_0^t |r_\eta^{\tti}(t-s)-R_\eta^{(M)}(t-s)|\d s\\
&\leq C\eps(t) + \int_0^t |r_\eta^{\tti}(t-s) - R_\eta^{(M)}(t-s)|\d s
\end{align*}
where here $\lim_{t \to \infty} \eps(t)=0$.
Recalling the sequences $\{a_m\}_{m=0}^M,\{c_m\}_{m=1}^M$ defining
$R_\eta^{(M)}$,
\begin{align*}
r_\eta^{\tti}(\tau)-R_\eta^{(M)}(\tau) &= \int_\iota^\infty
ae^{-a\tau}\mu_\eta(\d a) - \sum_{m=1}^M c_m^2e^{-a_m\tau}\\
&= \sum_{m=1}^M \int_{a_{m-1}}^{a_m} (ae^{-a\tau} - a_me^{-a_m\tau})\mu_\eta(\d a) + \int_{a:a>a_M} ae^{-a\tau} \mu_\eta(\d a),
\end{align*}
hence using the fact that $h(a) = ae^{-a\tau}$ satisfies $|h'(a)| \leq
2e^{-a\tau/2}$ and $|a_m - a_{m-1}|=1/\sqrt{M}$,
\begin{align*}
\int_0^t \big|r_\eta^{\tti}(\tau) - R_\eta^{(M)}(\tau)\big|\d\tau &\leq
\sum_{m=1}^M\int_{a_{m-1}}^{a_m} \Big(\int_0^t |ae^{-a\tau} - a_me^{-a_m\tau}|\d \tau\Big) \mu_\eta(\d a) + \int_{a:a>a_M} \Big(\int_0^t ae^{-a\tau}\d\tau\Big) \mu_\eta(\d a)\\
&\leq \sum_{m=1}^M\int_{a_{m-1}}^{a_m} \Big(\int_0^t
(2/\sqrt{M})e^{-a_{m-1}\tau/2}\d \tau\Big) \mu_\eta(\d a) + \mu_\eta([a_M,\infty))\\
&\leq \frac{4}{\sqrt{M}}\sum_{m=1}^M \int_{a_{m-1}}^{a_m} \frac{1}{a_{m-1}}
\mu_{\eta}(\d a) + \mu_\eta([a_M,\infty)) \\
&\leq \frac{4}{\iota\sqrt{M}}
\mu_\eta([\iota,\sqrt{M})) + \mu_\eta([\sqrt{M},\infty)) \leq \eps(M),
\end{align*}
where $\lim_{M \to \infty} \eps(M)=0$. This bounds the first term of
$\Delta_t$ by $\eps(M)+Ct \cdot \eps(t)$. Bounding also the second term
$\E|u^t-\tilde u^t|$ of $\Delta_t$ by Lemma \ref{lem:GP_couple}, we have
\[\sup_{t \in [T_0,T_0+T]} \Delta_t
\leq \eps(M)+\sqrt{T}\,\eps(T_0)+\sup_{t \in [T_0,T_0+T]}
Ct \cdot \eps(t)\]
which implies (\ref{eq:Delta_bound}) upon adjusting $\eps(T_0)$. This
completes the proof.
\end{proof}

Adapting part of the previous argument, we record here a uniform bound on
$\E(\theta^t)^4$ for the solution $\{\theta^t\}_{t \geq 0}$ of the DMFT
equation.

\begin{lemma}\label{lem:dmft_theta_uniform}
Suppose $c_\eta^\tti(0)-c_\eta^\tti(\infty)<\delta/\sigma^2$. Then
$\sup_{t\geq 0} \E (\theta^t)^4 \leq C$ 
and $\sup_{t \geq 0} \E (\theta_{M,T_0}^t)^4 \leq C$ for a constant $C>0$ and
all $M,T_0>0$.
 \end{lemma}
 \begin{proof}
We prove the statement for $\{\theta^t\}_{t \geq 0}$.
Let $A:[0,\infty) \rightarrow [0,\infty)$ be defined by
\[A(0)=\frac{\delta}{\sigma^2}-c_0,\qquad
A(\tau)=A(0)e^{-c_0\tau}-\int_0^t e^{-c_0(\tau-s)}r^\tti_\eta(s)\d s\]
for a small enough constant $c_0 \in (0,\iota)$. Here, by the conditions of
Definition \ref{def:regular},
$r_\eta^\tti(s)={-}{c_\eta^\tti}'(s)=\int_\iota^\infty ae^{-as}\d\mu_\eta(a)$,
and the same argument as in the preceding proof verifies that $\inf_{\tau \in
[0,\infty)} A(\tau)$ is bounded below by a positive constant
for a sufficiently small choice of $c_0 \in (0,\iota)$ and all $\tau \geq 0$.

Let $f:\R \to [0,\infty)$ be a smooth approximation to the absolute value,
satisfying $f(x)=|x|$ for all $|x| \geq 1$, and
$1+f'(x) \cdot x \geq f(x) \geq |x|$, $|f'(x)| \leq 1$, and $|f''(x)| \leq C$
for all $x \in \R$ and an absolute constant $C>0$.
Let $\bar{\theta}^t = \theta^t - \theta^\ast$, and set
$r^t = f(\bar{\theta}^t)+\int_0^t A(t-s)f(\bar{\theta}^s)\d s$.
Then by the DMFT equation (\ref{def:dmft_langevin_cont_theta}) and It\^o's
formula,
\begin{align*}
 \d \bar{\theta}^t&= \Big[{-}\frac{\delta}{\sigma^2} \bar{\theta}^t + (\log
g)'(\theta^t) + \int_0^t R_\eta(t,s)\bar{\theta}^s \d s + u^t\Big]\d t +
\sqrt{2}\,\d b^t,\\
\d r^t &= f'(\bar{\theta}^t)\d\bar\theta^t
+f''(\bar\theta^t)\d t+\Big[A(0)f(\bar{\theta}^t) + \int_0^t
A'(t-s)f(\bar{\theta}^s)\d s\Big]\d t,\\
\d (r^t)^4 &= 4(r^t)^3\d r^t+12(r^t)^2f'(\bar\theta^t)^2\d t.
 \end{align*}
Applying $r^t \geq 0$ and
the bounds $f'(\bar\theta^t)\bar\theta^t \geq f(\bar\theta^t)-1$,
$|f'(\bar\theta^t)| \leq 1$, $|\bar\theta^s| \leq f(\bar\theta^s)$, and
$|f''(\bar\theta^t)| \leq C$ from the definition of $f(\cdot)$, this gives
\begin{align*}
\frac{\d}{\d t}\E(r^t)^4 &\leq \E\Bigg[4(r^t)^3
\Bigg({-}\frac{\delta}{\sigma^2}[f(\bar\theta^t)-1]
+f'(\bar\theta^t)(\log g)'(\theta^t)+\int_0^t |R_\eta(t,s)|f(\bar\theta^s)\d s
+|u^t|+C \\
&\hspace{1in}+A(0)f(\bar\theta^t)+\int_0^t A'(t-s)f(\bar\theta^s)\d s\Bigg)
+12(r^t)^2\Bigg].
\end{align*}
Then, using $A(0)=\delta/\sigma^2-c_0$
and $A'(t-s)+r_\eta^\tti(t-s)={-}c_0A(t-s)$ from the definition of $A(\cdot)$,
\[\frac{\d}{\d t}\E(r^t)^4 \leq \E\Bigg[4(r^t)^3
\Bigg({-}c_0r^t+\frac{\delta}{\sigma^2}
+f'(\bar\theta^t)(\log g)'(\theta^t)+\int_0^t |R_\eta(t,s)-r_\eta^\tti(t-s)|
f(\bar\theta^s)\d s +|u^t|+C\Bigg)+12(r^t)^2\Bigg].\]
When $|\bar\theta^t| \geq 1$, we must have
$f'(\bar\theta^t)=\sign(\bar\theta^t)=|\bar\theta^t|/(\theta^t-\theta^*)$. 
Recalling the function $\kappa(r)$ from (\ref{eq:kappadef}), let us bound in
this case
\[f'(\bar\theta^t)[(\log g)'(\theta^t)-(\log g)'(\theta^*)]
=-|\bar\theta^t| \cdot \frac{-(\log g)'(\theta^t)+(\log
g)'(\theta^*)}{\theta^t-\theta^*}
\leq -\kappa(|\bar\theta^t|)|\bar\theta^t| \leq \kappa_0R_0,\]
where $\kappa_0R_0$ is the deterministic upper bound for ${-}\kappa(r)r$.
For $|\bar\theta^t| \leq 1$, let us apply instead the Lipschitz bound
$|f'(\bar\theta^t)[(\log g)'(\theta^t)-(\log g)'(\theta^*)]|
\leq L$ where $L$ is the Lipschitz constant of $(\log g)'$ under Assumption
\ref{assump:prior}(a). 
We also apply $|R_\eta(t,s)-r_\eta^\tti(t-s)| \leq \eps(t)$ from Definition
\ref{def:regular}, and $\int_0^t f(\bar \theta^s)\d s
\leq r^t/\inf_{\tau \in [0,\infty)} A(\tau)$ by the definition of $r^t$,
where we recall that
$\inf_{\tau \in [0,\infty)} A(\tau)$ is bounded below by a positive constant.
Thus, for some constant $C'>0$, this yields
 \begin{align*}
\frac{\d}{\d t}\E (r^t)^4 \leq {-}4c_0 \E (r^t)^4+C'\,\E\Big[\eps(t)(r^t)^4
+(r^t)^3(1+|(\log g)'(\theta^*)|+|u^t|)+(r^t)^2\Big].
\end{align*}
Since $u^t$ is a centered Gaussian variable,
we note that $\E(u^t)^4=3[\E(u^t)^2]^2=3C(t,t)^2$ which
is bounded uniformly for all $t \geq 0$
under Definition \ref{def:regular}. Also $\E[|(\log g)'(\theta^*)|^4]$ is finite
by the Lipschitz continuity of $(\log g)'$ and finiteness of moments
of $\theta^*$ under Assumption \ref{assump:model}. Then by H\"older's inequality,
$\E[(r^t)^3(1+|(\log g)'(\theta^*)|+|u^t|)+(r^t)^2] \leq C(\E[(r^t)^4]^{3/4}
+\E[(r^t)^4]^{1/2})$ for some $C>0$. Thus, for some $C,T,R>0$ sufficiently large
depending on $C',c_0$, the above implies
  \begin{align*}
  \frac{\d}{\d t} \E (r^t)^4 \leq C\,\E(r^t)^4,
\qquad \frac{\d}{\d t} \E (r^t)^4 \leq {-}4c_0\E(r^t)^4+c_0\E(r^t)^4<0
\text{ whenever } t \geq T \text{ and } \E(r^t)^4 \geq R.
  \end{align*}
This implies that $\sup_{t\geq 0} \E (r^t)^4$ is bounded by a constant depending
only on $C,T,R$. Then $\sup_{t \geq 0} \E (\theta^t)^4$ is also bounded since
$\theta^t=\bar\theta^t+\theta^*$ and
$|\bar\theta^t| \leq f(\bar\theta^t) \leq r^t$.

The argument to bound $\E (\theta_{M,T_0})^4$ is the same upon replacing
$R_\eta(t,s)$ and $u^t$ by $R_\eta^{(M)}(t,s)$ and $u_M^t$ for $s,t \geq T_0$,
and we omit the details for brevity.
 \end{proof}

\subsubsection{Convergence of the auxiliary process}

Extending the definition (\ref{eq:thetalimitlaw}) of
$\sP_{g_*,\omega_*;g,\omega}$, let
$\sP_{g_\ast,\omega_*;g,\omega}^{\otimes 2}$
denote the law of a triple $(\theta^*,\theta,\theta')$ where $\theta^*,\theta$
are generated according to (\ref{eq:scalarchannelprocess}) defining
$\sP_{g_*,\omega_*;g,\omega}$
and $\theta'$ is a second independent copy of $\theta$ drawn from the posterior
measure conditional on $y$.

\begin{lemma}\label{lem:aux_converge}
Suppose $c_\eta^\tti(0)-c_\eta^\tti(\infty)<\delta/\sigma^2$.
Fix any $M,T_0>0$, set $\omega^{(M)} = \delta/\sigma^2 -
\sum_{m=1}^M c_m^2/a_m$ and
$\omega_\ast^{(M)} = (\omega^{(M)})^2/c_\eta^\tti(\infty)$,
and let $\{\theta_{M,T_0}^t\}_{t \geq 0}$ be the
process (\ref{eq:dmft_aux_theta}). Then for any $T,T'>0$,
\begin{align*}
W_1\Big(\sP(\theta^\ast,\theta_{M,T_0}^{T_0+T},\theta_{M,T_0}^{T_0+T+T'}),
\sP_{g_\ast,\omega_\ast^{(M)};g,\omega^{(M)}}^{\otimes 2}\Big)
\leq C\sqrt{M}(e^{-cT}+e^{-cT'})
\end{align*}
for some constants $C,c>0$ not depending on $M,T_0,T,T'$.
\end{lemma}
\begin{proof}
Let $z \sim \N(0,c_\eta^\tti(\infty))$ and let
$\{\tilde{b}^t\}_{t \geq T_0}$ and $\{b^t_m\}_{t \geq 0}$ for
$m=1,\ldots,M$ be $M+1$ standard Brownian motions. These are all independent of
each other, of $\theta^*$, and of $\{\theta^t\}_{t \in [0,T]}$.
We note that the law of $\{\theta_{M,T_0}^t\}_{t \geq 0}$ 
defined by (\ref{eq:dmft_aux_theta}) coincides with the
marginal law of $\{\theta_{M,T_0}^t\}_{t \geq 0}$ in the joint process
\begin{align}
\theta^t_{M,T_0}&=\theta^t \text{ for } t\in [0,T_0],\notag\\
\d\theta^t_{M,T_0}&=\Big[{-}\frac{\delta}{\sigma^2}(\theta^t_{M,T_0}-\theta^\ast)+(\log g)'(\theta_{M,T_0}^t)+z+\sum_{m=1}^M c_m
x^t_m\Big]\d t+\sqrt{2}\,\d \tilde{b}^t \text{ for }
t>T_0,\label{eq:thetaauxevolve}\\
\d x^t_m &= [{-}a_m x^t_m + c_m(\theta^t_{M,T_0}-\theta^\ast)]\d t
 + \sqrt{2}\,\d b^t_m \text{ for } 1\leq m\leq M,\, t \geq 0 \label{eq:vevolve}
\end{align}
with initial conditions $x_1^0=\ldots=x_M^0=0$. Indeed, 
given $\{\theta_{M,T_0}^t\}_{t \geq 0}$, the equations (\ref{eq:vevolve}) for
$\{x_m^t\}_{t \geq 0}$ are linear and have the explicit solutions
\begin{equation}\label{eq:vmexplicit}
x_m^t=c_m \int_0^t e^{-a_m(t-s)}(\theta_{M,T_0}^s-\theta^*)\d s
+\int_0^t e^{-a_m(t-s)}\sqrt{2}\,\d b_m^s.
\end{equation}
Substituting these solutions into (\ref{eq:thetaauxevolve}) gives
(\ref{eq:dmft_aux_theta}), upon identifying
$u_M^t=z+\sum_{m=1}^M c_m\int_0^t e^{-a_m(t-s)}\sqrt{2}\,\d b_m^s$.
Here $\{u_M^t\}_{t \geq 0}$ is a centered Gaussian process independent of
$\theta^*$ and $\{\tilde b^t\}_{t \geq T_0}$, with covariance kernel exactly
$C_\eta^{(M)}(t,s)$ by (\ref{eq:CetaMinterp}). Thus the marginal law of
$\{\theta_{M,T_0}^t\}_{t \geq 0}$ coincides with the definition in
(\ref{eq:dmft_aux_theta}).

For $t \geq T_0$, let $x^t=(\theta_{M,T_0}^t,x_1^t,\ldots,x_M^t)$
and $b^t=(\tilde b^t,b_1^t,\ldots,b_M^t)$.
Conditional on $\theta^*$ and $z$,
the evolution of $\{x^t\}_{t \geq T_0}$ defined by
(\ref{eq:thetaauxevolve}--\ref{eq:vevolve}) is a standard (Markovian)
Langevin diffusion given by
\[\d x^t={-}\nabla H(x^t \mid \theta^*,z)\d t+\sqrt{2}\,\d b^t\]
with Hamiltonian
\begin{align}
H(x \mid \theta^*,z)&=H(\theta,x_1,\ldots,x_M \mid \theta^*,z)\\
&=\frac{1}{2}\Big(\underbrace{\frac{\delta}{\sigma^2}-\sum_{m=1}^{M}\frac{c_m^2}{a_m}}_{=\omega^{(M)}}\Big)
(\theta-\theta^\ast)^2 - \log g(\theta) - z\,\theta+\sum_{m=1}^{M}
\frac{a_m}{2}\Big(\frac{c_m}{a_m}(\theta-\theta^\ast)-x_m\Big)^2\notag\\
&=\frac{\omega^{(M)}}{2}(\theta-\theta^*-z/\omega^{(M)})^2
-\log g(\theta)+\sum_{m=1}^{M}
\frac{a_m}{2}\Big(\frac{c_m}{a_m}(\theta-\theta^\ast)-x_m\Big)^2+\text{const.},
\label{eq:thetaauxHamiltonian}
\end{align}
for an additive constant not depending on $x=(\theta,x_1,\ldots,x_M)$.
Note that the given condition
$c_\eta^\tti(0)-c_\eta^\tti(\infty)<\delta/\sigma^2$ implies that
$\omega^{(M)}>0$ strictly.

Convergence of $\{x^t\}_{t \geq T_0}$ in Wasserstein-1 to a stationary law
then follows from the results of
\cite{eberle2016reflection}: For any $x=(\theta,x_1,\ldots,x_M)$ and
$x'=(\theta',x_1',\ldots,x_M')$, we have
\[(x-x')^\top (\nabla H(x \mid \theta^*,z)-\nabla H(x' \mid \theta^*,z))
=(\theta-\theta')({-}(\log g)'(\theta)+(\log g)'(\theta'))
+(x-x')^\top L(x-x')\]
where
\[L=\begin{pmatrix} \frac{\delta}{\sigma^2} & {-}c_1 & \ldots & {-}c_M \\
{-}c_1 & a_1 & & \\
\vdots & & \ddots & \\
{-}c_M & & & a_M \end{pmatrix}.\]
By the positivity of the Schur complement
$\omega^{(M)}=\delta/\sigma^2-\sum_{m=1}^M c_m^2/a_m
\geq \delta/\sigma^2-(c_\eta^\tti(0)-c_\eta^\tti(\infty))$
and of $a_m \geq \iota$, this
matrix $L$ is strictly positive-definite, with smallest eigenvalue bounded away
from 0 independently of $M$. Then, recalling the function
$\kappa(r)$ from (\ref{eq:kappadef}),
\[(x-x')^\top (\nabla H(x \mid \theta^*,z)-\nabla H(x' \mid \theta^*,z)) \geq
(\theta-\theta')^2\kappa(|\theta-\theta'|)+c\|x-x'\|_2^2\]
for a constant $c>0$. Recalling also
that $\kappa(r)$ is positive for all $r>R_0$
and some $R_0>0$, and considering separately the cases where $|\theta-\theta'|
\leq R_0$ and $|\theta-\theta'|>R_0$, this verifies that
\[\inf_{\|x-x'\|_2=r}
\frac{(x-x')^\top (\nabla H(x \mid \theta^*,z)-\nabla H(x' \mid
\theta^*,z))}{\|x-x'\|_2^2}>c'\]
for all $r>R_0'$ and some constants $c',R_0'>0$.
Then by \cite[Corollary 3]{eberle2016reflection}, the Langevin diffusion
$\{x^t\}_{t \geq T_0}$ has the unique stationary law
\begin{equation}\label{eq:thetaauxstationary}
\sP^\infty(x) \propto \exp(-H(x \mid \theta^*,z)).
\end{equation}
Let us write $x^\infty \sim \sP^\infty$ and
$\langle f(x^\infty) \rangle$ for a sample and Gibbs average under this
stationary law. Let us write also $W_1(\cdot)$ for the
Wasserstein-1 distance conditional on $\theta^*,z$, and
$\sP(x^{T_0+T} \mid x^{T_0}=x)$ for the conditional
law of $x^{T_0+T}$ given $\theta^*,z$ and the initial condition $x^{T_0}=x$.
Then also by \cite[Corollary 2 and 3]{eberle2016reflection},
there exist constants $C,c>0$ such that for any $T>0$,
\[W_1(\sP(x^{T_0+T} \mid x^{T_0}=x),\sP^\infty)
\leq Ce^{-cT}W_1(\delta_x,\sP^\infty)
\leq Ce^{-cT}(\|x\|_2+\langle \|x^\infty\|_2 \rangle).\]
Similarly, for any $T'>0$,
\[W_1(\sP(x^{T_0+T+T'} \mid x^{T_0+T}=x),\sP^\infty)
\leq Ce^{-cT'} (\|x\|_2+\langle \|x^\infty\|_2 \rangle).\]
Combining the two conditional couplings that attain these Wasserstein-1
bounds, and taking the average over the sample path $\{x^t\}_{t \geq T_0}$
(which we denote by $\langle f(x^t) \rangle$, still conditional on $\theta^*,z$),
\[W_1(\sP(x^{T_0+T},x^{T_0+T+T'}),(\sP^\infty)^{\otimes 2})
\leq C(e^{-cT}+e^{-cT'})(\langle\|x^{T_0}\|_2\rangle+
\langle\|x^{T_0+T}\|_2\rangle+\langle\|x^\infty\|_2\rangle)\]
where $(\sP^\infty)^{\otimes 2}$ is the law of two independent samples from
$\sP^\infty$. The explicit form (\ref{eq:vmexplicit}) for each $\{x_m^t\}_{t
\geq 0}$ implies that $\langle |x_m^t| \rangle \leq c_m\int_0^t
e^{-\iota(t-s)}(\langle |\theta_{M,T_0}^s| \rangle
+|\theta^*|)\d s+(a_m)^{-1/2}$, and hence
\[\langle \|x^t\|_2 \rangle \leq C\sqrt{M}\Big(1+|\theta^*|+
\int_0^t e^{-\iota(t-s)} \langle |\theta_{M,T_0}^s| \rangle\d s\Big)\]
for a constant $C>0$. Then, taking the full expectation over $\theta^*,z$
and applying $\E\langle |\theta_{M,T_0}^t| \rangle \leq C$ by
Lemma \ref{lem:dmft_theta_uniform}, we get $\E
\langle \|x^t\|_2 \rangle \leq C'\sqrt{M}$ for a constant $C'>0$ and all
$t \geq 0$. Then, applying this above gives
\begin{equation}\label{eq:thetaauxW1condbound}
\E W_1(\sP(x^{T_0+T},x^{T_0+T+T'}),(\sP^\infty)^{\otimes 2})
\leq C\sqrt{M}(e^{-cT}+e^{-cT'})
\end{equation}
for some (different) constants $C,c>0$.

Finally, note that the stationary law $\sP^\infty(x)$ defined by
(\ref{eq:thetaauxstationary}) with Hamiltonian (\ref{eq:thetaauxHamiltonian})
describes a joint law (conditional on $\theta^*,z$) of $(\theta,x_1,\ldots,x_M)$
where $x_m \mid \theta$ is Gaussian and independent across $m=1,\ldots,M$,
and $\theta$ has marginal law given exactly by
$\sP(\theta \mid y)$ in the Gaussian convolution
model (\ref{eq:scalarchannel}) with observation
$y=\theta^*+z/\omega^{(M)}$. Here, the noise variable $z/\omega^{(M)}$ is
Gaussian with variance
$c_\eta^\tti(\infty)/(\omega^{(M)})^2=(\omega_*^{(M)})^{-1}$,
so the joint law of $\theta^*$ and the $(\theta,\theta')$-marginals of 
the conditional law $(\sP^\infty)^{\otimes 2}$ given $(\theta^*,z)$ is precisely
$\sP_{g_*,\omega_*^{(M)};g,\omega^{(M)}}^{\otimes 2}$. Then, taking the
$(\theta,\theta')$-marginals of the coupling (conditional on $\theta^*,z$)
that attains
$W_1(\sP(x^{T_0+T},x^{T_0+T+T'}),(\sP^\infty)^{\otimes 2})$
and combining with the identity coupling of $\theta^*$, we have
\[W_1(\sP(\theta^*,\theta_{M,T_0}^{T_0+T},\theta_{M,T_0}^{T_0+T+T'}),
\sP_{g_*,\omega_*^{(M)};g,\omega^{(M)}}^{\otimes 2})
\leq W_1(\sP(x^{T_0+T},x^{T_0+T+T'}),(\sP^\infty)^{\otimes 2}).\]
Taking the full expectation over $\theta^*,z$ on both sides and applying the
bound (\ref{eq:thetaauxW1condbound}) shows the lemma.
\end{proof}

\begin{lemma}\label{lem:aux_station_approx}
Suppose $c_\eta^\tti(0)-c_\eta^\tti(\infty)<\delta/\sigma^2$. For any $M>0$,
let
\begin{align*}
\omega^{(M)}&=\delta/\sigma^2-\sum_{m=1}^M c_m^2/a_m, \quad\omega_\ast^{(M)} =
(\omega^{(M)})^2/c_\eta^\tti(\infty),\\
\omega &= \delta/\sigma^2 - (c_\eta^\tti(0)-c_\eta^\tti(\infty)),
\quad \omega_\ast = \omega^2/c_\eta^\tti(\infty).
\end{align*}
Then $\lim_{M \to \infty} W_1(\sP_{g_\ast,\omega_\ast^{(M)};
g,\omega^{(M)}}^{\otimes 2},\sP_{g_\ast,\omega_\ast; g, \omega}^{\otimes 2})=0$.
\end{lemma}
\begin{proof}
Let $(\theta^*,\theta,\theta') \sim \sP_{g_*,\omega_*;g,\omega}^{\otimes 2}$,
i.e.\ $\theta,\theta'$ are two independent draws from the posterior law
$\sP_{g,\omega}(\theta \mid y)$ in the scalar Gaussian convolution model
(\ref{eq:scalarchannel}) where 
$y=\theta^*+{\omega^*}^{-1/2}z$ and $z \sim \N(0,1)$.
Let $\langle \cdot \rangle_{g,\omega}$
be average over $\theta,\theta'$ conditional on
$\theta^*,z$, and let $\mathcal{F}$ be the class of 1-Lipschitz functions
$f(\theta^*,\theta,\theta')$. Then, for any $f \in \mathcal{F}$,
\[\E_{g_*,\omega_*} \langle f(\theta^*,\theta,\theta') \rangle_{g,\omega}
=\E \frac{\int f(\theta^*,\theta,\theta')
\exp(-\frac{\omega}{2}[(\theta^*+{\omega^*}^{-1/2}z-\theta)^2
+(\theta^*+{\omega^*}^{-1/2}z-\theta')^2])g(\theta)g(\theta')\d(\theta,\theta')}
{\int \exp(-\frac{\omega}{2}[(\theta^*+{\omega^*}^{-1/2}z-\theta)^2
+(\theta^*+{\omega^*}^{-1/2}z-\theta')^2])g(\theta)g(\theta')\d(\theta,\theta')}\]
where $\E$ on the right side is over $\theta^* \sim g_*$ and $z \sim \N(0,1)$.
Writing $\langle \cdot \rangle$ for $\langle \cdot \rangle_{g,\omega}$ and
$\kappa_2$ for its associated posterior covariance,
the above is continuously-differentiable in $(\omega,\omega^*)$ with
\begin{align*}
\partial_\omega \E \langle f(\theta^*,\theta,\theta') \rangle&=\E
\Big[\kappa_2\Big(f(\theta^*,\theta,\theta'),{-}\tfrac{1}{2}[(\theta^*+{\omega^*}^{-1/2}z-\theta)^2
+(\theta^*+{\omega^*}^{-1/2}z-\theta')^2]\Big)\Big]\\
\partial_{\omega^*} \E \langle f(\theta^*,\theta,\theta') \rangle&=\E
\Big[\kappa_2\Big(f(\theta^*,\theta,\theta'),\tfrac{\omega z}{\omega_*^{3/2}}[(\theta^*+{\omega^*}^{-1/2}z-\theta)+(\theta^*+{\omega^*}^{-1/2}z-\theta')]\Big)\Big]
\end{align*}
By the 1-Lipschitz bound for $f$ and the identity $\Var
X=\frac{1}{2}\E[(X-X')^2]$ where $X'$ is an independent copy of $X$, we have
$\kappa_2(f(\theta^*,\theta,\theta'),f(\theta^*,\theta,\theta'))
\leq C(\kappa_2(\theta,\theta)+\kappa_2(\theta',\theta'))$ for an absolute
constant $C>0$. Then, applying Cauchy-Schwarz to $\kappa_2(\cdot)$ above,
we get that $(\omega,\omega_*) \mapsto
\E_{g_*,\omega_*} \langle f(\theta^*,\theta,\theta') \rangle_{g,\omega}$
is locally Lipschitz-continuous uniformly over $f \in \mathcal{F}$. Since
$\lim_{M \to \infty} \sum_{m=1}^M c_m^2/a_m=
\mu_\eta([\iota,\infty))=c_\eta^\tti(0)-c_\eta^\tti(\infty)$,
we have $\lim_{M \to \infty} (\omega^{(M)},\omega_*^{(M)})=(\omega,\omega_*)$.
Then this local Lipschitz continuity implies as desired
\[\lim_{M \to \infty}
W_1(\sP_{g_\ast,\omega_\ast^{(M)}; g,\omega^{(M)}}^{\otimes 2},
\sP_{g_\ast,\omega_\ast; g, \omega}^{\otimes 2})
=\lim_{M \to \infty} \sup_{f \in \mathcal{F}}
\Big|\E_{g_*,\omega_*} \langle f(\theta^*,\theta,\theta') \rangle_{g,\omega}
-\E_{g_*,\omega_*^{(M)}} \langle f(\theta^*,\theta,\theta')
\rangle_{g,\omega^{(M)}}\Big|=0.\]
\end{proof}

We now complete the proof of Lemma \ref{lem:theta_replica_eq}.

\begin{proof}[Proof of Lemma \ref{lem:theta_replica_eq}]
By Lemmas \ref{lem:auxiliary_V}, \ref{lem:aux_converge}, and
\ref{lem:aux_station_approx}, for any $M,T_0,T,T'>0$,
\begin{align*}
W_1(\sP(\theta^*,\theta^{T_0+T},\theta^{T_0+T+T'}),\sP_{g_\ast,\omega_*;g,\omega}^{\otimes
2}) \leq \eps(M)+2\sqrt{T+T'}\,\eps(T_0)+C\sqrt{M}(e^{-cT}+e^{-cT'}).
\end{align*}
Setting $T=T'=t$, choosing $T_0 \equiv T_0(t)$ so that
$\lim_{t \to \infty} T_0(t)=\infty$ and
$\lim_{t \to \infty} \sqrt{2t}\,\eps(T_0(t))=0$, and taking
$t \to \infty$ followed by $M \to \infty$, this shows
\[\lim_{t \to \infty}
W_1(\sP(\theta^*,\theta^{T_0(t)+t},\theta^{T_0(t)+2t}),\sP_{g_\ast,\omega_*;g,\omega}^{\otimes
2})=0.\]
In particular, we have the weak convergence in distribution of
$(\theta^*,\theta^{T_0(t)+t},\theta^{T_0(t)+2t})$ to
$\sP_{g_\ast,\omega_*;g,\omega}^{\otimes 2}$.
Lemma \ref{lem:dmft_theta_uniform} implies that
$(\theta^*,\theta^{T_0(t)+t},\theta^{T_0(t)+2t})$ is uniformly bounded in $L^4$
and hence uniformly integrable in $L^2$, so this implies
\begin{equation}\label{eq:twopointW2convergence}
\lim_{t \to \infty}
W_2(\sP(\theta^*,\theta^{T_0(t)+t},\theta^{T_0(t)+2t}),\sP_{g_\ast,\omega_*;g,\omega}^{\otimes
2})=0.
\end{equation}
Then, under Definition \ref{def:regular} and by definition of the law
$\sP_{g_*,\omega_*;g,\omega}^{\otimes 2}$, we have as desired
\begin{align*}
c_\theta^{\tti}(0) &= \lim_{t\rightarrow\infty} C_\theta(T_0(t)+t,T_0(t)+t)
=\lim_{t\rightarrow\infty} \E[(\theta^{T_0(t)+t})^2] = \E_{g_\ast,
\omega_\ast} \langle \theta^2\rangle_{g,\omega},\\
c_\theta^\tti(\infty) &= \lim_{t\rightarrow\infty} C_\theta(T_0(t)+t,T_0(t)+2t)
=\lim_{t\rightarrow\infty}
\E[\theta^{T_0(t)+t}\theta^{T_0(t)+2t}] =
\E_{g_\ast, \omega_\ast} \langle \theta\rangle_{g,\omega}^2,\\
c_\theta(\ast) &= \lim_{t\rightarrow\infty} C_\theta(T_0(t)+t,*)
=\lim_{t\rightarrow\infty} \E[\theta^{T_0(t)+t}\theta^\ast] =
\E_{g_\ast, \omega_\ast}[\langle \theta\rangle_{g,\omega}\theta^\ast].
\end{align*}
\end{proof}

\subsection{Analysis of $\eta$-equation}

We next derive from an analysis of the evolution
(\ref{def:dmft_langevin_cont_eta}) for $\{\eta^t\}_{t \geq 0}$
a representation of
$c_\eta^\tti(0),c_\eta^\tti(\infty)$ in terms of
$c_\theta^\tti(0),c_\theta^\tti(\infty),c_\theta(*)$.

\begin{lemma}\label{lem:eta_replica_eq}
It holds that
\begin{align}
c_\eta^{\tti}(0) &= \frac{\delta}{\sigma^4}\bigg[\frac{\E{\theta^*}^2 +
\sigma^2 + c_\theta^\tti(\infty) - 2c_\theta(\ast)}{\big(1+
\sigma^{-2}(c_\theta^{\tti}(0)-c_\theta^{\tti}(\infty))\big)^2} +
\frac{c_\theta^{\tti}(0)-c_\theta^{\tti}(\infty)}{1+\sigma^{-2}(c_\theta^{\tti}(0)-c_\theta^\tti(\infty))}\bigg],\label{eq:eta_replica_1}\\
c_\eta^\tti(\infty) &= \frac{\delta}{\sigma^4}\frac{\E{\theta^*}^2 +
\sigma^2 + c_\theta^\tti(\infty) -
2c_\theta(\ast)}{(1+\sigma^{-2}(c_\theta^{\tti}(0)-c_\theta^\tti(\infty)))^2}
\label{eq:eta_replica_2},
\end{align}
and in particular $c_\eta^\tti(0)-c_\eta^\tti(\infty)<\delta/\sigma^2$.
\end{lemma}

The argument is similar to the analysis of $\{\theta^t\}_{t \geq 0}$, where we
may approximate the dynamics of $\{\eta^t\}_{t \geq 0}$ at large times by a
Markovian joint evolution of a system $(\eta^t,x_1^t,\ldots,x_M^t)$. Our
argument here is simpler than before,
as the dynamics of $(\eta^t,x_1^t,\ldots,x_M^t)$ will
be linear, from which we may explicitly analyze the convergence of $\eta^t$
and show that it is independent of $M$; thus we will apply a simple Gronwall
argument to bound the propagation of the discretization error $\eps(M)$ over
time.

\subsubsection{Comparison with an auxiliary process}
We again fix a positive integer $M$, and define
$\{a_m\}_{m=0}^M$ and $\{c_m\}_{m=1}^M$ by
\[a_m=\iota+\frac{m}{\sqrt{M}} \text{ for } m=0,\ldots,M,
\qquad \frac{c_m^2}{a_m}=\mu_\theta([a_{m-1}, a_m))\]
with $\mu_\theta$ now instead of $\mu_\eta$.
For convenience, let us introduce $\xi^t=\eta^t+w^\ast-\eps$
and $v^t={-}w^t+w^*-\eps$, so the
DMFT equation (\ref{def:dmft_langevin_cont_eta}) for $\{\eta^t\}_{t \geq 0}$ is
equivalently
\begin{equation}\label{eq:dmft_xit}
\xi^t=-\frac{1}{\sigma^2}\int_0^t R_\theta(t,s)\xi^s\d s+v^t.
\end{equation}
Here, $\{v^t\}_{t \geq 0}$ is a centered Gaussian process with covariance
$\E[v^tv^s]=C_\theta(t,s)-C_\theta(t,*)-C_\theta(s,*)+\E(\theta^*)^2+\sigma^2$. We set
\[R_\theta^{(M)}(\tau)=\sum_{m=1}^M c_m^2 e^{-a_m\tau},
\qquad C_\theta^{(M)}(t,s)=\sum_{m=1}^M \frac{c_m^2}{a_m}
(e^{-a_m|t-s|}-e^{-a_m(t+s)})+c_\theta^\tti(\infty)\]
and define an auxiliary process $\{\xi_{M,T_0}^t\}_{t \geq 0}$ by
\begin{align}
\xi^t_{M,T_0} &= \xi^t \text{ for } t \in [0,T_0)\notag\\
\xi^t_{M,T_0} &= -\frac{1}{\sigma^2}\int_0^t R_\theta^{(M)}(t-s)\xi_{M,T_0}^s
\d s+v_M^t \text{ for } t \geq T_0\label{eq:xiaux}
\end{align}
where $\{v_M^t\}_{t \geq 0}$ is a centered Gaussian process with covariance
$\E[v_M^tv_M^s]=C_\theta^{(M)}(t,s)-2c_\theta(*)+\E(\theta^*)^2+\sigma^2$,
defined in the probability space of $\{\xi^t\}_{t \geq 0}$.
(We check in the proof of Lemma \ref{lem:eta_aux_converge} below that this is
indeed a positive-semidefinite covariance kernel.) We note that the process
$\{\xi_{M,T_0}^t\}_{t \geq 0}$ may be discontinuous at $T_0$; this is
inconsequential for our subsequent analysis.

\begin{lemma}\label{lem:eta_aux_approx}
For any $M,T_0,T>0$, there exists a coupling of $\{\xi^t\}_{t \geq 0}$
and $\{\xi^t_{M,T_0}\}_{t \geq 0}$ such that
\[\sup_{t\in[0,T_0+T]} \E(\xi^t-\xi^t_{M,T_0})^2
\leq Ce^{CT}(\eps(M)+\sqrt{T}\,\eps(T_0))\]
where $\eps(M)$ does not depend on $T_0,T$ and
$\eps(T_0)$ does not depend on $M,T$, and
$\lim_{M \to \infty} \eps(M)=0$ and $\lim_{T_0 \to \infty} \eps(T_0)=0$.
\end{lemma}
\begin{proof}
Applying the approximation (\ref{eq:Cthetaapproxinvariant})
and arguments analogous to Lemma \ref{lem:GP_couple}, we have that
\begin{align*}
&\sup_{s,t \in [T_0,T_0+T]} |\E[v_M^tv_M^s]-\E[v^tv^s]|\\
&\leq \sup_{s,t \in [T_0,T_0+T]} |C_\theta^{(M)}(t,s)-C_\theta(t,s)|
+|c_\theta(*)-C_\theta(t,*)|+|c_\theta(*)-C_\theta(s,*)|
\leq \eps(M)+\eps(T_0),
\end{align*}
and hence there exists a coupling of $\{v_{M,T_0}^t\}_{t \geq 0}$
and $\{v^t\}_{t \geq 0}$ such that
\[\sup_{t\in[T_0,T_0+T]}\E(v^t-v_M^t)^2 \leq \eps(M)+\sqrt{T}\,\eps(T_0).\]
We bound $\xi^t-\xi_{M,T_0}^t$ under this coupling of $\{v^t\}_{t \geq 0}$
with $\{v_M^t\}_{t \geq 0}$:
Let us write $\tilde \xi^t=\xi_{M,T_0}^t$. We have
$\xi^t=\tilde\xi^t$ for $t\in[0,T_0)$, while for $t\in[T_0,T_0+T]$, 
\begin{equation}\label{eq:xitdiffbound}
\E(\xi^t-\tilde \xi^t)^2 \leq 3\Big[\E\Big(\int_0^t R_\theta^{(M)}(t-s)|\xi^s
- \tilde{\xi}^s|\d s\Big)^2 + \E \Big(\int_0^t |R_\theta(t,s) -
R_\theta^{(M)}(t-s)||\xi^s|\d s\Big)^2 + \E (v^t - v_M^t)^2\Big].
\end{equation}
From the explicit definition of $R^{(M)}_\theta(t-s)$, the first term of
(\ref{eq:xitdiffbound}) satisfies
\begin{align*}
\E\Big(\int_0^t |R_\theta^{(M)}(t-s)||\xi^s - \tilde{\xi}^s|\d s\Big)^2 =
\E\Big(\int_{T_0}^t |R_\theta^{(M)}(t-s)||\xi^s - \tilde{\xi}^s|\d s\Big)^2 \leq C \int_{T_0}^t \E (\xi^s - \tilde{\xi}^s)^2\d s
\end{align*}
for a constant $C>0$. Following the argument used to bound
(\ref{eq:Delta_bound}), the second term of (\ref{eq:xitdiffbound}) is bounded by
$\eps(M)+C\eps(t)^2$ where $\eps(t) \to 0$ as $t \to
\infty$, while the third term is bounded by
$\eps(M)+\sqrt{T}\,\eps(T_0)$ under the above coupling. Then
by Gronwall's inequality,
\[\sup_{t\in [T_0,T_0+T]} \E(\xi^t-\tilde{\xi}^t)^2 \leq
Ce^{CT}\Big(\eps(M)+\sup_{t \in [T_0,T_0+T]} \eps(t)^2
+\sqrt{T}\,\eps(T_0)\Big),\]
which implies the lemma upon adjusting $\eps(T_0)$.
\end{proof}

\subsubsection{Convergence of the auxiliary process}

\begin{lemma}\label{lemma:sigmazpositive}
The value $\sigma_Z^2=\E {\theta^*}^2+c_\theta^\tti(\infty)
-2c_\theta(\ast)+\sigma^2$ is positive.
\end{lemma}
\begin{proof}
Let $\{\btheta^t\}_{t \geq 0}$ be the Langevin diffusion
(\ref{eq:langevinfixedprior}) for which
the DMFT system of Theorem \ref{thm:dmft_equilibrium} is the large-$(n,d)$
limit. By Theorem \ref{thm:dmft_response} to follow,
\[C_\theta(t,s)-C_\theta(t,*)-C_\theta(s,*)+\E(\theta^*)^2
=\lim_{n,d \to \infty}\E\Bigg[\frac{1}{d}\sum_{i=1}^d 
(\theta_i^t-\theta_i^*)(\theta_i^s-\theta_i^*)\Bigg]\]
Since $\{\btheta^t\}_{t \geq 0}$ is Markovian (conditional on
$\X,\y,\btheta^*$), we have for all $t \geq s$ that
\[\E\Bigg[\frac{1}{d}\sum_{i=1}^d 
(\theta_i^t-\theta_i^*)(\theta_i^s-\theta_i^*)\Bigg]
=\E\Bigg[\E\Bigg[\frac{1}{d}\sum_{i=1}^d 
(\theta_i^t-\theta_i^*)(\theta_i^s-\theta_i^*)\;\Bigg|\;
\btheta^s,\X,\y,\btheta^*\Bigg]\Bigg]
=\E\Bigg[\frac{1}{d}\sum_{i=1}^d 
(\theta_i^s-\theta_i^*)^2\Bigg] \geq 0,\]
hence $C_\theta(t,s)-C_\theta(t,*)-C_\theta(s,*)+\E(\theta^*)^2 \geq 0$.
Setting $s=t/2$ and taking the limit $t \to \infty$ under Definition
\ref{def:regular} shows $c_\theta^\tti(\infty)-2c_\theta(*)+\E {\theta^*}^2 \geq
0$, and the lemma follows.
\end{proof}

\begin{lemma}\label{lem:eta_aux_converge}
Let $c=(c_1,\ldots,c_M)$, $A = \diag(a_1,\ldots,a_M)$, $\Lambda = A +
cc^\top/\sigma^2$, and consider the 2-dimensional Gaussian law
$\N(0,\Sigma_M)$ with
\begin{align*}
\Sigma_M=\begin{pmatrix} \rho_M^2 & \kappa_M \\
\kappa_M & \rho_M^2 \end{pmatrix},
\qquad \kappa_M=\sigma_Z^2 \cdot \Big[1-c^\top\Lambda^{-1}c/\sigma^2\Big]^2,
\qquad \rho^2_M=\kappa_M+c^\top \Lambda^{-1}c,
\end{align*}
where $\sigma_Z^2=\E(\theta^*)^2 +
c_\theta^\tti(\infty)-2c_\theta(\ast)+\sigma^2$.
Then there exists an error $\eps(T)$ not depending on $T_0,M$ and satisfying
$\lim_{T \to \infty} \eps(T)=0$, such that for any $M,T_0,T,T'>0$,
\[W_2(\sP(\xi_{M,T_0}^{T_0+T},\xi_{M,T_0}^{T_0+T+T'}),\N(0,\Sigma_M)) \leq \eps(T)+\eps(T').\]
\end{lemma}
\begin{proof}
Let $z \sim \N(0,\sigma_Z^2)$, where $\sigma_Z^2>0$ by Lemma
\ref{lemma:sigmazpositive}, and let $\{b_m^t\}_{t \geq 0}$ for $m=1,\ldots,M$ be
standard Brownian motions. We assume these are independent of each other and of
$\{\xi^t\}_{t \in [0,T]}$. Then the law of
$\{\xi_{M,T_0}^t\}_{t \geq 0}$ coincides with
the marginal law of $\{\xi_{M,T_0}^t\}_{t \geq 0}$ in the joint process
\begin{align}
\xi_{M,T_0}^t&=\xi^t \text{ for } t \in [0,T_0)\notag\\
\xi_{M,T_0}^t&=\sum_{m=1}^M c_m x_m^t + z \text{ for } t \geq
T_0\label{eq:xiauxevolve}\\
\d x_m^t&=-[a_m x_m^t + c_m \xi_{M,T_0}^t/\sigma^2]\d t + \sqrt{2}\,\d b_m^t
\text{ for } 1 \leq m \leq M,\, t \geq 0\label{eq:xiauxxevolve}
\end{align}
with initial conditions $x_1^0=\ldots=x_M^0=0$. Indeed,
given $\{\xi_{M,T_0}^t\}_{t \geq 0}$, the equations (\ref{eq:xiauxxevolve})
for $\{x_m^t\}_{t \geq 0}$ have the explicit solutions
\begin{align*}
x_m^t=-\frac{1}{\sigma^2}\int_0^t c_me^{-a_m(t-s)} \xi_{M,T_0}^s \d s +
\int_0^t e^{-a_m(t-s)}\sqrt{2}\,\d b_m^s, 
\end{align*}
and substituting this into (\ref{eq:xiauxevolve}) gives (\ref{eq:xiaux}) upon
identifying $v_M^t=z+\int_0^t \sum_{m=1}^M c_me^{-a_m(t-s)}\sqrt{2}\,\d
b_m^s$. It is direct to check that $\{v_M^t\}_{t \geq 0}$ thus defined
has covariance $C_\theta^{(M)}(t,s)-2c_\theta(*)+\E(\theta^*)^2+\sigma^2$,
so this coincides with the law of $\{\xi_{M,T_0}^t\}_{t \geq 0}$
defined by (\ref{eq:xiaux}).

Let us denote $\tilde{\xi}^t=\xi_{M,T_0}^t$,
$x^t=(x_1^t,\ldots,x_M^t)$, and $b^t=(b_1^t,\ldots,b_M^t)$. For $t \geq T_0$,
the evolution of $(\tilde\xi^t,x^t) \in \R^{M+1}$ is a (Markovian)
Ornstein-Uhlenbeck process. Substituting (\ref{eq:xiauxevolve}) into
(\ref{eq:xiauxxevolve}), we have
\begin{align*}
\d x^t = -[\Lambda x^t + cz/\sigma^2]\d t + \sqrt{2}\,\d b^t \text{ for }
t \geq T_0
\end{align*}
where $c=(c_1,\ldots,c_M)$ and $\Lambda=A+cc^\top/\sigma^2$ with
$A=\diag(a_1,\ldots,a_M)$. This has the solution, for $t \geq T_0$,
\begin{align*}
x^t = e^{-\Lambda(t-T_0)}x^{T_0} + \frac{z}{\sigma^2}\Lambda^{-1}(e^{-\Lambda(t-T_0)} - I)c + \int_{T_0}^t
e^{-\Lambda(t-s)}\sqrt{2}\,\d b^s.
\end{align*}
Substituting back into (\ref{eq:xiauxevolve}),
\begin{equation}\label{eq:tildexiexplicit}
\tilde{\xi}^t = c^\top e^{-\Lambda(t-T_0)}x^{T_0}+z\Big[1 +
\frac{1}{\sigma^2}c^\top\Lambda^{-1}(e^{-\Lambda(t-T_0)} - I)c\Big] +
\int_{T_0}^t c^\top e^{-\Lambda(t-s)}\sqrt{2}\,\d b^s \text{ for } t \geq T_0.
\end{equation}
Here, we note that the equation (\ref{eq:dmft_xit}) implies that
$\{\xi^t\}_{t \geq 0}$ is itself a Gaussian process
(given by a linear functional of $\{v^t\}_{t \geq 0}$),
so $x^{T_0}$ with coordinates
\begin{equation}\label{eq:xT0}
x_m^{T_0}=\underbrace{-\frac{c_m}{\sigma^2}\int_0^{T_0}e^{-a_m(T_0-s)} \xi^s \d
s}_{=U_m} +
\underbrace{\int_0^{T_0} e^{-a_m(T_0-s)}\sqrt{2}\,\d b^s_m}_{=V_m}
\end{equation}
is a Gaussian vector. Consequently, the form (\ref{eq:tildexiexplicit}) shows
that for any $T,T'>0$,
$(\tilde{\xi}^{T_0+T},\tilde{\xi}^{T_0+T+T'})$ has a centered
bivariate Gaussian law. To conclude the proof of the lemma, it suffices to show
\begin{align}
|\E[(\tilde{\xi}^{T_0+T})^2]-\rho_M^2|,\,
|\E[(\tilde{\xi}^{T_0+T+T'})^2]-\rho_M^2|
 &\leq \eps(T)+\eps(T')\label{eq:xivarbound}\\
|\E[\tilde{\xi}^{T_0+T}\tilde{\xi}^{T_0+T+T'}]-\kappa_M| &\leq \eps(T)+\eps(T')
\label{eq:xicovbound}
\end{align}
for some errors $\eps(T),\eps(T')$ that hold uniformly over all $M,T_0>0$.

For (\ref{eq:xivarbound}), we may compute from the solution
(\ref{eq:tildexiexplicit}) that
\begin{align*}
\E[(\tilde{\xi}^{T_0+T})^2]
&=\underbrace{c^\top e^{-\Lambda T} \E[x^{T_0}(x^{T_0})^\top]e^{-\Lambda
T}c}_{=\mathrm{I}} + \underbrace{\sigma_Z^2
\cdot \Big[1 + \frac{1}{\sigma^2}c^\top\Lambda^{-1}(e^{-\Lambda T} -
I)c\Big]^2+c^\top  \Lambda^{-1}(I - e^{-2\Lambda T})c}_{=\mathrm{II}}.
\end{align*}
Observe that $\|\Lambda^{-1/2}c\|_2^2 \leq \sum_{m=1}^M c_m^2/a_m
=\sum_{m=1}^M \mu_\theta([a_{m-1},a_m)) \leq \mu_\theta([\iota,\infty))$.
Hence $\|\Lambda^{-1/2}c\|_2 \leq C$ for a constant $C>0$ not depending on $M$.
Since also $\lambda_{\min}(\Lambda) \geq \iota>0$, we have $\|e^{-\Lambda
T}\|_\op \leq e^{-\iota T}$, so
\[|\mathrm{II}-\rho_M^2| \leq \eps(T)\]
for an error $\eps(T)$ not depending on $M$.
To bound $\mathrm{I}$, write $x^{T_0}=U+V$ where
$U,V \in \R^M$ have the coordinates $U_m,V_m$ in (\ref{eq:xT0}).
Then, from the bound $\E[(u^\top x^{T_0})^2] \leq 2\E[(u^\top U)^2]
+2\E[(u^\top V)^2]$ for each unit vector $u \in \R^M$, we have
\[\|\E[x^{T_0}(x^{T_0})^\top]\|_\op
\leq 2\|\E[UU^\top]\|_\op+2\|\E[VV^\top]\|_\op.\]
For the second term,
$\|\E[VV^\top]\|_\op=\|\diag(a_m^{-1}(1-e^{-a_mT_0}))\|_\op \leq
\iota^{-1}$. For the first term,
\begin{align*}
\|\E[UU^\top]\|_\op \leq \E\|U\|_2^2
&=\E\sum_{m=1}^M \frac{c_m^2}{\sigma^4}\Big(\int_0^{T_0} e^{-a_m(T_0-s)}\xi^s \d s\Big)^2\\
&\leq \sum_{m=1}^M \frac{c_m^2}{\sigma^4} \int_0^{T_0} e^{-a_m(T_0-s)} \d s
\cdot \int_0^{T_0} e^{-a_m(T_0-s)}\E(\xi^s)^2 \d s.
\end{align*}
Noting that
$\E(\xi^t)^2=(\sigma^4/\delta)C_\eta(t,t) \leq C$
for all $t \geq 0$ under Definition \ref{def:regular},
this gives $\|\E[UU^\top]\|_\op \leq C'\sum_{m=1}^M c_m^2/a_m^2
\leq C'\mu_\theta([\iota,\infty))/\iota$. Combining these bounds shows
$\|\E[x^{T_0}(x^{T_0})^\top]\|_\op \leq C$ for a constant $C>0$ not depending
on $M,T_0$. Then, combining with the previous bounds $\|\Lambda^{-1/2}u\|_2
\leq C$ and $\lambda_{\min}(\Lambda) \geq \iota$,
this shows $|\mathrm{I}| \leq \eps(T)$, so
$|\E[(\tilde{\xi}^{T_0+T})^2]-\rho_M^2| \leq \eps(T)$.
The bound for $\E[(\tilde{\xi}^{T_0+T+T'})^2]$ in
(\ref{eq:xivarbound}) holds similarly.

For (\ref{eq:xicovbound}), we may compute similarly from
(\ref{eq:tildexiexplicit})
\begin{align*}
\E[(\tilde{\xi}^{T_0+T})\tilde \xi^{T_0+T+T'}]
&=u^\top e^{-\Lambda T}
\E[x^{T_0}(x^{T_0})^\top]e^{-\Lambda(T+T')}u\\
&\quad +\sigma_Z^2\cdot \Big[1 +
\frac{1}{\sigma^2}u^\top\Lambda^{-1}(e^{-\Lambda (T+T')} - I)u\Big]
\Big[1 + \frac{1}{\sigma^2}u^\top\Lambda^{-1}(e^{-\Lambda T} - I)u\Big]\\
&\quad + u^\top \Lambda^{-1}(e^{-\Lambda T'} - e^{-\Lambda(2T+T')})u,
\end{align*}
and the arguments to show (\ref{eq:xicovbound}) from this form
are the same as above.
\end{proof}

\begin{lemma}\label{lem:eta_aux_M_converge}
Consider the 2-dimensional Gaussian law $\N(0,\Sigma_\infty)$ with
\[\Sigma_\infty=\begin{pmatrix} \rho^2_\infty & \kappa_\infty \\
\kappa_\infty & \rho^2_\infty \end{pmatrix},
\quad \kappa_\infty = \frac{\E{\theta^*}^2 + \sigma^2 +
c_\theta^\tti(\infty) - 2c_\theta(\ast)}{(1 +
\sigma^{-2}(c_\theta^\tti(0)-c_\theta^\tti(\infty))^2},
\quad
\rho^2_\infty = \kappa_\infty+\frac{c_\theta^\tti(0)-c_\theta^\tti(\infty)}
{1 + \sigma^{-2}(c_\theta^\tti(0)-c_\theta^\tti(\infty))}.\]
Then $\lim_{M \to \infty} \|\Sigma_M-\Sigma_\infty\|_\op=0$.
\end{lemma}
\begin{proof}
This follows from noting that $c^\top \Lambda^{-1}c=\frac{\sum_{m=1}^M
c_m^2/a_m}{1 + \sigma^{-2}\sum_{m=1}^M c_m^2/a_m}$ via the Sherman-Morrison
identity, and
$\sum_{m=1}^M c_m^2/a_m \rightarrow \int_{\iota}^\infty \mu_\theta(\d
a)=c_\theta^\tti(0)-c_\theta^\tti(\infty)$ as $M\rightarrow\infty$. 
\end{proof}

We now complete the proof of Lemma \ref{lem:eta_replica_eq}.

\begin{proof}[Proof of Lemma \ref{lem:eta_replica_eq}]
By Lemmas \ref{lem:eta_aux_approx}, \ref{lem:eta_aux_converge}, and
\ref{lem:eta_aux_M_converge}, it holds that
\[W_2(\sP(\xi^{T_0+T},\xi^{T_0+T+T'}),\N(0,\Sigma_\infty)) \leq
Ce^{C(T+T')}(\eps(M)+\sqrt{T+T'}\,\eps(T_0))
+\eps(T)+\eps(T')+\eps(M).\]
Taking first the limit $M \to \infty$, then
choosing $T=T'=t$ and $T_0 \equiv T_0(t)$ such that $\lim_{t \to \infty}
T_0(t)=\infty$ and $\lim_{t \to \infty} e^{2Ct}\sqrt{2t}\,\eps(T_0(t))=0$
and taking $t \to \infty$, this shows
$W_2(\sP(\xi^{T_0(t)+t}, \xi^{T_0(t)+2t}),\N(0,\Sigma_\infty))
\rightarrow 0$ as $t\rightarrow \infty$. Under Definition \ref{def:regular},
this implies
\begin{align*}
\frac{\sigma^4}{\delta}c_\eta^{\tti}(0) &=
\lim_{t\rightarrow\infty} \frac{\sigma^4}{\delta}C_\eta(T_0(t)+t,T_0(t)+t)
=\lim_{t\rightarrow\infty}\E[(\xi^{T_0(t)+t})^2] = \rho^2_\infty,\\
\frac{\sigma^4}{\delta} c_\eta^\tti(\infty) &= 
\lim_{t\rightarrow\infty} \frac{\sigma^4}{\delta}C_\eta(T_0(t)+t,T_0(t)+2t)
=\lim_{t\rightarrow\infty} \E[\xi^{T_0(t)+t}\xi^{T_0(t)+2t}]=\kappa_\infty.
\end{align*}
This shows the desired forms of $c_\eta^\tti(0)$ and $c_\eta^\tti(\infty)$,
and we have also from these forms that
\[c_\eta^\tti(0)-c_\eta^\tti(\infty)=\frac{\delta}{\sigma^2}
\bigg[\frac{\sigma^{-2}(c_\theta^\tti(0)-c_\theta^\tti(\infty))}
{1+\sigma^{-2}(c_\theta^\tti(0)-c_\theta^\tti(\infty))}\bigg]<\frac{\delta}{\sigma^2}.\]
\end{proof}

\subsection{Completing the proof}

\begin{proof}[Proof of Theorem \ref{thm:dmft_equilibrium}]
By Lemmas \ref{lem:theta_replica_eq} and \ref{lem:eta_replica_eq}, we have five
equations (\ref{eq:theta_replica}), (\ref{eq:eta_replica_1}),
(\ref{eq:eta_replica_2}) for the five variables $c^{\tti}_\theta(0),
c_\theta^\tti(\infty), c_\theta(\ast), c_\eta^{\tti}(0), c_\eta^\tti(\infty)$.
Defining $\mse,\mse_*$ by (\ref{eq:mmse}), these equations show
\begin{align*}
\omega &= \frac{\delta}{\sigma^2} - (c_\eta^{\tti}(0)-c_\eta^\tti(\infty)) =
\frac{\delta}{\sigma^2 + (c_\theta^{\tti}(0)-c_\theta^\tti(\infty))} =
\frac{\delta}{\sigma^2+\mse},\\
\omega_\ast &= \frac{\omega^2}{c_\eta^\tti(\infty)} =
\frac{\delta}{\E{\theta^*}^2 + \sigma^2 + c_\theta(\infty) - 2c_\theta(\ast)} =
\frac{\delta}{\sigma^2 + \mse_\ast},
\end{align*}
as well as
\begin{align*}
\mse &= c_\theta^{\tti}(0)-c_\theta^\tti(\infty) = \E_{g_\ast,\omega_\ast}[\langle
\theta^2\rangle_{g,\omega} -  \langle \theta\rangle_{g,\omega}^2] =
\E_{g_\ast,\omega_\ast} \langle (\theta - \langle \theta
\rangle_{g,\omega})^2\rangle_{g,\omega},\\
\mse_\ast &= \E{\theta^\ast}^2 -
2\E_{g_\ast,\omega_\ast}[\theta^\ast\langle\theta\rangle_{g,\omega}] +
\E_{g_\ast,\omega_\ast} \langle\theta\rangle_{g,\omega}^2 =
\E_{g_\ast,\omega_\ast}(\theta^\ast - \langle\theta\rangle_{g,\omega})^2.
\end{align*}
This verifies that the fixed-point equations (\ref{eq:static_fixedpoint}) hold,
where it is clear that $\omega,\omega_*$ are uniquely defined from $\mse,\mse_*$
via (\ref{eq:static_fixedpoint}).
Defining $\ymse,\ymse_*$ by (\ref{eq:mmse}), we have also from the above forms
of $\omega,\omega_*$ that
\begin{align*}
\ymse&=\frac{\sigma^4}{\delta}(c_\eta^\tti(0)-c_\eta^\tti(\infty))
=\sigma^2\Big(1-\frac{\omega\sigma^2}{\delta}\Big),\\
\ymse^*&=\frac{\sigma^4}{\delta}(2c_\eta^\tti(0)-c_\eta^\tti(\infty))-\sigma^2
=\sigma^2+\frac{\omega\sigma^4}{\delta}\Big(\frac{\omega}{\omega_*}-2\Big),
\end{align*}
verifying (\ref{eq:ymmseomega}). Finally, the statement
(\ref{eq:dmftthetaconvergence})
is a consequence of (\ref{eq:twopointW2convergence}) shown in the proof of Lemma
\ref{lem:theta_replica_eq}.
\end{proof}

%% file: fixedprior.tex
\section{Analysis of fixed-prior Langevin dynamics under
LSI}\label{sec:fixedprior}

In this section, we prove Theorem \ref{thm:fixedalpha_dynamics}
and Corollary \ref{cor:fixedalpha_dynamics} that verify Definition
\ref{def:regular} and deduce the replica-symmetric limits for the
Bayes-optimal mean-squared-errors and free energy, under Assumption
\ref{assump:LSI} of a log-Sobolev inequality (LSI) for the posterior law.

\subsection{Preliminaries}\label{sec:langevinproperties}

\subsubsection{Properties of Langevin dynamics}

We review in this section two general results on a Langevin diffusion
of the form
\begin{equation}\label{eq:generaldiffusion}
\d \btheta^t=\nabla U(\btheta^t)\d t+\sqrt{2}\,\d\b^t
\end{equation}
with an equilibrium measure $e^{U(\btheta)}$. The first is
a fluctuation-dissipation relation for its
correlation and response functions at equilibrium, and the second is
a Bismut-Elworthy-Li representation for the spatial derivative
of its Markov semigroup. For bounded observables, similar
fluctuation-dissipation theorems have been stated and shown
in \cite{dembo2010markovian,chen2020mathematical} and Bismut-Elworthy-Li
formulae in \cite{bismut1984large,elworthy1994formulae}. 
We give versions of these results here for a class of unbounded observables
which may have linear growth
\[\cA=\{f \in C^2(\R^d,\R): \nabla f,\nabla^2 f \text{ are globally bounded}\},\]
and a class of drift coefficients
\begin{equation}\label{eq:classB}
\cB=\{U \in C^3(\R^d,\R): \nabla^2 U,\nabla^3 U \text{ are globally
bounded and H\"older continuous}\},
\end{equation}
drawing upon some analyses of our companion work \cite[Appendix A]{paper1}.

We write
\[P_tf(\btheta)=\E[f(\btheta^t) \mid \btheta^0=\btheta],
\qquad \L f(\btheta)=\nabla U^\top \nabla f(\btheta)+\Tr \nabla^2 f(\btheta)\]
for the Markov semigroup and infinitesimal generator
associated to (\ref{eq:generaldiffusion}).
It is shown in \cite[Proposition A.2]{paper1} that
\begin{equation}\label{eq:semigroupregularity}
f \in \cA,\,U \in \cB \quad \Rightarrow \quad
\nabla P_tf(\btheta),\nabla^2 P_t f(\btheta) \text{ are uniformly bounded over }
t \in [0,T],\,\btheta \in \R^d
\end{equation}
for any fixed $T>0$. In particular, $P_tf \in \cA$ for each fixed $t>0$.

\begin{lemma}\label{lemma:langevin_fdt_equi}
Suppose $U \in \cB$, and (\ref{eq:generaldiffusion}) has the unique
stationary distribution $q(\btheta)=e^{U(\btheta)}$ with finite third moments.
Let $\{\btheta^t\}_{t \geq 0}$ be the solution to (\ref{eq:generaldiffusion})
with initial condition $\btheta^0=\x$, and let $A \in \cA$ and $B \in \cB$.

\begin{enumerate}[(a)]
\item Define the response function
$R_{AB}^\x(t,s)=P_s(\nabla B^\top \nabla P_{t-s}A)(\x)$.
Then $R_{AB}^\x(t,s)$ satisfies the following condition:
Fix any continuous bounded function $h:[0,\infty) \to \R$. For each $\eps>0$,
let $\{\btheta^{t,\eps}\}_{t \geq 0}$ denote
the solution of the perturbed dynamics
\[\d\btheta^{t,\eps}=\nabla [U(\btheta^{t,\eps})+\eps h(t)B(\btheta^{t,\eps})]\d t+\sqrt{2}\,\d\b^t\]
with the same initial condition $\btheta^{0,\eps}=\x$.
Then for any $t>0$,
\[\lim_{\eps \to 0} \frac{1}{\eps}\Big(\E[A(\btheta^{t,\eps})
\mid \btheta^{0,\eps}=\x]-\E[A(\btheta^t) \mid \btheta^0=\x]\Big)
=\int_0^t R_{AB}^\x(t,s)h(s)\d s.\]

\item Define the correlation function
$C_{AB}^\x(t,s)=\E[A(\btheta^t)B(\btheta^s) \mid \btheta^0=\x]$.
Then for any $t \geq s \geq 0$, averaging over an initial condition
$\x \sim q$ drawn from the stationary distribution,
\[\partial_t \E_{\x \sim q} C_{AB}^\x(t,s)={-}\E_{\x \sim
q} R_{AB}^\x(t,s).\]
\end{enumerate}
\end{lemma}
\begin{proof}
Part (a) is an application of \cite[Proposition A.1]{paper1} of our companion
paper (specialized to this setting of dynamics with a fixed and non-adaptive
prior).

For part (b), we will use also from \cite[Proposition A.2]{paper1} that
for $A \in \cA$, we have $\partial_t P_tA=\L P_t A$.
Since the dynamics (\ref{eq:langevin_fixedq})
are Markovian with stationary distribution $q(\btheta)$, we have
\[\E_{\x \sim q} C_{AB}^\x(t,s)
=\E_{\x \sim q} [\E[A(\btheta^{t-s}) \mid \btheta^0=\x]B(\x)]
=\E_{\x \sim q} (B \cdot P_{t-s}A)[\x].\]
To differentiate under the integral in $t$, note that
$\partial_t(B \cdot P_{t-s}A)=B \cdot \L P_{t-s}A$.
By the uniform boundedness of
$\nabla P_tA,\nabla^2 P_tA$ over $t \in [0,T]$, the Lipschitz-continuity of
$\nabla B,\nabla U$, and finiteness of third moments of $q$,
we have that $(B \cdot \L P_t A)[\x]$ is uniformly integrable
with respect to $\x \sim q$ over $t \in [0,T]$.
Thus dominated convergence applies to show
\begin{align*}
\partial_t \E_{\x \sim q} C_{AB}^\x(t,s)
=\partial_t \E_{\x \sim q} (B \cdot P_{t-s}A)[\x]
=\E_{\x \sim q}(B \cdot \L P_{t-s} A)[\x].
\end{align*}
On the other hand, using also that both
$\nabla B^\top \nabla P_tA$ and $B \cdot \L P_tA$ are
integrable with respect to $\x \sim q$, we have via integration-by-parts
\begin{align*}
\E_{\x \sim q} R_{AB}^\x(t,s)&=\E_{\x \sim q}(\nabla B^\top \nabla
P_{t-s}A)[\x] =
\int q(\btheta)(\nabla B^\top \nabla P_{t-s}A)[\btheta]\d\btheta\\
&=-\int B(\btheta)
\sum_{j=1}^d \partial_j [q\,\partial_j (P_{t-s} A)](\btheta)\d\btheta\\
&= -\int B(\btheta)\,\big[q\,\Tr \nabla^2(P_{t-s} A) + \nabla (P_{t-s} A)^\top
\nabla q\big](\btheta)\d\btheta\\
&= -\int q(\btheta)B(\btheta)\,\big[\Tr \nabla^2(P_{t-s} A)+\nabla (P_{t-s}
A)^\top \nabla \log q\big](\btheta)\d\btheta\\
&= -\E_{\x \sim q}(B \cdot \L P_{t-s} A)[\x].
\end{align*}
\end{proof}

\begin{lemma}\label{lemma:BEL}
Suppose $U \in \cB$, and consider the solution
$(\btheta^t,\V^t) \in \R^d \times \R^{d \times d}$ to
\begin{equation}\label{eq:matrixV_sde}
\d\btheta^t=\nabla U(\btheta^t)\d t+\sqrt{2}\,\d \b^t,
\qquad \d\V^t=[\nabla^2 U(\btheta^t)]\V^t\d t
\end{equation}
with initial condition $(\btheta^0,\V^0)=(\x,\I)$,
adapted to the canonical filtration of the Brownian motion
$\{\b^t\}_{t \geq 0}$. Then for any $f \in \cA$ and any $t>0$,
\begin{align}
\nabla P_t f(\x)&=\E[\V^{t\top}\,\nabla f(\btheta^t)
\mid (\btheta^0,\V^0)=(\x,\I)]\label{eq:BEL1}\\
&=\frac{1}{t\sqrt{2}}\E\bigg[f(\btheta^t)\int_0^t {\V^s}^\top \d\b^s
\;\bigg|\;(\btheta^0,\V^0)=(\x,\I)\bigg]
\label{eq:BEL2}
\end{align}
\end{lemma}
\begin{proof}
The first identity (\ref{eq:BEL1}) is the statement of \cite[Eq.\ (184)]{paper1}
(again specialized to this setting of dynamics with a fixed and non-adaptive
prior).

For the second identity (\ref{eq:BEL2}), we use from \cite[Proposition
A.2]{paper1} that for $f \in \cA$ and any fixed $t \geq 0$, 
$(s,\btheta) \mapsto P_{t-s}f(\btheta)$ is $C^1$ in $s \in [0,t]$ and
$C^2$ in $\btheta$, with $\partial_s P_{t-s}f(\btheta)={-}\L
P_{t-s}f(\btheta)$. Then It\^o's formula applied to $g(s,\btheta)=P_{t-s}f(\btheta)$ gives
\begin{align*}
f(\btheta^t)=g(t,\btheta^t)&=g(0,\btheta^0)+\int_0^t
\partial_s g(s,\btheta^s)\d s
+\int_0^t \nabla_{\btheta} g(s,\btheta^s)^\top \d\btheta^s
+\int_0^t \Tr \nabla_{\btheta}^2 g(s,\btheta^s)\d s\\
&=P_t f(\x)
+\int_0^t (\partial_s+\L)P_{t-s}f(\btheta^s)\d s
+\sqrt{2}\int_0^t \nabla P_{t-s}f(\btheta^s)^\top \d\b^s\\
&=P_t f(\x)+\sqrt{2}\int_0^t \nabla P_{t-s}f(\btheta^s)^\top \d\b^s.
\end{align*}
Since $\nabla^2 U$ is bounded, $\{\V^t\}_{t \in [0,T]}$ is bounded over finite
time horizons, so $\int_0^t \V^{s\top} \d\b^s$ is a martingale.
Multiplying both sides by this martingale and taking expectations gives
\[\E\bigg[ f(\btheta^t)\int_0^t {\V^s}^\top \d\b^s
\;\bigg|\;(\btheta^0,\V^0)=(\x,\I)\bigg]
=\sqrt{2}\int_0^t \E\big[{\V^s}^\top \nabla
P_{t-s}f(\btheta^s) \mid (\btheta^0,\V^0)=(\x,\I)]\d s.\]
Since $P_t f \in \cA$, we may apply (\ref{eq:BEL1}) with $P_{t-s}f$ in place
of $f$ to get
\[\int_0^t \E\big[{\V^s}^\top \nabla
P_{t-s}f(\btheta^s) \mid (\btheta^0,\V^0)=(\x,\I)]\d s
=\int_0^t \nabla P_s(P_{t-s}f)(\x)\d s
=t \cdot \nabla P_t f(\x).\]
Substituting above and rearranging shows (\ref{eq:BEL2}).
\end{proof}

\subsubsection{Interpretation of the DMFT correlation and
response}\label{sec:CRmatrixdef}

We remark that under Assumption \ref{assump:prior}(a), the log-posterior density
$\log \sP_g(\btheta \mid \X,\y)$ belongs to the function class $\cB$, and
$\sP_g(\btheta \mid \X,\y)$ is the unique stationary distribution of
(\ref{eq:langevinfixedprior}).
Fixing $\X,\y,\btheta^*$, consider the coordinate functions
\begin{equation}\label{eq:coordfuncs}
e_j(\btheta)=\theta_j,\qquad e_j^*(\btheta)=\theta_j^*,
\qquad x_i(\btheta)=\frac{\sqrt{\delta}}{\sigma^2}([\X\btheta]_i-y_i).
\end{equation}
(Here, $e_j^*$ is a constant function not depending on $\btheta$.)
We define their associated correlation and response matrices
\begin{equation}\label{eq:CRmatrices}
\begin{gathered}
\bC_\theta(t,s)=(C_{e_je_k}^{\btheta_0}(t,s))_{j,k=1}^d,
\quad \bC_\theta(t,*)=(C_{e_je_k^*}^{\btheta_0}(t,0))_{j,k=1}^d,
\quad \bR_\theta(t,s)=(R_{e_je_k}^{\btheta_0}(t,s))_{j,k=1}^d\\
\bC_\eta(t,s)=(C_{x_jx_k}^{\btheta_0}(t,s))_{j,k=1}^n,
\qquad \bR_\eta(t,s)=(R_{x_jx_k}^{\btheta_0}(t,s))_{j,k=1}^n
\end{gathered}
\end{equation}
where $C_{AB}^{\btheta_0}(t,s)$ and $R_{AB}^{\btheta_0}(t,s)$
are the correlation and response functions as defined
in Lemma \ref{lemma:langevin_fdt_equi} for these coordinate functions,
under the dynamics (\ref{eq:langevinfixedprior}) with fixed prior $g(\cdot)$
and the given initial condition $\btheta^0$ of Assumption \ref{assump:model}.

The following result is a direct application of \cite[Theorem 2.8]{paper1}.

\begin{theorem}[\cite{paper1}]\label{thm:dmft_response}
Suppose Assumptions \ref{assump:model} and \ref{assump:prior}(a) hold, and let
$C_\theta,C_\eta,R_\theta,R_\eta$ be the correlation and response functions of
the DMFT system in Theorem \ref{thm:dmft_approx}(a) approximating the dynamics
(\ref{eq:langevinfixedprior}). Then almost surely as $n,d \to \infty$,
\[\begin{gathered}
d^{-1}\Tr \bC_\theta(t,s) \to C_\theta(t,s),
\qquad d^{-1}\Tr \bC_\theta(t,*) \to C_\theta(t,*),
\qquad n^{-1}\Tr \bC_\eta(t,s) \to C_\eta(t,s)\\
d^{-1}\Tr \bR_\theta(t,s) \to R_\theta(t,s),
\qquad n^{-1}\Tr \bR_\eta(t,s) \to R_\eta(t,s).
\end{gathered}\]
\end{theorem}

\subsection{Posterior bounds and Wasserstein-2
convergence}\label{sec:posteriorbounds}

Fixing the prior $g(\cdot)$ and the data $(\X,\y)$, let us write for convenience
\begin{equation}\label{eq:qshorthand}
q(\btheta)=\sP_g(\btheta \mid \X,\y)
\propto \exp\bigg({-}\frac{1}{2\sigma^2}\|\y-\X\btheta\|_2^2
+\sum_{j=1}^d \log g(\theta_j)\bigg)
\end{equation}
for the posterior density. The Langevin diffusion (\ref{eq:langevinfixedprior})
with this fixed prior is then
\begin{equation}\label{eq:langevin_fixedq}
\d \btheta^t=\nabla \log q(\btheta^t)\d t+\sqrt{2}\,\d\b^t.
\end{equation}
We will use the notations
\begin{align*}
\langle f(\btheta) \rangle=\E_{\btheta \sim q}[f(\btheta)],
\qquad
P_tf(\x)=\langle f(\btheta^t) \rangle_{\x}=\E[f(\btheta^t)
\mid \btheta^0=\x]
\end{align*}
where the former is an average under the posterior law $q(\cdot)$ conditional on
$\X,\y$, and the latter is an average over $\{\btheta^t\}_{t \geq 0}$
solving (\ref{eq:langevin_fixedq}) conditional on $\X,\y$ and also
the initial condition $\btheta^0=\x$. We write as shorthand
\[P_t(\x)=\langle\btheta^t \rangle_\x=(P_t e_1,\ldots,P_t e_d)[\x] \in \R^d.\]
We reserve $\langle f(\btheta^t) \rangle$ for the expectation conditional on
$\X,\y$ but averaging also over $\btheta^0$.

For constants $C_0,C_\LSI>0$, define the $(\X,\btheta^*,\beps)$-dependent event
\begin{align}
\event(C_0,C_\LSI)&=\Big\{\|\X\|_\op \leq C_0,\;\|\btheta^*\|_2^2,\|\beps\|_2^2
\leq C_0 d,\;
\text{the LSI } (\ref{eq:LSI}) \text{ holds for } q(\btheta)\Big\}.
\end{align}
Note that under Assumptions \ref{assump:model} and \ref{assump:LSI}(a),
this event holds almost surely for all large $n,d$ for some
sufficiently large choices of constants $C_0,C_\LSI>0$.
All subsequent constants $C,C',c,c'>0$ in this section may change from
instance to instance, and are dimension-free and depend only on
\begin{equation}\label{eq:constantdependence}
C_0,C_{\text{LSI}} \text{ above},\;\delta,\sigma^2,g_* \text{ of Assumption \ref{assump:model}},
\; C,c_0,r_0 \text{ of Assumption \ref{assump:prior}(a)}, \text{ and }
\log g(0).
\end{equation}

We record the following elementary bounds for the posterior expectation
$\langle f(\btheta) \rangle=\E_{\btheta \sim q} f(\btheta)$.

\begin{lemma}\label{lemma:elementarybounds}
Suppose Assumption \ref{assump:prior}(a) holds. Then on
the event where $\|\X\|_\op \leq C_0$, there exists a constant $C>0$ for which
\begin{align}
\langle \|\btheta\|_2^2 \rangle &\leq
C(d+\|\y\|_2^2),\label{eq:posteriormeansq}\\
\langle \|\nabla \log q(\btheta)\|_2^2 \rangle &\leq C(d+\|\y\|_2^2),
\label{eq:posteriorlogqmeansq}\\
\|\nabla^2 \log q(\btheta)\|_\mathrm{op} &\leq C.\label{eq:negativeCD}
\end{align}
In particular, on $\event(C_0,C_\LSI)$, for a constant $C'>0$ we have
$\langle \|\btheta\|_2^2 \rangle \leq C'd$ and
$\langle \|\nabla \log q(\btheta)\|_2^2 \rangle \leq C'd$.
\end{lemma}
\begin{proof}
(\ref{eq:negativeCD}) is immediate from the form of $\log
q(\btheta)$, the bound $\|\X\|_\op \leq C_0$,
and Assumption \ref{assump:prior}(a).

For (\ref{eq:posteriormeansq}), write $\E_g,\P_g$ for the expectation
and probability over the prior $\theta \sim g$ and
$\theta_j \overset{iid}{\sim} g$. We note that under
Assumption \ref{assump:prior}(a), we have
\[\log g(\theta)=\log g(0)+\theta(\log g)'(0)+\int_0^\theta \int_0^x (\log
g)''(u)\d u\,\d x \leq C(1+|\theta|)-(c_0/2)(|\theta|-r_0)^2
\leq C'-c'\theta^2\]
for some constants $C,C',c'>0$ depending only on the constants of
Assumption \ref{assump:prior}(a) and on $\log g(0)$. Then
$g$ is subgaussian, and for some constants $C,c>0$
(c.f.\ \cite[Eq.\ (3.1)]{vershynin2018high})
\begin{equation}\label{eq:priortail}
\E_g\|\btheta\|_2^2 \leq Cd,
\qquad \P_g[\|\btheta\|_2^2-\E_g\|\btheta\|_2^2 \geq du]
\leq Ce^{-cdu} \text{ for all } u \geq 1.
\end{equation}
Write
\[q(\btheta)=\frac{1}{Z}\exp\bigg({-}\frac{\|\y-\X\btheta\|_2^2}{2\sigma^2}\bigg)
\prod_{j=1}^d g(\theta_j),
\qquad Z=\E_g
\bigg[\exp\bigg({-}\frac{\|\y-\X\btheta\|_2^2}{2\sigma^2}\bigg)\bigg].\]
We have by Jensen's inequality
${-}\log Z \leq \E_g[\|\y-\X\btheta\|_2^2/2\sigma^2]
\leq C(d+\|\y\|_2^2+\E_g \|\btheta\|_2^2)
\leq C'(d+\|\y\|_2^2)$. Then for any $M>0$, also bounding the
exponential from above by 1,
\[\Big\langle\|\btheta\|_2^2 \1\{\|\btheta\|_2^2 \geq M\}\Big\rangle
\leq \frac{1}{Z}\,\E_g\Big[\|\btheta\|_2^2 \1\{\|\btheta\|_2^2 \geq M\}\Big]
\leq e^{C'(d+\|\y\|_2^2)}\E_g\Big[\|\btheta\|_2^2 \1\{\|\btheta\|_2^2 \geq M\}\Big].\]
Integrating the tail bound (\ref{eq:priortail}) shows that this is less than
$d+\|\y\|_2^2$
for $M=C(d+\|\y\|_2^2)$ and a sufficiently large choice of constant
$C>0$. Thus
\[\langle \|\btheta\|_2^2 \rangle
\leq M+\Big \langle \|\btheta\|_2^2 \1\{\|\btheta\|_2^2 \geq M\} \Big\rangle
\leq C'(d+\|\y\|_2^2).\]
This shows (\ref{eq:posteriormeansq}). Since
$\nabla \log q(\btheta)$ is $C$-Lipschitz by (\ref{eq:negativeCD}), and
$\|\nabla \log q(0)\|_2^2 \leq 2\|\X^\top\y/\sigma^2\|_2^2
+2d \cdot (\log g)'(0)^2 \leq C(d+\|\y\|_2^2)$, the statement
(\ref{eq:posteriorlogqmeansq}) follows from (\ref{eq:posteriormeansq}).
\end{proof}

\begin{remark}
In a later proof, we will require that (\ref{eq:posteriormeansq}) holds in a
form
\begin{equation}\label{eq:posteriormeansqsigmasq}
\langle \|\btheta\|_2^2 \rangle \leq Cd+(C/\sigma^2)\|\y\|_2^2
\end{equation}
for all large noise variances $\sigma^2>0$, where $C>0$ is a constant not
depending on $\sigma^2$. This may be seen from the above arguments:
Writing now $C,C'>0$ for constants not depending on $\sigma^2$,
the above shows ${-}\log Z \leq (C'/\sigma^2)(d+\|\y\|_2^2)$,
and hence $\langle\|\btheta\|_2^2 \1\{\|\btheta\|_2^2 \geq M\}\rangle \leq
d+\|\y\|_2^2/\sigma^2$ for $M=Cd+(C/\sigma^2)\|\y\|_2^2$ with a sufficiently
large choice of constant $C>0$.
\end{remark}

\begin{lemma}\label{lem:W2contraction}
Suppose Assumption \ref{assump:prior}(a) holds.
Let $\{\btheta^t\}_{t \geq 0}$ be the solution to (\ref{eq:langevin_sde}) with
initial condition $\btheta^0 \sim q_0$, let $q_t(\btheta^t)$
be the law of $\btheta^t$, and let $W_2(\cdot)$ the Wasserstein-2 distance,
all conditional on $\X,\btheta^*,\beps$ (and averaging over $\btheta^0$).
Then on the event $\event(C_0,C_\LSI)$, there exists a constant $C>0$ such that
\begin{equation}\label{eq:W2decay}
W_2(q_t,q) \leq Ce^{-(2/C_\LSI)t}
\,W_2(q_0,q) \text{ for all } t \geq 0.
\end{equation}
\end{lemma}
\begin{proof}
For $t \in [0,1]$ we may apply a simple synchronous coupling and
Gr\"onwall argument: Let $\{\btheta^t\}_{t \geq 0}$ and
$\{\tilde\btheta^t\}_{t \geq 0}$ be the solutions of (\ref{eq:langevin_sde})
with initial conditions $\btheta^0 \sim q_0$ and $\tilde\btheta^0 \sim q$,
coupled by the same Brownian motion. Then
$\frac{\d}{\d t}\|(\btheta^t-\tilde\btheta^t)\|_2 \leq \|\tfrac{\d}{\d t}(\btheta^t-\tilde\btheta^t)\|_2
=\|\nabla \log q(\btheta^t)-\nabla \log q(\tilde \btheta^t)\|_2
\leq C\|\btheta^t-\tilde \btheta^t\|_2$ by definition
of the Langevin equation (\ref{eq:langevin_sde}) and by (\ref{eq:negativeCD}).
Hence
\begin{equation}\label{eq:thetadiffbound}
\|\btheta^t-\tilde\btheta^t\|_2
\leq e^{Ct}\|\btheta^0-\tilde\btheta^0\|_2.
\end{equation}
Letting $(\btheta^0,\tilde\btheta^0)$ be the coupling of $(q_0,q)$ for which
$\langle \|\btheta^0-\tilde\btheta^0\|_2^2 \rangle=W_2(q_0,q)^2$, we have that
$(\btheta^t,\tilde\btheta^t)$ is a coupling of $(q_t,q)$, so
\[W_2(q_t,q)^2 \leq \langle \|\btheta^t-\tilde\btheta^t\|_2^2 \rangle
\leq e^{2Ct}\langle \|\btheta^0-\tilde \btheta^0\|_2^2 \rangle
=e^{2Ct}W_2(q_0,q)^2.\]
Thus $W_2(q_t,q) \leq C'W_2(q_0,q)$ for all $t \in [0,1]$, which implies
(\ref{eq:W2decay}) for $t \in [0,1]$ and some $C>0$.

For $t \geq 1$, under the curvature-dimension lower bound
$-\nabla^2 \log q(\btheta) \succeq {-}L\operatorname{Id}$ for a constant 
$L>0$ that is implied by (\ref{eq:negativeCD}), we apply from \cite[Lemma
4.2]{bobkov2001hypercontractivity} that
\begin{equation}\label{eq:smoothing}
\DKL(q_1\|q) \leq \left(\frac{1}{4\alpha}+\frac{L}{2}\right)
W_2(q_0,q)^2, \qquad \alpha=\frac{e^{2L}-1}{2L}.
\end{equation}
Under the LSI condition of $\event(C_0,C_\LSI)$,
we have the exponential contraction of
relative entropy (c.f.\ \cite[Theorem 5.2.1]{bakry2014analysis})
\begin{equation}\label{eq:entropycontraction}
\DKL(q_t\|q) \leq e^{-2(t-1)/C_\LSI}\DKL(q_1\|q)
\text{ for all } t \geq 1.
\end{equation}
We have also the $T_2$-transportation inequality
(c.f.\ \cite[Theorem 9.6.1]{bakry2014analysis})
\begin{equation}\label{eq:T2}
W_2(q_t,q)^2 \leq C_\LSI\,\DKL(q_t\|q),
\end{equation}
and (\ref{eq:W2decay}) for $t \geq 1$ follows
follows from combining (\ref{eq:smoothing}), (\ref{eq:entropycontraction}), and
(\ref{eq:T2}).
\end{proof}

\subsection{Properties of the correlation and response}

In this section, on the event $\event(C_0,C_\LSI)$,
we now show approximate time-translation-invariance at large
times for the correlation and response matrices
$\bC_\theta,\bC_\eta,\bR_\theta,\bR_\eta$ defined in Section
\ref{sec:CRmatrixdef}. We may write these using our Markov semigroup
notation as
\[
\begin{gathered}
\bC_\theta(t,s)=\Big(\big\langle e_k(\btheta^s)P_{t-s}e_j(\btheta^s)
\big\rangle_{\btheta^0}
\Big)_{j,k=1}^d, \qquad
\bC_\eta(t,s)=\Big(\big\langle x_k(\btheta^s)P_{t-s}x_j(\btheta^s)
\big\rangle_{\btheta^0}
\Big)_{j,k=1}^n,\\
\bR_\theta(t,s)=\Big(\big\langle \nabla e_k(\btheta^s)^\top \nabla P_{t-s}
e_j(\btheta^s) \big\rangle_{\btheta^0}\Big)_{j,k=1}^d,
\qquad
\bR_\eta(t,s)=\Big(\langle
\nabla x_k(\btheta^s)^\top \nabla P_{t-s}x_j(\btheta^s) \big\rangle_{\btheta^0}
\Big)_{j,k=1}^n.
\end{gathered}\]

\begin{lemma}\label{lemma:TTI}
Suppose Assumption \ref{assump:prior}(a) holds.
Let $\bC_\theta,\bC_\eta,\bR_\theta,\bR_\eta$ be defined
for the dynamics (\ref{eq:langevinfixedprior}), and set
\[\begin{gathered}
\bC_\theta^\infty(\tau)=\Big(\big\langle e_k(\btheta)P_\tau e_j(\btheta) \rangle
\Big)_{j,k=1}^d,\qquad
\bC_\eta^\infty(\tau) = \Big(\big\langle x_k(\btheta)
P_\tau x_j(\btheta) \big\rangle \Big)_{j,k=1}^n,\\
\qquad \bR_\theta^\infty(\tau)=\Big(\big\langle \nabla e_k(\btheta)^\top \nabla P_\tau
e_j(\btheta)\big\rangle\Big)_{j,k=1}^d,
\qquad \bR_\eta^\infty(\tau) = \Big(\big\langle \nabla x_k(\btheta)^\top
\nabla P_\tau x_j(\btheta)\big\rangle\Big)_{j,k=1}^n
\end{gathered}\]
where $\langle \cdot \rangle$ is expectation under the posterior law
$q(\cdot)$.
Then on $\event(C_0,C_\LSI) \cap \{\|\btheta^0\|_2^2 \leq C_0d\}$,
there exist constants $C,c>0$ such that for all $t \geq s \geq 0$,
\begin{align}
|\Tr \bC_\theta(t,s)-\Tr \bC_\theta^\infty(t-s)|
&\leq Cde^{-cs}\label{eq:Cthetacompare}\\
|\Tr \bR_\theta(t,s)-\Tr \bR_\theta^\infty(t-s)|
&\leq Cde^{-cs}\label{eq:Rthetacompare}\\
|\Tr \bC_\eta(t,s)-\Tr \bC_\eta^\infty(t-s)|
&\leq Cde^{-cs}\label{eq:Cetacompare}\\
|\Tr \bR_\eta(t,s)-\Tr \bR_\eta^\infty(t-s)|
&\leq Cde^{-cs}\label{eq:Retacompare}
\end{align}
\end{lemma}

\begin{proof}
Momentarily
let $q_t$ be the law of $\btheta^t$ conditional on $(\X,\y)$ and also on
a fixed initial condition $\btheta^0=\x$. For any fixed $t \geq 0$, denote by
$\bvarphi^t \sim q$ a random vector such that
$(\btheta^t,\bvarphi^t)$ is a coupling of $(q_t,q)$ for which
$\langle \|\btheta^t-\bvarphi^t\|_2^2 \rangle_\x=W_2(q_t,q)^2$, where
$W_2(\cdot)$ is the Wasserstein-2 distance conditional on $\X,\y$ and
$\btheta^0=\x$.
Then observe that for any $M$-Lipschitz function $f$, we have
\begin{equation}\label{eq:Lipschitzcouplingbound}
\langle \|f(\btheta^t)-f(\bvarphi^t)\|_2^2 \rangle_\x \leq
M^2\langle \|\btheta^t-\bvarphi^t\|_2^2\rangle_\x =M^2W_2(q_t,q)^2.
\end{equation}
Furthermore $W_2(q_t,q)^2 \leq Ce^{-ct}W_2(\delta_\x,q)^2
\leq 2Ce^{-ct}(\|\x\|_2^2+\langle \|\btheta\|_2^2 \rangle)$
for all $t \geq 0$ by Lemma \ref{lem:W2contraction}.
Then, applying (\ref{eq:Lipschitzcouplingbound})
with $f(\x)=\x$ and $f(\x)=\nabla \log q(\x)$, and applying
(\ref{eq:posteriormeansq}--\ref{eq:negativeCD}) on the event
$\event(C_0,C_\LSI)$, we have the basic estimates
\begin{equation}\label{eq:W2contraction}
\begin{gathered}
\langle \|\btheta^t-\bvarphi^t\|_2^2\rangle_\x \leq Ce^{-ct}(\|\x\|_2^2+d),
\quad \langle \|\nabla \log q(\btheta^t)-\nabla \log
q(\bvarphi^t)\|_2^2\rangle_\x
 \leq Ce^{-ct}(\|\x\|_2^2+d),\\
\langle \|\btheta^t\|_2^2\rangle_\x \leq C(\|\x\|_2^2+d),
\quad \langle \|\nabla \log q(\btheta^t)\|_2^2\rangle_\x \leq
C(\|\x\|_2^2+d)\\
\langle \|\bvarphi^t\|_2^2 \rangle=\langle \|\btheta\|_2^2 \rangle \leq Cd,
\quad \langle \|\nabla \log q(\bvarphi^t)\|_2^2 \rangle
=\langle \|\nabla \log q(\btheta)\|_2^2\rangle \leq Cd.
\end{gathered}
\end{equation}
We note that also
\begin{align}
\|P_t(\x)-P_t(\tilde\x)\|_2^2 &\leq e^{Ct}\|\x-\tilde\x\|_2^2,\label{eq:PtLipschitz}\\
\|P_t(\x)-P_t(\tilde\x)\|_2^2 &\leq Ce^{-ct}(\|\x\|_2^2+\|\tilde\x\|_2^2+d).\label{eq:Ptlongtime}
\end{align}
Indeed, (\ref{eq:PtLipschitz}) follows from (\ref{eq:thetadiffbound}) and
Jensen's inequality. Also by Jensen's inequality and (\ref{eq:W2contraction}),
\begin{align}\label{eq:mean_concentrate}
\|P_t(\x)-\langle \btheta \rangle\|_2^2=\|\langle \btheta^t
-\bvarphi^t \rangle_\x\|_2^2 \leq Ce^{-ct}(\|\x\|_2^2+d),
\end{align}
and applying this bound for
both $P_t(\x)$ and $P_t(\tilde\x)$ yields (\ref{eq:Ptlongtime}).

For (\ref{eq:Cthetacompare}), note that for any $s,\tau \geq 0$,
\begin{equation}\label{eq:Cthetasemigroupform}
\Tr \bC_\theta(s+\tau,s)
=\sum_{j=1}^d \langle e_j(\btheta^s)P_\tau e_j(\btheta^s) \rangle_{\btheta^0}.
\end{equation}
Now let $q_s$ be the law of $\btheta^s$ conditional on $(\X,\y)$
and the given initial condition $\btheta^0$ of Assumption \ref{assump:prior},
and let
$(\btheta^s,\bvarphi^s)$ be the optimal Wasserstein-2 coupling of $(q_s,q)$
as above. Then
\begin{align}
|\Tr \bC_\theta(s+\tau,s)-\Tr \bC_\theta^\infty(\tau)|
&\leq \sum_{j=1}^d \big\langle\big|e_j(\btheta^s)P_\tau e_j(\btheta^s)
-e_j(\bvarphi^s)P_\tau e_j(\bvarphi^s)\big|\big\rangle_{\btheta^0}\notag\\
&\leq \sum_{j=1}^d \big\langle\big|(\theta_j^s-\varphi_j^s)P_\tau
e_j(\btheta^s)\big|\big\rangle_{\btheta^0}
+\big\langle\big|\varphi_j^s(P_\tau e_j(\btheta^s)-P_\tau
e_j(\bvarphi^s))\big|\big\rangle_{\btheta^0}\notag\\
&\leq \underbrace{\big\langle
\|\btheta^s-\bvarphi^s\|_2^2\big\rangle_{\btheta^0}^{1/2}
\big\langle\|P_\tau(\btheta^s)\|_2^2\big\rangle_{\btheta^0}^{1/2}}_{\mathrm{I}}
+\underbrace{\big\langle\|\bvarphi^s\|_2^2\big\rangle^{1/2}
\big\langle\|P_\tau(\btheta^s)-P_\tau(\bvarphi^s)\|_2^2\big\rangle_{\btheta^0}^{1/2}}_{\mathrm{II}}
\label{eq:Cthetacomparedecomp}
\end{align}
where we recall our shorthand
$P_t(\x)=\langle\btheta^t \rangle_\x \in \R^d$.

We have $\mathrm{I} \leq Cde^{-cs}$ for all $s \geq 0$ by
(\ref{eq:W2contraction}) and 
$\langle \|P_\tau(\btheta^s)\|_2^2 \rangle_{\btheta^0} \leq \langle
\|\btheta^{s+\tau}\|_2^2 \rangle_{\btheta^0}$
which follows from Jensen's inequality. For $\mathrm{II}$,
by (\ref{eq:PtLipschitz}) and (\ref{eq:W2contraction}), we have
\[\big\langle\|P_\tau(\btheta^s)-P_\tau(\bvarphi^s)\|_2^2
\big\rangle_{\btheta^0}
\leq e^{C\tau}\langle \|\btheta^s-\bvarphi^s\|_2^2 \rangle_{\btheta^0}
\leq e^{C\tau} \cdot Cde^{-cs}.\]
Choosing a large enough constant $s_0>0$, for $\tau \leq s/s_0$, this gives
$\langle \|P_\tau(\btheta^s)-P_\tau(\bvarphi^s)\|_2^2 \rangle_{\btheta^0}
\leq C'de^{-c's}$. For $\tau>s/s_0$, applying instead
(\ref{eq:Ptlongtime}) and (\ref{eq:W2contraction}), we have
$\langle \|P_\tau(\btheta^s)-P_\tau(\bvarphi^s)\|_2^2 \rangle_{\btheta^0}
\leq Ce^{-c\tau}(\langle \|\btheta^s\|_2^2 \rangle_{\btheta^0}
+\langle \|\bvarphi^s\|_2^2
\rangle+d) \leq C'de^{-c's}$. Thus, for some $C,c>0$,
\[\big\langle\|P_\tau(\btheta^s)-P_\tau(\bvarphi^s)\|_2^2
\big\rangle_{\btheta^0}
\leq Cde^{-cs} \text{ for all } s,\tau \geq 0.\]
Thus also $\mathrm{II} \leq Cde^{-cs}$, and applying these bounds for
$\mathrm{I}$ and $\mathrm{II}$ to (\ref{eq:Cthetacomparedecomp})
shows (\ref{eq:Cthetacompare}).

For (\ref{eq:Rthetacompare}), note that
\begin{equation}\label{eq:Rthetaform}
\Tr \bR_\theta(s+\tau,s)=\sum_{j=1}^d \langle (\partial_j P_\tau
e_j)[\btheta^s]\rangle_{\btheta^0},
\qquad \Tr \bR_\theta^\infty(\tau)=\sum_{j=1}^d \langle (\partial_j P_\tau
e_j)[\btheta] \rangle.
\end{equation}
Let $\d P_t(\x) \in \R^{d\times d}$ be the Jacobian of the vector map
$\x \mapsto P_t(\x)$. By (\ref{eq:BEL1}) of Lemma \ref{lemma:BEL} applied
with $f=e_j$ for each $j=1,\ldots,d$, we have
\begin{equation}\label{eq:Ptjac1}
\d P_t(\x) = \langle \V^t \rangle_\x
\end{equation}
where (with slight extension of the notation)
we write $\langle \cdot \rangle_\x$ for
the average over $\{\btheta^t,\V^t\}_{t \geq 0}$ solving
(\ref{eq:matrixV_sde}) with initial condition $(\btheta^0,\V^0)=(\x,\I)$.
For $t \geq 1$, let us write also $\nabla P_t e_j(\x)=\nabla P_1 f(\x)$ with
$f=P_{t-1}e_j$. Noting that $f \in \cA$
by (\ref{eq:semigroupregularity}),
we may apply (\ref{eq:BEL2}) of Lemma \ref{lemma:BEL} with this $f$.
Doing so for each $j=1,\ldots,d$ gives
\begin{align}\label{eq:Ptjac2}
\d P_t(\x) = \frac{1}{\sqrt{2}}
\bigg\langle P_{t-1}(\btheta^1) \bigg(\int_0^1 (\V^s)^\top \d
\b^s\bigg)^\top\bigg\rangle_{\x} \text{ for } t \geq 1.
\end{align}
In particular,
\begin{equation}\label{eq:tracePtjac}
\sum_{j=1}^d (\partial_j P_\tau e_j)[\x]
=\langle \Tr \V^\tau \rangle_{\x}=\frac{1}{\sqrt{2}}
\bigg\langle P_{\tau-1}(\btheta^1)^\top \int_0^1 (\V^s)^\top \d
\b^s\bigg\rangle_{\x}
\end{equation}
with the second equality holding for $\tau \geq 1$.

Now let $\{\btheta^t,\V^t\}_{t \geq 0}$ and $\{\tilde\btheta^t,\tilde \V^t\}_{t
\geq 0}$ be the solutions to (\ref{eq:matrixV_sde})
with initial conditions $(\btheta^0,\V^0)=(\x,\I)$ and
$(\tilde \btheta^0,\tilde \V^0)=(\tilde\x,\I)$,
coupled by the same Brownian motion $\{\b^t\}_{t \geq 0}$,
and write $\langle \cdot \rangle_{\x,\tilde\x}$ for the associated
average over $\{\btheta^t,\V^t,\tilde\btheta^t,\tilde\V^t\}_{t \geq 0}$
conditional on these initial conditions.
By the form of (\ref{eq:matrixV_sde}) and by (\ref{eq:negativeCD}),
$\frac{\d}{\d t}\|\V^t\|_\op
\leq \|\nabla^2 \log q(\btheta^t) \cdot \V^t\|_\op
\leq C\|\V^t\|_\op$, so
\begin{align}\label{eq:V_op_bound}
\|\V^t\|_\op \leq e^{Ct}\|\V^0\|_\op=e^{Ct}.
\end{align}
Then also
\begin{align*}
\frac{\d}{\d t}\|\V^t-\tilde\V^t\|_F &\leq
\|[\nabla^2 \log q(\btheta^t)-\nabla^2 \log q(\tilde \btheta^t)]\V^t\|_F
+\|[\nabla^2 \log q(\tilde \btheta^t)](\V^t-\tilde \V^t)\|_F\\
&\leq \|\nabla^2 \log q(\btheta^t)-\nabla^2 \log q(\tilde \btheta^t)\|_F
\|\V^t\|_\op+\|\nabla^2 \log q(\tilde \btheta^t)\|_\op \|\V^t-\tilde\V^t\|_F.
\end{align*}
Applying $\nabla^2 \log q(\btheta^t)-\nabla^2 \log q(\tilde \btheta^t)
=\diag((\log q)''(\theta^t_j)-(\log q)''(\tilde \theta^t_j))$,
the bound $\|\nabla^2 \log q(\btheta)\|_\op \leq C$ from
(\ref{eq:negativeCD}), $|(\log g)'''(\theta)| \leq C$ under Assumption
\ref{assump:prior}(a), and (\ref{eq:thetadiffbound}),
\[\frac{\d}{\d t}\|\V^t-\tilde\V^t\|_F 
\leq C\|\btheta^t-\tilde\btheta^t\|_2\|\V^t\|_\op
+C\|\V^t-\tilde\V^t\|_F
\leq Ce^{Ct}\|\x-\tilde\x\|_2 \cdot e^{Ct}+C\|\V^t-\tilde\V^t\|_F.\]
Integrating this bound,
\[\|\V^t-\tilde\V^t\|_F \leq C\|\x-\tilde\x\|_2 \text{ for all } t \in [0,1].\]
So it follows from the first equality of (\ref{eq:tracePtjac})
that for $\tau \in [0,1]$,
\[\left|\sum_{j=1}^d \partial_j P_\tau e_j(\x)
-\partial_j P_\tau e_j(\tilde\x)\right|
=\Big|\big\langle \Tr (\V^\tau-\tilde \V^\tau) \big\rangle_{\x,\tilde\x}\Big|
\leq \sqrt{d}\,\big\langle \|\V^\tau-\tilde \V^\tau\|_F\big\rangle_{\x,\tilde\x}
\leq C\sqrt{d}\|\x-\tilde\x\|_2.\]
Hence by (\ref{eq:Rthetaform}) and (\ref{eq:W2contraction}), for $\tau \in [0,1]$,
\begin{align}\label{eq:response_bound_1}
|\Tr \bR_\theta(s+\tau,s)-\Tr \bR_\theta^\infty(\tau)|
\leq C\sqrt{d}\,\langle\|\btheta^s-\bvarphi^s\|_2 \rangle_{\btheta^0}
\leq C'de^{-cs}.
\end{align}
For $\tau \geq 1$, we apply instead the second equality of (\ref{eq:tracePtjac}) and Cauchy-Schwarz to obtain
\begin{align}\label{eq:response_bound_2}
\notag&\sqrt{2}\left|\sum_{j=1}^d \partial_j P_\tau e_j(\x)
-\partial_j P_\tau e_j(\tilde\x)\right|\\
\notag&\leq \bigg|\bigg\langle P_{\tau-1}(\btheta^1)^\top\int_0^1 (\V^s)^\top\d\b^s
-P_{\tau-1}(\tilde \btheta^1)^\top\int_0^1 (\tilde
\V^s)^\top\d\b^s\bigg\rangle_{\x,\tilde\x}\bigg|\\
\notag&\leq \Big\langle \Big\|P_{\tau-1}(\btheta^1)
-P_{\tau-1}(\tilde\btheta^1)\Big\|_2^2 \Big\rangle_{\x,\tilde\x}^{1/2}
\bigg\langle \bigg\|\int_0^1 (\V^s)^\top\d\b^s\bigg\|_2^2\bigg\rangle_{\x}^{1/2}+
\Big\langle \Big\|P_{\tau-1}(\tilde\btheta^1)\Big\|_2^2 \Big\rangle_{\tilde\x}^{1/2}
\bigg\langle \bigg\|\int_0^1
(\V^s-\tilde\V^s)^\top\d\b^s\bigg\|_2^2\bigg\rangle_{\x,\tilde\x}^{1/2}\\
\notag&=\Big\langle \Big\|P_{\tau-1}(\btheta^1)
-P_{\tau-1}(\tilde\btheta^1)\Big\|_2^2 \Big\rangle_{\x,\tilde\x}^{1/2}
\bigg\langle \int_0^1 \|\V^s\|_F^2\d s\bigg\rangle_{\x}^{1/2}+
\Big\langle \Big\|P_{\tau-1}(\tilde\btheta^1)\Big\|_2^2 \Big\rangle_{\tilde\x}^{1/2}
\bigg\langle \int_0^1 \|\V^s-\tilde\V^s\|_F^2\d s\bigg\rangle_{\x,\tilde\x}^{1/2}\\
&\leq C\sqrt{d}\Big\langle \Big\|P_{\tau-1}(\btheta^1)
-P_{\tau-1}(\tilde\btheta^1)\Big\|_2^2 \Big\rangle_{\x,\tilde\x}^{1/2}
+C\|\x-\tilde\x\|_2\Big\langle \Big\|P_{\tau-1}(\tilde\btheta^1)\Big\|_2^2
\Big\rangle_{\tilde\x}^{1/2}.
\end{align}
Note that $\langle \|P_{\tau-1}(\tilde \btheta^1)\|_2^2 \rangle_{\tilde \x} \leq \langle
\|\tilde \btheta^\tau\|_2^2 \rangle_{\tilde\x} \leq C(\|\tilde\x\|_2^2+d)$ by (\ref{eq:W2contraction}).
Applying (\ref{eq:PtLipschitz}--\ref{eq:Ptlongtime}),
(\ref{eq:thetadiffbound}), and (\ref{eq:W2contraction}),
\begin{align*}
\Big\langle\Big\|P_{\tau-1}(\btheta^1)-P_{\tau-1}(\tilde\btheta^1)\Big\|_2^2
\Big\rangle_{\x,\tilde\x}
&\leq e^{2C(\tau-1)}\langle \|\btheta^1-\tilde\btheta^1\|_2^2 \rangle_{\x,\tilde\x}
\leq Ce^{2C\tau}\|\x-\tilde\x\|_2^2 \text{ for all } \tau \geq 1,\\
\Big\langle \Big\|P_{\tau-1}(\btheta^1)-P_{\tau-1}(\tilde\btheta^1)\Big\|_2^2
\Big \rangle_{\x,\tilde\x}
&\leq Ce^{-c(\tau-1)}(\langle \|\btheta^1\|_2^2 \rangle_\x+\langle
\|\tilde\btheta^1\|_2^2 \rangle_{\tilde\x}+d)\\
&\leq C'e^{-c\tau}(\|\x\|_2^2+\|\tilde\x\|_2^2+d) \text{ for all } \tau \geq 2.
\end{align*}
Choosing a large enough constant $s_0>0$, if $\tau \in [1,s/s_0]$, then we may
apply the former bound, (\ref{eq:W2contraction}), and Cauchy-Schwarz
to (\ref{eq:response_bound_2}) to get
\begin{align}\label{eq:response_bound_3}
|\Tr \bR_\theta(s+\tau,s)-\Tr \bR_\theta^\infty(\tau)|
&\leq \bigg\langle\bigg|\sum_{j=1}^d \partial_j P_\tau e_j(\btheta^s)
-\partial_j P_\tau e_j(\bvarphi^s)\bigg|\bigg\rangle_{\btheta^0}\notag\\
&\leq Ce^{C\tau}\sqrt{d}\,\langle \|\btheta^s-\bvarphi^s\|_2 \rangle_{\btheta^0}
+C\,\Big\langle\|\btheta^s-\bvarphi^s\|_2(\|\bvarphi^s\|_2+\sqrt{d})\Big\rangle_{\btheta^0}\notag\\
&\leq C'd(e^{C\tau}+1)e^{-cs} \leq C''de^{-c's}.
\end{align}
If $\tau \geq s/s_0$, applying instead the latter bound to
(\ref{eq:response_bound_2}),
\begin{align}\label{eq:response_bound_4}
|\Tr \bR_\theta(s+\tau,s)-\Tr \bR_\theta^\infty(\tau)| &\leq
Ce^{-c\tau}\sqrt{d}\,\big(\langle
\|\btheta^s\|_2^2\rangle_{\btheta^0}+\langle\|\bvarphi^s\|_2^2\rangle+d\big)^{1/2}
+C\,\Big\langle \|\btheta^s-\bvarphi^s\|_2(\|\bvarphi^s\|_2+\sqrt{d})\Big
\rangle_{\btheta^0}\notag\\
&\leq C'de^{-c's}.
\end{align}
Combining these bounds for $\tau \in [0,1]$, $\tau \in [1,s/s_0]$, and
$\tau \geq s/s_0$ in (\ref{eq:response_bound_1}), (\ref{eq:response_bound_3}),
and (\ref{eq:response_bound_4}) shows (\ref{eq:Rthetacompare}).

The arguments for $\bC_\eta$ and $\bR_\eta$ in
(\ref{eq:Cetacompare}--\ref{eq:Retacompare}) are similar:
For (\ref{eq:Cetacompare}), recall the definitions (\ref{eq:coordfuncs})
and note that
\begin{align*}
&\Big|\Tr \bC_\eta(s+\tau,s) - \Tr \bC_\eta^\infty(\tau)\Big|\\
&\leq \sum_{i=1}^n \Big\langle\Big|x_i(\btheta^s)P_\tau
x_i(\btheta^s)-x_i(\bvarphi^s)P_\tau x_i(\bvarphi^s)\Big|\Big\rangle_{\btheta^0}\\
&\leq \sum_{i=1}^n \Big\langle\Big|\big(x_i(\btheta^s) - x_i(\bvarphi^s)\big)
P_\tau x_i(\btheta^s)\Big|\Big\rangle_{\btheta^0} + \Big\langle\Big|x_i(\bvarphi^s)\big(P_\tau x_i(\btheta^s) -
P_\tau x_i(\bvarphi^s)\big)\Big|\Big\rangle_{\btheta^0}\\
&\leq \frac{\delta}{\sigma^4}\Big\langle \pnorm{\X(\btheta^s - \bvarphi^s)}{2}^2\Big\rangle_{\btheta^0}^{1/2}
\Big\langle \pnorm{\X P_\tau (\btheta^s)-\y}{2}^2\Big\rangle_{\btheta^0}^{1/2}
+\frac{\delta}{\sigma^4}\Big\langle \pnorm{\X\bvarphi^s-\y}{2}^2\Big\rangle^{1/2}
\Big\langle \pnorm{\X P_\tau (\btheta^s) - \X P_\tau
(\bvarphi^s)}{2}^2\Big\rangle_{\btheta^0}^{1/2}.
\end{align*}
The desired result (\ref{eq:Cetacompare}) follows from
the conditions $\|\X\|_\op \leq C_0$,
$\|\y\|_2^2 \leq C_0d$, and the preceding bounds for
(\ref{eq:Cthetacomparedecomp}).

For (\ref{eq:Retacompare}), note that
\begin{equation}\label{eq:Retaform}
\Tr \bR_\eta(s+\tau, s) = \frac{\delta}{\sigma^4}\big\langle\Tr \X[\d P_\tau(\btheta^s)]\X^\top\big\rangle_{\btheta^0},
\qquad \Tr \bR_\eta^\infty(\tau) = \frac{\delta}{\sigma^4}\big\langle\Tr \X[\d
P_\tau(\btheta)]\X^\top\big\rangle.
\end{equation}
Let $\{\btheta^t,\V^t\}_{t \geq 0}$ and $\{\tilde\btheta^t,\tilde \V^t\}_{t\geq
0}$ be the solutions of (\ref{eq:matrixV_sde}) with initial conditions
$(\x,\I)$ and $(\tilde \x,\I)$. If $\tau \in [0,1]$, we apply
(\ref{eq:Ptjac1}) to obtain
\[|\Tr \X[\d P_\tau(\x)]\X^\top-\Tr \X[\d P_\tau(\tilde \x)]\X^\top|
\leq \big\langle \pnorm{\V^\tau - \tilde{\V}^\tau}{F}\big\rangle_{\x,\tilde\x}
\cdot \pnorm{\X^\top \X}{F}
\leq \sqrt{d}\,\pnorm{\X}{\op}^2 \cdot \big\langle\pnorm{\V^\tau -
\tilde{\V}^\tau}{F}\big\rangle_{\x,\tilde\x},\]
which leads to the bound (\ref{eq:response_bound_1}) up to a different constant
depending on the bound $C_0$ for $\pnorm{\X}{\op}$. If $\tau \geq 1$, we apply
(\ref{eq:Ptjac2}) to obtain 
\begin{align*}
&\sqrt{2}\,\Big|\Tr \X[\d P_\tau(\x)]\X^\top - \Tr \X[\d
P_\tau(\tilde\x)]\X^\top\Big|\\
&\leq \Big\langle\Big|\Big(P_{\tau-1}(\btheta^1) -
P_{\tau-1}(\tilde{\btheta}^1)\Big)^\top \X^\top \X \Big(\int_0^1 {\V^s}^\top \d
\b^s\Big)\Big|\Big\rangle_{\x,\tilde\x} + \Big\langle\Big|P_{\tau-1}(\tilde{\btheta}^1)^\top \X^\top \X
\Big(\int_0^1 (\V^s - \tilde{\V}^s)^\top \d
\b^s\Big)\Big|\Big\rangle_{\x,\tilde\x}\\
&\leq \pnorm{\X}{\op}^2 \bigg[\Big\langle\pnorm{P_{\tau-1}(\btheta^1) -
P_{\tau-1}(\tilde{\btheta}^1)}{}^2\Big\rangle_{\x,\tilde\x}^{1/2} \bigg\langle \int_0^1
\|\V^s\|_F^2 \d s\bigg\rangle_\x^{1/2} +
\Big\langle\pnorm{P_{\tau-1}(\tilde{\btheta}^1)}{}^2\Big\rangle_{\tilde\x}^{1/2}
\bigg\langle \int_0^1 \|\V^s - \tilde{\V}^s\|_F^2 \d s\bigg\rangle_{\x,\tilde\x}^{1/2}\bigg].
\end{align*}
This can be bounded in the same way as (\ref{eq:response_bound_2}),
(\ref{eq:response_bound_3}), and (\ref{eq:response_bound_4})
up to different constants depending on the bound $C_0$ for $\pnorm{\X}{\op}$.
This shows (\ref{eq:Retacompare}).
\end{proof}

\begin{lemma}\label{lem:CRexpdecay}
Suppose Assumption \ref{assump:prior}(a) holds.
Let $\{\btheta^t\}_{t \geq 0}$ be the solution to
(\ref{eq:langevinfixedprior}).
Then on the event $\event(C_0,C_\LSI) \cap \{\|\btheta^0\|_2^2 \leq C_0d\}$,
there exist constants $C,c>0$ such that for all $t \geq s \geq 0$,
\begin{align}
\big|\Tr \bC_\theta(t,s)-P_s(\btheta^0)^\top\langle\btheta \rangle\big|
&\leq Cde^{-c(t-s)}\label{eq:Cthetadecay}\\
|\Tr \bR_\theta(t,s)| &\leq
Cde^{-c(t-s)}\label{eq:Rthetadecay}\\
\bigg|\Tr \bC_\eta(t,s)-
\frac{\delta}{\sigma^4}(\X P_s(\btheta^0)-\y)^\top(\X\langle\btheta\rangle-\y)\bigg|
&\leq Cde^{-c(t-s)}\label{eq:Cetadecay}\\
|\Tr \bR_\eta(t,s)| &\leq
Cde^{-c(t-s)}\label{eq:Retadecay}
\end{align}
and furthermore
\begin{align}
\big|P_s(\btheta^0)^\top \langle\btheta
\rangle-\|\langle\btheta\rangle\|_2^2\big|
&\leq Cde^{-cs}\label{eq:Cthetainitdecay}\\
\big|(\X P_s(\btheta^0)-\y)^\top
(\X\langle\btheta\rangle-\y)-\|\X\langle\btheta\rangle-\y\|_2^2
\big| &\leq Cde^{-cs}\label{eq:Cetainitdecay}
\end{align}
\end{lemma}
\begin{proof}
For (\ref{eq:Cthetadecay}) and (\ref{eq:Cthetainitdecay}), observe that
\begin{align*}
\big|\Tr \bC_\theta(s+\tau,s)-P_s(\btheta^0)^\top \langle\btheta\rangle\big|
&=\big|\big\langle \btheta^{s\top}(P_\tau
\btheta^s-\langle\btheta\rangle)\big\rangle_{\btheta^0}\big|\\
&\leq \big\langle\|\btheta^s\|_2^2\big\rangle_{\btheta^0}^{1/2}
\big\langle\|P_\tau \btheta^s-\langle\btheta\rangle\|_2^2\big\rangle_{\btheta^0}^{1/2}
\leq Cde^{-c\tau},
\end{align*}
the last inequality applying (\ref{eq:W2contraction}) and
(\ref{eq:mean_concentrate}). Similarly
\[\big|P_s(\btheta^0)^\top \langle\btheta\rangle-\|\langle\btheta
\rangle\|_2^2\big|
=\big|(P_s\btheta^0-\langle\btheta \rangle)^\top \langle\btheta \rangle\big|
\leq \|\langle\btheta\rangle\|_2 \cdot \|P_s\btheta^0-\langle\btheta\rangle\|_2
\leq Cde^{-cs}.\]

For (\ref{eq:Rthetadecay}), recall from (\ref{eq:Rthetaform}) that
$\Tr \bR_\theta(s+\tau,s)=\sum_{j=1}^d \langle (\partial_j P_\tau
e_j)[\btheta^s]\rangle_{\btheta^0}$.
Then by the first equality of (\ref{eq:tracePtjac}) and (\ref{eq:V_op_bound}),
we have $|\Tr \bR_\theta(s+\tau,s)| \leq C d$ for any $\tau \in [0,1]$.
For $\tau \geq 1$, we apply instead the second equality of (\ref{eq:tracePtjac}) where
$\int_0^t {\V^s}^\top \d \b^s$ is a martingale. Then
$\langle \langle \btheta\rangle^\top \int_0^1 {\V^s}^\top \d \b^s\rangle_\x=0$
for any initial condition $\x \in \R^d$, so for any $\tau \geq 1$,
\begin{align*}
\left|\sum_{j=1}^d \partial_j P_\tau e_j(\x)\right|
&=\left|\bigg\langle\Big(P_{\tau-1}(\btheta^1)-\langle\btheta\rangle\Big)^\top
\int_0^1 {\V^s}^\top \d \b^s\bigg\rangle_\x\right|\\
&\leq \big\langle\pnorm{P_{\tau-1}(\btheta^1) -
\langle\btheta\rangle}{}^2\rangle_\x^{1/2}
\bigg\langle\int_0^1 \|\V^s\|_F^2 \d s\bigg\rangle_\x^{1/2}
\leq Ce^{-c\tau}\sqrt{d}(\|\x\|_2+\sqrt{d}),
\end{align*}
the last inequality
using the estimates (\ref{eq:mean_concentrate}) and (\ref{eq:V_op_bound}). Then
by (\ref{eq:W2contraction}) and Jensen's inequality,
\begin{align*}
|\Tr \bR_\theta(s+\tau,s)| \leq
\bigg\langle\bigg|\sum_{j=1}^d (\partial_j P_\tau
e_j)[\btheta^s]\bigg|\bigg\rangle_{\btheta^0}
\leq C'de^{-c'\tau}.
\end{align*}
Combining these cases $\tau \in [0,1]$ and $\tau \geq 1$ gives
(\ref{eq:Rthetadecay}).

The arguments for (\ref{eq:Cetadecay}), (\ref{eq:Retadecay}), and
(\ref{eq:Cetainitdecay}), are analogous to the above, and we omit these
for brevity.
\end{proof}

\subsection{The DMFT system is approximately-TTI}

We now prove Theorem \ref{thm:fixedalpha_dynamics}, that under the log-Sobolev
condition of Assumption \ref{assump:LSI}(a), the DMFT system of Theorem
\ref{thm:dmft_approx}(a) is approximately-TTI in the
sense of Definition \ref{def:regular}.

\begin{lemma}\label{lem:fixedalpha_C}
Under Assumptions \ref{assump:model}, \ref{assump:prior}(a), and
\ref{assump:LSI}(a), the DMFT
system prescribed by Theorem \ref{thm:dmft_approx}(a)
satisfies the conditions of Definition \ref{def:regular}(1)
with $\eps(t)=Ce^{-ct}$ for some constants $C,c>0$.
\end{lemma}

\begin{proof}
We restrict to the almost sure event where the
convergence statements of Theorem \ref{thm:dmft_response} hold, and where
$\event(C_0,C_\LSI) \cap \{\|\btheta^0\|_2^2 \leq C_0d\}$ holds
for all large $n,d$.

Consider first the statements for $C_\theta(t,s)$. Applying
$\|\btheta^0\|_2^2 \leq C_0d$ and
(\ref{eq:Cthetadecay}) of Lemma \ref{lem:CRexpdecay},
for some constants $C,c>0$,
\begin{equation}\label{eq:Cthetalimsmalls}
\limsup_{n,d \to \infty} \big|d^{-1}\Tr \bC_\theta(t,s)
-d^{-1}P_s(\btheta^0)^\top \langle\btheta \rangle\big|
\leq Ce^{-ct} \text{ for all } s \leq t/2.
\end{equation}
By Theorem \ref{thm:dmft_response}, 
$\lim_{n,d \to \infty} d^{-1} \Tr \bC_\theta(t,s)=C_\theta(t,s)$ for all
$t \geq s \geq 0$. Then, for each $s \geq 0$ and $t \geq 2s$,
\[\limsup_{n,d \to \infty} d^{-1}
P_s(\btheta^0)^\top\langle\btheta\rangle
\leq C_\theta(t,s)+Ce^{-ct},
\qquad \liminf_{n,d \to \infty} d^{-1}
P_s(\btheta^0)^\top\langle\btheta \rangle
\geq C_\theta(t,s)-Ce^{-ct}.\]
Taking $t \to \infty$ on the right side of both statements
shows that for each $s \geq 0$, there exists a limit
\begin{equation}\label{eq:cthetainit}
\tilde c_\theta(s):=\lim_{n,d \to \infty} d^{-1}
P_s(\btheta^0)^\top \langle \btheta \rangle
=\lim_{t \to \infty} C_\theta(t,s).
\end{equation}
Next, (\ref{eq:Cthetainitdecay}) of Lemma \ref{lem:CRexpdecay} implies
for some $C,c>0$,
\begin{equation}\label{eq:tildecsbound}
\limsup_{n,d \to \infty} \big|d^{-1}
P_s(\btheta^0)^\top \langle\btheta\rangle
-d^{-1}\|\langle\btheta\rangle\|_2^2\big| \leq Ce^{-cs} \text{ for all } s \geq 0.
\end{equation}
Then
\[\limsup_{n,d \to \infty} d^{-1}\|\langle\btheta\rangle\|_2^2
\leq \tilde c_\theta(s)+Ce^{-cs},
\qquad \liminf_{n,d \to \infty} d^{-1}\|\langle\btheta\rangle\|_2^2
\geq \tilde c_\theta(s)-Ce^{-cs}.\]
Taking $s \to \infty$ on the right side of both statements shows that there
exists a limit
\begin{equation}\label{eq:cthetainfty}
c_\theta^\tti(\infty):=\lim_{n,d \to \infty} d^{-1}\|\langle\btheta\rangle\|_2^2
=\lim_{s \to \infty} \tilde c_\theta(s).
\end{equation}

Now consider $\bC_\theta^\infty(\tau)$ as defined in Lemma \ref{lemma:TTI}.
Let ${-}\L=\int_0^\infty a\,\d E_a$ be the spectral decomposition of ${-}\L$
as a positive, self-adjoint operator on $L^2(q)$
(c.f.\ \cite[Theorem A.4.2]{bakry2014analysis}), where $\{E_a\}_{a \geq 0}$ is a
family of orthogonal projections onto an increasing family of closed linear
subspaces of $L^2(q)$. In particular, $E_0f=\langle f(\btheta) \rangle$
is the projection onto the constant functions.
For each $\tau \geq 0$ and all $f,g \in L^2(q)$, we then have
\begin{equation}\label{eq:spectraldecomp}
\langle f(\btheta)P_\tau g(\btheta) \rangle=\int_0^\infty e^{-a \tau}\d \langle
f(\btheta)E_a g(\btheta) \rangle
\end{equation}
understood as a Stieltjes integral with respect to the bounded-variation
function $a \mapsto \langle f(\btheta)E_a g(\btheta) \rangle$
(c.f.\ \cite[Proposition 3.1.6(iii)]{bakry2014analysis}).
The LSI on the event $\event(C_0,C_\LSI)$ implies a
spectral gap, i.e.\ the spectrum of ${-}\L$ is
included in $\{0\} \cup [1/C_\LSI,\infty)$.
Thus, fixing any constant $\iota \in (0,1/C_\LSI)$, we have
\begin{align*}
d^{-1}\Tr \bC_\theta^\infty(\tau)
=d^{-1}\sum_{j=1}^d \langle e_j(\btheta)P_\tau e_j(\btheta) \rangle
&=d^{-1}\sum_{j=1}^d \langle e_j(\btheta)E_0e_j(\btheta) \rangle
+d^{-1}\sum_{j=1}^d \int_\iota^\infty e^{-a\tau} \d\langle e_j(\btheta)E_a
e_j(\btheta) \rangle\\
&=d^{-1}\|\langle \btheta\rangle\|_2^2
+d^{-1}\sum_{j=1}^d \int_\iota^\infty e^{-a\tau} \d\langle e_j(\btheta)E_a
e_j(\btheta) \rangle,
\end{align*}
the first equality applying (\ref{eq:spectraldecomp}), and
the second equality applying $\sum_j \langle e_j(\btheta)E_0e_j(\btheta)
\rangle=\sum_j \langle \theta_j\langle \theta_j\rangle\rangle
=\|\langle \btheta\rangle \|^2$.
Define $(n,d,\X,\y)$-dependent scalars $c_{\theta,d},m_{\theta,d}>0$
and a positive measure $\mu_{\theta,d}$ on $[\iota,\infty)$ by
\begin{equation}\label{eq:muddef}
c_{\theta,d}=d^{-1}\|\langle \btheta \rangle\|_2^2, \quad 
m_{\theta,d}=d^{-1}\sum_{j=1}^d \int_\iota^\infty \d\langle e_j(\btheta)E_a e_j(\btheta)
\rangle, \quad
\mu_{\theta,d}(S)=d^{-1}\sum_{j=1}^d \int_S \d\langle e_j(\btheta)E_a e_j(\btheta)
\rangle,
\end{equation}
noting that $a \mapsto d^{-1}\sum_{j=1}^d
\langle e_j(\btheta)E_a e_j(\btheta) \rangle$ is nondecreasing and hence
defines a valid distribution function for $\mu_{\theta,d}$. Then
\begin{equation}\label{eq:Cthetamuexpr}
d^{-1}\Tr \bC_\theta^\infty(\tau)
=c_{\theta,d}+\int_\iota^\infty e^{-a\tau} \mu_{\theta,d}(\d a),
\qquad m_{\theta,d}=\mu_{\theta,d}([\iota,\infty)).
\end{equation}

Applying (\ref{eq:Cthetamuexpr}) with $\tau=0$,
\begin{equation}\label{eq:cmtheta}
c_{\theta,d}+m_{\theta,d}=d^{-1}\Tr \bC_\theta^\infty(0)=d^{-1}\langle \|\btheta\|_2^2\rangle
\leq C.
\end{equation}
In particular, $\mu_{\theta,d}$ is finite and uniformly bounded in total
variation norm for all $(n,d)$.
We claim that $\tau \mapsto d^{-1}\Tr \bC_\theta^\infty(\tau)$ is uniformly
equicontinuous over all $(n,d)$: Observe that
\begin{align}
\Big|d^{-1}\Tr \bC_\theta^\infty(\tau)-d^{-1}\Tr \bC_\theta^\infty(\tau')\Big|
&=\Big|d^{-1}\big\langle \btheta^\top [P_\tau-P_{\tau'}](\btheta)\big\rangle\Big|\notag\\
&\leq d^{-1}\langle\|\btheta\|_2^2\rangle^{1/2} \cdot
\langle\|[P_\tau-P_{\tau'}](\btheta)\|_2^2\rangle^{1/2}\notag\\
&=d^{-1}\langle\|\btheta\|_2^2\rangle^{1/2} \cdot
\langle\|\btheta-P_{|\tau-\tau'|}(\btheta)\|_2^2\rangle^{1/2}.\label{eq:equicontinuous}
\end{align}
\cite[Theorem II.2.1]{kunita1984stochastic} implies
$\|P_t(\x)-\x\|_2^2 \leq C(1+\|\x\|_2^2)t$ for 
all $t \in [0,1]$ and a constant $C>0$.
This and (\ref{eq:W2contraction}) imply
that the right side of (\ref{eq:equicontinuous}) is at most
$C'|\tau-\tau'|$ for a constant $C'>0$ and all $|\tau-\tau'| \leq 1$,
so $\tau \mapsto d^{-1}\Tr \bC_\theta^\infty(\tau)$ is uniformly
equicontinuous as claimed. We note that for any $M>0$, by the relation
(\ref{eq:Cthetamuexpr}),
$c_{\theta,d}+\mu_{\theta,d}([0,M))+e^{-M\tau}\mu_{\theta,d}([M,\infty))
\geq d^{-1}\Tr \bC_\theta^\infty(\tau)$. Then setting $\tau=1/M$ and rearranging
yields
\[(1-e^{-1})\mu_{\theta,d}([M,\infty)) \leq c_{\theta,d}+m_{\theta,d}-d^{-1}\Tr
\bC_\theta^\infty(1/M)=d^{-1}\Tr \bC_\theta^\infty(0)-d^{-1}\Tr
\bC_\theta^\infty(1/M).\]
So this uniform equicontinuity implies 
that the measures $\mu_{\theta,d}$ are uniformly tight.

Then, there exists a subsequence $\{(n_k,d_k)\}_{k \geq 1}$ of $(n,d)$ along
which $\mu_{\theta,d} \Rightarrow \mu_\theta$ weakly, for some finite
positive measure $\mu_\theta$ on $[\iota,\infty)$. Recalling also that
$c_{\theta,d}=d^{-1}\|\langle\btheta\rangle\|_2^2 \to c_\theta^\tti(\infty)$ as $n,d \to \infty$
by the definition (\ref{eq:cthetainfty}), and setting
\begin{equation}\label{eq:cthetadef}
c_\theta^\tti(\tau)=c_\theta^\tti(\infty)+\int_\iota^\infty e^{-a\tau}\mu_\theta(\d a),
\end{equation}
this weak convergence applied to (\ref{eq:Cthetamuexpr}) implies
$\lim_{k \to \infty} d_k^{-1}\Tr \bC_\theta^\infty(\tau)=c_\theta^\tti(\tau)$.
Combining this with the convergence
$\lim_{k \to \infty} d_k^{-1}\Tr \bC_\theta(s+\tau,s)=C_\theta(s+\tau,s)$
by Theorem \ref{thm:dmft_response}, for any $s,\tau \geq 0$ we have
\begin{equation}\label{eq:Cthetattiapprox}
\Big|C_\theta(s+\tau,s)-c_\theta^\tti(\tau)\Big|
\leq \limsup_{k \to \infty} \Big|d_k^{-1}\Tr \bC_\theta(s+\tau,s)
-d_k^{-1}\Tr \bC_\theta^\infty(\tau)\Big| \leq Ce^{-cs},
\end{equation}
where the last inequality holds by (\ref{eq:Cthetacompare}).
Since $C_\theta(t,s)$ is non-random,
this implies that $c_\theta^\tti(\tau)$ is also non-random for every
$\tau \geq 0$, and thus also the measure $\mu_\theta$ is non-random. This shows (\ref{eq:Cthetaapproxinvariant}).

The statement (\ref{eq:Cetaapproxinvariant}) follows analogously:
By arguments parallel to (\ref{eq:cthetainit}) and (\ref{eq:cthetainfty}), applying
Theorem \ref{thm:dmft_response} and (\ref{eq:Cetadecay}) and
(\ref{eq:Cetainitdecay}) shows that there exist limits
\begin{align}
\tilde c_\eta(s)&:=\lim_{n,d \to \infty}
\frac{\delta}{n\sigma^4} (\X P_s(\btheta^0)-\y)^\top
(\X\langle \btheta \rangle-\y)=\lim_{t \to \infty} C_\eta(t,s),
\label{eq:cetainit}\\
c_\eta^\tti(\infty)&:=\lim_{n,d \to \infty} \frac{\delta}{n\sigma^4}
\|\X\langle \btheta \rangle-\y\|_2^2=\lim_{s \to \infty} \tilde c_\eta(s).
\label{eq:cetainfty}
\end{align}
Note that
\[n^{-1}\Tr \bC_\eta^\infty(\tau)
=n^{-1}\sum_{i=1}^n
\big\langle x_i(\btheta)P_\tau x_i(\btheta)\big\rangle
=\frac{\delta}{n\sigma^4}\|\X\langle \btheta \rangle-\y\|_2^2
+\frac{1}{n}\sum_{i=1}^n \int_\iota^\infty
e^{-a\tau}\d\langle x_i(\btheta)E_a x_i(\btheta) \rangle.\]
Defining
\begin{equation}\label{eq:cetamuexpr}
c_{\eta,n}=\frac{\delta}{n\sigma^4}\|\X\langle \btheta \rangle-\y\|_2^2,
\quad \mu_{\eta,n}(S)=\frac{1}{n}\sum_{i=1}^n \int_S
\d\langle x_i(\btheta)E_a x_i(\btheta) \rangle,
\quad m_{\eta,n}=\mu_{\eta,n}([\iota,\infty)),
\end{equation}
we have
\begin{equation}\label{eq:cmeta}
c_{\eta,n}+m_{\eta,n}=n^{-1}\Tr \bC_\eta^\infty(0)
=\frac{\delta}{n\sigma^4}\langle \|\X\btheta-\y\|_2^2 \rangle \leq C.
\end{equation}
So along some subsequence $\{(n_k,d_k)\}_{k \geq 1}$, we have
$c_{\eta,n} \to c_\eta^\tti(\infty)$, $\mu_{\eta,n} \Rightarrow \mu_\eta$
weakly for a finite positive measure $\mu_\eta$ on $[\iota,\infty)$, and
$\lim_{k \to \infty} n_k^{-1}\Tr \bC_\eta^\infty(\tau)
=c_\eta^\tti(\tau)$ for the quantity
\[c_\eta^\tti(\tau)=c_\eta^\tti(\infty)+\int_\iota^\infty e^{-a\tau}\mu_\eta(\d a).\]
By an argument parallel to (\ref{eq:Cthetattiapprox}) using Theorem \ref{thm:dmft_response} and (\ref{eq:Cetacompare}), this shows
$|C_\eta(s+\tau,s)-c_\eta^\tti(\tau)| \leq Ce^{-cs}$, establishing (\ref{eq:Cetaapproxinvariant}).

Finally, for (\ref{eq:Cthetastarlim}), observe that by Theorem
\ref{thm:dmft_response}, $\lim_{n,d \to \infty} d^{-1}P_s(\btheta^0)^\top
\btheta^*=C_\theta(s,*)$. Noting that
\[\limsup_{n,d \to \infty}
d^{-1}\big|(P_s(\btheta^0)-\langle\btheta\rangle)^\top \btheta^*\big|
\leq \limsup_{n,d \to \infty}
d^{-1}\|P_s \btheta^0-\langle\btheta \rangle\|_2 \cdot \|\btheta^*\|_2
\leq Ce^{-cs}\]
by (\ref{eq:mean_concentrate}), this implies the existence of the limit
\begin{equation}\label{eq:cthetastar}
c_\theta(*):=\lim_{n,d \to \infty} d^{-1}\langle\btheta \rangle^\top \btheta^*
=\lim_{s \to \infty} C_\theta(s,*),
\end{equation}
which satisfies $|C_\theta(s,*)-c_\theta(*)| \leq Ce^{-cs}$.
This shows (\ref{eq:Cthetastarlim}).
\end{proof}

\begin{lemma}\label{lem:fixedalpha_R}
Under Assumptions \ref{assump:model}, \ref{assump:prior}(a), and
\ref{assump:LSI}(a), the DMFT
system prescribed by Theorem \ref{thm:dmft_approx}(a)
satisfies the conditions of Definition \ref{def:regular}(2)
with $\eps(t)=Ce^{-ct}$ for some constants $C,c>0$.
\end{lemma}
\begin{proof}
We again restrict to the almost sure event where the convergence statements of
Theorem \ref{thm:dmft_response} hold,
and where $\event(C_0,C_\LSI) \cap \{\|\btheta^0\|_2^2 \leq C_0d\}$ holds
for all large $n,d$.

Consider first $R_\theta(t,s)$. By
(\ref{eq:Rthetadecay}) of Lemma \ref{lem:CRexpdecay} and the convergence
$R_\theta(t,s)=\lim_{n,d \to \infty} d^{-1}\bR_\theta(t,s)$ of Theorem
\ref{thm:dmft_response},
\begin{equation}\label{eq:Rthetaapproxsmalls}
|R_\theta(t,s)| \leq Ce^{-ct} \text{ for all } s \leq t/2.
\end{equation}
For $s \geq t/2$, note that
the forms of (\ref{eq:Cthetamuexpr}) and (\ref{eq:cthetadef}) imply that both
$d^{-1}\Tr \bC_\theta^\infty(\tau)$ and $c_\theta^\tti(\tau)$
are convex and differentiable in $\tau \geq 0$. Then, 
along the subsequence $\{(n_k,d_k)\}_{k \geq 1}$
of the preceding proof, the pointwise convergence
$\lim_{k \to \infty} d_k^{-1}\Tr
\bC_\theta^\infty(\tau)=c_\theta^\tti(\tau)$  implies also
$\lim_{k \to \infty} d_k^{-1} \partial_\tau \Tr \bC_\theta^\infty(\tau)
=\partial_\tau c_\theta^\tti(\tau)$
for each $\tau \geq 0$ (c.f.\ \cite[Theorem 25.7]{rockafellar1997convex}). By
the fluctuation-dissipation relation of Lemma \ref{lemma:langevin_fdt_equi}
applied with $A=B=e_j$ for each $j=1,\ldots,d$, we have
$\partial_\tau \Tr \bC_\theta^\infty(\tau)={-}\Tr \bR_\theta^\infty(\tau)$.
Then, defining $r_\theta^\tti(\tau)={-}\partial_\tau c_\theta^\tti(\tau)$,
this shows $\lim_{k \to \infty} d_k^{-1} \Tr \bR_\theta^\infty(\tau)
=r_\theta^\tti(\tau)$. Combining with
$\lim_{k \to \infty} d_k^{-1} \Tr \bR_\theta(s+\tau,s)=R_\theta(s+\tau,s)$
from Theorem \ref{thm:dmft_response}, for any $s,\tau \geq 0$ we have that
\[\big|R_\theta(s+\tau,s)-r_\theta^\tti(\tau)\big|
\leq \limsup_{k \to \infty}
\Big|d_k^{-1}\Tr \bR_\theta(s+\tau,s)-d_k^{-1}\Tr \bR_\theta^\infty(\tau)\Big|
\leq Ce^{-cs},\]
where the last inequality applies (\ref{eq:Rthetacompare}).
In particular, for any $t \geq 0$,
\begin{equation}\label{eq:Rthetaapproxlarges}
|R_\theta(t,s)-r_\theta^\tti(t-s)| \leq Ce^{-c't}
\text{ for all } s \in [t/2,t].
\end{equation}
Together, (\ref{eq:Rthetaapproxsmalls}) and (\ref{eq:Rthetaapproxlarges})
imply (\ref{eq:Rthetaapproxinvariant}).
The statement (\ref{eq:Retaapproxinvariant}) follows analogously, and we omit
this for brevity.
\end{proof}

\begin{proof}[Proof of Theorem \ref{thm:fixedalpha_dynamics}]
This follows from Lemmas \ref{lem:fixedalpha_C} and \ref{lem:fixedalpha_R}.
\end{proof}

\subsection{Limit MSE and free energy}

We now show Corollary \ref{cor:fixedalpha_dynamics} on the asymptotic
values of the mean-squared-errors and the free energy.

\begin{proposition}\label{prop:mseconcentration}
Suppose Assumptions \ref{assump:model}, \ref{assump:prior}(a), and
\ref{assump:LSI}(a) hold. Let $\YMSE_*$ and the marginal likelihood
$\sP_g(\y \mid \X)$ be as defined in Corollary \ref{cor:fixedalpha_dynamics}.
Let $\E[\cdot\mid \X]$ denote the expectation with respect to $\theta_j^*
\overset{iid}{\sim} g_*$ and $\eps_i \overset{iid}{\sim} \N(0,\sigma^2)$
conditioning on $\X$. Then almost surely,
\begin{equation}\label{eq:mseconcentration}
\lim_{n,d \to \infty}
d^{-1}\log \sP_g(\y \mid \X)-d^{-1}\E[\log \sP_g(\y \mid \X) \mid \X]=0,
\qquad \lim_{n,d \to \infty} \YMSE_*-\E[\YMSE_* \mid \X]=0.
\end{equation}
\end{proposition}
\begin{proof}
We condition on $\X$ throughout, and restrict to the $\X$-dependent event
\[\{\|\X\|_\op \leq C_0 \text{ and } (\ref{eq:LSI}) \text{ holds}\}.\]
Note that by assumption, this event holds a.s.\ for all large $n,d$ and
does not depend on $\btheta^*,\beps$.

For the first statement, let us consider
\begin{align*}
Z(\btheta^*,\beps)=\log \int
\exp\bigg({-}\frac{1}{2\sigma^2}\|\X\btheta^*+\beps-\X\btheta\|_2^2
+\sum_{j=1}^d \log g(\theta_j)\bigg)\d\btheta
\end{align*}
(which coincides with $\log P_g(\y \mid \X)$ up to an additive constant)
as a function of $(\btheta^*,\beps)$. Then
\begin{align*}
\nabla_{\btheta^*} Z(\btheta^*,\beps)
={-}\frac{1}{\sigma^2} \X^\top(\X\btheta^*+\beps-\X\langle \btheta \rangle),
\qquad \nabla_{\beps} Z(\btheta^*,\beps)
={-}\frac{1}{\sigma^2} (\X\btheta^*+\beps-\X\langle \btheta \rangle).
\end{align*}
Under Assumption \ref{assump:model}, note that
$\btheta^*$ and $\beps$ have independent
subgaussian entries, so there are constants $C_1,c>0$ such that
(c.f.\ \cite[Eq.\ (3.1)]{vershynin2018high})
\begin{equation}\label{eq:thetaepsnormbound}
\P[\|\btheta^*\|_2^2+\|\beps\|_2^2>C_1d] \leq e^{-cd}.
\end{equation}
When $\|\btheta^*\|_2^2+\|\beps\|_2^2 \leq C_1d$, we have the bound
$\langle \|\btheta\|_2^2 \rangle \leq Cd$ from (\ref{eq:posteriormeansq}).
Applying this and $\|\X\|_\op \leq C_0$,
\[\|\nabla_{(\btheta^*,\beps)} Z(\btheta^*,\beps)\|_2
\1\{\|\btheta^*\|_2^2+\|\beps\|_2^2 \leq C_1d\} \leq L\sqrt{d}\]
for a constant $L>0$. Thus $Z(\btheta^*,\beps)$ is $L\sqrt{d}$-Lipschitz on
$\{\|\btheta^*\|_2^2+\|\beps\|_2^2 \leq C_1d\}$, so its Lipschitz extension
\[\tilde Z(\btheta^*,\beps)=\inf_{\x \in \R^{d+n}:\|\x\|_2^2 \leq C_1d}
Z(\x)+L\sqrt{d}\|\x-(\btheta^*,\beps)\|_2\]
is globally $L\sqrt{d}$-Lipschitz on $\R^{d+n}$
and $\tilde Z(\btheta^*,\beps)=Z(\btheta^*,\beps)$ over
$\{\|\btheta^*\|_2^2+\|\beps\|_2^2 \leq C_1d\}$.
Under Assumption \ref{assump:model}, the joint distribution of
$(\btheta^*,\beps)$ satisfies a log-Sobolev inequality by tensorization,
implying the Lipschitz concentration
\begin{equation}\label{eq:concentrationII}
\P[|\tilde Z(\btheta^*,\beps)-\E[\tilde Z(\btheta^*,\beps) \mid \X]|
\geq td \mid \X] \leq 2e^{-t^2d/(2L^2)}.
\end{equation}
We may bound
\begin{align*}
&|\E[\tilde Z(\btheta^*,\beps) \mid \X]-\E[Z(\btheta^*,\beps) \mid \X]|\\
&\leq \E\Big[\1\{\|\btheta^*\|_2^2+\|\beps\|_2^2 \geq C_1d\}
\Big(|\tilde Z(\btheta^*,\beps)|+|Z(\btheta^*,\beps)|\Big)\;\Big|\;\X\Big]\\
&\leq \P[\|\btheta^*\|_2^2+\|\beps\|_2^2 \geq C_1d \mid \X]^{1/2}
\Big((\E[\tilde Z(\btheta^*,\beps) \mid \X]^2)^{1/2}
+(\E[Z(\btheta^*,\beps) \mid \X]^2)^{1/2}\Big)
\end{align*}
Applying the upper bound $Z(\btheta^*,\beps) \leq \log \int \exp(\sum_{j=1}^d
\log g(\theta_j))\d\btheta=0$, Jensen's inequality lower bound
$Z(\btheta^*,\beps) \geq
\E_g[-\frac{1}{2\sigma^2}\|\X\btheta^*+\beps-\X\btheta\|_2^2]$
where $\E_g[\cdot]$ is the expectation over $\theta_j \overset{iid}{\sim} g$,
and $|\tilde Z(\btheta^*,\beps)-Z(0)|
=|\tilde Z(\btheta^*,\beps)-\tilde Z(0)|
\leq L\sqrt{d}(\|\btheta^*\|_2^2+\|\beps\|_2^2)^{1/2}$, we obtain
\[|\E[\tilde Z(\btheta^*,\beps) \mid \X]-\E[Z(\btheta^*,\beps) \mid \X]|
\leq \P[\|\btheta^*\|_2^2+\|\beps\|_2^2 \geq C_1d \mid \X]^{1/2} \cdot Cd
\leq e^{-c'd}\]
for all large $n,d$, the last inequality applying (\ref{eq:thetaepsnormbound}).
Thus (\ref{eq:concentrationII}) and (\ref{eq:thetaepsnormbound}) imply
\[\P[|Z(\btheta^*,\beps)-\E[Z(\btheta^*,\beps) \mid \X]| \geq td+e^{-c'd} \mid \X]
\leq 2e^{-t^2d/(2L^2)}+e^{-cd},\]
implying the first statement of (\ref{eq:mseconcentration}) by the 
Borel-Cantelli lemma.

For the second statement, let us write
\[n\YMSE_*(\btheta^*,\beps)=\|\X\btheta^*+\beps-\X\langle \btheta
\rangle\|_2^2\]
viewed also as a function of $(\btheta^*,\beps)$. Writing $\kappa_2(\cdot)$ for
the covariance associated to the posterior mean $\langle \cdot \rangle$,
differentiating in $(\btheta^*,\beps)$ gives, for any unit vectors $\u \in \R^d$
and $\v \in \R^n$,
\begin{align*}
\u^\top \nabla_{\btheta^*}[n\YMSE_*]
&=2(\X\btheta^*+\beps-\X\langle\btheta \rangle)^\top \X\u
-\frac{2}{\sigma^2}\kappa_2\Big(\btheta^\top\X^\top\X\u,
(\X\btheta^*+\beps-\X\langle\btheta \rangle)^\top \X\btheta\Big),\\
\v^\top \nabla_{\beps}[n\YMSE_*]
&=2(\X\btheta^*+\beps-\X\langle\btheta \rangle)^\top \v
-\frac{2}{\sigma^2}\kappa_2\Big(\btheta^\top\X^\top\v,
(\X\btheta^*+\beps-\X\langle\btheta \rangle)^\top \X\btheta\Big).
\end{align*}
The Poincar\'e inequality implied by the assumed LSI for
$\sP_g(\btheta \mid \X,\y)$ shows, for any vector $\x \in \R^d$,
\[\kappa_2(\x^\top\btheta,\u^\top\btheta) \leq C\|\x\|_2^2.\]
On the event $\{\|\btheta^*\|_2^2+\|\beps\|_2^2 \leq C_1d\}$,
applying this Poincar\'e bound, Cauchy-Schwarz for $\kappa_2(\cdot)$,
and $\|\X\|_\op \leq C_0$ and $\langle \|\btheta\|_2^2 \rangle \leq Cd$ from
(\ref{eq:posteriormeansq}), we obtain
$|\u^\top \nabla_{\btheta^*}[n\YMSE_*]| \leq C\sqrt{d}$ and
$|\v^\top \nabla_{\beps}[n\YMSE_*]| \leq C\sqrt{d}$ for any unit vectors
$\u,\v$, and hence
\[\|\nabla_{\btheta^*,\beps}[n\YMSE_*]\|_2
\1\{\|\btheta^*\|_2^2+\|\beps\|_2^2 \leq C_1d\} \leq L\sqrt{d}\]
for some constant $L>0$. So $n\YMSE_*$ is 
$L\sqrt{d}$-Lipschitz in $(\btheta^*,\beps)$
on $\{\|\btheta^*\|_2^2+\|\beps\|_2^2 \leq C_1d\}$.
For any $(\btheta^*,\beps)$, we also have the bound
$|n\YMSE_*(\btheta^*,\beps)| \leq C(\|\btheta^*\|_2^2+\|\beps\|_2^2)^{1/2}$
by (\ref{eq:posteriormeansq}), so the second statement of
(\ref{eq:mseconcentration}) follows from the same Lipschitz extension and
concentration argument as above.
\end{proof}

\begin{proof}[Proof of Corollary \ref{cor:fixedalpha_dynamics}(a)]
We restrict to the almost sure event where $\event(C_0,C_\LSI)$ holds
for all large $n,d$.
Observe that by (\ref{eq:muddef}) and (\ref{eq:cmtheta}),
\[\MSE=d^{-1}\langle \|\btheta-\langle \btheta \rangle\|_2^2 \rangle
=d^{-1}\big(\langle\|\btheta\|_2^2\rangle-\|\langle\btheta\rangle\|_2^2\big)
=m_d=\mu_d([\iota,\infty)),\]
so $\lim_{n,d \to \infty}
\MSE=\mu_\theta([\iota,\infty))=c_\theta^\tti(0)-c_\theta^\tti(\infty)$
by (\ref{eq:cthetadef}). Also
\[\MSE_*=d^{-1}\|\btheta^*-\langle \btheta \rangle\|_2^2
=d^{-1}\big(\|\btheta^*\|_2^2
-2\langle \btheta \rangle^\top \btheta^*+\|\langle \btheta \rangle\|_2^2\big),\]
so
$\lim_{n,d \to \infty} \MSE_*=\E[{\theta^*}^2]
-2c_\theta(*)+c_\theta^\tti(\infty)$ by Assumption \ref{assump:model} and the
definitions (\ref{eq:cthetainfty}) and (\ref{eq:cthetastar}).
Thus $\MSE \to \mse$ and $\MSE_* \to \mse_*$ for the quantities
$\mse,\mse_*$ defined in (\ref{eq:mmse}).

Similarly
\[\YMSE=n^{-1}\big\langle\|\X\btheta-\X\langle \btheta \rangle\|_2^2 \big\rangle
=n^{-1}\big(\langle\|\X\btheta-\y\|_2^2\rangle-\|\X\langle\btheta\rangle-\y\|_2^2\big).\]
Then by (\ref{eq:cetamuexpr}) and (\ref{eq:cmeta}),
$\lim_{n,d \to \infty} n^{-1}\|\X\langle \btheta \rangle-\y\|_2^2
=\frac{\sigma^4}{\delta}c_\eta^\tti(\infty)$ and
$ \YMSE=\frac{\sigma^4}{\delta}\mu_{\eta,n}([\iota,\infty))
\rightarrow\frac{\sigma^4}{\delta}(c_\eta^\tti(0)-c_\eta^\tti(\infty))=\ymse$
as defined in (\ref{eq:mmse}). For $\YMSE_*$,
writing $\E[\cdot \mid \X]$ for the expectation over $(\btheta^*,\beps)$ as in
Proposition \ref{prop:mseconcentration}, observe first that
\begin{align*}
n^{-1}\E[\|\X\langle \btheta \rangle-\y\|_2^2 \mid \X]
&=n^{-1}\E[\|\X\langle \btheta \rangle-\X\btheta^*\|_2^2 \mid \X]
-2n^{-1}\E[\beps^\top(\X\langle \btheta \rangle-\X\btheta^*)+\sigma^2 \mid \X],
\end{align*}
and Gaussian integration-by-parts gives
\begin{align*}
\E[\beps^\top(\X\langle \btheta \rangle-\X\btheta^*) \mid \X]
&=\E[\beps^\top \X\langle \btheta \rangle \mid \X]
=\E[\langle \|\X\btheta-\X\btheta^*\|_2^2 \rangle \mid \X]
-\E[\|\X\langle \btheta \rangle-\X\btheta^*\|_2^2 \mid \X]\\
&=\E[\langle \|\X\btheta-\X\langle \btheta \rangle\|_2^2 \rangle \mid \X]
=n\,\E[\YMSE \mid \X].
\end{align*}
Thus
\begin{equation}\label{eq:YMSErelation}
\E[\YMSE_* \mid \X]=n^{-1}\E[\|\X\langle \btheta \rangle-\X\btheta^*\|_2^2 \mid
\X]=n^{-1}\E[\|\X\langle \btheta \rangle-\y\|_2^2 \mid \X]
+2\E[\YMSE \mid \X]-\sigma^2.
\end{equation}
We remark that $n^{-1}\|\X\langle \btheta \rangle-\y\|_2^2$ and $\YMSE$ are
bounded for all large $n,d$ on the event $\event(C_0,C_\LSI)$, by the
bound for $\langle \|\btheta\|_2^2 \rangle \leq C$ from (\ref{eq:posteriormeansq}).
Thus, applying $\YMSE \to \ymse$ and
$n^{-1}\|\X\langle\btheta \rangle-\y\|_2^2 \to
\frac{\sigma^4}{\delta}c_\eta^\tti(\infty)$ as argued above and dominated
convergence, the right side of (\ref{eq:YMSErelation}) converges to
$\ymse_*=\frac{\sigma^4}{\delta}(2c_\eta^\tti(0)-c_\eta^\tti(\infty))-\sigma^2$
as defined in (\ref{eq:mmse}). Then the concentration of $\YMSE_*$ established
in Proposition \ref{prop:mseconcentration} combined with (\ref{eq:YMSErelation})
show $\lim_{n,d \to \infty} \YMSE_*=\ymse_*$.

To show the last statement (\ref{eq:fixedalpha_equilibriumlaw}), conditional on
$\X,\btheta^*,\beps$ and averaging over the initial condition $\btheta^0 \sim
q_0=g_0^{\otimes d}$, let $q_t$ be the conditional law of $\btheta^t$.
Consider a coupling of a posterior sample
$\btheta \sim q$ with $\btheta^t \sim q_t$ such that
$\langle \|\btheta^t-\btheta\|_2^2 \rangle=W_2(q_t,q)^2$, where $\langle \cdot
\rangle$ denotes the expectation under this coupling and $W_2(\cdot)$ is the
Wasserstein-2 distance, both conditional on $\X,\btheta^*,\beps$.
For a given realization of $(\btheta^t,\btheta)$ from this coupling,
considering the
coordinatewise coupling of $\frac{1}{d}\sum_{j=1}^d
\delta_{(\theta_j^*,\theta_j^t)}$ with
$\frac{1}{d}\sum_{j=1}^d \delta_{(\theta_j^*,\theta_j)}$ shows
\[W_2\left(\frac{1}{d}\sum_{j=1}^d \delta_{(\theta_j^*,\theta_j^t)},
\frac{1}{d}\sum_{j=1}^d \delta_{(\theta_j^*,\theta_j)} \right)^2
\leq \frac{1}{d}\sum_{j=1}^d (\theta_j^t-\theta_j)^2
=\frac{1}{d}\|\btheta^t-\btheta\|_2^2.\]
Then
\[\Bigg\langle
W_2\Bigg(\frac{1}{d}\sum_{j=1}^d \delta_{(\theta_j^*,\theta_j^t)},
\frac{1}{d}\sum_{j=1}^d \delta_{(\theta_j^*,\theta_j)}\Bigg)^2 \Bigg\rangle
\leq \frac{1}{d}\big\langle \|\btheta^t-\btheta\|_2^2 \big\rangle
=\frac{1}{d} W_2(q_t,q)^2.\]
Applying Lemmas \ref{lemma:elementarybounds} and \ref{lem:W2contraction},
$W_2(q_t,q)^2 \leq Ce^{-ct}(\langle
\|\btheta\|_2^2\rangle+\langle \|\btheta^0\|_2^2 \rangle) \leq C'de^{-ct}$
on the event $\event(C_0,C_\LSI)$,
for some constants $C,C',c>0$. So on this event,
\begin{equation}\label{eq:equilibriumlaw_tmp1}
\limsup_{n,d \to \infty}
\Bigg\langle
W_2\Bigg(\frac{1}{d}\sum_{j=1}^d \delta_{(\theta_j^*,\theta_j^t)},
\frac{1}{d}\sum_{j=1}^d \delta_{(\theta_j^*,\theta_j)}\Bigg)^2
\Bigg\rangle \leq C'e^{-ct}.
\end{equation}

Now by Theorem \ref{thm:dmft_approx}(a), for each fixed $t \geq 0$,
almost surely with respect to the randomness of both
$\X,\btheta^*,\beps$ and
$\btheta^0,\{\b^t\}_{t \geq 0}$ defining $\{\btheta^t\}_{t \geq 0}$, we have
\begin{equation}\label{eq:W2asconvergence}
\lim_{n,d \to \infty} W_2\Bigg(\frac{1}{d}\sum_{j=1}^d
\delta_{(\theta_j^*,\theta_j^t)},\;\sP(\theta^*,\theta^t)\Bigg)^2=0
\end{equation}
where $\sP(\theta^*,\theta^t)$ here is the law of $(\theta^*,\theta^t)$ in
the DMFT system. To take an expectation over the randomness of
$\btheta^0$ and $\{\b^t\}_{t \geq 0}$, note that
from the definition (\ref{eq:langevin_fixedq}), we have
\[\btheta^t=\btheta^0+\int_0^t \nabla_{\btheta} \log q(\btheta^s)\,\d s
+\sqrt{2}\,\b^t
=\btheta^0+\int_0^t \Big[\frac{1}{\sigma^2}\X^\top(\y-\X\btheta^s)
+(\log g)'(\btheta^s)\Big]\d s+\sqrt{2}\,\b^t,\]
where $(\log g)'$ is applied entrywise.
Then on $\event(C_0,C_\LSI)$,
by the Lipschitz continuity of $(\log g)'(\theta)$, this implies for a constant
$C>0$ that
\[d^{-1/2}\|\btheta^t\|_2 \leq \int_0^t
Cd^{-1/2}\|\btheta^s\|_2\,\d s
+Ct+d^{-1/2}\|\btheta^0\|_2+\sqrt{2}\,d^{-1/2}\|\b^t\|_2.\]
Then for any $T>0$, Gronwall's inequality gives, for a constant $C>0$,
\begin{equation}\label{eq:Gronwallargument}
\sup_{t \in [0,T]} d^{-1/2}\|\btheta^t\|_2 \leq
Ce^{CT}\Big(T+d^{-1/2}\|\btheta^0\|_2+d^{-1/2}
\sup_{t \in [0,T]} \|\b^t\|_2\Big)
\end{equation}
For any $p>1$, applying
\[\Big(\sup_{t \in [0,T]} d^{-1}\|\b^t\|_2^2\Big)^p \leq
\sup_{t \in [0,T]} d^{-1}\sum_{j=1}^d |b_j^t|^{2p}
\leq d^{-1}\sum_{j=1}^d \sup_{t \in [0,T]} |b_j^t|^{2p}\]
and Doob's $L^p$-maximal inequality, we have
that $\langle (\sup_{t \in [0,T]} d^{-1}\|\b^t\|_2^2)^p \rangle$ is bounded by a
$(T,p)$-dependent constant. Similarly $\langle (d^{-1}\|\btheta^0\|_2)^p
\rangle$ is bounded by a $(T,p)$-dependent constant, so
\[\Big\langle \Big(\sup_{t \in [0,T]} d^{-1}\|\btheta^t\|_2^2\Big)^p
\Big\rangle \leq C_{T,p}\]
for a constant $C_{T,p}>0$, where $\langle \cdot \rangle$ averages over
$\btheta^0$ and $\{\b^t\}_{t \geq 0}$.
Since $W_2(\frac{1}{d}\sum_{j=1}^d
\delta_{(\theta_j^*,\theta_j^t)},\sP(\theta^*,\theta^t))^2 \leq
C(d^{-1}\|\btheta^*\|_2^2+d^{-1}\|\btheta^t\|_2^2+\E(\theta^*)^2+\E(\theta^t)^2)$,
this implies on the event $\event(C_0,C_\LSI)$ that
\begin{equation}\label{eq:W2UI}
\bigg\langle \sup_{t \in [0,T]} W_2\bigg(\frac{1}{d}\sum_{j=1}^d
\delta_{(\theta_j^*,\theta_j^t)},\;\sP(\theta^*,\theta^t)\bigg)^{2p} \bigg\rangle
\leq C_{T,p}'
\end{equation}
for a different constant $C_{T,p}'>0$. In particular, for any fixed $t \geq 0$
and $p>1$, the squared Wasserstein-2
distance in (\ref{eq:W2asconvergence}) is uniformly bounded in $L^p$ and hence
uniformly integrable with respect to $\langle \cdot \rangle$ for all large
$n,d$, so dominated convergence implies, almost surely,
\begin{equation}\label{eq:equilibriumlaw_tmp2}
\lim_{n,d \to \infty} \Bigg\langle W_2\Bigg(\frac{1}{d}\sum_{j=1}^d
\delta_{(\theta_j^*,\theta_j^t)},\;\sP(\theta^*,\theta^t)\Bigg)^2\Bigg\rangle=0.
\end{equation}

Combining (\ref{eq:equilibriumlaw_tmp1}) and
(\ref{eq:equilibriumlaw_tmp2}) shows that
for any fixed $t \geq 0$, almost surely,
\[\limsup_{n,d \to \infty}
\Bigg\langle W_2\Bigg(\frac{1}{d}\sum_{j=1}^d \delta_{(\theta_j^*,\theta_j)},
\sP_{g_*,\omega_*;g,\omega}\Bigg)^2 \Bigg\rangle \leq C\Big(e^{-ct}
+W_2(\sP(\theta^*,\theta^t),\; \sP_{g_*,\omega_*;g,\omega})^2\Big).\]
By Theorem \ref{thm:dmft_equilibrium}, we have
\[\lim_{t \to \infty} W_2(\sP(\theta^*,\theta^t),\;
\sP_{g_*,\omega_*;g,\omega})=0\]
so taking the limit $t \to \infty$ shows (\ref{eq:fixedalpha_equilibriumlaw}).
\end{proof}

To show Corollary \ref{cor:fixedalpha_dynamics}(b) on the asymptotic free
energy, we will apply an I-MMSE argument, together with the following
proposition which guarantees continuity of $\mse,\mse_*$ in the noise
variance $\sigma^2$.
In the later proof of Theorem \ref{thm:adaptivealpha_dynamics}, we will require 
also continuity in the prior parameter $\alpha$; thus we establish both
statements here.

\begin{lemma}\label{lemma:mmse_lipschitz}
Suppose Assumptions \ref{assump:model} and \ref{assump:prior}(b) hold.
Fix any open subset $O \subset \R^K$, and suppose also that Assumption
\ref{assump:LSI} holds for each $g \in \{g(\cdot,\alpha):\alpha \in O\}$,
where the constant $C_\LSI>0$ is uniform over $\alpha \in O$.
Consider any noise variance $\tilde \sigma^2 \geq \sigma^2$, and define
$\mse(\tilde \sigma^2,\alpha),\mse_*(\tilde\sigma^2,\alpha)$ by
(\ref{eq:mmse}) via the (approximately-TTI) DMFT limit of the Langevin dynamics
(\ref{eq:langevinfixedprior}) with a fixed prior $g(\cdot,\alpha)$
in the linear model (\ref{eq:linearmodel}) with noise variance $\tilde\sigma^2$.

Then over any compact interval
$I \subset [\sigma^2,\infty)$ and compact subset $S \subset O$,
$\mse(\tilde\sigma^2,\alpha),\mse_*(\tilde\sigma^2,\alpha)$ are Lipschitz functions of
$(\tilde\sigma^2,\alpha) \in I \times S$.
\end{lemma}
\begin{proof}
Consider noise/prior parameters $(s^2,\alpha)$ and 
$(\tilde s^2,\tilde \alpha)$, where $s^2,\tilde s^2 \geq \sigma^2$.
Let us couple the linear models with noise
variances $s^2$ and $\tilde s^2$ by $\y=\X\btheta^*+s\z$
and $\tilde\y=\X\btheta^*+\tilde s\z$, where $\z \sim \N(0,\I)$.
Fixing $\X,\btheta^*,\z$, let us denote
\[U(\btheta)={-}\frac{1}{2s^2}\|\X\btheta^*+s \z
-\X\btheta\|_2^2+\sum_{j=1}^d \log g(\theta_j,\alpha)\]
so that $q(\btheta) \propto e^{U(\btheta)}$ is the posterior law given $(\X,\y)$
under parameters $(s^2,\alpha)$. Denote similarly $\tilde U(\btheta)$ 
with $(\tilde s^2,\tilde \alpha)$ in place of $(s^2,\alpha)$,
and $\tilde q(\btheta)\propto e^{\tilde U(\btheta)}$
as the posterior law given $(\X,\tilde\y)$.
We condition on $\X,\btheta^*,\z$ and restrict to the event
\[\event'(C_0,C_\LSI)=\{\|\X\|_\op \leq C_0,\,\|\btheta^*\|_2^2,\|\z\|_2^2 \leq C_0d,
\;(\ref{eq:LSI}) \text{ holds for both } q \text{ and } \tilde q\},\]
which by assumption holds a.s.\ for all large $n,d$. We first derive a bound on
the Wasserstein-2 distance between $q$ and $\tilde q$, conditional on
$\X,\btheta^*,\z$.

Let $\{\btheta^t\}_{t \geq 0}$ be the
Langevin diffusion (\ref{eq:langevin_fixedq}) with fixed prior $g(\cdot,\alpha)$
and stationary distribution $q(\btheta)$, initialized as $\btheta^0 \sim q_0$
where $q_0$ has finite second moment and finite entropy. Let us write
$\langle f(\btheta^t) \rangle$ for the expectation over $\btheta^0$ and
$\{\b^t\}_{t \geq 0}$ defining (\ref{eq:langevin_fixedq}),
conditional on $\X,\btheta^*,\z$.
We apply the following argument of
\cite{vempala2019rapid} to bound the KL-divergence $\DKL(q_t\|\tilde q)$
conditional on $\X,\btheta^*,\z$:
Differentiating this KL-divergence in time,
\begin{align*}
\frac{\d}{\d t} \DKL(q_t\|\tilde q)
&=\frac{\d}{\d t} \int q_t(\log q_t-\log \tilde q)\\
&=\int \Big(\frac{\d}{\d t} q_t\Big)(\log q_t-\log \tilde q)
+\underbrace{\int \frac{q_t}{q_t}\Big(\frac{\d}{\d t} q_t\Big)}_{=0}
=\int \Big(\frac{\d}{\d t} q_t\Big)(\log q_t-\tilde U+\log \tilde Z).
\end{align*}
The law of $\btheta^t$ conditional on $\X,\btheta^*,\z$
admits a density $q_t$ which is described by the Fokker-Planck equation
\[\frac{\d}{\d t} q_t=\nabla \cdot [q_t\nabla(\log q_t-U)]\]
with initial condition $q_t|_{t=0}=q_0$.
Then, applying this Fokker-Planck equation and
integrating by parts, we obtain
\begin{align*}
\frac{\d}{\d t} \DKL(q_t\|\tilde q)
&={-}\int q_t\,\nabla(\log q_t-U)^\top \nabla(\log q_t-\tilde U)\\
&={-}\int q_t\,\|\nabla(\log q_t-\tilde U)\|_2^2
-\int q_t\,\nabla(\tilde U-U)^\top \nabla(\log q_t-\tilde U)\\
&\leq {-}(1/2)\int q_t\,\|\nabla(\log q_t-\tilde U)\|_2^2
+(1/2)\int q_t\|\nabla(\tilde U-U)\|_2^2,
\end{align*}
the last step applying Cauchy-Schwarz for the second term. By the LSI for
$\tilde q$, the first term (the relative Fisher information) is lower bounded as
\[\int q_t\,\|\nabla(\log q_t-\tilde U)\|_2^2
=\int q_t\,\Big\|\nabla \log \frac{q_t}{\tilde q}\Big\|_2^2
\geq \frac{1}{2C_\LSI}\,\DKL(q_t\|\tilde q).\]
Thus
\[\frac{\d}{\d t} \DKL(q_t\|\tilde q)
\leq {-}\frac{1}{4C_\LSI}\,\DKL(q_t\|\tilde q)
+\frac{1}{2}\,\underbrace{\langle \|\nabla \tilde U(\btheta^t)-\nabla
U(\btheta^t)\|_2^2 \rangle}_{:=\Delta(t)}.\]
Integrating this inequality shows,
for some constants $C,c>0$ depending only on $C_\LSI$ and for any $T>0$,
\begin{equation}\label{eq:KLcomparison}
\DKL(q_T \| \tilde q) \leq C\Big(\sup_{t \in [0,T]} \Delta(t)
+e^{-cT}\DKL(q_0\|\tilde q)\Big).
\end{equation}

We now specialize (\ref{eq:KLcomparison}) to the initialization $q_0=\tilde q$,
and bound $\Delta(t)$. We have
\[\Delta(t)
 \leq \Bigg\langle \bigg\|\frac{1}{s^2}\X^\top(\X\btheta^*+s
\z-\X\btheta^t)-\frac{1}{\tilde s^2}\X^\top(\X\btheta^*+\tilde s\z
-\X\btheta^t)\bigg\|_2^2+\sum_{j=1}^d \Big(\partial_\theta
\log g(\theta_j^t,\alpha)-\partial_\theta
\log g(\theta_j^t,\tilde \alpha)\Big)^2\Bigg \rangle.\]
Let $C,C',C''>0$ be constants depending on the compact sets $S,I$ of the lemma
statement and changing
from instance to instance. For $\alpha,\tilde\alpha \in S$ and
$s^2,\tilde s^2 \in I$,
\[|s^{-2}-\tilde s^{-2}| \leq C|s^2-\tilde s^2|,
\quad |s^{-1}-\tilde s^{-1}| \leq C|s^2-\tilde s^2|,
\quad |\partial_\theta \log g(\theta;\alpha)-\partial_\theta \log g(\theta;\tilde\alpha)|
\leq C\|\alpha-\tilde\alpha\|_2,\]
the last inequality holding by Assumption \ref{assump:prior}(b). Thus
\[\Delta(t) \leq
C\Big[\|\X\|_\op^4(\|\btheta^*\|_2^2+\langle \|\btheta^t\|_2^2\rangle)
+\|\X\|_\op^2\|\z\|_2^2\Big](s^2-\tilde s^2)^2
+Cd\|\alpha-\tilde\alpha\|_2^2.\]
On the event $\event'(C_0,C_\LSI)$, we have
$\langle \|\btheta^t\|_2^2 \rangle \leq
C(\langle \|\btheta^0\|_2^2 \rangle+d)$ by (\ref{eq:W2contraction}),
and $\langle \|\btheta^0\|_2^2 \rangle \leq Cd$ under the initialization
$q_0=\tilde q$ which holds also by (\ref{eq:W2contraction}).
Applying these bounds together
with $\|\X\|_\op \leq C$, $\|\btheta^*\|_2^2 \leq Cd$, and
$\|\z\|_2^2 \leq Cd$ by definition of $\event'(C_0,C_\LSI)$, we have 
\[\sup_{t \geq 0}
\Delta(t) \leq C'd(s^2-\tilde s^2)^2+C'd\|\alpha-\tilde\alpha\|_2^2.\]
Applying this and $q_0=\tilde q$ to (\ref{eq:KLcomparison}),
we have on the event $\event'(C_0,C_\LSI)$ that
\[\sup_{t \geq 0} \DKL(q_t\|\tilde q)
\leq Cd(s^2-\tilde s^2)^2+Cd\|\alpha-\tilde\alpha\|_2^2.\]
By lower-semicontinuity of KL-divergence and the $T_2$-transportation
inequality for $\tilde q$ implied by the LSI
(c.f.\ \cite[Theorem 9.6.1]{bakry2014analysis}),
\begin{equation}\label{eq:W2locallylipschitz}
W_2(q,\tilde q)^2 \leq C\DKL(q\|\tilde q)
\leq C\liminf_{t \to \infty} \DKL(q_t\|\tilde q)
\leq C'd(s^2-\tilde s^2)^2+C'd\|\alpha-\tilde\alpha\|_2^2.
\end{equation}
This gives our desired bound on the Wasserstein-2
distance between $q$ and $\tilde q$.

Now let $\langle f(\btheta) \rangle_q$ and $\langle f(\btheta) \rangle_{\tilde
q}$ be the posterior expectations under $q$ (given $\y$) and $\tilde q$ (given
$\tilde \y$). Then by Jensen's inequality,
\[\|\langle \btheta \rangle_q-\langle \btheta \rangle_{\tilde q}\|_2 \leq
W_2(q,\tilde q).\]
Applying $|\|\x\|_2^2-\|\y\|_2^2| \leq \|\x-\y\|_2 \cdot \|\x+\y\|_2$ and
Cauchy-Schwarz, also
\[|\langle \|\btheta\|_2^2 \rangle_q-\langle \|\btheta\|_2^2 \rangle_{\tilde q}|
\leq W_2(q,\tilde q) \cdot \sqrt{2\langle \|\btheta\|_2^2 \rangle_q
+2\langle \|\btheta\|_2^2 \rangle_{\tilde q}}
\leq C\sqrt{d}\,W_2(q,\tilde q)\]
where the last inequality applies
$\langle \|\btheta\|_2 \rangle_q \leq
\langle \|\btheta\|_2^2 \rangle_q^{1/2} \leq C\sqrt{d}$
on $\event(C_0,C_\LSI)$ by (\ref{eq:W2contraction}),
and similarly for $\tilde q$. Then, denoting by $\MSE(s^2,\alpha)$
and $\MSE(\tilde s^2,\tilde \alpha)$ the values of $\MSE$ as defined in
Corollary \ref{cor:fixedalpha_dynamics} under $q$ and $\tilde q$, we have
\begin{align*}
|\MSE(s^2,\alpha)-\MSE(\tilde s^2,\tilde\alpha)|
&=\big|d^{-1}\langle \|\btheta-\langle \btheta\rangle_q\|_2^2\rangle_q
-d^{-1}\langle \|\btheta-\langle \btheta\rangle_{\tilde q}\|_2^2\rangle_{\tilde
q}\big|\\
&\leq d^{-1}\big|\langle \|\btheta\|_2^2 \rangle_q
-\langle \|\btheta\|_2^2 \rangle_{\tilde q}\big|
+d^{-1}\big|\|\langle \btheta \rangle_q\|_2^2
-\|\langle \btheta \rangle_{\tilde q}\|_2^2\big|\\
&\leq \frac{C'W_2(q,\tilde q)}{\sqrt{d}}
\leq C''|s^2-\tilde s^2|+C''\|\alpha-\tilde \alpha\|_2.
\end{align*}
Similarly
\begin{align*}
|\MSE_*(s^2,\alpha)-\MSE_*(\tilde s^2,\tilde\alpha)|
&=\big|d^{-1}\|\btheta^*-\langle \btheta\rangle_q\|_2^2
-d^{-1}\|\btheta^*-\langle \btheta\rangle_{\tilde q}\|_2^2\big|\\
&\leq 2d^{-1}\|\btheta^*\|_2\|\langle \btheta\rangle_q-
\langle \btheta\rangle_{\tilde q}\|_2
+d^{-1}\big|\|\langle \btheta \rangle_q\|_2^2
-\|\langle \btheta \rangle_{\tilde q}\|_2^2\big|\\
&\leq \frac{C'W_2(q,\tilde q)}{\sqrt{d}}
\leq C''|s^2-\tilde s^2|+C''\|\alpha-\tilde \alpha\|_2.
\end{align*}
Since $\event'(C_0,C_\LSI)$ holds a.s.\ for all large $n,d$,
and Corollary \ref{cor:fixedalpha_dynamics}(a)
already proven shows $\lim_{n,d \to \infty} \MSE=\mse$ and
$\lim_{n,d \to \infty} \MSE_*=\mse_*$ a.s.\ at both $(s^2,\alpha)$ and $(\tilde
s^2,\tilde \alpha)$, this implies
\[|\mse(s^2,\alpha)-\mse(\tilde s^2,\tilde\alpha)|,
|\mse_*(s^2,\alpha)-\mse_*(\tilde s^2,\tilde\alpha)|
\leq C|s^2-\tilde s^2|+C\|\alpha-\tilde \alpha\|_2,\]
so $\mse(s^2,\alpha)$ and $\mse_*(s^2,\alpha)$ are locally Lipschitz as desired.
\end{proof}

\begin{proof}[Proof of Corollary \ref{cor:fixedalpha_dynamics}(b)]
We apply Corollary \ref{cor:fixedalpha_dynamics}(a) and an
I-MMSE relation for mismatched Gaussian channels. 
Write $\E[\cdot \mid \X]$ for the expectation over $(\btheta^*,\beps)$ conditional on $\X$
as in Proposition
\ref{prop:mseconcentration}. Let
\[I(\y,\btheta^*)
=\E\left[\log \frac{\sP(\y \mid \btheta^*,\X)}{\sP_{g_*}(\y \mid
\X)}\;\bigg|\;\X\right]
={-}\E[\log \sP_{g_*}(\y \mid \X) \mid \X]
-\frac{n}{2}(1+\log 2\pi\sigma^2)\]
be the signal-observation mutual information in the linear model
(\ref{eq:linearmodel}) conditional on $\X$, where
$\sP(\y \mid \btheta^*,\X)$ is the Gaussian likelihood of $\y$ and
$\sP_{g_*}(\y \mid \X)$ is the marginal likelihood
(\ref{eq:marginallikelihood}) under the true prior $g_*$. Then
\begin{align}
\E[\log \sP_g(\y \mid \X) \mid \X]&={-}\DKL(\sP_{g_*}(\y \mid \X) \| \sP_g(\y
\mid \X))
+\E[\log \sP_{g_*}(\y \mid \X) \mid \X]\notag\\
&={-}\DKL(\sP_{g_*}(\y \mid \X) \| \sP_g(\y \mid \X))
-I(\y,\btheta^*)-\frac{n}{2}(1+\log 2\pi\sigma^2)
\label{eq:freeenergyIMMSE}
\end{align}
where here and throughout the proof, $\DKL(\cdot)$ denotes the
KL-divergence also conditional on $\X$.

Let us denote the inverse noise variance by $s^{-1}=\sigma^2$ and write
\begin{equation}\label{eq:Esgdef}
E(s,g)=\E[\YMSE_* \mid \X]
=n^{-1}\E[\|\X\langle \btheta \rangle-\X\btheta^*\|^2 \mid \X]
\end{equation}
for the expected $\YMSE_*$ in the linear model (\ref{eq:linearmodel})
with assumed prior $g$ and noise variance
$s^{-1}$. We clarify that this means $\langle \cdot \rangle$ in
(\ref{eq:Esgdef}) is the posterior average under the law
\[\sP_g(\btheta \mid \X,\y) \propto \exp\bigg({-}\frac{s}{2}
\|\y-\X\btheta\|_2^2+\sum_{j=1}^d \log g(\theta_j)\bigg)\]
and $\E[\cdot\mid \X]$ is
the expectation over $(\btheta^*,\beps)$ where 
$\beps$ also has variance $s^{-1}$. We write also
$I[s],\DKL[s]$ for the above quantities $I(\y,\btheta^*)$ and
$\DKL(\sP_{g_*}(\y \mid \X) \|\sP_g(\y \mid \X))$ in this model with noise
variance $s^{-1}$. Then
\cite[Theorem 2]{guo2005mutual} and \cite[Eq.\ (24)]{verdu2010mismatched} show
the I-MMSE relations
\[\frac{\d}{\d s}I[s]=\frac{n}{2}E(s,g_*), \qquad
\frac{\d}{\d s}\DKL[s]=\frac{n}{2}\big(E(s,g)-E(s,g_*)\big).\]
For any fixed $n,d$ and $\X$,
in the limit $s \to 0$, it is direct to check
that $I[s] \to 0$ and $\DKL[s] \to 0$. Thus, for
$I(\y,\btheta^*) \equiv I[\sigma^{-2}]$ and
$\DKL(\sP_{g_*}(\y \mid \X) \| \sP_g(\y \mid \X)) \equiv
\DKL[\sigma^{-2}]$ in the original model with noise variance $\sigma^2$,
integrating these I-MMSE relations shows
\begin{equation}\label{eq:IMMSEcombined}
\DKL(\sP_{g_*}(\y \mid \X) \|\sP_g(\y \mid \X))+I(\y,\btheta^*)
=\frac{n}{2}\int_0^{\sigma^{-2}} E(s,g)\d s.
\end{equation}
Assumption \ref{assump:LSI}(b) ensures that the posterior LSI (\ref{eq:LSI})
holds a.s.\ in the
model with any noise variance $s^{-1} \in [\sigma^2,\infty)$. Then
applying Corollary \ref{cor:fixedalpha_dynamics}(a) already shown and the
concentration of $\YMSE_*$ in Proposition \ref{prop:mseconcentration},
we have $E(s,g) \to \ymse_*(s,g)$ a.s.\ for each $s^{-1} \in (\sigma^2,\infty)$,
where $\ymse_*(s,g)$ is defined by (\ref{eq:mmse}) via the DMFT limit of the
Langevin dynamics (\ref{eq:langevinfixedprior}) with fixed prior $g(\cdot)$ in
the linear model with noise variance $s^{-1}$. To apply dominated convergence,
we note that on the event $\|\X\|_\op \leq C_0$,
by the extension (\ref{eq:posteriormeansqsigmasq}) of (\ref{eq:posteriormeansq}), we have $E(s,g) \leq C$ for a constant $C>0$
uniformly over all $s \in [0,\sigma^{-2}]$ and all $n,d$.
Then, since $\|\X\|_\op \leq C_0$ holds a.s.\ for all large $n,d$,
taking the limit $n,d \to \infty$ and
applying the bounded convergence theorem to (\ref{eq:IMMSEcombined})
shows that almost surely,
\begin{equation}\label{eq:IMMSE}
\lim_{n,d \to \infty}
\frac{1}{d}\Big(\DKL(\sP_{g_*}(\y \mid \X) \|\sP_g(\y \mid
\X))+I(\y,\btheta^*)\Big)
=\frac{\delta}{2}\int_0^{\sigma^{-2}} \ymse_*(s,g)\d s.
\end{equation}

Let us now fix the assumed prior $g(\cdot)$, write $\ymse_*(s) \equiv
\ymse_*(s,g)$, and
let $(\mse(s),\mse_*(s),\omega(s),\omega_*(s))$ denote the fixed points
(\ref{eq:static_fixedpoint}) corresponding to $\ymse_*(s)$.
Recall the marginal density $\sP_{g,\omega}(y)$ of the
scalar channel model (\ref{eq:scalarchannelmarginal}), and define
\begin{align}
&f(\omega,\omega_*,s)
={-}\E_{g_*,\omega_*} \log \sP_{g,\omega}(y)
-\frac{1}{2}\left(2\delta+\log \frac{2\pi}{\omega}
-\delta \log \frac{\delta s}{\omega}
+(1-\delta)\frac{\omega}{\omega_*}
+\frac{\omega}{s}\Big(\frac{\omega}{\omega_*}-2\Big)\right)\label{eq:fomega}\\
&=\underbrace{\frac{\omega}{2}\E\,\theta^{*2}
-\E\log \int \exp\Big(\omega
\theta(\theta^*+\omega_*^{-1/2}z)-\frac{\omega}{2}\theta^2\Big)g(\theta)\d\theta}_{:=\mathrm{I}}
-\underbrace{\frac{1}{2}\Big(2\delta-\delta\log\frac{\delta
s}{\omega}-\frac{\delta\omega}{\omega_*}+\frac{\omega}{s}\Big(\frac{\omega}{\omega_*}-2\Big)\Big)}_{:=\mathrm{II}}.\notag
\end{align}
Here, the expectations in the second line are over
$\theta^* \sim g_*$ and $z \sim \N(0,1)$, and we have
applied the explicit form of $\sP_{g,\omega}(y)$ and evaluated
$\E_{g_*,\omega_*}$ under the true model $y=\theta^*+\omega_*^{-1/2}z$ with
some some algebraic simplification. We now claim that
\begin{equation}\label{eq:IMMSEalt}
\frac{\delta}{2}\int_0^s \ymse_*(t)\d t=f(\omega(s),\omega_*(s),s)
\end{equation}
for all $s \in (0,\sigma^{-2})$. To show this, it suffices to check
$\lim_{s \to 0} f(\omega(s),\omega_*(s),s)=0$ and $\frac{\d}{\d
s}f(\omega(s),\omega_*(s),s)=\frac{\delta}{2}\ymse_*(s)$,
which we may do as follows:
\begin{itemize}
\item Let $\MSE(s),\MSE_*(s)$ denote the values of $\MSE,\MSE_*$ in a linear
model with noise variance $s^{-1}$. On the event $\|\X\|_\op \leq C_0$, the
bound (\ref{eq:posteriormeansqsigmasq}) implies that $\MSE(s),\MSE_*(s) \leq
C(1+s\|\y\|_2^2/d)$ for a constant $C>0$ (independent of $s$)
and for all $s^{-1} \in (\sigma^2,\infty)$.
Taking the almost sure limit as $n,d \to \infty$ shows that $\mse(s),\mse_*(s)
\leq C$. In particular, in the limit $s \to 0$, we have that
$\mse(s),\mse_*(s)$ remain bounded, so
$\omega(s),\omega_*(s) \sim \delta s$ by the fixed point relation
(\ref{eq:static_fixedpoint}). Then
$\omega(s) \to 0$, $\omega_*(s) \to 0$, $\omega(s)/s \to \delta$, and
$\omega(s)/\omega_*(s) \to 1$ as $s \to 0$. Applying this to (\ref{eq:fomega})
shows
\[\lim_{s \to 0} f(\omega(s),\omega_*(s),s)=0.\]
\item Differentiating the term $\mathrm{I}$ of (\ref{eq:fomega})
in $\omega,\omega_*$ and applying
Gaussian integration-by-parts with respect to $z \sim \N(0,1)$, we may check
that
\begin{align*}
\partial_\omega \mathrm{I}
&=\frac{1}{2}\E \langle (\theta^*-\theta)^2 \rangle_{g,\omega}
-\frac{\omega}{\omega_*}\E \langle (\theta-\langle \theta \rangle_{g,\omega})^2
\rangle_{g,\omega},\\
\partial_{\omega_*} \mathrm{I}
&=\frac{\omega^2}{2\omega_*^2}\E \langle (\theta-\langle \theta
\rangle_{g,\omega})^2\rangle_{g,\omega}.
\end{align*}
Then at the fixed points $(\omega,\omega_*)=(\omega(s),\omega_*(s))$, we have
\begin{align*}
\partial_\omega \mathrm{I}|_{(\omega,\omega_*)=(\omega(s),\omega_*(s))}
&=\frac{1}{2}(\mse(s)+\mse_*(s))-\frac{\omega(s)}{\omega_*(s)}\mse(s)\\
\partial_{\omega_*} \mathrm{I}|_{(\omega,\omega_*)=(\omega(s),\omega_*(s))}
&=\frac{\omega(s)^2}{2\omega_*(s)^2}\mse(s).
\end{align*}
Applying $\mse(s)=\delta/\omega(s)-\sigma^2$ and $\mse_*(s)=\delta/\omega_*(s)-\sigma^2$ by
(\ref{eq:static_fixedpoint}) and comparing with the derivatives of the second
term $\mathrm{II}$ of (\ref{eq:fomega}), this verifies
\begin{equation}\label{eq:omegastationarity}
\partial_\omega f(\omega(s),\omega_*(s),s)=0,
\qquad \partial_{\omega_*} f(\omega(s),\omega_*(s),s)=0.
\end{equation}
Furthermore, direct calculation shows that at
$(\omega,\omega_*)=(\omega(s),\omega_*(s))$,
\[\partial_s f(\omega(s),\omega_*(s),s)=\frac{\delta\sigma^2}{2}
+\frac{\omega(s)\sigma^4}{2}\Big(\frac{\omega(s)}{\omega_*(s)}-2\Big)
=\frac{\delta}{2}\ymse_*(s),\]
the second equality using (\ref{eq:ymmseomega}).
Lemma \ref{lemma:mmse_lipschitz} implies that
$\mse(s),\mse_*(s),\omega(s),\omega_*(s)$ are locally Lipschitz,
and hence absolutely continuous, over $s \in (0,\sigma^{-2})$. Then
also $s \mapsto f(\omega(s),\omega_*(s),s)$ is absolutely continuous, and
we may differentiate by the chain rule to get
\begin{align*}
\frac{\d}{\d s} f(\omega(s),\omega_*(s),s)
&=\partial_\omega f(\omega(s),\omega_*(s),s)
\cdot \omega'(s)+\partial_{\omega_*} f(\omega(s),\omega_*(s),s)
\cdot \omega_*'(s)+\partial_s f(\omega(s),\omega_*(s),s)\\
&=\partial_s f(\omega(s),\omega_*(s),s)=\frac{\delta}{2}\,\ymse_*(s).
\end{align*}
\end{itemize}
Combining the above arguments verifies the claim (\ref{eq:IMMSEalt}).

Applying (\ref{eq:IMMSE}) and (\ref{eq:IMMSEalt}) to (\ref{eq:freeenergyIMMSE})
and writing
$(\omega,\omega_*)=(\omega(\sigma^{-2}),\omega_*(\sigma^{-2}))$ for the
fixed points at the original noise variance $\sigma^2$, this shows
\[\lim_{n,d \to \infty} d^{-1}\E[\log \sP_g(\y \mid \X) \mid \X]
=-f(\omega,\omega_*,\sigma^{-2})
-\frac{\delta}{2}(1+\log 2\pi\sigma^2).\]
Applying concentration of $d^{-1}\log \sP_g(\y \mid \X)$ with respect to
$\E[\cdot \mid \X]$
which is established in Propostion \ref{prop:mseconcentration}, and
substituting the form of $f$ in (\ref{eq:fomega}), this
shows Corollary \ref{cor:fixedalpha_dynamics}(b).
\end{proof}

%% file: adaptiveprior.tex
\section{Analysis of empirical Bayes Langevin dynamics}\label{sec:adaptiveprior}

In this section, we prove Theorem \ref{thm:adaptivealpha_dynamics} on the
adaptive empirical Bayes dynamics with time-varying prior parameter
$\widehat\alpha^t$, and discuss further the examples of Section
\ref{sec:examples}.

\subsection{General analysis under uniform LSI}

We introduce a few notational shorthands: Conditional on
$\X,\btheta^*,\beps$, let
\[q_\alpha(\btheta) \equiv \sP_{g(\cdot,\alpha)}(\btheta \mid \X,\y)\]
be the posterior law under the prior parameter $\alpha$.
We write $\langle \cdot \rangle_\alpha$ for its posterior expectation. For
$\btheta\in\R^d$, define
\begin{equation}\label{eq:barqalpha}
\bar \sP_{\btheta}=\frac{1}{d}\sum_{j=1}^d \delta_{(\theta_j^*,\theta_j)},
\qquad \bar \sP_\alpha=\langle \bar \sP_{\btheta} \rangle_\alpha.
\end{equation}
Thus $\bar \sP_\alpha$ is a $(\X,\btheta^*,\beps)$-dependent joint law over variables
$(\theta^*,\theta)$ which satisfies
\begin{equation}\label{eq:barqalphaproperty}
\E_{(\theta^*,\theta) \sim \bar \sP_\alpha} f(\theta^*,\theta)
=\frac{1}{d}\sum_{j=1}^d \langle f(\theta_j^*,\theta_j) \rangle_\alpha
=\frac{1}{d}\sum_{j=1}^d \int f(\theta_j^*,\theta_j) q_\alpha(\btheta)\d\btheta.
\end{equation}
We write $\theta \sim \bar \sP_\alpha$ as shorthand for the $\theta$-marginal of
$(\theta^*,\theta) \sim \bar \sP_\alpha$.

We note that under Assumptions \ref{assump:prior}(b) and
\ref{assump:compactalpha}, all constants in (\ref{eq:constantdependence}) are
uniform over $g \in \{g(\cdot,\alpha):\alpha \in O\}$ for the bounded domain $O$
of Assumption \ref{assump:compactalpha},
where a uniform bound for $|\log g(0,\alpha)|$ follows from
$|\log g(0,\alpha)| \leq |\log g(0,0)|
+\|\nabla_\alpha(\log g(0,0))\|_2 \cdot \|\alpha\|_2
+C\|\alpha\|_2^2$ as
implied by (\ref{eq:logggradientbound}) of Assumption \ref{assump:prior}(b).
Hence the bounds of Section \ref{sec:posteriorbounds} hold uniformly over
$\alpha \in O$. In particular, from (\ref{eq:posteriormeansq}),
\begin{equation}\label{eq:posteriormeansquniform}
\sup_{\alpha \in O} \langle \|\btheta\|_2^2 \rangle_\alpha
\leq C(d+\|\y\|_2^2)
\end{equation}
on an event $\{\|\X\|_\op \leq C_0\}$ that holds a.s.\ for all large $n,d$.

We first prove Lemma \ref{lemma:gradalphaF} on the derivatives of
$F,\widehat F$ and uniform convergence of
$\widehat F,\nabla \widehat F$ over $S \subset O$.

\begin{proof}[Proof of Lemma \ref{lemma:gradalphaF}]
For (a), differentiating
\[\widehat F(\alpha)={-}\frac{1}{d}\log
\int \Big(\frac{1}{2\pi\sigma^2}\Big)^{n/2}
\exp\Big({-}\frac{1}{2\sigma^2}\|\y-\X\btheta\|_2^2
+\sum_{j=1}^d \log g(\theta_j,\alpha)\Big)\d\btheta\]
and applying the property (\ref{eq:barqalphaproperty}), we have
\begin{equation}\label{eq:hatFderivcalculation}
\nabla \widehat F(\alpha)={-}\frac{1}{d}\sum_{j=1}^d \langle \nabla_\alpha
\log g(\theta_j,\alpha) \rangle_\alpha
=-\E_{\theta \sim \bar \sP_\alpha} \nabla_\alpha \log g(\theta,\alpha).
\end{equation}
For the form of $\nabla F(\alpha)$, define analogously to (\ref{eq:fomega})
\begin{equation}\label{eq:fomegaalpha}
f(\omega,\omega_*,\alpha)
={-}\E_{g_*,\omega_*}\log \sP_{g(\cdot,\alpha),\omega}(y)
-\frac{1}{2}\left(2\delta+\log \frac{2\pi}{\omega}
-\delta \log \frac{\delta}{\omega\sigma^2}
+(1-\delta)\frac{\omega}{\omega_*}
+\omega\sigma^2\Big(\frac{\omega}{\omega_*}-2\Big)\right)
\end{equation}
where the dependence on $\alpha$ is in $\sP_{g(\cdot,\alpha),\omega}(y)$.
For any
$\alpha \in O$, let $\omega(\alpha),\omega_*(\alpha)$ be the fixed points
$\omega,\omega_*$ defined by (\ref{eq:mmse}) via the DMFT system for the
dynamics (\ref{eq:langevinfixedprior}) with fixed
prior $g \equiv g(\cdot,\alpha)$. (This DMFT system is approximately-TTI 
for each $\alpha \in O$ by Assumption \ref{assump:compactalpha} and Theorem
\ref{thm:fixedalpha_dynamics}, hence $\omega(\alpha),\omega_*(\alpha)$ are
well-defined.) Then
\begin{equation}\label{eq:Ffrelation}
F(\alpha)=f(\omega(\alpha),\omega_*(\alpha),\alpha)
+\frac{\delta}{2}(1+\log 2\pi\sigma^2).
\end{equation}
By the same calculations as (\ref{eq:omegastationarity}),
at the fixed points $(\omega(\alpha),\omega_*(\alpha))$, we have
$\partial_\omega f(\omega(\alpha),\omega_*(\alpha),\alpha)=0$ and
$\partial_{\omega_*} f(\omega(\alpha),\omega_*(\alpha),\alpha)=0$. 
By Lemma \ref{lemma:mmse_lipschitz}, $\omega(\alpha),\omega_*(\alpha)$
are locally Lipschitz and hence absolutely continuous over $\alpha \in O$. Then
$F(\alpha)$ is also absolutely continuous over $\alpha \in O$, and
differentiating by the chain rule gives
\begin{align*}
\nabla F(\alpha)&=\nabla_\alpha
f(\omega,\omega_*,\alpha)\Big|_{(\omega,\omega_*)=(\omega(\alpha),\omega_*(\alpha))}\\
&={-}\nabla_\alpha\bigg[\E_{g_*,\omega_*}\log \sP_{g(\cdot,\alpha),\omega}
(y)\bigg]\bigg|_{(\omega,\omega_*)=(\omega(\alpha),\omega_*(\alpha))}\\
&={-}\nabla_\alpha\bigg[\E_{g_*,\omega_*} \log\int
\Big(\frac{\omega}{2\pi}\Big)^{1/2}
\exp\Big({-}\frac{\omega}{2}(y-\theta)^2+\log g(\theta,\alpha)\Big)\d\theta
\bigg]
\bigg|_{(\omega,\omega_*)=(\omega(\alpha),\omega_*(\alpha))}
\end{align*}
By definition $\sP_\alpha$ is the joint law of $(\theta^*,\theta)$ under the
generative process where $(\theta^*,y)$ are drawn from the Gaussian convolution
model defining this expectation $\E_{g_*,\omega_*}$,
and where $\theta \sim \sP_{g(\cdot,\alpha),\omega}(\theta
\mid y)$. Hence, evaluating $\nabla_\alpha$ above gives
\[\nabla F(\alpha)={-}\E_{\theta \sim \sP_\alpha} \nabla_\alpha \log g(\theta,\alpha).\]

For (b), let $S \subset O$ be any compact subset of the domain $O$ in
Assumption \ref{assump:compactalpha}, 
and let $Q$ be a countable dense subset of $O$.
Define
\[\event(C_0,C_\LSI)=\{
\|\X\|_\op \leq C_0,\;\text{(\ref{eq:LSI}) holds for } q_\alpha(\btheta) \equiv
\sP_{g(\cdot,\alpha)}(\btheta \mid \X,\y) \text{ for every } \alpha \in O\}.\]
Assumptions \ref{assump:model} and
\ref{assump:compactalpha} ensure for some $C_0,C_\LSI>0$ that
$\event(C_0,C_\LSI)$ holds a.s.\ for all large $n,d$,
where this event depends only on $\X$ and not on $\btheta^*,\beps$.

We restrict to the almost-sure event where the convergence statements of
Corollary \ref{cor:fixedalpha_dynamics} and Proposition
\ref{prop:mseconcentration} hold for every $\alpha \in Q$, and where
$\event(C_0,C_\LSI)$ holds for all large $n,d$.
Note that Corollary \ref{cor:fixedalpha_dynamics} shows
$\widehat F(\alpha) \to F(\alpha)$ for each $\alpha \in Q$. 
To strengthen this to uniform convergence over $S$, note that
Assumption \ref{assump:prior}(b) implies
$\partial_\theta \nabla_\alpha \log g(\theta,\alpha)$ is uniformly
bounded over $(\theta,\alpha) \in \R \times S$, so
\[\|\nabla_\alpha \log g(\theta,\alpha)\|_2 \leq 
\|\nabla_\alpha \log g(0,\alpha)\|_2+C|\theta|.\]
Then, since $\nabla_\alpha \log g(0,\alpha)$ is bounded
over $\alpha \in S$ by compactness of $S$, and $\sup_{\alpha \in S}
\langle \|\btheta\|_2^2 \rangle_\alpha \leq Cd$
by (\ref{eq:posteriormeansquniform}), we have
\begin{align}
\sup_{\alpha \in S}
\|\E_{\theta \sim \bar \sP_\alpha} \nabla_\alpha \log g(\theta,\alpha)\|_2
&\leq \sup_{\alpha \in S}
\frac{1}{d}\sum_{j=1}^d \langle\|\nabla_\alpha \log g(\theta_j,\alpha)\|_2
\rangle_\alpha\notag\\
&\leq \sup_{\alpha \in S}
\|\nabla_\alpha \log g(0,\alpha)\|_2
+\frac{C}{d}\sum_{j=1}^d \langle |\theta_j| \rangle_\alpha \leq C'.\label{eq:nablaFboundO}
\end{align}
This shows $\nabla \widehat F(\alpha)$ is bounded over any compact
subset $S \subset O$. Then for any compact $S \subset O$,
the functions $\widehat F(\alpha)$ for all $n,d$ are equicontinuous in a
neighborhood of each point
$\alpha \in S$, and hence are uniformly equicontinuous over $S$ since a finite
number of such neighborhoods cover $S$. 
Then by Arzela-Ascoli, the
convergence $\widehat F(\alpha) \to F(\alpha)$ for each $\alpha \in Q$ implies
uniform convergence over $\alpha \in S$.

We next show the pointwise convergence $\nabla \widehat F(\alpha) \to \nabla
F(\alpha)$ for each $\alpha \in Q$. Recalling our definition of
$\bar \sP_{\btheta}$ in (\ref{eq:barqalpha}), and applying
Jensen's inequality and the convexity $W_2(\lambda\sP+(1-\lambda)\sP',\sQ)^2
\leq \lambda W_2(\sP,\sQ)^2+(1-\lambda)W_2(\sP',\sQ)^2$ of the squared
Wasserstein-2 distance,
\[W_2(\bar \sP_\alpha,\sP_\alpha)^2
\leq \langle W_2(\bar \sP_{\btheta},\sP_\alpha)^2 \rangle_\alpha.\]
For each $\alpha \in Q$, 
the right side converges to 0 as $n,d \to \infty$
by the statement (\ref{eq:fixedalpha_equilibriumlaw}) of
Corollary \ref{cor:fixedalpha_dynamics}(a). Thus
$\lim_{n,d \to \infty} W_2(\bar \sP_\alpha,\sP_\alpha)=0$.
Assumption \ref{assump:prior}(b) ensures that $\nabla_\alpha
\log g(\theta,\alpha)$ is Lipschitz in $\theta$,
so this Wasserstein-2 convergence implies
\[\lim_{n,d \to \infty} \nabla \widehat F(\alpha)
=\lim_{n,d \to \infty} \E_{\theta \sim \bar \sP_\alpha} \nabla_\alpha \log
g(\theta,\alpha)=\E_{\theta \sim \sP_\alpha} \nabla_\alpha \log
g(\theta,\alpha)=\nabla F(\alpha)\]
for each $\alpha \in Q$, as claimed.

To extend this to uniform convergence over any compact subset $S \subset O$,
we differentiate (\ref{eq:hatFderivcalculation}) a second time.
Writing $\Var_\alpha,\Cov_\alpha$ for the
variance and covariance under $\langle \cdot \rangle_\alpha$,
\begin{equation}\label{eq:hessF}
\nabla^2 \widehat F(\alpha)=
-\frac{1}{d}\bigg\langle\sum_{j=1}^d \nabla_\alpha^2 \log
g(\theta_j,\alpha)\bigg \rangle_\alpha
-\frac{1}{d}\Cov_\alpha\bigg[\sum_{j=1}^d \nabla_\alpha \log
g(\theta_j,\alpha)\bigg].
\end{equation}
The first term is uniformly bounded over $\alpha \in S$, by
the same argument as showing boundedness of $\nabla \widehat F(\alpha)$ above.
For the second term, on the event $\event(C_0,C_\LSI)$,
for every unit vector $v \in \R^K$ and $\alpha \in S$,
\[\Var_\alpha\bigg[\sum_{j=1}^d v^\top \nabla_\alpha \log
g(\theta_j,\alpha)\bigg]
\leq (C_\LSI/2)\bigg\langle
\sum_{j=1}^d \Big(v^\top \partial_\theta \nabla_\alpha \log
g(\theta_j,\alpha)\Big)^2\bigg \rangle_\alpha\]
by the Poincar\'e inequality for $q_\alpha$ implied by its LSI. Since
$\partial_\theta \nabla_\alpha \log g(\theta,\alpha)$ is
bounded over $\alpha \in S$, the second term of (\ref{eq:hessF}) is also
bounded on $\event(C_0,C_\LSI)$.
Thus $\nabla^2 \widehat F(\alpha)$ is uniformly bounded over $\alpha
\in S$ for all large $n,d$. This implies as above that for any compact $S
\subset O$, the functions $\nabla \widehat F(\alpha)$ for all large $n,d$ are
uniformly equicontinuous on $S$, so
$\nabla \widehat F(\alpha) \to \nabla F(\alpha)$ uniformly over $\alpha \in S$.
This shows part (b).

For part (c), note that if $g^*=g(\cdot,\alpha^*)$, then
\[\E[\widehat F(\alpha) \mid \X]-\E[\widehat F(\alpha^*) \mid \X]
=d^{-1}\DKL(\sP_{g(\cdot,\alpha^*)}(\y \mid \X) \|
\sP_{g(\cdot,\alpha)}(\y \mid \X)) \geq 0,\]
where here $\DKL(\cdot)$ is the KL-divergence conditional on $\X$. Thus
$\alpha^*$ is a minimizer of $\alpha \mapsto \E[\widehat F(\alpha) \mid \X]$
over $\R^K$. Applying the convergence $\widehat F(\alpha)-\E[\widehat
F(\alpha) \mid \X] \to 0$ for each $\alpha \in Q$ from Proposition
\ref{prop:mseconcentration}, we have also
$\E[\widehat F(\alpha) \mid \X] \to F(\alpha)$ for each $\alpha \in Q$. Note that
\[\nabla_\alpha \E[\widehat F(\alpha) \mid \X]
={-}\E[\E_{\theta \sim \bar \sP_\alpha}
\nabla_\alpha \log g(\theta,\alpha) \mid \X],\]
and that $\sup_{\alpha \in S} \E[\langle \|\btheta\|_2^2 \rangle_\alpha \mid \X]
\leq \E[\sup_{\alpha \in S}\langle \|\btheta\|_2^2 \rangle_\alpha \mid \X]
\leq Cd$ on $\event(C_0,C_\LSI)$, by (\ref{eq:posteriormeansquniform}).
Then the argument (\ref{eq:nablaFboundO}) shows also that
$\nabla_\alpha \E[\widehat F(\alpha) \mid \X]$
is uniformly bounded and equicontinuous over $\alpha \in S$, hence
\[\lim_{n,d \to \infty} \sup_{\alpha \in S} |\E[\widehat
F(\alpha) \mid \X]-F(\alpha)|=0.\]
Since $\alpha^*$ is a minimizer of
$\E[\widehat F(\alpha) \mid \X]$, this implies that
$F(\alpha) \geq F(\alpha^*)$ for every $\alpha \in S$. Since this holds for
every compact subset $S \subset O$, this shows part (c).
\end{proof}

We proceed to prove Theorem \ref{thm:adaptivealpha_dynamics}.
Let $\{\btheta^t,\widehat\alpha^t\}_{t \geq 0}$ be the solution of the
adaptive Langevin equations (\ref{eq:langevin_sde}--\ref{eq:gflow}).
Let $\{\alpha^t\}_{t \geq 0}$ be
the (deterministic) $\alpha$-component of the DMFT limit of
$\{\widehat\alpha^t\}_{t \geq 0}$ prescribed by Theorem
\ref{thm:dmft_approx}(b), and consider the SDE
\begin{equation}\label{eq:tildethetaSDE}
\d\tilde\btheta^t=\nabla_{\tilde\btheta}\bigg({-}\frac{1}{2\sigma^2}\|\y-\X\tilde\btheta^t\|_2^2
+\sum_{j=1}^d \log g(\tilde \theta_j^t,\alpha^t)\bigg)\d t+\sqrt{2}\,\d\b^t
\end{equation}
which replaces $\widehat\alpha^t$ by $\alpha^t$. We couple
$\{\tilde\btheta^t\}_{t \geq 0}$ to
$\{\btheta^t,\widehat\alpha^t\}_{t \geq 0}$ via
the same initial
conditions $\tilde\btheta^0=\btheta^0$ and $\alpha^0$ of Assumption
\ref{assump:model}, and via the same Brownian motion $\{\b^t\}_{t \geq 0}$.

We write $q_t$ for the density of $\tilde\btheta^t$ conditional on
$\X,\btheta^*,\beps$ and averaging over $\tilde\btheta^0$,
where $q_0=g_0^{\otimes d}$ is the initial density of
$\tilde\btheta^0=\btheta^0$. In parallel to (\ref{eq:barqalpha}), we denote
\begin{equation}\label{eq:barqt}
\bar \sP_{\tilde\btheta^t}=\frac{1}{d}\sum_{j=1}^d \delta_{(\theta_j^*,\tilde
\theta_j^t)},
\qquad \bar \sP_t=\langle \bar \sP_{\tilde\btheta^t} \rangle
\end{equation}
where $\langle \cdot \rangle$ is the average with respect to
$\tilde\btheta^t \sim q_t$, i.e.\ the average over $\tilde\btheta^0$ and
$\{\b^t\}_{t \geq 0}$. We write
$\tilde \theta^t \sim \bar \sP_t$ for the $\tilde \theta^t$-marginal
of a sample $(\theta^*,\tilde \theta^t) \sim \bar \sP_t$.

\begin{lemma}\label{lemma:Vdecreasing}
Under Assumptions \ref{assump:model} and \ref{assump:prior}(b),
there exists a unique solution
$\{\tilde\btheta^t\}_{t \geq 0}$ to (\ref{eq:tildethetaSDE}).
Letting $q_t$ be the above conditional density of $\tilde\btheta^t$, and
letting $V(q,\alpha)$ be the Gibbs free energy (\ref{eq:gibbsenergy}),
almost surely
\begin{equation}\label{eq:Vdecreasing}
\limsup_{n,d \to \infty}
\sup_{t \in [0,T]} \frac{\d}{\d t} \Big(V(q_t,\alpha^t)+R(\alpha^t)\Big) \leq 0.
\end{equation}
\end{lemma}
\begin{proof}
Fixing $C_0>0$ large enough, let
\[\event(C_0)=\{\|\X\|_\op \leq C_0,\;
\|\btheta^*\|_2^2,\|\beps\|_2^2 \leq C_0d\}.\]
We restrict to the event where the almost-sure convergence statements of
Theorem \ref{thm:dmft_approx}(b) hold and where
$\event(C_0)$ holds for all large $n,d$.

Since $\{\alpha^t\}_{t \geq 0}$ is continuous, for each $T>0$, there exists a
compact ball $S_T$ for which $\alpha^t \in S_T$ for all $t \in [0,T]$.
By Assumption \ref{assump:prior}(b), $(\theta,\alpha)
\mapsto \partial_\theta \log g(\theta,\alpha)$ restricted to $\alpha \in S_T$
is Lipschitz. Then the drift of (\ref{eq:tildethetaSDE}) is Lipschitz over
each time horizon $[0,T]$, so (\ref{eq:tildethetaSDE}) 
admits a unique solution $\{\tilde\btheta^t\}_{t \in [0,T]}$
(c.f.\ \cite[Theorem II.1.2]{kunita1984stochastic}) over $t \in [0,T]$ for
every $T \geq 0$, and hence also over all $t \geq 0$. We note that
\[\frac{\d}{\d t}(\tilde \btheta^t-\btheta^t)
=\frac{1}{\sigma^2}\X^\top \X(\btheta^t-\tilde\btheta^t)
+\Big[\partial_\theta \log g(\tilde \theta_j^t,\alpha^t)-
\partial_\theta \log g(\theta_j^t,\widehat\alpha^t)\Big]_{j=1}^d.\]
Applying again the Lipschitz property of
$(\theta,\alpha) \mapsto \partial_\theta
\log g(\theta,\alpha)$ over $\alpha \in S_T$
and the bound $\|\X\|_\op \leq C_0$, there is a constant $C>0$ depending on
$C_0,T$ such that
\[\bigg\|\frac{1}{\sqrt{d}}\frac{\d}{\d t}(\tilde \btheta^t-\btheta^t)\bigg\|_2
\leq \frac{C}{\sqrt{d}}
\|\tilde \btheta^t-\btheta^t\|_2+C\|\alpha^t-\widehat\alpha^t\|_2.\]
Since $\sup_{t \in [0,T]} \|\alpha^t-\widehat\alpha^t\|_2 \to 0$
by Theorem \ref{thm:dmft_approx}(b), a Gronwall argument implies
\begin{equation}\label{eq:tildethetaerror}
\lim_{n,d \to \infty}
\sup_{t \in [0,T]} \frac{1}{\sqrt{d}}\|\tilde\btheta^t-\btheta^t\|_2=0.
\end{equation}

By the DMFT equation (\ref{def:dmft_langevin_alpha}),
the evolution of $\alpha^t$ is given by
\begin{equation}\label{eq:alphaflowEB}
\frac{\d}{\d t} \alpha^t=\E_{\theta^t \sim \sP(\theta^t)}
\nabla_\alpha \log g(\theta^t,\alpha^t)-\nabla R(\alpha^t)
\end{equation}
where $\sP(\theta^t)$ is the law of the DMFT variable $\theta^t$.
The law $q_t$ of $\tilde\btheta^t$ satisfies the Fokker-Planck equation
\begin{equation}\label{eq:FokkerPlanck}
\frac{\d}{\d t} q_t(\tilde\btheta)
=\nabla_{\tilde\btheta} \cdot \bigg[q_t(\tilde\btheta)\nabla_{\tilde\btheta}\bigg(
\frac{1}{2\sigma^2}\|\y-\X\tilde\btheta\|_2^2-\sum_{j=1}^d
\log g(\tilde \theta_j,\alpha^t)+\log q_t(\tilde\btheta)\bigg)\bigg].
\end{equation}
Then, using (\ref{eq:alphaflowEB}) and (\ref{eq:FokkerPlanck}) to differentiate
$V(q_t,\alpha^t)+R(\alpha^t)$,
\begin{align}
&\frac{\d}{\d t} \Big(V(q_t,\alpha^t)+R(\alpha^t)\Big)
={-}\frac{1}{d}\underbrace{\int
\bigg\|\nabla_{\tilde\btheta}\Big(\frac{1}{2\sigma^2}\|\y-\X\tilde\btheta\|_2^2-\sum_{j=1}^d
\log g(\tilde\theta_j,\alpha^t)+\log q_t(\tilde\btheta) \Big)
\bigg\|_2^2 q_t(\tilde\btheta)\d\tilde\btheta}_{:=\FI_t}\label{eq:defFI}\\
&\hspace{0.5in}-\Big(\E_{\theta^t \sim \sP(\theta^t)}\nabla_\alpha \log
g(\theta^t,\alpha^t)-\nabla R(\alpha^t)\Big)^\top \Big(\int \frac{1}{d}\sum_{j=1}^d \nabla_\alpha
\log g(\tilde\theta_j,\alpha^t)q_t(\tilde\btheta)\d\tilde\btheta-\nabla R(\alpha^t)\Big).\notag
\end{align}
Here, the first term $\FI_t$ (the relative Fisher information) arises
from differentiation in $q_t$ and integration-by-parts in $\tilde\btheta$,
while the second term arises from differentiation in $\alpha^t$. 
Recalling the notation (\ref{eq:barqt}),
\[\int \frac{1}{d}\sum_{j=1}^d \nabla_\alpha
\log g(\tilde\theta_j,\alpha^t)q_t(\tilde\btheta)\d\tilde\btheta
=\E_{\tilde \theta^t \sim \bar \sP_t} \nabla_\alpha \log g(\tilde \theta^t,\alpha^t)\]
so we may write the above as 
\begin{equation}\label{eq:Gibbstimederiv}
\begin{aligned}
&\frac{\d}{\d t} \Big(V(q_t,\alpha^t) + R(\alpha^t)\Big)\\
&={-}\frac{1}{d}\,\FI_t
-\Big\|\E_{\tilde \theta^t \sim \bar \sP_t} \nabla_\alpha \log
g(\tilde \theta^t,\alpha^t)-\nabla R(\alpha^t)\Big\|^2\\
&\hspace{0.5in}
+\underbrace{\Big(\E_{\tilde \theta^t \sim \bar \sP_t}
\nabla_\alpha \log g(\tilde \theta^t,\alpha^t)
-\E_{\theta^t \sim \sP(\theta^t)}
\nabla_\alpha \log g(\theta^t,\alpha^t)
\Big)^\top \Big(\E_{\tilde \theta^t \sim \bar \sP_t} \nabla_\alpha \log
g(\tilde \theta^t,\alpha^t)-\nabla R(\alpha^t)\Big)}_{:=\Delta_t}.
\end{aligned}
\end{equation}
By the convexity 
$W_2(\lambda\sP+(1-\lambda)\sP',\sQ)^2 \leq 
\lambda W_2(\sP,\sQ)^2+(1-\lambda)W_2(\sP',\sQ)^2$ and Jensen's inequality,
\begin{equation}\label{eq:barqbound}
\sup_{t \in [0,T]} W_2(\bar \sP_t,\sP(\theta^*,\theta^t))^2
\leq \sup_{t \in [0,T]} \langle W_2(\bar \sP_{\tilde
\btheta^t},\sP(\theta^*,\theta^t))^2 \rangle
\leq \Big\langle \sup_{t \in [0,T]} W_2(\bar \sP_{\tilde
\btheta^t},\sP(\theta^*,\theta^t))^2\Big\rangle,
\end{equation}
where $\langle \cdot \rangle$ is the average over $\tilde \btheta^t \sim q_t$,
and $\sP(\theta^*,\theta^t)$ is the joint law of the DMFT variables
$(\theta^*,\theta^t)$.
By Theorem \ref{thm:dmft_approx}(b) and (\ref{eq:tildethetaerror}), for any
fixed $T>0$ we have
\begin{equation}\label{eq:W1asconvergence}
\sup_{t \in [0,T]}
W_2(\bar \sP_{\tilde \btheta^t},\sP(\theta^*,\theta^t))^2
\leq \sup_{t \in [0,T]} 2W_2(\bar \sP_{\tilde \btheta^t},\bar \sP_{\btheta^t})^2
+2W_2(\bar \sP_{\btheta^t},\sP(\theta^*,\theta^t))^2 \to 0
\end{equation}
almost surely as $n,d \to \infty$. The same arguments as leading to
(\ref{eq:W2UI}) show that
$\sup_{t \in [0,T]} W_2(\bar \sP_{\tilde \btheta^t},\sP(\theta^*,\theta^t))^2$
is uniformly integrable with respect to $\langle \cdot \rangle$ for all large
$n,d$. Then applying (\ref{eq:W1asconvergence})
and dominated convergence to bound the right side of (\ref{eq:barqbound}), we
get
\begin{equation}\label{eq:barqconvergence}
\lim_{n,d \to \infty} \sup_{t \in [0,T]}
W_2(\bar \sP_t,\sP(\theta^*,\theta^t))^2=0.
\end{equation}
Finally, applying that
$(\theta,\alpha) \mapsto \nabla_\alpha \log g(\theta,\alpha)$ is uniformly
Lipschitz over $\alpha \in S_T$ by Assumption \ref{assump:prior}(b),
this Wasserstein-2 convergence implies
\[\lim_{n,d \to \infty}
\sup_{t \in [0,T]} \Big|\E_{\theta \sim \bar \sP_t}
\nabla_\alpha \log g(\theta,\alpha^t)
-\E_{\theta \sim \sP(\theta^t)}
\nabla_\alpha \log g(\theta,\alpha^t)\Big|=0,\]
hence $\lim_{n,d \to \infty} \sup_{t \in [0,T]} |\Delta_t|=0$ for the quantity
$\Delta_t$ of (\ref{eq:Gibbstimederiv}).
As the first two terms of (\ref{eq:Gibbstimederiv}) are non-positive, this shows
(\ref{eq:Vdecreasing}).
\end{proof}

\begin{proof}[Proof of Theorem \ref{thm:adaptivealpha_dynamics}]
Let $S \subset O \subset \R^K$ be the domains of
Assumption \ref{assump:compactalpha}.
Fixing sufficiently large constants $C_0,C_\LSI>0$, define
\[\event(C_0,C_\LSI)=\{\|\X\|_\op \leq C_0,\;
\|\btheta^*\|_2^2,\|\beps\|_2^2 \leq C_0d,\;
\text{and (\ref{eq:LSI}) holds for } q_\alpha \text{ for every } \alpha \in O\}.\]
We restrict to the event where the almost-sure convergence statements of Theorem
\ref{thm:dmft_approx}(b) and Lemma \ref{lemma:gradalphaF} hold, and where
$\event(C_0,C_\LSI)$ holds for all large $n,d$. Throughout, $C,C',c>0$
denote constants that may depend on $C_0,C_\LSI$ and change
from instance to instance.

On the event $\event(C_0,C_\LSI)$, we first note that by It\^o's formula,
\[\d\|\tilde\btheta^t\|_2^2
=2(\tilde\btheta^t)^\top\Big[\Big(\frac{1}{\sigma^2}\X^\top(\y-\X\tilde\btheta^t)
+\Big(\partial_\theta \log g(\tilde\theta_j^t,\alpha^t)\Big)_{j=1}^d\Big)\d t+\sqrt{2}\,\d\b^t\Big]+(2d)\d t,\]
and hence
\[\frac{\d}{\d t} \langle d^{-1}\|\tilde\btheta^t\|_2^2 \rangle
\leq C(1+\langle d^{-1/2}\|\tilde\btheta^t\|_2 \rangle)+2\bigg\langle
d^{-1}\sum_{j=1}^d \tilde \theta_j^t
\cdot \partial_\theta \log g(\tilde \theta_j^t,\alpha^t)\bigg\rangle\]
for a constant $C>0$. Under the convexity-at-infinity condition of
Assumption \ref{assump:prior}(b), there exist constants $C,c>0$ for which
$\theta \cdot \partial_\theta \log g(\theta,\alpha^t)
\leq C|\theta|-c\theta^2$ for all $\theta \in \R$ and $\alpha^t \in S$. Applying
this and Cauchy-Schwarz to the above, we have for some constants $C',c'>0$ that
$\frac{\d}{\d t} \langle d^{-1}\|\tilde\btheta^t\|_2^2 \rangle
\leq C'-c'\langle d^{-1}\|\tilde\btheta^t\|_2^2 \rangle$. This implies on
$\event(C_0,C_\LSI)$ that
\begin{equation}\label{eq:tildethetabound}
d^{-1}\langle \|\tilde\btheta^t\|_2^2 \rangle \leq C
\end{equation}
for a constant $C>0$ and all $t \geq 0$. The arguments leading to
(\ref{eq:W2UI}) show that for any fixed $t \geq 0$,
$d^{-1}\|\tilde\btheta^t\|_2^2$ is uniformly integrable with respect to
$\langle \cdot \rangle$ for all large $n,d$. Since $\event(C_0,C_\LSI)$ holds
a.s.\ for all large $n,d$, and $\lim_{n,d \to \infty}
d^{-1}\|\tilde \btheta^t\|_2^2=(\theta^t)^2$ a.s.\ by Theorem
\ref{thm:dmft_approx}(b) and (\ref{eq:tildethetaerror})
where $\theta^t$ here is the $\theta$-component of the limiting DMFT system,
this implies also
\begin{equation}\label{eq:DMFTthetabounded}
\E(\theta^t)^2 \leq C
\end{equation}
for all $t \geq 0$. Furthermore, for any $s \leq t$,
applying
\[\tilde\btheta^t-\tilde\btheta^s=\int_s^t \Big[\frac{1}{\sigma^2}
\X^\top(\y-\X\tilde\btheta^r)+\Big(\partial_\theta \log g(\tilde
\theta_j^r,\alpha^r)\Big)_{j=1}^d\Big]\d r+\sqrt{2}(\b^t-\b^s)\]
and uniform Lipschitz continuity of $\theta \mapsto \partial_\theta \log
g(\theta,\alpha^r)$ for $\alpha^r \in S$,
we have on $\event(C_0,C_\LSI)$ that
\[d^{-1/2}\|\tilde \btheta^t-\tilde \btheta^s\|_2
\leq \int_s^t Cd^{-1/2}\|\tilde \btheta^r-\tilde \btheta^s\|_2\,\d r
+C(t-s)(1+d^{-1/2}\|\tilde\btheta^s\|_2)+\sqrt{2}d^{-1/2}\|\b^t-\b^s\|_2.\]
Then by Gronwall's inequality,
\[d^{-1/2}\|\tilde \btheta^t-\tilde \btheta^s\|_2
\leq Ce^{C(t-s)}\Big(C(t-s)(1+d^{-1/2}\|\tilde\btheta^s\|_2)
+d^{-1/2}\sup_{r \in [s,t]}
\|\b^r-\b^s\|_2\Big).\]
Then applying (\ref{eq:tildethetabound}) and Doob's maximal inequality shows
\begin{equation}\label{eq:tildethetacontinuity}
d^{-1}\langle \|\tilde \btheta^t-\tilde \btheta^s\|_2^2 \rangle
\leq C(t-s) \text{ for all } s \leq t \text{ with } t-s \leq 1.
\end{equation}

We now show that for a constant $C>0$,
\begin{equation}\label{eq:intgradFsquared}
\int_0^\infty \|\nabla F(\alpha^t)+\nabla R(\alpha^t)\|_2^2 \d t<C.
\end{equation}
We remind the reader that $q_t$ is the law of $\tilde\btheta^t$ (conditioned on
$\X,\btheta^*,\beps$) and $q_{\alpha^t}$ is the posterior law of $\btheta$ under
the prior $g \equiv g(\cdot,\alpha^t)$.
On $\event(C_0,C_\LSI)$, the LSI for $q_{\alpha^t}$ and its implied
$T_2$-transportation inequality
(c.f.\ \cite[Theorem 9.6.1]{bakry2014analysis}) imply
for the Fisher information term $\FI_t$ of (\ref{eq:defFI}) that
\[\FI_t \geq C_{\LSI}^{-1} \DKL(q_t\| q_{\alpha^t}) \geq 
C_{\LSI}^{-2} W_2(q_t,q_{\alpha^t})^2\]
for all $t \geq 0$. The average marginal
distribution of coordinates of $\tilde \btheta^t \sim q_t$ is the
$\tilde\theta^t$-marginal of $\bar \sP_t$ defined
in (\ref{eq:barqt}),
and that of $\btheta \sim q_{\alpha^t}$ is the $\theta$-marginal of
$\bar \sP_{\alpha^t}$ as defined in (\ref{eq:barqalpha}).
Considering the coordinatewise coupling of $\bar \sP_t,\bar \sP_{\alpha^t}$,
we see that
$W_2(\bar \sP_t, \bar \sP_{\alpha^t})^2 \leq d^{-1}W_2(q_t,q_{\alpha^t})^2$, so
\begin{equation}\label{eq:FIW2qbar}
d^{-1}\FI_t \geq C_{\LSI}^{-2}W_2(\bar \sP_t,\bar \sP_{\alpha^t})^2.
\end{equation}
Applying this and the uniform Lipschitz continuity of
$\theta \mapsto \nabla_\alpha \log g(\theta,\alpha)$ over $\alpha \in O$
guaranteed by Assumption \ref{assump:prior}(b),
\[\Big\|\E_{\tilde \theta^t \sim \bar \sP_t} \nabla_\alpha \log g(\tilde
\theta^t,\alpha^t)
-\E_{\theta \sim \bar \sP_{\alpha^t}} \nabla_\alpha \log g(\theta,\alpha^t)
\Big\|_2^2
\leq C'\,W_2(\bar \sP_t, \bar \sP_{\alpha^t})^2
\leq Cd^{-1}\FI_t.\]
Then applying this as a lower bound for $d^{-1}\FI_t$ in
(\ref{eq:Gibbstimederiv}), and applying also $C^{-1}(a-b)^2+b^2 \geq c_0a^2$
for a constant $c_0>0$ and all $a,b \in \R$, we get from
(\ref{eq:Gibbstimederiv}) that
\[\frac{\d}{\d t} \Big(V(q_t,\alpha^t)+R(\alpha^t)\Big)
\leq {-}c_0\Big\|\E_{\theta \sim \bar \sP_{\alpha^t}} \nabla_\alpha \log
g(\theta,\alpha^t)-\nabla R(\alpha^t)\Big\|^2+\Delta_t.\]
Now note from Lemma \ref{lemma:gradalphaF} that
\[\E_{\theta \sim \bar \sP_{\alpha^t}} \nabla_\alpha \log
g(\theta,\alpha^t)={-}\nabla \widehat F(\alpha^t).\]
Applying $\sup_{t \in [0,T]} \Delta_t \to 0$ and the uniform convergence
$\nabla \widehat F(\alpha) \to \nabla F(\alpha)$ over $\alpha \in S$
from Lemma \ref{lemma:gradalphaF}, this shows a strengthening of
(\ref{eq:Vdecreasing}): for any $t\in[0,T]$,
\[\limsup_{n,d \to \infty} 
\frac{\d}{\d t} \Big(V(q_t,\alpha^t)+R(\alpha^t)\Big)
\leq {-}c_0\|\nabla F(\alpha^t)+\nabla
R(\alpha^t)\|_2^2.\]
Then for any $T>0$,
\[c_0\int_0^T \|\nabla F(\alpha^t)+\nabla R(\alpha^t)\|_2^2\,\d t
\leq \limsup_{n,d \to \infty}
V(q_0,\alpha^0)+R(\alpha^0)-V(q_T,\alpha^T)-R(\alpha^T).\]
Note that by the definition of $V(q,\alpha)$ in (\ref{eq:gibbsenergy})
and the conditions of finite moments and finite entropy for $g_0$ in Assumption
\ref{assump:model},
$V(q_0,\alpha^0)=V(g_0^{\otimes d},\alpha^0)$ is bounded above by a
constant on $\event(C_0,C_\LSI)$ for all large $n,d$. Also
by the definition (\ref{eq:gibbsenergy}),
\[V(q_T,\alpha^T) \geq \frac{1}{d}\DKL(q_T\|g(\cdot,\alpha^T)^{\otimes d})
+\frac{n}{2d}\log 2\pi\sigma^2
\geq \frac{n}{2d}\log 2\pi\sigma^2\]
which is bounded below by a constant for all $T$ and all large $n,d$.
Then, applying also $R(\alpha^0) \leq C$ and $R(\alpha^T) \geq 0$ and
taking the limit $n,d \to \infty$ followed by
$T \to \infty$, we obtain the claimed bound (\ref{eq:intgradFsquared}).

Consider the set
\[\Crit=\{\alpha \in S:\nabla F(\alpha)+\nabla R(\alpha)=0\}.\]
Suppose by contradiction that $\{\alpha^t\}_{t \geq 0}$ has a limit
point $\alpha^\infty \in S$ that does not belong to $\Crit$. 
Lemma \ref{lemma:gradalphaF} implies that $\nabla F(\alpha)+\nabla R(\alpha)$
is continuous over $\alpha \in O$, so
$\|\nabla F(\alpha)+\nabla R(\alpha)\|_2>\delta$ for all $\alpha \in
B_\delta(\alpha^\infty):=\{\alpha:\|\alpha-\alpha^\infty\|_2<\delta\}$
and some $\delta>0$. However, Assumption \ref{assump:prior}(b) and the DMFT
equation (\ref{def:dmft_langevin_alpha}) imply
\begin{equation}\label{eq:dtalphabound}
\bigg\|\frac{\d}{\d t}\alpha^t\bigg\|_2
\leq \E_{\theta \sim \sP(\theta^t)}\|\nabla_\alpha \log g(\theta,\alpha^t)\|_2
+\|\nabla R(\alpha^t)\|_2
\leq C(1+\E|\theta^t|+\|\alpha^t\|_2) \leq C'
\end{equation}
for some constants $C,C'>0$ and all $t \geq 0$, where the last inequality
applies (\ref{eq:DMFTthetabounded}) and the assumption $\alpha^t \in S$.
Then for each $t_0 \geq 0$ such that
\begin{equation}\label{eq:alphat0}
\alpha^{t_0} \in B_{\delta/2}(\alpha^\infty)
\end{equation}
we must have
$\alpha^t \in B_\delta(\alpha^\infty)$ for all $t \in [t_0-c\delta,t_0+c\delta]$
and some constant $c>0$. Then $\int_{t_0-c\delta}^{t_0+c\delta} \|\nabla
F(\alpha^t)+\nabla R(\alpha^t)\|_2^2\,\d t \geq 2c\delta^3$. The condition (\ref{eq:alphat0}) must
hold for infinitely many times $t_0$ because $\alpha^\infty$ is a limit point of
$\{\alpha^t\}_{t \geq 0}$,
but this contradicts (\ref{eq:intgradFsquared}). Thus we must have
$\alpha^\infty \in \Crit$. Since this holds for
every limit point $\alpha^\infty$ of $\{\alpha^t\}_{t \geq 0}$, and $S$ is
compact, this implies $\lim_{t \to \infty} \dist(\alpha^t,\Crit)=0$.
If furthermore all points
of $\Crit$ are isolated, then the limit point $\alpha^\infty$
of $\{\alpha^t\}_{t \geq 0}$ must be unique, and
\[\lim_{t \to \infty} \alpha^t=\alpha^\infty.\]

For the remaining statements (\ref{eq:posteriorconverges}), fix any
$\eps>0$. Choosing $T(\eps)$ such that
$\|\alpha^t-\alpha^\infty\|_2<\eps/2$ for all $t>T(\eps)$, we then have
$\limsup_{n,d \to \infty} \|\widehat\alpha^t-\alpha^\infty\|_2<\eps$
by Theorem \ref{thm:dmft_approx}(b), showing the first statement of
(\ref{eq:posteriorconverges}).
For the second statement of (\ref{eq:posteriorconverges}), we note from
(\ref{eq:Gibbstimederiv}) that
\[\frac{\d}{\d t}\Big(V(q_t,\alpha^t)+R(\alpha^t)\Big) \leq -\frac{1}{d}\FI_t
+\Delta_t.\]
Then, by the same arguments as above, for some constant $C>0$ and every $T>0$,
\[\limsup_{n,d \to \infty} \int_0^T d^{-1}\FI_t
\leq \limsup_{n,d \to \infty}
V(q_0,\alpha^0)+R(\alpha^0)-V(q_T,\alpha^T)-R(\alpha^T) \leq C.\]
Recalling (\ref{eq:FIW2qbar}), this implies
\begin{equation}\label{eq:integratedW2tmp}
\limsup_{n,d \to \infty} \int_0^T W_2(\bar \sP_t,\bar \sP_{\alpha^t})^2\d t
\leq C.
\end{equation}
For each fixed $t \geq 0$, we have
\begin{equation}\label{eq:barPapproximatesDMFT}
\lim_{n,d \to \infty} W_2(\bar \sP_t,\sP(\theta^*,\theta^t))^2=0
\end{equation}
by (\ref{eq:barqconvergence}). We have also by Jensen's inequality
for the squared Wasserstein-2 distance
and (\ref{eq:fixedalpha_equilibriumlaw}) of Corollary
\ref{cor:fixedalpha_dynamics}(a),
\begin{equation}\label{eq:barPconvergence}
\limsup_{n,d \to \infty} W_2(\bar \sP_{\alpha^t},\sP_{\alpha^t})^2
\leq \limsup_{n,d \to \infty} \big\langle
W_2(\bar \sP_{\btheta},\sP_{\alpha^t})^2 \big\rangle_{\alpha^t}=0
\end{equation}
where $\langle \cdot \rangle_{\alpha^t}$ is the average over $\btheta \sim
q_{\alpha^t}$ defining $\bar\sP_{\btheta}$.
Then, combining (\ref{eq:barPapproximatesDMFT})
and (\ref{eq:barPconvergence}), we have that
$\lim_{n,d \to \infty} W_2(\bar \sP_t,\bar \sP_{\alpha^t})
=W_2(\sP(\theta^*,\theta^t),\sP_{\alpha^t})$.
Applying this and Fatou's lemma to (\ref{eq:integratedW2tmp}), we obtain the bound
$\int_0^T W_2(\sP(\theta^*,\theta^t),\sP_{\alpha^t})^2 \d t
\leq C$. Since $T>0$ is arbitrary, taking $T \to \infty$ gives
\begin{equation}\label{eq:integratedW2bound}
\int_0^\infty W_2(\sP(\theta^*,\theta^t),\sP_{\alpha^t})^2 \d t \leq C.
\end{equation}

For any $s \leq t$, considering the coordinatewise coupling gives
$W_2(\bar \sP_s,\bar \sP_t)^2
\leq d^{-1} \langle \|\tilde \btheta^s-\tilde \btheta^t\|_2^2 \rangle 
\leq C(t-s)$, where the second inequality holds for a constant $C>0$ and all
$t-s \in [0,1]$ by (\ref{eq:tildethetacontinuity}). Also
\begin{equation}\label{eq:W2barPlipschitz}
W_2(\bar \sP_{\alpha^s},\bar \sP_{\alpha^t})^2
\leq d^{-1}W_2(q_{\alpha^s},q_{\alpha^t})^2 \leq C\|\alpha^t-\alpha^s\|_2^2
\leq C'(t-s)^2
\end{equation}
by the Wasserstein-2 Lipschitz continuity of $q_\alpha$
over $\alpha \in S$ shown in (\ref{eq:W2locallylipschitz}), and 
the bound (\ref{eq:dtalphabound}) for $\d\alpha^t/\d t$. Then taking the limit
$n,d \to \infty$ using (\ref{eq:barPapproximatesDMFT}) and
(\ref{eq:barPconvergence}), this shows
\[|W_2(\sP(\theta^*,\theta^t),\sP_{\alpha^t})^2
-W_2(\sP(\theta^*,\theta^s),\sP_{\alpha^s})^2| \leq C(t-s)\]
for all $t-s \in [0,1]$. Then $t \mapsto
W_2(\sP(\theta^*,\theta^t),\sP_{\alpha^t})^2$ is Lipschitz, so
(\ref{eq:integratedW2bound}) implies
\begin{equation}\label{eq:W2Pthetatalphat}
\lim_{t \to \infty}
W_2(\sP(\theta^*,\theta^t),\sP_{\alpha^t})^2=0.\end{equation}
We have similarly to (\ref{eq:W2barPlipschitz}) that
$W_2(\bar \sP_{\alpha^t},\bar \sP_{\alpha^\infty})^2
\leq d^{-1}W_2(q_{\alpha^t},q_{\alpha^\infty})^2
\leq C\|\alpha^t-\alpha^\infty\|_2^2$. Hence by (\ref{eq:barPconvergence}), also
$W_2(\sP_{\alpha^t},\sP_{\alpha^\infty})^2 \leq C\|\alpha^t-\alpha^\infty\|_2^2$,
so
\begin{equation}\label{eq:W2alphatalphainfty}
\lim_{t \to \infty} W_2(\sP_{\alpha^t},\sP_{\alpha^\infty})^2=0.
\end{equation}
Combining (\ref{eq:W2Pthetatalphat}) and (\ref{eq:W2alphatalphainfty}) show that
for any $\eps>0$, there exists $T(\eps)>0$ such that
$W_2(\sP(\theta^*,\theta^t),\sP_{\alpha^\infty})<\eps$ for all $t \geq T(\eps)$.
The second statement of (\ref{eq:posteriorconverges}) follows from this and
the almost sure convergence $\lim_{n,d \to \infty} W_2(\frac{1}{d}\sum_j
\delta_{(\theta_j^*,\theta_j^t)},\sP(\theta^*,\theta^t))=0$ ensured
by Theorem \ref{thm:dmft_approx}(b).
\end{proof}

\subsection{Analysis of examples}\label{sec:exampledetails}

\begin{proof}[Analysis of Examples \ref{ex:gaussianmean} and
\ref{ex:logconcave}]

We prove the claims in Example \ref{ex:logconcave} that Assumptions
\ref{assump:prior}(b) and \ref{assump:compactalpha} hold,
and that $\Crit$ consists of the unique point $\alpha=\alpha^*$. (Then these
claims hold also in Example \ref{ex:gaussianmean} for the Gaussian prior,
which is a special case.)

Assumption \ref{assump:prior}(b) is immediate from the
given conditions for $f(x)$. For Assumption \ref{assump:compactalpha}, let us
first show that there exists a compact interval $S \subset \R$ for which
$\{\alpha^t\}_{t \geq 0}$ is confined to $S$ (for all $t \geq 0$):
By Lemma \ref{lemma:Vdecreasing} (which does not require Assumption
\ref{assump:compactalpha}), for each fixed $t \geq 0$, almost surely
\begin{equation}\label{eq:Vdiffbound}
\limsup_{n,d \to \infty} V(q_t,\alpha^t)-V(q_0,\alpha^0) \leq 0.
\end{equation}
By the Gibbs variational principle (\ref{eq:gibbsvariational})
and the lower bound ${-}\log g(\theta,\alpha)=f(\theta-\alpha)
\geq f(0)+\frac{c_0}{2}(\theta-\alpha)^2$,
\begin{align*}
V(q_t,\alpha^t) &\geq \widehat F(\alpha^t)\\
&={-}\frac{1}{d}\log \int 
(2\pi\sigma^2)^{-n/2}
\exp\Big({-}\frac{1}{2\sigma^2}\|\y-\X\btheta\|_2^2
+\sum_{j=1}^d \log g(\theta_j,\alpha^t)\Big)\d\btheta\\
&\geq {-}\frac{1}{d}\log \int (2\pi\sigma^2)^{-n/2}
\exp\Big({-}\frac{1}{2\sigma^2}\|\y-\X\btheta\|_2^2
-f(0)d-\sum_{j=1}^d \frac{c_0}{2}(\theta_j-\alpha^t)^2\Big)\d\btheta
\end{align*}
Applying $\|\X\|_\op \leq C$ a.s.\ for all large $n,d$,
it is readily checked by explicit evaluation of this integral over
$\btheta$ that
\[V(q_t,\alpha^t) \geq C+\frac{c_0}{2}(\alpha^t)^2
-\frac{1}{2d}\Big(\frac{\X^\top \y}{\sigma^2}+c_0\alpha^t\1\Big)^\top
\Big(\frac{\X^\top\X}{\sigma^2}+c_0\I\Big)^{-1}
\Big(\frac{\X^\top \y}{\sigma^2}+c_0\alpha^t\1\Big)\]
a.s.\ for all large $n,d$ and a constant $C \in \R$ depending on
$\sigma^2,\delta,f(0),c_0,\alpha^*$,
where here $\1$ denotes the all-1's vector in $\R^d$. We have
\[\lim_{n,d \to \infty}
\frac{1}{d}\,\1^\top\Big(\frac{\X^\top \X}{\sigma^2}+c_0\I\Big)^{-1}\1
=\sigma^2 G({-}\sigma^2 c_0) < \frac{1}{c_0}\]
strictly, where $G(z)=\lim d^{-1}\Tr (\X^\top \X-z\I)^{-1}$ denotes
the Stieltjes transform of the Marcenko-Pastur spectral limit of $\X^\top \X$
\cite[Theorem 2.5]{bloemendal2014isotropic}. 
Applying this to lower-bound the quadratic term in $\alpha^t$ above,
and applying Cauchy-Schwarz to lower-bound the linear term, we get
\[V(q_t,\alpha^t) \geq C'+c'(\alpha^t)^2\]
for some constants $C' \in \R$ and $c'>0$.
Now applying this and $V(q_0,\alpha^0)=V(g_0^{\otimes d},\alpha^0) \leq C$
to (\ref{eq:Vdiffbound}), we deduce that $(\alpha^t)^2$ is uniformly
bounded over all $t \geq 0$, i.e.\ there exists a compact interval $S$ 
for which $\alpha^t \in S$ for all $t \geq 0$, as claimed. By enlarging $S$, we
may assume without loss of generality $\alpha^* \in S$.
Then, taking $O$ to be any neighborhood of $S$, the
remaining LSI condition of Assumption \ref{assump:compactalpha} holds by
the strong convexity of $f(x)$ and Proposition \ref{prop:LSI}.

We now show that $F(\alpha)$ is strictly convex on $O$, by
showing convexity of the original negative log-likelihood $\widehat F(\alpha)$:
Fixing sufficiently large and small constants $C_0,c>0$,
let us restrict to the event
\[\event=\{\|\X\|_\op \leq C_0,\; \|\X\1\|_2 \geq c\sqrt{d}\}\]
which holds a.s.\ for all large $n,d$.
Recalling the form of $\nabla^2 \widehat F(\alpha)$ from (\ref{eq:hessF})
and applying this with ${-}\log g(\theta,\alpha)=f(\theta-\alpha)$,
\[\widehat F''(\alpha)=
\frac{1}{d}\bigg\langle\sum_{j=1}^d f''(\theta_j-\alpha) \bigg \rangle_\alpha
-\frac{1}{d}\Var_\alpha\bigg[\sum_{j=1}^d f'(\theta_j-\alpha)\bigg]\]
where $\langle \cdot \rangle_\alpha$ is the average under the posterior law
corresponding to $g(\cdot,\alpha)$, and $\Var_\alpha$ is its posterior
variance. Since $f(x)$ is strictly convex, the posterior density of $\theta$ is
strictly log-concave for each fixed $\alpha$. Then, denoting
\[\v_\alpha(\btheta)=\Big(f''(\theta_j-\alpha)\Big)_{j=1}^d \in \R^d,
\qquad \D_\alpha(\btheta)=\diag\Big(f''(\theta_j-\alpha)\Big)_{j=1}^d \in \R^{d
\times d},\]
the Brascamp-Lieb inequality \cite[Theorem 4.9.1]{bakry2014analysis} implies
\[\Var_\alpha\bigg[\sum_{j=1}^d f'(\theta_j-\alpha)\bigg]
\leq \bigg\langle \v_\alpha(\btheta)^\top \bigg(\D_\alpha(\btheta)
+\frac{\X^\top\X}{\sigma^2}\bigg)^{-1}\v_\alpha(\btheta)\bigg\rangle_\alpha\]
Observing also that $\sum_{j=1}^d f''(\theta_j-\alpha)
=\v_\alpha(\btheta)^\top\D_\alpha(\btheta)^{-1}\v_\alpha(\btheta)$,
this shows
\[\widehat F''(\alpha)
\geq \frac{1}{d}\,\bigg\langle \v_\alpha(\btheta)^\top 
\bigg[\D_\alpha(\btheta)^{-1}
-\bigg(\D_\alpha(\btheta)+\frac{\X^\top\X}{\sigma^2}\bigg)^{-1}\bigg]
\v_\alpha(\btheta)\bigg\rangle_\alpha.\]
Applying the Woodbury matrix identity
and $0 \preceq \sigma^2\I+\X\D_\alpha(\btheta)^{-1}\X^\top \preceq C'\,\I$
on the event $\event$ for some constant $C'>0$,
\begin{align*}
\widehat F''(\alpha)
&\geq \frac{1}{d}\,\bigg\langle \v_\alpha(\btheta)^\top 
\bigg[\D_\alpha(\btheta)^{-1}\X^\top
\bigg(\sigma^2\,\I+\X\D_\alpha(\btheta)^{-1}\X^\top\bigg)^{-1}
\X\D_\alpha(\btheta)^{-1}\bigg]\v_\alpha(\btheta)\bigg\rangle_\alpha\\
&\geq \frac{1}{C'd} \bigg\langle \v_\alpha(\btheta)^\top 
\D_\alpha(\btheta)^{-1}\X^\top \X\D_\alpha(\btheta)^{-1}\v_\alpha(\btheta)
\bigg\rangle_\alpha
=\frac{1}{C'd}\1^\top\X^\top\X\1 \geq c',
\end{align*}
the last inequality holding for some $c'>0$ on $\event$. Thus, on $\event$,
$\widehat F(\alpha)-(c'/2)\alpha^2$ is convex over $\alpha \in \R$.
Since $\event$ holds a.s.\ for all large $n,d$ and $F(\alpha)$ is the
almost-sure pointwise limit of $\widehat F(\alpha)$, this implies that
$F(\alpha)-(c'/2)\alpha^2$ is also convex \cite[Theorem
10.8]{rockafellar1997convex}, so $F(\alpha)$ is strongly convex as claimed.
Lemma \ref{lemma:gradalphaF}(c) implies that $\nabla F(\alpha^*)=0$, i.e.\
$\alpha^*$ is a point of $\Crit$, so by this convexity it is the unique point of
$\Crit$.
\end{proof}

\begin{proposition}\label{prop:alphaconfined}
In the setting of Theorem \ref{thm:adaptivealpha_dynamics}, suppose
$R(\alpha)$ is given by (\ref{eq:ralpha}--\ref{eq:Ralpha}) with
$\|\alpha^0\|_2 \leq D$. Then there exists a constant $C(g_*,g_0,\alpha^0)>0$
depending only on $(g_*,g_0,\alpha^0)$ such that
the DMFT process $\{\alpha^t\}_{t \geq 0}$ satisfies
\[\|\alpha^t\|_2 \leq D+C(g_*,g_0,\alpha^0)\Big(\frac{1+\delta}{\sigma^2}+1\Big)
\text{ for all } t \geq 0.\]
\end{proposition}
\begin{proof}
By Lemma \ref{lemma:Vdecreasing}, for each fixed $t \geq 0$, almost surely
\begin{equation}\label{eq:VRdiffbound}
\limsup_{n,d \to \infty} \Big(V(q_t,\alpha^t)+R(\alpha^t)\Big)
-\Big(V(q_0,\alpha^0)+R(\alpha^0)\Big) \leq 0.
\end{equation}
By definition of $V(q,\alpha)$ in (\ref{eq:gibbsenergy}), we have
\begin{align*}
V(q_0,\alpha^0)=V(g_0^{\otimes d},\alpha^0)
&=\frac{1}{d}\int \frac{1}{2\sigma^2}\|\y-\X\btheta\|^2
\prod_{j=1}^d g_0(\theta_j)\d\theta_j
+\DKL(g_0\|g(\cdot,\alpha^0))+\frac{n}{2d}\log 2\pi\sigma^2,\\
V(q_t,\alpha^t)&=\frac{1}{d}\int \frac{1}{2\sigma^2}\|\y-\X\btheta\|^2
q_t(\btheta)\d\btheta
+\frac{1}{d}\DKL(q_t\|g(\cdot,\alpha^t)^{\otimes d})
+\frac{n}{2d}\log 2\pi\sigma^2.
\end{align*}
Let $\Pi_\X \in \R^{n \times n}$ be the orthogonal projection onto the
column span of $\X$. Then, applying the above forms with
$\DKL(q_t\|g(\cdot,\alpha^t)^{\otimes d})
\geq 0$ and noting that
$\|(\I-\Pi_\X)(\y-\X\btheta)\|^2=\|(\I-\Pi_\X)\y\|^2$ which
does not depend on $\btheta$, we have
\begin{align*}
&V(q_0,\alpha^0)-V(q_t,\alpha^t)\\
&\leq \frac{1}{d}\int \frac{1}{2\sigma^2}\|\y-\X\btheta\|^2
\prod_{j=1}^d g_0(\theta_j)\d\theta_j
-\frac{1}{d}\int \frac{1}{2\sigma^2}\|\y-\X\btheta\|^2
q_t(\btheta)\d\btheta+\DKL(g_0\|g(\cdot,\alpha^0))\\
&=\frac{1}{d}\int \frac{1}{2\sigma^2}\|\Pi_\X(\y-\X\btheta)\|^2
\prod_{j=1}^d g_0(\theta_j)\d\theta_j
-\frac{1}{d}\int \frac{1}{2\sigma^2}\|\Pi_\X(\y-\X\btheta)\|^2
q_t(\btheta)\d\btheta+\DKL(g_0\|g(\cdot,\alpha^0))\\
&\leq \frac{1}{d}\int \frac{1}{2\sigma^2}\|\Pi_\X(\y-\X\btheta)\|^2
\prod_{j=1}^d g_0(\theta_j)\d\theta_j+\DKL(g_0\|g(\cdot,\alpha^0)).
\end{align*}
Let us apply
\[\|\Pi_\X(\y-\X\btheta)\|^2 \leq 2\|\Pi_\X\beps\|^2+
2\|\X(\btheta^*-\btheta)\|_2^2,\]
$\|\X\|_\op^2 \leq C(1+\delta)$, and $\|\Pi_\X \beps\|^2 \leq
C\min(n,d)\sigma^2$ for a universal constant $C>0$ a.s.\ for all large $n,d$.
Then, for a constant $C(g_*,g_0,\alpha^0)>0$ depending only on
$g_*,g_0,\alpha^0$,
\[\limsup_{n,d \to \infty}
V(q_0,\alpha^0)-V(q_t,\alpha^t)
\leq C(g_*,g_0,\alpha^0)\Big(\frac{1+\delta}{\sigma^2}+1\Big).\]
Applying this to (\ref{eq:VRdiffbound}) and noting that $R(\alpha^0)=0$ because
$\|\alpha^0\| \leq D$, for every $t \geq 0$ we get
\[R(\alpha^t) \leq C(g_*,g_0,\alpha^0)\Big(\frac{1+\delta}{\sigma^2}+1\Big).\]
The lemma follows from this bound and the condition
$R(\alpha) \geq \|\alpha\|-D$ whenever $\|\alpha\| \geq D+1$.
\end{proof}

\begin{proof}[Proof of Proposition \ref{prop:largedelta}]
Fix any $s^2=\sigma^2/\delta>0$. Throughout this proof, constants may depend on
$s^2$ but not on $\delta$.
Proposition \ref{prop:alphaconfined} implies that
there exists a constant radius $D'>0$ (depending on $s^2$ but not on $\delta$)
such that for any $\delta>1$,
\[\alpha^t \in \ball(D') \text{ for all } t \geq 0.\]
Set $S=\overline{\ball(D')}$ and
$O=\ball(D'+1)$. Then for each fixed $\alpha \in O$,
Assumption \ref{assump:prior}(b) implies
\begin{equation}\label{eq:loggunifbounds}
C \geq
{-}\partial_\theta^2 \log g(\theta,\alpha) \geq \begin{cases} c_0 & \text{ for
} |\theta| \geq r_0\\
{-}C & \text{ for all } \theta \in \R \end{cases}
\end{equation}
for some $C,r_0,c_0>0$ uniformly over $\alpha \in O$.
By this bound (\ref{eq:loggunifbounds}) and Proposition \ref{prop:LSI}(b),
for some sufficiently large $\delta_0=\delta_0(s^2)>0$, $\sigma^2=\delta s^2$,
and all $\delta \geq \delta_0$,
the LSI (\ref{eq:LSI}) must hold for $g=g(\cdot,\alpha)$ and each
$\alpha \in O$. This verifies Assumption \ref{assump:compactalpha}.

Throughout the remainder of the proof, let $C,C',c,c'>0$
denote constants not depending on $\delta$ that may change from instance to
instance. We compare the optimization landscape of $F(\alpha)$ with that
of $G_{s^2}(\alpha)$ over $O$. Let
$\mse(\alpha)$, $\mse_*(\alpha)$, $\omega(\alpha)$, $\omega_*(\alpha)$ be as
defined by (\ref{eq:mmse}) and (\ref{eq:static_fixedpoint}) for the prior
$g=g(\cdot,\alpha)$. We first bound
$\mse(\alpha)$, $\mse_*(\alpha)$, $\omega(\alpha)$, $\omega_*(\alpha)$:
Write as shorthand $\langle \cdot \rangle=
\langle \cdot \rangle_{g(\cdot,\alpha),\omega(\alpha)}$ for the posterior
expectation in the scalar channel model (\ref{eq:scalarchannel}). We have
\[\langle (y-\theta) \rangle^2
=\frac{1}{Z}\int (y-\theta)^2 e^{-\frac{\omega(\alpha)}{2}(y-\theta)^2}
g(\theta,\alpha) \d\theta,
\qquad Z=\int e^{-\frac{\omega(\alpha)}{2}(y-\theta)^2}g(\theta,\alpha)\d\theta.\]
We separate the integrals over the sets
$\{\theta:e^{-\frac{\omega(\alpha)}{2}(y-\theta)^2} \leq Z\}$
and $\{\theta:e^{-\frac{\omega(\alpha)}{2}(y-\theta)^2}>Z\}$, and on the latter 
set apply the upper bound $(y-\theta)^2 \leq {-}\frac{2}{\omega(\alpha)}\log Z$.
This gives
\[\langle (y-\theta) \rangle^2
\leq \int (y-\theta)^2 \1\{e^{-\frac{\omega(\alpha)}{2}(y-\theta)^2} \leq Z\}
g(\theta,\alpha)\d\theta-\tfrac{2}{\omega(\alpha)}\log Z
\leq 2\int (y-\theta)^2 g(\theta,\alpha)\d\theta,\]
the last inequality applying Jensen's inequality to bound $\log Z \geq
\int -\frac{\omega(\alpha)}{2}(y-\theta)^2 g(\theta,\alpha)\d\theta$.
It is clear from the lower bounds of (\ref{eq:loggunifbounds}) and the
boundedness of $\log g(0,\alpha)$ and $\partial_\theta \log g(0,\alpha)$ over
$\alpha \in \overline{O}$ that $\int \theta^2 g(\theta,\alpha)\d\theta<C$ for
some constant $C>0$, for all $\alpha \in O$. Thus this inequality shows
\[\langle (y-\theta) \rangle^2 \leq C(1+y^2),\]
which implies also
\[\langle (\langle \theta \rangle-\theta) \rangle^2
\leq \langle (y-\theta) \rangle^2 \leq C(1+y^2),
\quad \langle \theta \rangle^2
\leq 2y^2+2(y-\langle \theta \rangle)^2
\leq 2y^2+2\langle (y-\theta)^2 \rangle
\leq C'(1+y^2).\]
Taking expectations over $y=\theta^*+\omega_*(\alpha)^{-1/2}z$ with
$\theta^* \sim g_*$ and $z \sim \N(0,1)$, we get $\mse(\alpha),\mse_*(\alpha)
\leq C(1+\omega_*(\alpha)^{-1})$. Then applying 
$\omega_*(\alpha)^{-1}=(\sigma^2+\mse_*(\alpha))/\delta
\leq s^2+C(1+\omega_*(\alpha)^{-1})/\delta$,
for all $\delta>\delta_0$ sufficiently large,
this implies $\omega_*(\alpha)^{-1} \leq C'$. This in turn shows
by $\mse(\alpha),\mse_*(\alpha) \leq C(1+\omega_*(\alpha)^{-1})$ that
\begin{equation}\label{eq:msebounds}
\mse(\alpha),\mse_*(\alpha) \leq C.
\end{equation}
Let $o_\delta(1)$ denote a quantity that converges to 0 uniformly over
$\alpha \in O$ as $\delta \to \infty$ (fixing $s^2=\sigma^2/\delta$).
Then, applying (\ref{eq:msebounds}) to the fixed point equations
$\omega(\alpha)=\delta/(\sigma^2+\mse(\alpha))$ and
$\omega_*(\alpha)=\delta/(\sigma^2+\mse_*(\alpha))$, we have
\begin{equation}\label{eq:omegabounds}
\omega(\alpha)^{-1}=s^2+o_\delta(1),\qquad \omega_*(\alpha)^{-1}
=s^2+o_\delta(1).
\end{equation}

We recall from Lemma \ref{lemma:mmse_lipschitz} that
$\omega(\alpha),\omega_*(\alpha)$ must be
continuous functions of $\alpha \in O$. We now argue via the implicit function
theorem that for all $\delta>\delta_0$ sufficiently large, these are in
fact continuously-differentiable over $\alpha \in O$. For this,
fix any $\alpha \in O$ and consider the map
\begin{equation}\label{eq:fixedpointmap}
f_\alpha(\omega,\omega_*)=
\begin{pmatrix}
\omega^{-1}-\delta^{-1}(\sigma^2+\E_{g_*,\omega_*}[\langle (\theta-\langle
\theta \rangle_{g(\cdot,\alpha),\omega})^2\rangle_{g(\cdot,\alpha),\omega}]) \\
\omega_*^{-1}-\delta^{-1}(\sigma^2+\E_{g_*,\omega_*}[(\theta^*-
\langle \theta \rangle_{g(\cdot,\alpha),\omega})^2]\end{pmatrix}.
\end{equation}
Thus (\ref{eq:mmse}) and (\ref{eq:static_fixedpoint}) imply that
$0=f_\alpha(\omega(\alpha),\omega_*(\alpha))$. Let us momentarily write as
shorthand $\E=\E_{g_*,\omega_*}$ and $\langle \cdot \rangle=\langle \cdot
\rangle_{g(\cdot,\alpha),\omega}$. Expressing
$y=\theta^*+\omega_*^{-1/2}z$,
$\E$ may be understood as the expectation over $\theta^* \sim g_*$ and $z \sim
\N(0,1)$. The expected posterior average $\E \langle \cdot \rangle$
is given explicitly by
\[\E \langle f(\theta) \rangle
=\E \frac{\int f(\theta) e^{H_\alpha(\theta,\omega,\omega_*)}\d\theta}
{\int e^{H_\alpha(\theta,\omega,\omega_*)}\d\theta},
\qquad H_\alpha(\theta,\omega,\omega_*)
=\omega(\theta^*+\omega_*^{-1/2}z)\theta-\frac{\omega}{2}\theta^2
+\log g(\theta,\alpha),\]
and the derivatives in
$(\omega,\omega_*)$ may be computed via differentiation of $H_\alpha$.
Let us denote by $\kappa_j(\cdot)$ the $j^\text{th}$ mixed cumulant associated
to the posterior mean $\langle \cdot \rangle=\langle \cdot
\rangle_{g(\cdot,\alpha),\omega}$, i.e.\
\[\kappa_1(f(\theta))=\langle f(\theta) \rangle,
\qquad \kappa_2(f(\theta),g(\theta))=\langle f(\theta)g(\theta) \rangle
-\langle f(\theta) \rangle\langle g(\theta) \rangle,\]
etc. Then 
$\E[\langle (\theta-\langle \theta \rangle)^2\rangle]
=\E[\kappa_2(\theta,\theta)]$ and
$\E[(\theta^*-\langle \theta \rangle)^2]
=\E[(\theta^*-\kappa_1(\theta))^2]$, and differentiating in
$(\omega,\omega_*)$ gives
\begin{align*}
\partial_\omega \E[\langle (\theta-\langle
\theta \rangle)^2\rangle]
&=\E[\kappa_3(\theta,\theta,\partial_\omega
H_\alpha(\theta,\omega,\omega_*))],\\
\partial_{\omega_*} \E[\langle (\theta-\langle
\theta \rangle)^2\rangle]
&=\E[\kappa_3(\theta,\theta,\partial_{\omega_*}
H_\alpha(\theta,\omega,\omega_*))],\\
\partial_\omega \E[(\theta^*-\langle \theta \rangle)^2]
&=\E[{-2}(\theta^*-\kappa_1(\theta))\kappa_2(\theta,
\partial_\omega H_\alpha(\theta,\omega,\omega_*))]\\
\partial_{\omega_*} \E[(\theta^*-\langle \theta \rangle)^2]
&=\E[{-2}(\theta^*-\kappa_1(\theta))\kappa_2(\theta,
\partial_{\omega_*} H_\alpha(\theta,\omega,\omega_*))]
\end{align*}
We note that each absolute moment $\E_{g_*,\omega_*(\alpha)}[\langle |\theta|^k
\rangle_{g(\cdot,\alpha),\omega(\alpha)}]$ is bounded by a constant over
$\alpha \in O$, by continuity of this quantity in $\alpha$ and compactness of
$\overline{O}$. Then it is direct to check that each of the above four
derivatives evaluated at $(\omega,\omega_*)=(\omega(\alpha),\omega_*(\alpha))$
is also bounded by a constant over $\alpha \in O$. This implies that
the derivative of the map $f_\alpha(\omega,\omega_*)$ in
(\ref{eq:fixedpointmap}) satisfies
\begin{equation}\label{eq:domegaf}
\d_{\omega,\omega_*}
f_\alpha(\omega,\omega_*)\Big|_{(\omega,\omega_*)=(\omega(\alpha),\omega_\ast(\alpha))}
=\begin{pmatrix} {-}\omega^{-2}+o_\delta(1) &
o_\delta(1) \\ o_\delta(1) & -\omega_*^{-2}+o_\delta(1) \end{pmatrix}
\bigg|_{(\omega,\omega_*)=(\omega(\alpha),\omega_\ast(\alpha))}
=-s^2\,\I+o_\delta(1),
\end{equation}
where the last equality applies (\ref{eq:omegabounds}).
In particular, for $\delta>\delta_0$ sufficiently large, this derivative is
invertible. Since $f_\alpha(\omega,\omega_*)$ is continuously-differentiable
in $(\omega,\omega_*,\alpha)$ (where differentiability
in $\alpha$ is ensured by Assumption \ref{assump:prior}(b) for
$\log g(\theta,\alpha)$), the implicit function theorem implies that for
each $\alpha_0 \in O$, there exists a unique continuously-differentiable
extension of the root $(\omega(\alpha_0),\omega_*(\alpha_0))$ of
$0=f_{\alpha_0}(\omega(\alpha_0),\omega_*(\alpha_0))$ to a solution of
$0=f_\alpha(\omega,\omega_*)$ in an open neighborhood of $\alpha_0$.
This extension must then coincide with $\omega(\alpha),\omega_*(\alpha)$,
because Lemma \ref{lemma:mmse_lipschitz} ensures that
$\omega(\alpha),\omega_*(\alpha)$ are continuous in $\alpha$.
Thus $\omega(\alpha),\omega_*(\alpha)$ are continuously-differentiable
in $\alpha \in O$, as claimed. The
implicit function theorem shows also that their first derivatives are given by 
\[\begin{pmatrix} \nabla_\alpha \omega^\top \\ \nabla_\alpha \omega_*^\top
\end{pmatrix}={-}[\d_{\omega,\omega_*} f_\alpha]^{-1} \d_\alpha
f_\alpha\Big|_{(\omega,\omega_*)=(\omega(\alpha),\omega_*(\alpha))}.\]
We may check as above that the $\alpha$-derivatives
\begin{align*}
\partial_{\alpha_j} \E[\langle (\theta-\langle
\theta \rangle)^2\rangle]
&=\E[\kappa_3(\theta,\theta,\partial_{\alpha_j}
H_\alpha(\theta,\omega,\omega_*))],\\
\partial_{\alpha_j} \E[(\theta^*-\langle \theta \rangle)^2]
&=\E[{-2}(\theta^*-\kappa_1(\theta))\kappa_2(\theta,
\partial_{\alpha_j} H_\alpha(\theta,\omega,\omega_*))]
\end{align*}
evaluated at $(\omega,\omega_*)=(\omega(\alpha),\omega_*(\alpha))$
are also both bounded by a constant over $\alpha \in \overline{O}$. By the
definition of $f_\alpha$, this implies
$\d_\alpha f_\alpha|_{(\omega,\omega_*)=(\omega(\alpha),\omega_*(\alpha))}
=o_\delta(1)$, so together with (\ref{eq:domegaf}), this shows also
\begin{equation}\label{eq:gradomegabounds}
\nabla_\alpha \omega(\alpha)=o_\delta(1),
\quad \nabla_\alpha \omega_*(\alpha)=o_\delta(1).
\end{equation}

Recall from Lemma \ref{lemma:gradalphaF} that
\begin{equation}\label{eq:nablaFrecap}
\nabla F(\alpha)={-}\E_{\theta \sim \sP_\alpha}
\nabla_\alpha \log g(\theta,\alpha)
=-\E_{g_*,\omega_*(\alpha)}\langle \nabla_\alpha \log g(\theta,\alpha)
\rangle_{g(\cdot,\alpha),\omega(\alpha)}.
\end{equation}
Applying continuity of $(\omega,\omega_*) \mapsto
\E_{g_*,\omega_*}\langle \nabla_\alpha \log g(\theta,\alpha)
\rangle_{g(\cdot,\alpha),\omega}$ and the approximations
$\omega(\alpha)^{-1},\omega_*(\alpha)^{-1}=s^2+o_\delta(1)$ shown above,
we have
\begin{equation}\label{eq:gradFcompare}
\nabla F(\alpha)={-}\E_{g_*,s^{-2}}\langle \nabla_\alpha \log g(\theta,\alpha)
\rangle_{g(\cdot,\alpha),s^{-2}}+o_\delta(1)
=\nabla G_{s^2}(\alpha)+o_\delta(1),
\end{equation}
where $G_{s^2}(\alpha)={-}\E_{g_*,s^{-2}}[\log \sP_{g(\cdot,\alpha),s^{-2}}(y)]$
is the negative population log-likelihood (\ref{eq:Galpha}) in the scalar
channel model with fixed noise variance $s^2$. 
[Note that fixing an arbitrary point $\alpha_0 \in O$ and integrating this
gradient approximation over $\alpha \in O$, this also implies
\[F(\alpha)=G(\alpha)+(F(\alpha_0)-G(\alpha_0))+o_\delta(1),\]
i.e.\ $F$ approximately coincides with $G$ up to an additive shift.]
Furthermore, the above continuous-differentiability
of $\omega(\alpha),\omega_*(\alpha)$ and
(\ref{eq:nablaFrecap}) imply $F(\alpha)$ is
twice continuously-differentiable over $\alpha \in O$, and differentiating
$\nabla F(\alpha)$ by the chain rule gives
\begin{align}
\partial_{\alpha_i}\partial_{\alpha_j}
F(\alpha)&=-\partial_\omega\Big(\E_{g_*,\omega_*(\alpha)}\langle
\partial_{\alpha_i} \log g(\theta,\alpha)
\rangle_{g(\cdot,\alpha),\omega(\alpha)}\Big)
\cdot \partial_{\alpha_j} \omega(\alpha)\notag\\
&\hspace{1in}-\partial_{\omega_*} \Big(\E_{g_*,\omega_*(\alpha)}\langle
\partial_{\alpha_i} \log g(\theta,\alpha)
\rangle_{g(\cdot,\alpha),\omega(\alpha)}\Big)
\cdot \partial_{\alpha_j} \omega_*(\alpha)\notag\\
&\hspace{1in}-\partial_{\alpha_j} \Big(\E_{g_*,\omega_*(\alpha)}\langle
\partial_{\alpha_i} \log g(\theta,\alpha)
\rangle_{g(\cdot,\alpha),\omega(\alpha)}\Big).\label{eq:HessFform}
\end{align}
Writing again $\E=\E_{g_*,\omega_*}$, 
$\langle \cdot \rangle=\langle \cdot \rangle_{g(\cdot,\alpha),\omega}$,
and $\kappa_j$ for the cumulants with respect to $\langle \cdot \rangle$,
we have
\begin{align*}
\partial_\omega \E \langle \partial_{\alpha_i} \log g(\theta,\alpha) \rangle
&=\E[\kappa_2(\partial_{\alpha_i} \log g(\theta,\alpha),
\partial_\omega H_\alpha(\theta,\omega,\omega_*))]\\
\partial_{\omega_*}
\E \langle \partial_{\alpha_i} \log g(\theta,\alpha) \rangle
&=\E[\kappa_2(\partial_{\alpha_i} \log g(\theta,\alpha),
\partial_{\omega_*} H_\alpha(\theta,\omega,\omega_*))],
\end{align*}
and these are bounded at $(\omega,\omega_*)=(\omega(\alpha),\omega_*(\alpha))$
over all $\alpha \in O$. Furthermore
\[\partial_{\alpha_j} \E \langle
\partial_{\alpha_i} \log g(\theta,\alpha) \rangle
=\E \langle
\partial_{\alpha_i} \partial_{\alpha_j} \log g(\theta,\alpha)
\rangle+\E[\kappa_2(\partial_{\alpha_i} \log g(\theta,\alpha),
\partial_{\alpha_j} \log g(\theta,\alpha))].\]
Applying these and the bounds (\ref{eq:gradomegabounds}) to
(\ref{eq:HessFform}),
\begin{align*}
&\partial_{\alpha_i}\partial_{\alpha_j} F(\alpha)\\
&=-\E_{g_*,\omega_*(\alpha)} \langle
\partial_{\alpha_i} \partial_{\alpha_j} \log g(\theta,\alpha)
\rangle_{g(\cdot,\alpha),\omega(\alpha)}
-\E_{g_*,\omega_*(\alpha)} \Cov_{\langle g(\cdot,\alpha),\omega(\alpha) \rangle}
(\partial_{\alpha_i} \log g(\theta,\alpha),
\partial_{\alpha_j} \log g(\theta,\alpha))+o_\delta(1)\\
&=\underbrace{-\E_{g_*,s^{-2}} \langle
\partial_{\alpha_i} \partial_{\alpha_j} \log g(\theta,\alpha)
\rangle_{g(\cdot,\alpha),s^{-2}}
-\E_{g_*,s^{-2}} \Cov_{\langle g(\cdot,\alpha),s^{-2} \rangle}
(\partial_{\alpha_i} \log g(\theta,\alpha),
\partial_{\alpha_j} \log g(\theta,\alpha))}_{=\partial_{\alpha_i}
\partial_{\alpha_j} G_{s^2}(\alpha)}+o_\delta(1)\\
\end{align*}
Thus we have shown
\begin{equation}\label{eq:hessFapprox}
\nabla^2 F(\alpha)=\nabla^2 G_{s^2}(\alpha)+o_\delta(1)
\end{equation}
where again $o_\delta(1)$ converges to 0 uniformly over $\alpha \in O$ as
$\delta \to \infty$.

The approximation (\ref{eq:gradFcompare}) implies that
$\nabla F+\nabla R$ converges uniformly to
$\nabla G_{s^2}+\nabla R$ over $\alpha \in O$,
as $\delta \to \infty$. Then for all $\delta>\delta_0$ sufficiently large and
for some
function $\iota:[\delta_0,\infty) \to (0,\infty)$ satisfying
$\iota(\delta) \to 0$ as $\delta \to \infty$, each
point of $\Crit \cap \,\ball(D)=\{\alpha \in \ball(D):\nabla F(\alpha)=0\}$
must fall within a ball of radius $\iota(\delta)$ around a point of
$\Crit_G=\{\alpha \in \ball(D):\nabla G_{s^2}(\alpha)=0\}$.
The approximation (\ref{eq:hessFapprox}) further
implies that for each such ball around a point $\alpha_0 \in \Crit_G$,
$\nabla^2 F$ converges uniformly to
$\nabla^2 G_{s^2}$ on this ball, as $\delta \to \infty$. If
$\nabla^2 G_{s^2}(\alpha_0)$ is non-singular, then for all $\delta>\delta_0$
sufficiently large, an argument via the topological degree shows that
there must be exactly one point of $\Crit$ in this ball (having the same index
as $\alpha_0$ as a critical point of $G_{s^2}$) --- see e.g.\ \cite[Lemma
5]{mei2018landscape}. This shows statements (1) and (2) of the proposition.

As a direct consequence of these statements, if
$\alpha^*$ is the unique point of $\Crit_G$ and $\nabla^2 G_{s^2}(\alpha^*)$
is non-singular, then there is a unique point of $\Crit \cap \,\ball(D)$.
If furthermore $g_*(\theta)=g(\theta,\alpha^*)$, then this point of
$\Crit \cap \,\ball(D)$ must be $\alpha^*$ itself, since
$\nabla F(\alpha^*)=0$ by Lemma \ref{lemma:gradalphaF}(c).
\end{proof}

\begin{proof}[Analysis of Example \ref{ex:gaussianmeanmixture}]
We verify Assumption \ref{assump:prior}(b)
for Example \ref{ex:gaussianmeanmixture} of the
Gaussian mixture model with varying means. Let $\iota \in \{1,\ldots,K\}$
denote the mixture component of $\theta$, and let
$\langle f(\iota,\theta) \rangle=\E[f(\iota,\theta) \mid \theta]$
denote the posterior average over $\iota$ given $\theta \sim
\N(\alpha_\iota,\omega_0^{-1})$ and prior $\P[\iota=k]=p_k$.
Let $\kappa_2(\cdot)$ denote the
covariance associated to $\langle \cdot \rangle$. Then, since
\[\log g(\theta,\alpha)=\log \sum_{k=1}^K p_k \sqrt{\frac{\omega_k}{2\pi}}
\exp\Big({-}\frac{\omega_k}{2}(\theta-\alpha_k)^2\Big),\]
the derivatives of $\log g(\theta,\alpha)$ up to order 2 are given by
\begin{equation}\label{eq:loggderivmeanmix}
\begin{aligned}
\partial_\theta \log g(\theta,\alpha)
&=\langle \omega_\iota(\alpha_\iota-\theta)\rangle,
\quad \partial_\theta^2 \log g(\theta,\alpha)
=\kappa_2\big(\omega_\iota(\alpha_\iota-\theta),\omega_\iota(\alpha_\iota-\theta)\big)-\langle
\omega_\iota \rangle\\
\partial_{\alpha_i} \log g(\theta,\alpha)
&=\omega_i(\theta-\alpha_i)\langle \1_{\iota=i} \rangle,
\quad \partial_{\alpha_i}\partial_\theta \log g(\theta,\alpha)
=\omega_i(\theta-\alpha_i)\kappa_2\big(\1_{\iota=i},\omega_\iota(\alpha_\iota-\theta)\big)+\omega_i\langle \1_{\iota=i} \rangle,\\
\partial_{\alpha_i}\partial_{\alpha_j} \log g(\theta,\alpha)
&=\omega_i\omega_j(\theta-\alpha_i)(\theta-\alpha_j)\kappa_2(\1_{\iota=i},\1_{\iota=j})-\1_{i=j}\omega_i\langle \1_{\iota=i} \rangle\\
\end{aligned}
\end{equation}
In particular, $|\partial_\theta \log g(\theta,\alpha)| \leq
C(1+|\theta|+|\alpha_i|)$ and $|\partial_{\alpha_i} \log g(\theta,\alpha)| \leq
C(1+|\theta|+|\alpha_i|)$, showing (\ref{eq:logggradientbound}).

To bound the high-order derivatives of $\log g(\theta,\alpha)$
locally over $\alpha \in \R^K$, 
let $k_{\max} \in \{1,\ldots,K\}$ be the (unique)
index corresponding to the smallest
value of $\omega_k$. For any compact subset $S \subset \R^K$, there exist
constants $B(S),c_0(S)>0$ depending on the fixed values $\{p_1,\ldots,p_K\}$,
$\{\omega_1,\ldots,\omega_K\}$ and $S$ such that for all $\alpha \in S$, we have
\[\frac{\omega_k}{2}(\theta-\alpha_k)^2 \geq \frac{\omega_{k_{\max}}}{2}
(\theta-\alpha_{k_{\max}})^2
+c_0(S)\theta^2 \text{ for any } \theta>B(S) \text{ and all } k \neq k_{\max}.\]
This implies there exists a constant $C(S)>0$ for which
\[\langle \1_{\iota \neq k_{\max}} \rangle \leq C(S)e^{-c_0(S)\theta^2}
\text{ for all } \theta>B(S) \text{ and } \alpha \in S.\]
Let $\iota'$ denote an independent copy of $\iota$ under its posterior law
given $\theta$. Then for any $\theta>B$, any $\alpha \in S$,
and any $k \in \{1,\ldots,K\}$,
the posterior variance of $\1_{\iota=k}$ is bounded as
\[\langle |\1_{\iota=k}-\langle \1_{\iota=k} \rangle|^2 \rangle
\leq \langle |\1_{\iota=k}-\1_{\iota'=k}|^2 \rangle
\leq 4 \langle \1_{\iota \neq k_{\max} \text{ or }
\iota' \neq k_{\max}} \rangle \leq C'(S)e^{-c_0(S)\theta^2},\]
and similarly
\[\langle |\omega_\iota(\theta-\alpha_\iota)-\langle
\omega_\iota(\theta-\alpha_\iota) \rangle|^2 \rangle
\leq C'(S)(1+\theta^2)e^{-c_0(S)\theta^2}.\]
Applying these bounds and H\"older's inequality, all posterior covariances in
(\ref{eq:loggderivmeanmix}) are exponentially small in
$\theta^2$ for $\theta>B(S)$, implying that all derivatives of order 2 in
(\ref{eq:loggderivmeanmix}) are bounded over $\alpha \in S$ and $\theta>B(S)$.
Similarly they are bounded over $\alpha \in S$ and $\theta<-B(S)$, and hence
also bounded uniformly over $\alpha \in S$ and $\theta \in \R$ since we may
bound the cumulants trivially by a constant $C(S)$ for $\theta \in [-B(S),B(S)]$.
The same argument bounds all mixed cumulants of $\1_{\iota=k}$ and
$\omega_\iota(\alpha_\iota-\theta)$ of orders 3 and 4, and hence also all
partial derivatives of $\log g(\theta,\alpha)$ of orders 3 and 4
over $\alpha \in S$ and $\theta \in \R$. These
arguments show also that as $\theta \to \pm \infty$, uniformly over $\alpha \in
S$,
$\kappa_2(\omega_\iota(\alpha_\iota-\theta),\omega_\iota(\alpha_\iota-\theta))
\to 0$ and $\langle \omega_\iota \rangle \to \omega_{k_{\max}}$, so
$\partial_\theta^2[-\log g(\theta,\alpha)] \to \omega_{k_{\max}}>0$, verifying
all statements of Assumption \ref{assump:prior}(b).
\end{proof}

\begin{proof}[Analysis of Example \ref{ex:gaussianweightmixture}]
We verify Assumption \ref{assump:prior}(b)
in Example \ref{ex:gaussianweightmixture} for the
Gaussian mixture model with fixed mixture means/variances and varying weights.
Again let $\iota \in \{0,\ldots,K\}$ denote the mixture component of $\theta$,
and let $\langle f(\iota,\theta) \rangle=\E[f(\iota,\theta) \mid \theta]$
denote the posterior average over $\iota$ given $\theta \sim
\N(\mu_\iota,\omega_\iota^{-1})$ and prior
$\P[\iota=k]=e^{\alpha_k}/(e^{\alpha_0}+\ldots+e^{\alpha_K})$.
Let $\kappa_2(\cdot)$ denote the
covariance associated to $\langle \cdot \rangle$, and
in addition, let $\langle \cdot \rangle_{\text{prior}}$ and
$\kappa_2^{\text{prior}}$ denote the mean and covariance
over $\iota$ drawn from the prior
$\P[\iota=k]=e^{\alpha_k}/(e^{\alpha_0}+\ldots+e^{\alpha_K})$.
Then, since
\[\log g(\theta,\alpha)=\log\sum_{k=0}^K e^{\alpha_k}
\sqrt{\frac{\omega_k}{2\pi}}\exp\Big({-}\frac{\omega_k}{2}(\theta-\mu_k)^2\Big)
-\log\sum_{k=0}^K e^{\alpha_k},\]
the derivatives of $\log g(\theta,\alpha)$ up to order 2 are given by
\begin{equation}\label{eq:gaussianmixturederivs}
\begin{gathered}
\partial_{\alpha_i} \log g(\theta,\alpha)=\langle \1_{\iota=i} \rangle
-\langle \1_{\iota=i}\rangle_{\text{prior}},
\quad \partial_{\alpha_i} \partial_{\alpha_j}
\log g(\theta,\alpha)=\kappa_2(\1_{\iota=i},\1_{\iota=j})
-\kappa_2^{\text{prior}}(\1_{\iota=i},\1_{\iota=j}),\\
\partial_\theta \log g(\theta,\alpha)
=\langle \omega_\iota(\mu_\iota-\theta) \rangle,
\quad \partial_{\alpha_i}\partial_\theta
\log g(\theta,\alpha)
=\kappa_2\big(\1_{\iota=i},\omega_\iota(\mu_\iota-\theta)\big),\\
\partial_\theta^2 \log g(\theta,\alpha)
=\kappa_2\big(\omega_\iota(\mu_\iota-\theta),
\omega_\iota(\mu_\iota-\theta)\big)-\langle \omega_\iota \rangle.
\end{gathered}
\end{equation}
In particular, this shows $\sum_k \partial_{\alpha_k} \log
g(\theta,\alpha)=1-1=0$,
so $\nabla_\alpha \log g(\theta,\alpha)$ always belongs to the subspace
$E=\{\alpha \in \R^{K+1}:\alpha_0+\ldots+\alpha_K=0\}$. Also $\nabla R(\alpha)=r'(\|\alpha\|_2) \cdot \frac{\alpha}{\|\alpha\|_2} \in E$ if $\alpha \in E$.
Furthermore,
$|\partial_{\alpha_i} \log g(\theta,\alpha)| \leq C$ and
$|\partial_\theta \log g(\theta,\alpha)| \leq C(1+|\theta|)$, showing
(\ref{eq:logggradientbound}).

To bound the higher-order derivatives of $\log g(\theta,\alpha)$
locally over $\alpha \in E$,
let $k_{\max} \in \{0,\ldots,K\}$ be the
index corresponding to the smallest $\omega_k$, and among these the largest
$\mu_k$ (if there are multiple $\omega_k$'s equal to the smallest value).
Then for some constants $B,c_0>0$ depending only on the fixed values
$\{\mu_0,\ldots,\mu_K\}$ and $\{\omega_0,\ldots,\omega_K\}$, we have
\[\frac{\omega_k}{2}(\theta-\mu_k)^2
\geq \frac{\omega_{k_{\max}}}{2}
(\theta-\mu_{k_{\max}})^2+c_0\theta \text{ for any } \theta>B \text{ and all }
k \neq k_{\max}.\]
This implies, for any compact subset $S \subset E$, there is a constant $C(S)>0$
for which
\[\langle \1_{\iota \neq k_{\max}} \rangle
\leq C(S)e^{-c_0\theta} \text{ for all } \theta>B \text{ and } \alpha \in S.\]
Then the same arguments as in the preceding example show
\[\langle |\omega_\iota-\langle \omega_\iota \rangle|^2 \rangle
\leq C'(S)e^{-c_0\theta}, \quad
\langle |\omega_\iota(\mu_\iota-\theta)-\langle \omega_\iota(\mu_\iota-\theta)
\rangle|^2 \rangle
\leq C'(S)(1+\theta)^2e^{-c_0\theta},\]
implying via Cauchy-Schwarz that each order-2 derivative in
(\ref{eq:gaussianmixturederivs}) is bounded over
$\alpha \in S$ and $\theta>B$. Similarly it is bounded over
$\alpha \in S$ and $\theta<-B$, hence also for all $\alpha \in S$
and $\theta \in \R$. The same argument applies to bound the
mixed cumulants of $\omega_\iota$ and
$\omega_\iota(\mu_\iota-\theta)$ of orders 3 and 4, and thus the partial
derivatives of $\log g(\theta,\alpha)$ of orders 3 and 4.
This shows also
$\lim_{\theta \to \infty} \partial_\theta^2 [-\log g(\theta,\alpha)]
=\omega_{k_{\max}}>0$ uniformly over $\alpha \in S$, and a similar statement
holds for $\theta \to {-}\infty$,
establishing all conditions of Assumption \ref{assump:prior}(b).
\end{proof}

%% file: globalconditions.tex
\section{Proof of Theorem \ref{thm:dmft_approx}}\label{appendix:globalconditions}

Theorem \ref{thm:dmft_approx}(a) follows immediately from \cite[Theorem
2.5]{paper1}, upon identifying $s(\theta,\alpha)$ of \cite[Theorem 2.5]{paper1}
as $(\log g)'(\theta)$ (with no dependence on $\alpha$)
and $\cG(\alpha,\sP)=0$. The required conditions of
\cite[Assumption 2.2]{paper1} for $s(\cdot)$ hold by Assumption
\ref{assump:prior}(a), and the conditions of
\cite[Assumption 2.3]{paper1} for $\cG(\cdot)$ are vacuous.

For Theorem \ref{thm:dmft_approx}(b), consider first the following
global version of Assumption \ref{assump:prior}(b):
\begin{assumption}\label{assump:priorglobal}
$\log g(\theta,\alpha)$ and $R(\alpha)$ are thrice continuously-differentiable and satisfy
(\ref{eq:logggradientbound}), and the conditions (\ref{eq:localconditions}) hold for constants
$C,r_0,c_0>0$ globally over all $\alpha \in \R^K$. 
\end{assumption}
Under Assumption \ref{assump:priorglobal},
Theorem \ref{thm:dmft_approx}(b) again follows from
\cite[Theorem 2.5]{paper1} upon identifying $s(\theta,\alpha)=\partial_\theta \log
g(\theta,\alpha)$ and $\cG(\alpha,\sP)=\E_{\theta \sim \sP}[\nabla_\alpha \log
g(\theta,\alpha)]-\nabla R(\alpha)$, where all conditions of
\cite[Assumptions 2.2 and 2.3]{paper1} may be checked from these
conditions of Assumption \ref{assump:priorglobal}.

To show Theorem \ref{thm:dmft_approx}(b) under the weaker local conditions
of Assumption \ref{assump:prior}(b), we may apply the following truncation
argument: Note first that
twice continuous-differentiability of $\log g(\theta,\alpha)$ and $R(\alpha)$
imply that $\nabla_{(\theta,\alpha)} \log g(\theta,\alpha)$
and $\nabla R(\alpha)$ are locally Lipschitz. Together with the global
linear growth conditions of (\ref{eq:logggradientbound}), this implies that
there exists a unique (non-explosive) solution
$\{(\btheta^t,\widehat\alpha^t)\}_{t \geq 0}$
to the joint diffusion (\ref{eq:langevin_sde}--\ref{eq:gflow}) for all times
(c.f.\ \cite[Theorem 12.1]{rogers2000diffusions}). Furthermore, since
\begin{align*}
\btheta^t&=\btheta^0+\int_0^t \Big({-}\frac{1}{2\sigma^2}\X^\top(\X\btheta^s-\y)
+\big(\partial_\theta \log g(\theta_j^s,\widehat\alpha^s)\big)_{j=1}^d\Big)\d s
+\sqrt{2}\,\b^t\\
\widehat\alpha^t&=\widehat\alpha^0+\int_0^t
\Big(\frac{1}{d}\sum_{j=1}^d \nabla_\alpha \log
g(\theta_j^s,\widehat\alpha^s)-\nabla_\alpha R(\widehat\alpha^s)\Big)\d s,
\end{align*}
under the growth conditions (\ref{eq:logggradientbound}), this solution satisfies
the bounds
\begin{align*}
\|\btheta^t\|_2 &\leq \|\btheta^0\|_2+C\int_0^t \Big(\|\X\|_\op^2\|\btheta^s\|_2
+\|\X\|_\op\|\y\|_2+\sqrt{d}+\|\btheta^s\|_2+\sqrt{d}\|\widehat\alpha^s\|_2\Big)\d
s+\sqrt{2}\,\|\b^t\|_2\\
\|\widehat\alpha^t\|_2 &\leq \|\widehat\alpha^0\|_2+C\int_0^t
\Big(1+\|\widehat\alpha^s\|_2+\|\btheta^s\|_2/\sqrt{d}\Big)\d s
\end{align*}
Fixing any $T>0$, by the conditions of Assumption \ref{assump:prior}
and a standard maximal inequality for Brownian motion (see \cite[Lemma
4.7]{paper1}) there exists a constant $C_0>0$ large enough such that the event
\[\event=\Big\{\|\X\|_\op \leq C_0,\,\|\y\|_2 \leq C_0\sqrt{d},\,
\|\btheta^0\|_2 \leq C_0\sqrt{d},\,\|\widehat\alpha^0\|_2 \leq C_0,\,
\sup_{t \in [0,T]} \|\b^t\|_2 \leq C_0\sqrt{d}\Big\}\]
holds a.s.\ for all large $n,d$. Then by a Gronwall argument, for a constant
$M=M(T,C_0)>0$,
\[\sup_{t \in [0,T]} \frac{\|\btheta^t\|_2}{\sqrt{d}}+\|\widehat\alpha^t\|_2
<M\]
holds on $\event$. Applying the conditions of Assumption \ref{assump:prior}(b) with $S=\{\alpha:\|\alpha\|_2 \leq M\}$, there exist functions $g_M:\R \times \R^K \to \R$ and $R_M:\R^K \to \R$
such that $g_M(\theta,\alpha)=g(\theta,\alpha)$
and $R_M(\alpha)=R(\alpha)$ for all $\|\alpha\| \leq M$,
and $g_M$ and $R_M$ satisfy Assumption \ref{assump:priorglobal}. Let
$\{(\btheta_M^t,\widehat\alpha_M^t)\}_{t \geq 0}$ be the solution of
(\ref{eq:langevin_sde}--\ref{eq:gflow}) defined with $g_M(\cdot)$ and
$R_M(\cdot)$ in place of $g(\cdot)$ and $R(\cdot)$, and let $\boldeta_M^t=\X\btheta_M^t$.
Then as argued above, Theorem \ref{thm:dmft_approx}(b) holds for
$\{(\btheta_M^t,\boldeta_M^t,\widehat\alpha_M^t)\}_{t \geq 0}$, showing that
a.s.\ as $n,d \to \infty$,
\begin{equation}\label{eq:truncated_dmft}
\begin{aligned}
\frac{1}{d}\sum_{j=1}^d \delta_{\theta_j^*,\{\theta_{M,j}^t\}_{t \in [0,T]}}
&\overset{W_2}{\to} \sP(\theta^*,\{\theta_M^t\}_{t \in [0,T]})\\
\frac{1}{n}\sum_{i=1}^n \delta_{\eta_i^*,\eps_i,\{\eta_{M,i}^t\}_{t \in [0,T]}}
&\overset{W_2}{\to} \sP(\eta^*,\eps,\{\eta_M^t\}_{t \in [0,T]})\\
\{\widehat \alpha^t\}_{t \in [0,T]} &\to \{\alpha_M^t\}_{t \in [0,T]}
\end{aligned}
\end{equation}
for limiting processes defined by the DMFT equations
(\ref{def:dmft_covuw}--\ref{def:covarianceresponse}) also with
$g_M(\cdot)$ and $R_M(\cdot)$ in place of $g(\cdot)$ and $R(\cdot)$. Since 
$\{(\btheta^t,\boldeta^t,\widehat\alpha^t)\}_{t \in [0,T]}
=\{(\btheta_M^t,\boldeta_M^t,\widehat\alpha_M^t)\}_{t \in [0,T]}$ a.s.\ for all large $n,d$, this implies that (\ref{eq:truncated_dmft}) holds also with
$\{(\btheta^t,\boldeta^t,\widehat\alpha^t)\}_{t \in [0,T]}$ in place of
$\{(\btheta_M^t,\boldeta_M^t,\widehat\alpha_M^t)\}_{t \in [0,T]}$. Furthermore,
the deterministic limit process $\{\alpha_M^t\}_{t \in [0,T]}$ must satisfy
$\|\alpha_M^t\| \leq M$ for all $t \in [0,T]$, so the solution up to time $T$
of the DMFT equations (\ref{def:dmft_covuw}--\ref{def:covarianceresponse}) with
$g_M(\cdot)$ and $R_M(\cdot)$ is also a solution of these equations with
$g(\cdot)$ and $R(\cdot)$. This proves Theorem \ref{thm:dmft_approx}(b) under Assumption \ref{assump:prior}(b).

%% file: dmftgaussian.tex
\section{Correlation and response functions for a Gaussian prior}\label{appendix:gaussian}

For illustration, we check Definition \ref{def:regular} explicitly 
for the dynamics (\ref{eq:langevinfixedprior}) with a Gaussian prior
\[g(\theta)=\sqrt{\frac{\lambda}{2\pi}}\exp\Big(
{-}\frac{\lambda\theta^2}{2}\Big).\]
Then (\ref{eq:langevinfixedprior}) is the Ornstein-Uhlenbeck process
\begin{equation}\label{eq:langevin_OU}
\d\btheta^t=\Big[{-}\Big(\frac{\X^\top\X}{\sigma^2}+\lambda\I\Big)\btheta^t+
\frac{\X^\top \y}{\sigma^2}\Big]\d t+\sqrt{2}\,\d\b^t.
\end{equation}

\begin{lemma}\label{lemma:OUCR}
Under Assumption \ref{assump:model}, let
\[\mu=\lim_{n,d \to \infty} \frac{1}{d}\sum_{i=1}^d
\delta_{\lambda_i(\X^\top\X/\sigma^2)}\]
be the almost-sure limit of the
empirical eigenvalue distribution of $\X^\top \X/\sigma^2$.
Then for the dynamics (\ref{eq:langevin_OU}) with a fixed Gaussian prior,
the corresponding DMFT system prescribed by Theorem \ref{thm:dmft_approx}(a)
has the correlation and response functions
\begin{align*}
C_\theta(t,s)&=\int \bigg[\E(\theta^0)^2 \cdot e^{-(\lambda+x)(t+s)}
+\frac{\E(\theta^*)^2 x^2+x}{(\lambda+x)^2}
(1-e^{-(\lambda+x)t})(1-e^{-(\lambda+x)s})\\
&\hspace{1.5in}+\frac{1}{\lambda+x}
\Big(e^{-(\lambda+x)|t-s|}-
e^{-(\lambda+x)(t+s)}\Big)\bigg]\mu(\d x)\\
C_\theta(t,*)&=\int \frac{\E(\theta^*)^2 x}{\lambda+x}(1-e^{-(\lambda+x)t})
\mu(\d x)\\
R_\theta(t,s)&=\int e^{-(\lambda+x)(t-s)}\mu(\d x)\\
C_\eta(t,s)&=\int\bigg[\E(\theta^0)^2 x e^{-(\lambda+x)(t+s)}
+(\E(\theta^*)^2 x+1)\Big(\frac{x}{\lambda+x}(1-e^{-(\lambda+x)t})-1\Big)
\Big(\frac{x}{\lambda+x}(1-e^{-(\lambda+x)s})-1\Big)\\
&\hspace{1in}+(\delta-1)+\frac{x}{\lambda+x}\Big(e^{-(\lambda+x)|t-s|}-
e^{-(\lambda+x)(t+s)}\Big)\bigg]\mu(\d x),\\
R_\eta(t,s)&=\int x e^{-(\lambda+x)(t-s)}\mu(\d x).
\end{align*}
\end{lemma}
\begin{proof}
Setting $\bLambda=\frac{\X^\top\X}{\sigma^2}+\lambda\I$,
the dynamics (\ref{eq:langevin_OU}) have the explicit solution
\begin{align}
\btheta^t&=e^{-\bLambda t}\btheta^0
+\bLambda^{-1}\Big(\I-e^{-\bLambda t}\Big)\frac{\X^\top\y}{\sigma^2}
+\int_0^t e^{-\bLambda(t-s)}\sqrt{2}\,\d\b^s\notag\\
&=e^{-\bLambda t}\btheta^0
+\bLambda^{-1}\Big(\I-e^{-\bLambda
t}\Big)\Big(\frac{\X^\top\X}{\sigma^2}\btheta^*+\frac{\X^\top\beps}{\sigma^2}\Big)+\int_0^t
e^{-\bLambda(t-s)}\sqrt{2}\,\d\b^s\label{eq:OUexplicit}
\end{align}

Recall the definitions of $e_j(\btheta)$ and $x_i(\btheta)$ from
(\ref{eq:coordfuncs}) and the associated correlation and response matrices
(\ref{eq:CRmatrices}).
Under Assumption \ref{assump:model}, applying the explicit form
(\ref{eq:OUexplicit}) and independence of
$\X,\btheta^0,\btheta^*,\beps$, it is direct to check that almost surely,
\begin{align*}
\lim_{n,d \to \infty}\frac{1}{d}\Tr \bC_\theta(t,s)
&=\lim_{n,d \to \infty} \frac{1}{d}\langle {\btheta^t}^\top \btheta^s \rangle\\
&=\lim_{n,d \to \infty} \frac{1}{d}\Tr \bigg(
\E(\theta^0)^2 \cdot e^{-\bLambda(t+s)}
+\E(\theta^*)^2 \cdot \Big(\frac{\X^\top
\X}{\sigma^2}\Big)(\I-e^{-\bLambda t})\bLambda^{-2}(\I-e^{-\bLambda s})
\Big(\frac{\X^\top \X}{\sigma^2}\Big)\\
&\hspace{0.5in}+\int_0^{t \wedge s} 2e^{-\bLambda(t+s-2r)}\d r\bigg)
+\frac{1}{d}\Tr\bigg(\E\eps^2 \cdot \frac{\X}{\sigma^2}
(\I-e^{-\bLambda t})\bLambda^{-2}(\I-e^{-\bLambda s})\frac{\X^\top}{\sigma^2}
\bigg)\\
&=\int \bigg[\E(\theta^0)^2 \cdot e^{-(\lambda+x)(t+s)}
+\frac{\E(\theta^*)^2 x^2+x}{(\lambda+x)^2}
(1-e^{-(\lambda+x)t})(1-e^{-(\lambda+x)s})\\
&\hspace{0.5in}+\frac{1}{\lambda+x}
\Big(e^{-(\lambda+x)|t-s|}-
e^{-(\lambda+x)(t+s)}\Big)\bigg]\mu(\d x)
\end{align*}
and
\begin{align*}
\lim_{n,d \to \infty} \frac{1}{d}\Tr \bC_\theta(t,*)
=\lim_{n,d \to \infty} \frac{1}{d}\langle {\btheta^t}^\top \btheta^* \rangle
&=\lim_{n,d \to \infty} \frac{1}{d}\Tr\Big(\E(\theta^*)^2
\bLambda^{-1}(\I-e^{-\bLambda t})\frac{\X^\top\X}{\sigma^2}\Big)\\
&=\int \frac{\E(\theta^*)^2 x}{\lambda+x}(1-e^{-(\lambda+x)t})\mu(\d x).
\end{align*}
Furthermore, the above form (\ref{eq:OUexplicit}) for $\btheta^t$ implies
\begin{equation}\label{eq:OUsemigroup}
P_t(\btheta)=e^{-\bLambda t}\btheta+
+\bLambda^{-1}\Big(\I-e^{-\bLambda
t}\Big)\Big(\frac{\X^\top\X}{\sigma^2}\btheta^*+\frac{\X^\top\beps}{\sigma^2}\Big).
\end{equation}
Then $\nabla P_te_j(\btheta)=\nabla[\e_j^\top P_t(\btheta)]=e^{-\bLambda t}\e_j$ is a constant function not depending on
$\btheta$, and
\[\lim_{n,d \to \infty} \frac{1}{d}\Tr \bR_\theta(t,s)
=\lim_{n,d \to \infty} \frac{1}{d}\sum_{j=1}^d [\nabla e_j^\top \nabla P_{t-s}e_j](\btheta^s)
=\lim_{n,d \to \infty} \frac{1}{d}\Tr e^{-\bLambda (t-s)}=\int
e^{-(\lambda+x)(t-s)}\mu(\d x).\]
By Theorem \ref{thm:dmft_response},
this shows the forms of $C_\theta(t,s)$, $C_\theta(t,*)$, and $R_\theta(t,s)$.

From (\ref{eq:OUexplicit}) and (\ref{eq:OUsemigroup}), we have also
\begin{align*}
\X\btheta^t-\y&=\X e^{-\bLambda t}\btheta^0
+\X\Big(\bLambda^{-1}\Big(\I-e^{-\bLambda t}\Big)\frac{\X^\top\X}{\sigma^2}
-\I\Big)\btheta^*
+\Big(\X\bLambda^{-1}\Big(\I-e^{-\bLambda t}\Big)
\frac{\X^\top}{\sigma^2}-\I_n\Big)\beps\\
&\hspace{1in}+\int_0^t \X e^{-\bLambda(t-s)}\sqrt{2}\,\d\b^s
\end{align*}
and $\nabla P_t x_i(\btheta)=(\sqrt{\delta}/\sigma)
\nabla[\e_i^\top \X^\top P_t(\btheta)]
=(\sqrt{\delta}/\sigma)e^{-\bLambda t}\X^\top
\e_i$. Then
\begin{align*}
&\lim_{n,d \to \infty} \frac{1}{n}\Tr \bC_\eta(t,s)
=\lim_{n,d \to \infty} \frac{1}{n}
\cdot \frac{\delta}{\sigma^2}(\X\btheta^t-\y)^\top(\X\btheta^s-\y)\\
&=\lim_{n,d \to \infty} \frac{1}{d\sigma^2}\Tr\bigg(
\E(\theta^0)^2 \cdot e^{-\bLambda t}\X^\top \X e^{-\bLambda s}\\
&\hspace{1in}+\E(\theta^*)^2 \cdot \Big(\bLambda^{-1}\Big(\I-e^{-\bLambda
t}\Big)\frac{\X^\top\X}{\sigma^2}-\I\Big)\X^\top \X
\Big(\bLambda^{-1}\Big(\I-e^{-\bLambda
s}\Big)\frac{\X^\top\X}{\sigma^2}-\I\Big) \\
&\hspace{1in}
+\int_0^{t \wedge s} 2\,e^{-\bLambda(t-r)}\X^\top \X e^{-\bLambda(s-r)}
\d r\bigg)\\
&\hspace{0.5in}+\frac{1}{d\sigma^2}\Tr\bigg(
\E\eps^2 \cdot \Big(\X\bLambda^{-1}\Big(\I-e^{-\bLambda t}\Big)
\frac{\X^\top}{\sigma^2}-\I_n\Big)^\top
\Big(\X\bLambda^{-1}\Big(\I-e^{-\bLambda s}\Big)
\frac{\X^\top}{\sigma^2}-\I_n\Big)\bigg)
\\
&=\int\bigg[\E(\theta^0)^2 x e^{-(\lambda+x)(t+s)}
+(\E(\theta^*)^2 x+1)\Big(\frac{x}{\lambda+x}(1-e^{-(\lambda+x)t})-1\Big)
\Big(\frac{x}{\lambda+x}(1-e^{-(\lambda+x)s})-1\Big)\\
&\hspace{1in}+(\delta-1)+\frac{x}{\lambda+x}\Big(e^{-(\lambda+x)|t-s|}-
e^{-(\lambda+x)(t+s)}\Big)\bigg]\mu(\d x),
\end{align*}
and
\begin{align*}
\lim_{n,d \to \infty} \frac{1}{n}\Tr \bR_\eta(t,s)
&=\lim_{n,d \to \infty} \frac{1}{n}
\sum_{i=1}^n [\nabla x_i^\top \nabla P_{t-s}x_i](\btheta^s)
=\lim_{n,d \to \infty} \frac{1}{n} \cdot \frac{\delta}{\sigma^2}
\sum_{i=1}^n \e_i^\top \X e^{-\bLambda(t-s)}\X^\top \e_i\\
&=\lim_{n,d \to \infty} \frac{1}{d\sigma^2}\Tr 
\X e^{-\bLambda(t-s)}\X^\top 
=\int x e^{-(\lambda+x)(t-s)}\mu(\d x).
\end{align*}
By Theorem \ref{thm:dmft_response}, this shows the forms of $C_\eta(t,s)$
and $R_\eta(t,s)$.
\end{proof}

From Lemma \ref{lemma:OUCR} it is apparent that the approximations
(\ref{eq:Cthetaapproxinvariant}--\ref{eq:Cthetastarlim}) and
(\ref{eq:Rthetaapproxinvariant}--\ref{eq:Retaapproxinvariant}) hold with
$\eps(t)=Ce^{-ct}$ and
\[\begin{gathered}
c_\theta^\init(s)={-}\int
\frac{\E(\theta^*)^2x^2+x}{(\lambda+x)^2}e^{-(\lambda+x)s}\mu(\d x),
\qquad
c_\theta^\tti(\infty)=\int \frac{\E(\theta^*)^2x^2+x}{(\lambda+x)^2}\mu(\d x)\\
c_\theta^\tti(\tau)=c_\theta^\tti(\infty)
+\int \frac{1}{\lambda+x}e^{-(\lambda+x)\tau}\mu(\d x),
\qquad r_\theta^\tti(\tau)=\int e^{-(\lambda+x)\tau} \mu(\d x),
\qquad c_\theta(*)=\int \frac{\E(\theta^*)^2 x}{\lambda+x}\mu(\d x)\\
c_\eta^\init(s)=\int \frac{(\E(\theta^*)^2 x+1)\lambda x}{(\lambda+x)^2}
e^{-(\lambda+x)s}\mu(\d x),
\qquad c_\eta^\tti(\infty)=\int \frac{(\E(\theta^*)^2
x+1)\lambda^2}{(\lambda+x)^2}\mu(\d x)+\delta-1\\
c_\eta^\tti(\tau)=c_\eta^\tti(\infty)+
\int \frac{x}{\lambda+x}e^{-(\lambda+x)\tau}\mu(\d x),
\qquad r_\eta^\tti(\tau)=\int xe^{-(\lambda+x)\tau}\mu(\d x).
\end{gathered}\]
These functions $c_\theta^\tti,c_\eta^\tti$ have the forms
(\ref{eq:cttiforms}) for the positive, finite measures
$\mu_\theta(\d a)=a^{-1}\mu(\d(a-\lambda))$
and $\mu_\eta(\d a)=[(a-\lambda)/a]\mu(\d(a-\lambda))$ supported on
$a \in [\lambda,\infty)$. Furthermore, these functions
$c_\theta^\tti,c_\eta^\tti,r_\theta^\tti,r_\eta^\tti$ satisfy
the fluctuation-dissipation relations (\ref{eq:crttiFDT}), verifying all
conditions of Definition \ref{def:regular}.

%% file: appendix_LSI.tex
\section{Sufficient conditions for a log-Sobolev inequality}\label{appendix:LSI}

We prove Proposition \ref{prop:LSI} on a log-Sobolev inequality for the 
posterior law.

\begin{lemma}\label{lemma:univariateLSI}
Under Assumption \ref{assump:prior}(a), the prior density $g(\cdot)$ satisfies
the LSI (\ref{eq:LSIprior}). Furthermore, consider the law
\begin{equation}\label{eq:univariatetilt}
\sP(\theta)=\frac{g(\theta)e^{{-}\frac{a}{2}\theta^2+b\theta}}{Z},
\qquad Z=\int g(\theta)e^{{-}\frac{a}{2}\theta^2+b\theta}\d\theta.
\end{equation}
For any $a>0$ and $b \in \R$, this law $\sP(\theta)$ also satisfies the LSI
(\ref{eq:LSIprior}). Both statements hold with the
constant $C_\LSI=(4/c_0)\exp(8r_0^2(c_0+C)^2/(\pi c_0))$ where $C,c_0,r_0$ are
the constants of Assumption \ref{assump:prior}(a).
\end{lemma}
\begin{proof}
Applying $x=\min(x,{-c_0})+\max(x+c_0,0)$, define
\begin{align*}
\ell_-(\theta)
&=\log g(0)+(\log g)'(0) \cdot \theta+\int_0^\theta \int_0^x
\min\Big((\log g)''(u),-c_0\Big)\d u\,\d x\\
\ell_+(\theta)&=\int_0^\theta \int_0^x \max\Big((\log g)''(u)+c_0,0\Big)
\d u\,\d x
\end{align*}
so that $\log g(\theta)=\ell_-(\theta)+\ell_+(\theta)$. Then set
\begin{align*}
\tilde \ell_-(\theta)&={-}\log Z-a\theta^2+b\theta+\ell_-(\theta)
\end{align*}
so that $\log \sP(\theta)=\tilde \ell_-(\theta)+\ell_+(\theta)$.
By definition we have
$\ell_-''(\theta)=\min((\log g)''(\theta),-c_0) \leq -c_0$ and also
$\tilde \ell_-''(\theta)=-a+\ell_-''(\theta) \leq -c_0$.
We have $(\log g)''(u)+c_0 \leq 0$
for all $|u|>r_0$ and $(\log g)''(u)+c_0 \leq c_0+C$ for all $|u| \leq r_0$.
Hence $|\ell_+'(\theta)| \leq r_0(c_0+C)$. Thus both $\log g(\theta)$ and
$\log \sP(\theta)$ are sums of a $c_0$-strongly-log-concave
potential $\ell_-(\theta)$ or $\tilde \ell_-(\theta)$ and a
$r_0(c_0+C)$-Lipschitz perturbation $\ell_+(\theta)$. Then
\cite[Lemma 2.1]{bardet2018functional} shows that both laws satisfy a LSI
with constant $C_\LSI=(4/c_0)\exp(8r_0^2(c_0+C)^2/(\pi c_0))$.
\end{proof}

\begin{proof}[Proof of Proposition \ref{prop:LSI}]
Under condition (a), the posterior density is strongly
log-concave, satisfying
\[{-}\nabla^2 \log \sP_g(\btheta \mid \X,\y)
=\frac{1}{\sigma^2}\X^\top \X-\diag\Big((\log g)''(\theta_j)\Big)_{j=1}^d
\succeq c_0\I.\]
Hence (\ref{eq:LSI}) with $C_{\LSI}=2/c_0$ follows from the Bakry-Emery
criterion. Clearly this holds for any noise variance $\sigma^2>0$,
verifying Assumption \ref{assump:LSI}.

The proof under condition (b) is an adaptation of the argument of
\cite{bauerschmidt2019very}; see also \cite[Theorem 3.4]{fan2023gradient}
and \cite{montanari2024provably} for similar specializations to
the linear model. Under the conditions for $\X$ of
Assumption \ref{assump:model}, by the Bai-Yin law
(\cite[Theorem 3.1]{yin1988limit}), for any $\eps>0$ the event
\[\event(\X)=\Big\{(\sqrt{\delta}-1)_+^2-\eps \leq \lambda_{\min}(\X^\top \X) 
\leq \lambda_{\max}(\X^\top \X) \leq (\sqrt{\delta}+1)^2+\eps\Big\}\]
hold a.s.\ for all large $n,d$ (where $\delta=\lim n/d$). Thus, choosing some
sufficiently small $\eps>0$ and setting
\[\kappa=(\sqrt{\delta}-1)_+^2-2\eps,
\quad
\tau^2=\sigma^2\Big([(\sqrt{\delta}+1)^2-(\sqrt{\delta}-1)_+^2+3\eps]^{-1}-\eps
\Big),
\quad \bSigma=\sigma^2(\X^\top\X-\kappa\,\I)^{-1}-\tau^2\,\I,\]
we have $\X^\top\X-\kappa\,\I \succeq \eps\,\I$,
$(\X^\top\X-\kappa\,\I)^{-1} \succeq (\frac{\tau^2}{\sigma^2}+\eps)\,\I$, and
hence $\bSigma \succeq \eps\sigma^2$ on $\event(\X)$.
Since $\sigma^2(\X^\top\X-\kappa\,\I)^{-1}=\bSigma+\tau^2\I$,
we have the Gaussian convolution identity
\[e^{-\frac{1}{2\sigma^2}\btheta^\top (\X^\top \X-\kappa\I)\btheta}
\propto \int e^{-\frac{1}{2\tau^2}\|\btheta-\bvarphi\|_2^2}
e^{-\frac{1}{2}\bvarphi^\top \bSigma^{-1} \bvarphi}\d\bvarphi.\]
Then the posterior density $\sP_g(\btheta \mid \X,\y)$
satisfies
\begin{align*}
\sP_g(\btheta \mid \X,\y) &\propto \exp\Big({-}\frac{1}{2\sigma^2}\|\y-\X\btheta\|^2\Big)
\prod_{j=1}^d g(\theta_j)\\
&\propto
\exp\Big({-}\frac{1}{2\sigma^2}\btheta^\top(\X^\top\X-\kappa\I)\btheta\Big)
\prod_{j=1}^d
g(\theta_j)\exp\Big({-}\frac{\kappa}{2\sigma^2}\theta_j^2+\frac{\x_j^\top\y}{\sigma^2}\theta_j\Big)\\
&\propto \int e^{-\frac{1}{2}\bvarphi^\top \bSigma^{-1}\bvarphi}
\prod_{j=1}^d g(\theta_j)\exp\Big({-}\frac{\kappa}{2\sigma^2}\theta_j^2
-\frac{1}{2\tau^2}(\theta_j-\varphi_j)^2+\frac{\x_j^\top
\y}{\sigma^2}\theta_j\Big)\d\bvarphi.
\end{align*}
Defining
\begin{align*}
q_{\varphi_j}(\theta_j)&=\frac{1}{Z_j(\varphi_j)}
g(\theta_j)\exp\Big({-}\frac{\kappa}{2\sigma^2}\theta_j^2
-\frac{1}{2\tau^2}(\theta_j-\varphi_j)^2+\frac{\x_j^\top
\y}{\sigma^2}\theta_j\Big),\\
Z_j(\varphi_j)&=\int g(\theta_j)\exp\Big({-}\frac{\kappa}{2\sigma^2}\theta_j^2
-\frac{1}{2\tau^2}(\theta_j-\varphi_j)^2+\frac{\x_j^\top
\y}{\sigma^2}\theta_j\Big)\d\theta_j,\\
\mu(\bvarphi)&=\frac{
e^{-\frac{1}{2}\bvarphi^\top\bSigma^{-1}\bvarphi}\prod_{j=1}^d Z_j(\varphi_j)}
{\int e^{-\frac{1}{2}{\bvarphi'}^\top\bSigma^{-1}\bvarphi'}
\prod_{j=1}^d Z_j(\varphi_j')\d\bvarphi'}
\end{align*}
this gives the mixture-of-products representation
\[\sP_g(\btheta \mid \X,\y)=\int \underbrace{\prod_{j=1}^d
q_{\varphi_j}(\theta_j)}_{:=q_{\bvarphi}(\btheta)}
\mu(\bvarphi) \d\bvarphi.\]
Then for any $f \in C^1(\R^d)$,
\begin{equation}\label{eq:entdecomp}
\Ent[f(\btheta)^2 \mid \X,\y]=\E_{\bvarphi \sim \mu} \Ent_{\btheta \sim
q_{\bvarphi}}f(\btheta)^2+\Ent_{\bvarphi \sim \mu} \E_{\btheta \sim
q_{\bvarphi}}f(\btheta)^2.
\end{equation}

For the first term of (\ref{eq:entdecomp}),
note that inside the exponential defining $q_{\varphi_j}(\theta_j)$ we have
$\kappa \geq -2\eps$ and $\tau^2 \leq \sigma^2((1+3\eps)^{-1}-\eps)$, so
the coefficient of $\theta_j^2$ is negative for sufficiently
small $\eps>0$. Then
by Lemma \ref{lemma:univariateLSI}, $q_{\varphi_j}(\theta_j)$ satisfies
the univariate LSI (\ref{eq:LSIprior}) with constant
$C_\LSI:=(4/c_0)\exp(8r_0^2(c_0+C)^2/(\pi c_0))$. So the product law
$q_{\bvarphi}$ satisfies the LSI with the same constant by tensorization, and
\begin{equation}\label{eq:entdecompbound1}
\E_{\bvarphi \sim \mu} \Ent_{\btheta \sim q_{\bvarphi}}f(\btheta)^2
\leq C_\LSI\,\E_{\bvarphi \sim \mu} \E_{\btheta \sim q_{\bvarphi}}
\|\nabla f(\btheta)\|_2^2=C_\LSI\,\E_{\btheta \sim q} \|\nabla f(\btheta)\|_2^2.
\end{equation}

For the second term of (\ref{eq:entdecomp}), note that
\begin{align*}
{-}\nabla_{\bvarphi}^2 \log \mu(\bvarphi)
=\bSigma^{-1}-\diag\Big((\log Z_j)''(\varphi_j)\Big)_{j=1}^d
=\bSigma^{-1}
+\diag\Big(\frac{1}{\tau^2}-\frac{1}{\tau^4}\Var_{\theta_j \sim
q_{\varphi_j}}[\theta_j]\Big)_{j=1}^d.
\end{align*}
The LSI for $q_{\varphi_j}$ implies
$\Var_{\theta_j \sim q_{\varphi_j}}[\theta_j] \leq (C_\LSI/2)$ by its implied
Poincar\'e inequality. Applying
$(\sqrt\delta+1)^2-(\sqrt\delta-1)_+^2=4\sqrt\delta\1\{\delta>1\}
+(\sqrt\delta+1)^2\1\{\delta \leq 1\}$
and the given condition (b) for $\sigma^2$,
we see that for a sufficiently small choice of $\eps>0$,
we have $\tau^2 \geq C_\LSI/[2(1-\eps)]$. Then this gives
${-}\nabla_{\bvarphi}^2 \log \mu(\bvarphi) \succeq (\eps/\tau^2)\I$.
Then by the Bakry-Emery criterion, $\mu$ satisfies the LSI $\Ent_{\bvarphi \sim
\mu} f(\bvarphi)^2 \leq (2\tau^2/\eps)\E_{\bvarphi \sim \mu}\|\nabla
f(\bvarphi)\|_2^2$, hence
\[\Ent_{\bvarphi \sim \mu} \E_{\btheta \sim q_{\bvarphi}} f(\btheta)^2 \leq \frac{2\tau^2}{\eps}\,
\E_{\bvarphi \sim \mu} \|\nabla_{\bvarphi} (\E_{\btheta \sim q_{\bvarphi}} f(\btheta)^2)^{1/2}\|_2^2.\]
Denote by $q_{\bvarphi^{-j}}$ the product of components of $q_{\bvarphi}$
other than the $j^\text{th}$. We compute
\[\partial_{\varphi_j} (\E_{\btheta \sim q_{\bvarphi}} f(\btheta)^2)^{1/2}
=\frac{\partial_{\varphi_j}\E_{\btheta \sim q_{\bvarphi}} f(\btheta)^2}
{2(\E_{\btheta \sim q_{\bvarphi}} f(\btheta)^2)^{1/2}}
=\frac{\E_{\btheta^{-j} \sim q_{\bvarphi^{-j}}} \partial_{\varphi_j}\E_{\theta_j
\sim q_{\varphi_j}}f(\btheta)^2}{2(\E_{\btheta \sim q_{\bvarphi}} f(\btheta)^2)^{1/2}}
=\frac{\E_{\btheta^{-j} \sim q_{\bvarphi^{-j}}}
\Cov_{\theta_j \sim q_{\varphi_j}}[f(\btheta)^2,\theta_j]}
{2\tau^2(\E_{\btheta \sim q_{\bvarphi}} f(\btheta)^2)^{1/2}}.\]
We apply \cite[Proposition 2.2]{ledoux2001logarithmic}:
For any law $\nu$ on $\R^d$ satisfying a LSI $\Ent_\nu f^2 \leq C_\LSI\,\E_\nu
\|\nabla f\|_2^2$, and for any smooth functions $f,g:\R^d \to \R$,
\[\Cov_\nu[f^2,g] \leq C \sup_{\btheta \in \R^d} \|\nabla g(\btheta)\|_2
\cdot (\E_\nu f^2)^{1/2}(\E_\nu \|\nabla f\|_2^2)^{1/2}\]
where $C$ depends only on the LSI constant $C_\LSI$ of $\nu$.
Applying this to the univariate law $\nu=q_{\varphi_j}$, followed by
Cauchy-Schwarz,
\[\E_{\btheta^{-j} \sim q_{\bvarphi^{-j}}} \Cov_{\theta_j \sim
q_{\varphi_j}}[f(\btheta)^2,\theta_j]
\leq C_1(\E_{\btheta \sim q_{\bvarphi}}f(\btheta)^2)^{1/2}(\E_{\btheta \sim
q_{\bvarphi}}(\partial_{\theta_j} f(\btheta))^2)^{1/2}\]
where $C_1$ depends only on the LSI constant $C_\LSI$ for $q_{\varphi_j}$. 
Summing over $j=1,\ldots,d$ gives
\begin{align*}
\|\nabla_{\bvarphi} (\E_{\btheta \sim q_{\bvarphi}}f(\btheta)^2)^{1/2}\|_2^2
&=\sum_{j=1}^d\big[\partial_{\varphi_j} (\E_{\btheta \sim q_{\bvarphi}}
f(\btheta)^2))^{1/2}\big]^2 \\
&\leq \sum_{j=1}^d \Big(\frac{C_1}{2\tau^2}\Big)^2
\E_{\btheta \sim q_{\bvarphi}}(\partial_{\theta_j} f(\btheta))^2
=\Big(\frac{C_1}{2\tau^2}\Big)^2 \E_{\btheta \sim
q_{\bvarphi}} \|\nabla f(\btheta)\|_2^2.
\end{align*}
Thus
\begin{equation}\label{eq:entdecompbound2}
\Ent_{\bvarphi \sim \mu} \E_{\btheta \sim q_{\bvarphi}} f(\btheta)^2
\leq \frac{2\tau^2}{\eps}
\Big(\frac{C_1}{2\tau^2}\Big)^2
\E_{\btheta \sim \mu} \E_{\btheta \sim q_{\bvarphi}} \|\nabla f(\btheta)\|_2^2
=\frac{C_1^2}{2\eps\tau^2} \E[\|\nabla f(\btheta)\|_2^2 \mid \X,\y].
\end{equation}
Applying (\ref{eq:entdecompbound1}) and (\ref{eq:entdecompbound2}) to
(\ref{eq:entdecomp}) completes the proof of (\ref{eq:LSI}), on the above event
$\event(\X)$. Since the given
condition (b) also holds for all $\tilde \sigma^2 \geq \sigma^2$ when it holds
for $\sigma^2$, this verifies Assumption \ref{assump:LSI}.

Finally, under condition (c), we note that on the above event $\event(\X)$ we
have
\[{-}\nabla^2 \log \sP_g(\btheta \mid \X,\y)
=\frac{1}{\sigma^2}\X^\top \X-\diag\Big((\log g)''(\theta_j)\Big)_{j=1}^d
\succeq \eps \I\]
for all $\sigma^2 \leq [(\sqrt{\delta}-1)^2-\eps]/(C+\eps)$,
so (\ref{eq:LSI}) holds by the Bakry-Emery criterion. Choosing $\eps>0$
small enough, under condition (c) we have
$[(\sqrt{\delta}-1)^2-\eps]/(C+\eps)>4C_0\sqrt{\delta}$, so that
(\ref{eq:LSI}) holds with $C_\LSI=2/\eps$ for
$\sigma^2>[(\sqrt{\delta}-1)^2-\eps]/(C+\eps)$ by the analysis of condition (b).
Thus, on $\event(\X)$, (\ref{eq:LSI}) holds with a uniform constant $C_\LSI>0$
for all $\sigma^2>0$, again verifying Assumption \ref{assump:LSI}.
\end{proof}

%% file: auxiliary.tex
\section{Auxiliary lemmas}

\begin{lemma}[Coupling of Gaussian processes]\label{lem:couple_general_cov}
Let $\{K(t,s)\}_{t,s\in[0,T]}$ and $\{\tilde{K}(t,s)\}_{t,s\in[0,T]}$ be two
positive semidefinite covariance kernels such that for some $\eps>0$ and $C_0>0$
\begin{align}\label{eq:cov_point_diff}
\sup_{t,s\in[0,T]} |K(t,s) - \tilde{K}(t,s)| \leq \eps,
\end{align}
and
\begin{align}\label{eq:GP_regularity}
\sup_{t,s\in[0,T]} K(t,t) + K(s,s) - 2K(t,s) \leq C_0|t-s|.
\end{align}
Then there exists a coupling of the two mean-zero
Gaussian processes $\{u^t\}_{t\in[0,T]}$ and
$\{\tilde{u}^t\}_{t\in[0,T]}$ with covariances $\E[u^tu^s]=K(t,s)$ and
$\E[\tilde u^t\tilde u^s]=\tilde K(t,s)$ such that
\begin{align*}
\sup_{t\in[0,T]} \E[(u^t - \tilde{u}^t)^2]  \leq (6C_0+3)\sqrt{T\eps}+15\eps.
\end{align*}
\end{lemma}
\begin{proof}
Fix $\gamma>0$, and let $\floor{t}=\max\{i\gamma: i\in \Z_+, i\gamma
\leq t\}$ where $\Z_+=\{0,1,2,\ldots\}$. Let $v^t = u^{\floor{t}}$
and $\tilde v^t=\tilde u^{\floor{t}}$ so that
$\E[v^tv^s] = K(\floor{t},\floor{s})$ and
$\E[\tilde v^t\tilde v^s] = \tilde K(\floor{t},\floor{s})$.
Then by (\ref{eq:GP_regularity}) and (\ref{eq:cov_point_diff}),
\begin{align*}
\sup_{t\in[0,T]} \E(u^t - v^t)^2 \leq C_0\gamma,
\qquad \sup_{t\in[0,T]} \E(\tilde u^t - \tilde v^t)^2 \leq C_0\gamma+4\eps.
\end{align*}
Let $X=(v^0,v^{\gamma},v^{2\gamma},\ldots,v^{\floor{T}}) \in \R^N$, where here
$N \leq T/\gamma+1$, and similarly let
$\tilde X = (\tilde v^0,\tilde v^{\gamma},\tilde v^{2\gamma},\ldots,\tilde
v^{\floor{T}}) \in \R^N$. Let $\Sigma,\tilde \Sigma \in \R^{N\times N}$ be the
covariance matrices of $X,\tilde X$, so
$\Sigma_{ij} = K(i\gamma, j\gamma)$ and
$\tilde\Sigma_{ij} = \tilde{K}(i\gamma, j\gamma)$. Coupling $X$ and $\tilde X$
by $X=\Sigma^{1/2}Z$ and $\tilde X=\tilde \Sigma^{1/2}Z$ where $Z \sim \N(0,I)$,
for each $i=1,\ldots,N$,
\[\E(X_i-\tilde X_i)^2=\e_i^\top(\Sigma^{1/2}-\tilde\Sigma^{1/2})^2\e_i
\leq \|\Sigma^{1/2}-\tilde\Sigma^{1/2}\|_\op^2
\stackrel{(*)}{\leq} \|\Sigma - \tilde\Sigma\|_\op \leq N\eps.\]
Here $(*)$ follows from \cite[Theorem X.1.3]{bhatia2013matrix},
and the last inequality applies (\ref{eq:cov_point_diff}). Then we have
\begin{align*}
\sup_{t\in [0,T]} \E (u^t - \tilde{u}^t)^2 &\leq 3 \Big[\sup_{t\in [0,T]} \E
(u^t - v^t)^2 + \sup_{t\in [0,T]} \E (v^t - \tilde{v}^t)^2 + \sup_{t\in [0,T]}
\E (\tilde{u}^t - \tilde{v}^t)^2\Big]\\
&\leq 6C_0 \gamma + 12\eps+3\max_{i=1}^N \E(X_i-\tilde X_i)^2
\leq 6C_0 \gamma + 12\eps+3(T/\gamma+1)\eps.
\end{align*}
The conclusion follows by choosing $\gamma=\sqrt{T\eps}$.
\end{proof}
